\documentclass[11pt,leqno]{article}
\usepackage{amsmath,amssymb,mathrsfs,amsthm,bm,ddag}
\usepackage{eucal}

\usepackage{tikz}
\usetikzlibrary{intersections, calc, arrows}
\usetikzlibrary{decorations.markings}

\usepackage{xy}
\input{xy}
\xyoption{all}

\usepackage[hypertex]{hyperref}

\usepackage{threeparttable}

\usepackage{xspace}

\usepackage[geometry]{ifsym}

\usepackage{comment}

\makeatletter
\def\@seccntformat#1{\csname the#1\endcsname.\hspace{2ex}}

 \renewcommand{\subsection}%
  {\@startsection{subsection}%
  {2}%
  {\z@}%
  {2ex}
  {0ex}
  {\reset@font\normalsize\bfseries}}%

 \@addtoreset{equation}{subsection}

 \newcommand{\nsection}{\@startsection{section}{1}{\z@}%
     {-5ex}
     {1ex}
     {\reset@font\center\large\sc}}

 \renewenvironment{thebibliography}[1]
 {\nsection*{\refname\@mkboth{\refname}{\refname}}%
   \list{\@biblabel{\@arabic\c@enumiv}}%
   {\settowidth
   \labelwidth{\@biblabel{#1}}%
   \leftmargin
	\labelwidth
        \advance
	 \leftmargin
	 \labelsep
         \@openbib@code
         \usecounter{enumiv}%
         \let\p@enumiv\@empty
	 \parskip=0pt
	 \itemsep=1pt
	 \parsep=1pt
	 \itemindent=\z@
         \renewcommand\theenumiv{\@arabic\c@enumiv}}%
   	 \sloppy
   	 \clubpenalty4000
   	 \@clubpenalty\clubpenalty
   	 \widowpenalty4000%
   	 \footnotesize
   	 \sfcode`\.\@m}
  	 {\def\@noitemerr
    	 {\@latex@warning{Empty `thebibliography' environment}}%
   	 \endlist}

\makeatother

\newtheoremstyle{thm}
 {1em}
 {3pt}
 {\itshape}
 {}
 {\bf}
 {. ---}
 {0.5em}
 {}
\newtheoremstyle{dfn}
 {1em}
 {3pt}
 {}
 {}
 {\bf}
 {. {---}}
 {0.5em}
 {}

\swapnumbers
\theoremstyle{thm}
\newtheorem{thm}[subsection]{Theorem}
\newtheorem{lem}[subsection]{Lemma}
\newtheorem*{lem*}{Lemma}
\newtheorem{cor}[subsection]{Corollary}
\newtheorem*{cor*}{Corollary}
\newtheorem{prop}[subsection]{Proposition}
\newtheorem*{prop*}{Proposition}

\newtheorem*{conj*}{Conjecture}
\newtheorem*{thm*}{Theorem}

\theoremstyle{dfn}
\newtheorem{dfn}[subsection]{Definition}
\newtheorem*{dfn*}{Definition}
\newtheorem{ex}[subsection]{Example}
\newtheorem*{ex*}{Example}

\newtheorem*{rem*}{Remark}

\newcommand{\shom}{\mathop{\mc{H}om}\nolimits}

\usepackage{color}
\newenvironment{meta}{
\noindent \color{red}
\sffamily[}{\upshape]}

\usepackage{framed}

\specialcomment{comme}
{\begingroup\ttfamily\color{blue}}{\endgroup}
\excludecomment{comme}

\newsavebox{\circlebox}
\savebox{\circlebox}{\fontencoding{OMS}\selectfont\char13}
\newlength{\circleboxwdht}
\newcommand{\ccirc}[1]{
  \setlength{\circleboxwdht}{\wd\circlebox}
  \addtolength{\circleboxwdht}{\dp\circlebox}
  \raisebox{0.2\dp\circlebox}{
    \parbox[][\circleboxwdht][c]{\wd\circlebox}
    {\centering\scriptsize #1}}
  \llap{\usebox{\circlebox}}
}

\newcommand{\SH}{\mc{SH}}

\renewcommand{\H}{\mc{H}}
\newcommand{\Cat}{\mc{C}\mr{at}}

\newcommand{\Shv}{\mc{S}\mr{hv}}
\newcommand{\PSh}{\mc{PS}\mr{hv}}

\newcommand{\Spc}{\mc{S}\mr{pc}}
\newcommand{\sSet}{\mc{S}\mr{et}_{\Delta}}
\newcommand{\Cart}{\mc{C}\mr{art}}

\newcommand{\coCart}{\mr{co}\mc{C}\mr{art}}
\newcommand{\Str}{\mc{S}\mr{tr}}
\newcommand{\bp}{{}^{\backprime}}
\newcommand{\Prcat}{\mc{P}\mr{r}}
\newcommand{\PrL}{\mc{P}\mr{r}^{\mr{L}}}
\newcommand{\PrLsym}{\mc{P}\mr{r}^{\mr{L},\otimes}}
\newcommand{\LinCat}{\mc{L}\mr{in}\mc{C}\mr{at}}
\newcommand{\RMod}{\mr{RMod}}
\newcommand{\basech}{*}
\newcommand{\twoLinCat}{\mbf{LinCat}}

\newcommand{\Tri}{\mc{T}\mr{ri}}
\newcommand{\Fin}{\mc{F}\mr{in}_*}
\newcommand{\Mor}{\mr{Mor}}

\DeclareMathSymbol{\mlq}{\mathord}{operators}{``}
\DeclareMathSymbol{\mrq}{\mathord}{operators}{`'}

\newcounter{adjprop}

\begin{document}
\title{Enhanced bivariant homology theory attached to six functor
formalism}
\author{Tomoyuki Abe}
\maketitle

\begin{abstract}
 Bivariant theory is a unified framework for cohomology and Borel-Moore
 homology theories. In this paper, we extract an $\infty$-enhanced
 bivariant homology theory from Gaitsgory-Rozenblyum's six functor
 formalism.
\end{abstract}

\section*{Introduction}
Grothendieck's 6-functor formalism is very powerful in cohomology
theory. At the same time, if we want to axiomatize the formalism, it
requires a long list of relations between these functors, and when we
wish to establish a 6-functor formalism for some cohomology theory, we
need tremendous amount of work to verify these axioms.
After ideas of Lurie, Gaitsgory and Rozenblyum constructed a very
general machinery to construct an $\infty$-enhanced 6-functor formalisms
from minimal amount of data.
To proceed, let us recall the category of correspondences used
by Gaitsgory and Rozenblyum. We fix a base scheme $S$.
The $(\infty,2)$-category $\mbf{Corr}^{\mr{prop}}_{\mr{sep};\mr{all}}$
($\mbf{Corr}$ for short) has the objects the same as the category of
$S$-schemes $\mr{Sch}(S)$ (or its subcategory).
A morphism $F\colon X\rightarrow Y$ is a diagram of the form
\begin{equation*}
 \xymatrix{Z_F\ar[d]_g\ar[r]^-{f}&X\\
  Y,&}
\end{equation*}
where $g$ is separated. Given $1$-morphisms $F,G\colon X\rightarrow Y$,
a $2$-morphism $F\Rightarrow G$ is a diagram
\begin{equation*}
 \xymatrix{
  Z_F\ar[rd]^-{\alpha}\ar@/_10pt/[rdd]\ar@/^10pt/[rrd]&&\\
 &Z_G\ar[r]\ar[d]&X\\
 &Y,&}
\end{equation*}
where $\alpha$ is proper. Let $\mbf{Pres}$ be the $(\infty,2)$-category
of presentable stable $\infty$-categories with colimit commuting maps as
morphisms and natural transformations as $2$-morphisms. Gaitsgory and
Rozenblyum interpret a 6-functor formalism as a $2$-functor
$\mbf{D}\colon\mbf{Corr}\rightarrow\mbf{Pres}$.
In fact, for a correspondence $F$ as above, we have
$\mbf{D}(F)\colon\mbf{D}(X)\rightarrow\mbf{D}(Y)$. This encodes the data
of the functor $g_!f^*$.
The functoriality of $g_!$ with respect to proper morphism is
encoded in $2$-morphisms. Since $\mbf{D}(F)$ is a map in $\mbf{Pres}$,
it admits a right adjoint, which encodes the data of $g^!$, $f_*$.
If we need $\otimes$, $\shom$, we need to consider the $\infty$-category
of commutative algebra objects in the $\infty$-category of presentable
$\infty$-categories, but we do not go into that far in this
introduction.

Even though the data of 6-functor is
encoded very beautifully, it is not straightforward to extract concrete
data. For example, if we wish to extract an {\em $\infty$-functor} of
cohomology theory
\begin{equation*}
 \mc{H}^*\colon\mr{Sch}(S)^{\mr{op}}\rightarrow\mr{Sp},
\end{equation*}
where $\mr{Sp}$ is the $\infty$-category of spectra,
this already does not seem to be straightforward from the definition.
A goal of this paper is to ``decode'' the data from
Gaitsgory-Rozenblyum's 6-functor formalism so that we can handle it more
easily.

Let us go into more precise statement.
In general, when we are given a 6-functor formalism, we can attach 4
kinds of (co)homology theories: cohomology, Borel-Moore homology,
homology, compact support cohomology. First two and the last two
theories possess essentially the same information via duality theory.
Thus, we may focus on the first two theories. Given $f\colon
X\rightarrow S$ in $\mr{Sch}(S)$, cohomology and Borel-Moore homology of
$X$ can be defined by
\begin{equation*}
 \mr{H}^*(X):=\mr{Map}_{D(S)}(\mbf{1}_S,f_*f^*\mbf{1}_S),
  \qquad
  \mr{H}^{\mr{BM}}_*(X):=\mr{Map}_{D(S)}(\mbf{1}_S,f_*f^!\mbf{1}_S).
\end{equation*}
The cohomology theory is contravariant with respect to any morphism, and
Borel-Moore homology is covariant with respect to proper morphism.
Seemingly completely different theories,
Fulton and MacPherson \cite{FM} unified
these two theories into the so called {\em bivariant homology theory}.
Let $g\colon X\rightarrow Y$ be a morphism in $\mr{Sch}(S)$.
Then we define
\begin{equation*}
 \mr{H}(g):=\mr{Map}_{D(Y)}(\mbf{1}_Y,g_*g^!\mbf{1}_Y)
  \simeq
  \mr{Map}_{D(Y)}(g_!\mbf{1}_X,\mbf{1}_{Y}).
\end{equation*}
By definition, we have $\mr{H}^*(X)\simeq\mr{H}(\mr{id}_X)$,
$\mr{H}^{\mr{BM}}_*(X)\simeq\mr{H}(X\rightarrow S)$.
The main result of this paper gives an $\infty$-enhancement of the
bivariant homology.
In order to make this precise, we consider the category of
arrows $\widetilde{\mr{Ar}}$. Namely, the objects consist of
$S$-morphisms $X\rightarrow Y$. For morphisms, we do {\em not} use the
evident one: a morphism from $f'\colon X'\rightarrow Y'$ to $f\colon
X\rightarrow Y$ consists of a diagram of the following form:
\begin{equation*}
 \xymatrix{
  X'\ar[d]&X\times_{Y}Y'\ar[d]\ar@{}[rd]|\square\ar[r]\ar[l]_-{\alpha}
  &X\ar[d]\\
 Y'\ar@{=}[r]&Y'\ar[r]&Y
  }
\end{equation*}
where $\alpha$ is proper. We may check by hand that we have a morphism
$\mr{H}(f)\rightarrow\mr{H}(f')$. It is even not too hard to check that
bivariant homology theory is a (ordinary) functor
$\mr{H}\colon\widetilde{\mr{Ar}}^{\mr{op}}\rightarrow\mr{h}\mr{Sp}$.
Our main result gives an $\infty$-enhancement of this functor.
A simplified version can be written as follows
(cf.\ Theorem \ref{mainthmcons} for the detail):

\begin{thm*}
 Given a 6-functor formalism $\mbf{Corr}\rightarrow\mbf{Pres}$,
 there exists an $\infty$-functor
 \begin{equation*}
  \mc{H}\colon\widetilde{\mr{Ar}}^{\mr{op}}\rightarrow\mr{Sp}
 \end{equation*}
 such that $\mr{h}\mc{H}\simeq\mr{H}$ as a functor
 $\widetilde{\mr{Ar}}\rightarrow\mr{h}\mr{Sp}$.
\end{thm*}
One of the obstacles of constructing such a functor is that the
functoriality of $\mr{H}^*(X)$ comes from $1$-morphism of $\mbf{Corr}$,
whereas that of $\mr{H}^{\mr{BM}}_*(X)$ comes from $2$-morphism of
$\mbf{Corr}$. In order to combine these two morphisms into one functor
as in the theorem, we need to ``integrate'' these two types of morphisms.

Our main motivation of the theorem is to construct such a functor for
theory of motives.
We plan to use the functor above to construct certain
elements in Chow groups which appear in ramification theory of
$\ell$-adic sheaves.
Since we need ``gluing'' of elements in Chow groups,
$\infty$-enhancement is crucial.

Before concluding the introduction, let us see the organization of this
paper. Throughout this paper, we use the language of $\infty$-categories
freely. In \S\ref{Notrecol}, we collect some preliminaries on
$\infty$-categories. Most of the material in this section should be more
or less well-known to experts, but we write here since we could not find
references. In \S\ref{dualconst}, we establish some duality type
theorem. Via straightening/unstraightening construction of Lurie,
Cartesian and coCartesian fibrations correspond to each other, and
contain essentially the same information, as long as we are considering
morphisms which preserve (co)Cartesian edges.
However, it is fairly inexplicit in nature if we pass through
straightening/unstraightening construction. We construct an explicit
model for such correspondence. This construction naturally appears in
\S\ref{constfun}. In \S\ref{stablecat}, we define the $(\infty,2)$-category
of stable $R$-linear $\infty$-categories. We heavily use the language of
(generalized) $\infty$-operads. The construction has already
appeared in \cite{GR}. The main construction is carried out in
\S\ref{constfun} and \S\ref{funtobitheo}. In \S\ref{constfun} we define
a lax functor
$\mbf{Corr}\dashrightarrow\mbf{B}\mr{Sp}^{\circledast}$, where
$\mbf{B}\mr{Sp}^{\circledast}$ is the $(\infty,2)$-category with single
object and morphisms corresponding to objects of $\mr{Sp}$. The
composition is defined by the monoidal structure of $\mr{Sp}$.
This functor sends the $1$-morphism $f\colon X\rightarrow Y$ to
$\mr{H}(f)$, and encodes the complete data of bivariant homology
theory. However, to go from this $(\infty,2)$-functor to the functor we
are looking for, we need one step more, which is carried out in
\S\ref{funtobitheo}.
Finally, in \S\ref{secExam}, we collect some examples of 6-functor
formalisms in the sense of Gaitsgory-Rozenblyum. Most of the part of
this section has already appeared elsewhere, but some of the
sources are not published and not even available in arXiv,
we decided to include this for the sake of completeness.

\subsection*{Conventions and notations}\mbox{}\\
\indent
When we say $\infty$-categories, it always mean quasi-categories,
in particular, $(\infty,1)$-categories.
We do {\em not} abbreviate $\infty$-category as category.
In principle, we follow the conventions of Lurie in \cite{HTT},
\cite{HA}. Exceptions are that we call $\infty$-operad what Lurie calls
planar $\infty$-operads, and that we denote by $\Spc$ the
$\infty$-category of spaces.

As in \cite{HTT}, we denote $\mr{N}((\sSet^+)^{\circ})$ by $\Cat_\infty$.
By definition, an object of $\Cat_\infty$ is a marked simplicial set $\mc{C}^{\natural}$ where $\mc{C}$ is an $\infty$-category.
Unless any confusion may arise, the object $\mc{C}^{\natural}$ is denoted simply by $\mc{C}$.
We denote by $\widehat{\Cat}_\infty$ the $\infty$-category of possibly large $\infty$-categories.

We denote by $\mbf{\Delta}$ the simplex category, whose objects will be
denoted by $[n]$ for $n\in\mb{N}$ as usual. A morphism
$[n]\rightarrow[m]$ corresponds to a function.
We denote by $\sigma^i\colon[0]\rightarrow[n]$ the map sending $0$
to $i\in[n]$. We denote by $\rho^i\colon[1]\rightarrow[n]$ for $0<i\leq
n$ the map sending $0$ to $i-1$ and $1$ to $i$. Both of these are inert
maps. We also denote by $d^i\colon[n-1]\rightarrow[n]$ increasing map
which avoids $i\in[n]$.

In principle, we use calligraphic fonts ({\it e.g.}\ $\mc{C}$) for
$\infty$-categories, and bold fonts ({\it e.g.}\ $\mbf{C}$) for
$(\infty,2)$-categories. For a map of simplicial sets $X\rightarrow S$
and a vertex $s\in S$, we denote by $X_s$ the fiber product
$X\times_{S,s}\Delta^0$. An equivalence of ($\infty$-)categories is
denoted by $\simeq$, and an isomorphism of simplicial sets is denoted by
$\cong$. For an $\infty$-category $\mc{C}$, the space of morphisms is
denoted by $\mr{Map}_{\mc{C}}(-,-)$.

We denote by $(-)\times^{\mr{cat}}_{(-)}(-)$ for a product
{\em in $\Cat_\infty$} in order to clarify the difference between the
fiber product as simplicial sets. If $f\colon\mc{D}\rightarrow\mc{C}$ is
a categorical fibration of $\infty$-categories and
$g\colon\mc{E}\rightarrow\mc{C}$ is a functor of $\infty$-categories,
then the functor
$\mc{D}\times_{\mc{C}}\mc{E}\rightarrow
\mc{D}\times^{\mr{cat}}_{\mc{C}}\mc{E}$ is a categorical equivalence.

\subsection*{Acknowledgment}\mbox{}\\
The author would like to thank Deepam Patel for continuous
discussions and encouragements.
Without him, this work would not have appeared.
He also thanks Andrew MacPherson, Rune Haugseng, Adeel Khan for some
discussions. This work is supported by JSPS KAKENHI Grant Numbers
16H05993, 18H03667, 20H01790.

\section{Some preliminaries on $\infty$-categories}
\label{Notrecol}
We will fix some notations, and recall some constructions in
$\infty$-category theory. The expositions are informal when there are
references.

\subsection{}
\label{defpushpullsim}
Let $f\colon S\rightarrow T$ be a map of simplicial sets.
Then we have the base change functor
$f^*\colon(\sSet)_{/T}\rightarrow(\sSet)_{/S}$.
As in \cite[4.1.2.7]{HTT}, $f^*$ admits a right adjoint $f_*$.
More explicitly, for $X\rightarrow S$, $f_*X\rightarrow T$ is the
simplicial set having the following universal property: for any simplicial set
$K$ over $T$, we have the following isomorphism as simplicial sets:
\begin{equation*}
 \mr{Fun}_T(K,f_*X)\cong\mr{Fun}_S(K\times_T S,X).
\end{equation*}

\subsection{}
\label{constsimplcat}
Assume we are given a functor of $\infty$-categories
$F\colon\mc{C}\rightarrow\Cat_\infty$. Since
$\Cat_\infty\simeq\mr{N}(\Cat_\infty^{\Delta})$ (where
$\Cat_\infty^{\Delta}$ is the simplicial category of $\infty$-categories
in \cite[3.0.0.1]{HTT}), we have the simplicial functor
$\mf{C}[F]\colon\mf{C}[\mc{C}]\rightarrow\Cat_\infty^{\Delta}$
(where $\mf{C}$ is the functor defined in \cite[1.1.5]{HTT}).
Now, we have the simplicial functor
$\mr{MAP}\colon(\Cat_\infty^{\Delta})^{\mr{op}}\times
\Cat_\infty^{\Delta}\rightarrow\Cat_\infty^{\Delta}$ sending
$(\mc{E},\mc{E}')$ to $\mr{Fun}(\mc{E},\mc{E}')$. Fix
$\mc{D}\in\Cat_\infty^{\Delta}$. Then we have the functor
\begin{equation*}
 \mr{Fun}(F,\mc{D})^{\Delta}
  \colon
  \mf{C}[\mc{C}]^{\mr{op}}
  \xrightarrow{\mf{C}[F]^{\mr{op}}\times\{\mc{D}\}}
  (\Cat_\infty^{\Delta})^{\mr{op}}\times\Cat_\infty^{\Delta}
  \xrightarrow{\mr{MAP}}
  \Cat_\infty^{\Delta}.
\end{equation*}
Taking the adjoint, we get the functor
$\mr{Fun}(F,\mc{D})\colon\mc{C}^{\mr{op}}\rightarrow\Cat_\infty$.

The unstraightening of this functor has an alternative description.
Let $f\colon\mc{X}\rightarrow\mc{C}^{\mr{op}}$ be a Cartesian
fibration. By \cite[3.2.2.12]{HTT}, the structural map
$f_*(\mc{D}\times\mc{X})\rightarrow\mc{C}^{\mr{op}}$ is a
coCartesian fibration.
This coCartesian fibration is denoted by $\Phi^{\mr{co}}(f,\mc{D})$.
By (dual version of) \cite[7.3]{GHN},
$\Phi^{\mr{co}}(\mr{Un}_{\mc{C}^{\mr{op}}}(F),\mc{D})$ is equivalent to
the unstraightening of $\mr{Fun}(F,\mc{D})$.
Dually, given a coCartesian fibration $g\colon\mc{Y}\rightarrow\mc{C}$,
we define $\Phi^{\mr{Cart}}(g,\mc{D}):=g_*(\mc{D}\times\mc{Y})$, which
is a Cartesian fibration over $\mc{C}$.

\subsection{}
\label{gammaanddual}
Let $\Gamma$ be the category whose objects are the pairs $([n],i)$ where
$i\in[n]$. A morphism $([n],i)\rightarrow([n'],i')$
consists of a map $\alpha\colon[n']\rightarrow[n]$ such that
$i\leq\alpha(i')$. We have the evident functor
$\gamma\colon\Gamma\rightarrow\mbf{\Delta}^{\mr{op}}$ sending $([n],i)$
to $[n]$.
This is a Cartesian fibration. The fiber over
$[n]\in\mbf{\Delta}^{\mr{op}}$ is $\Delta^n$.
We can check easily that this Cartesian fibration is equivalent to the
unstraightening of the evident functor
$\Delta^{\bullet}\colon\mbf{\Delta}\rightarrow\Cat_\infty$ sending
$[n]\in\mbf{\Delta}$ to $\Delta^n$.

The coCartesian fibration
$\gamma^\vee\colon\Gamma^\vee\rightarrow\mbf{\Delta}$ with the same
straightening can also be defined easily.
It is the category of objects $([n],i)$
where $i\in[n]$ and a map $([n],i)\rightarrow([n'],i')$ is a map
$\alpha\colon[n]\rightarrow[n']$ in $\mbf{\Delta}$ such that
$\alpha(i)\leq i'$. We have the evident functor
$\gamma^\vee\colon\Gamma^\vee\rightarrow\mbf{\Delta}$, which is a
coCartesian fibration.

\subsection{}
\label{maxminKancpx}
We denote by $\Spc$ the $\infty$-category of spaces. We have the functor
$\Spc\rightarrow\Cat_\infty$ by viewing a spaces as an
$\infty$-category. Let us see that this inclusion functor admits both
left and right adjoints.
Let $S$ be a simplicial set. We put the contravariant model structure on
$(\sSet)_{/S}$ and Cartesian model structure on
$(\sSet^+)_{/S}$.
Consider pairs of adjoint functors:
\begin{equation*}
 \xymatrix{
  (\sSet)_{/S}
  \ar@<.5ex>[r]^-{\iota}&
  (\sSet^+)_{/S},
  \ar@<.5ex>[l]^-{\theta}
  }
  \qquad
  \xymatrix{
  (\sSet^+)_{/S}
  \ar@<.5ex>[r]^-{\mu}&
  (\sSet)_{/S}.
  \ar@<.5ex>[l]^-{\iota}
  }
\end{equation*}
Here $\iota(X):=X^{\sharp}$, namely all the edges are marked,
$\mu(X,\mc{E}):=X$, and $\theta(X,\mc{E})$ be the simplicial subset of
$X$ consisting of all the simplices $\sigma$ such that every edge of
$\sigma$ belongs to $\mc{E}$.
We claim that the above pairs are Quillen adjunctions.
The second one is a Quillen adjunction by \cite[3.1.5.1]{HTT}.
For the first one, since the adjointness is easy to check, it suffices
to show that $\iota$ preserves cofibrations and weak equivalences.
Preservation of cofibrations is obvious.
The preservation of weak equivalences is shown in the proof of
\cite[3.1.5.6]{HTT}: for a morphism of simplicial sets
$f\colon X\rightarrow Y$ over $S$,
the induced map $f^\sharp\colon X^\sharp\rightarrow Y^\sharp$ is a
Cartesian equivalence if and only if $f$ is a contravariant equivalence.
We also have $\theta\circ\iota\simeq\mr{id}$,
$\mu\circ\iota\simeq\mr{id}$. Since $\iota$ preserves fibrant objects,
these imply that $\mr{R}\theta\circ\mr{L}\iota\simeq\mr{id}$, and
$\mr{L}\mu\circ\mr{R}\iota\simeq\mr{id}$.
In the special case where $S=\Delta^0$, we have the following result by
\cite[5.2.4.6]{HTT}, which is originally due to Joyal
\cite[6.15, 6.27]{J}:

\begin{lem*}
 \label{propspccatinf}
 The functor $\iota\colon\Spc\rightarrow\Cat_\infty$ admits a right
 adjoint $\theta$ and a left adjoint $\mu$ such that
 $\theta\circ\iota\simeq\mr{id}$, $\mu\circ\iota\simeq\mr{id}$.
 In particular, $\iota$ is fully faithful and commutes with small limits
 and colimits.
\end{lem*}
In the sequel, for a Cartesian or coCartesian fibration
$X\rightarrow S$ we often denote $\theta X$ by $X^{\simeq}_{/S}$.
When $S=\Delta^0$, we omit $/\Delta^0$ and simply write $X^{\simeq}$.

\begin{lem}
 \label{cartedgedetelem}
 Consider the following diagrams
 \begin{equation*}
  \xymatrix{
   X\ar[rr]^-{h}\ar[rd]_{f}&&Y\ar[ld]^{g}\\
  &S,&}\qquad
     \xymatrix{
     X\times_{S,f(e)}\Delta^1\ar[rr]^-{h'}\ar[rd]_{f'}&&
     Y\times_{S,f(e)}\Delta^1\ar[ld]^{g'}\\
  &\Delta^1,&}
 \end{equation*}
 such that $f$, $g$ are coCartesian fibrations, $h$ is an inner
 fibration which preserves coCartesian edges, and $e$ is an edge in
 $X$. Then $e$ is $h$-Cartesian if and only if $e$ is
 $h'$-Cartesian.
\end{lem}
\begin{proof}
 The proof is almost a copy of \cite[5.2.2.3]{HTT}.
 The ``only if'' direction is obvious, so let us show the ``if'' direction.
 Since we need to
 check a certain right lifting property with respect to
 $\Lambda^n_n\rightarrow\Delta^n$, we may assume that
 $S=\Delta^n$, in particular, $S$ is an $\infty$-category.
 Let $e\colon x\rightarrow y$ be an $h'$-Cartesian edge.
 By \cite[2.4.4.3]{HTT}, it suffices to show that the
 following diagram is homotopy pullback diagram for any $z\in X$:
 \begin{equation*}
  \xymatrix{
   \mr{Map}(z,x)\ar[r]\ar[d]&\mr{Map}(z,y)\ar[d]\\
  \mr{Map}(h(z),h(x))\ar[r]&\mr{Map}(h(z),h(y)).
   }
 \end{equation*}
 If there is no map from $f(z)$ to $f(x)$, we have nothing to prove, so
 we may assume that there {\em is} a map, in fact a unique map,
 $\epsilon\colon f(z)\rightarrow f(x)$. Let $\epsilon'\colon
 z\rightarrow z_0$ be a $f$-coCartesian edge lifting $\epsilon$.
 For $w\in X$, consider the following diagrams:
 \begin{equation*}
  \xymatrix{
   \mr{Map}(z_0,w)\ar[r]\ar[d]&\mr{Map}(z,w)\ar[d]\\
  \mr{Map}(f(z_0),f(w))\ar[r]&\mr{Map}(f(z),f(w)),
   }\qquad
   \xymatrix{
   \mr{Map}(h(z_0),h(w))\ar[r]\ar[d]&\mr{Map}(h(z),h(w))\ar[d]\\
  \mr{Map}(f(z_0),f(w))\ar[r]&\mr{Map}(f(z),f(w)).
   }
 \end{equation*}
 Both diagrams are homotopy pullback diagram. Indeed, the left one is a
 homotopy pullback since $\epsilon'$ is $f$-coCartesian, and the right
 one is since $h(\epsilon')$ is $g$-coCartesian by the assumption.
 Thus, if $w$ is either $x$ or $y$, the top horizontal maps are
 equivalences. Thus, we are reduced to showing that
 \begin{equation*}
  \xymatrix{
   \mr{Map}(z_0,x)\ar[r]\ar[d]&\mr{Map}(z_0,y)\ar[d]\\
  \mr{Map}(h(z_0),h(x))\ar[r]&\mr{Map}(h(z_0),h(y)).
   }
 \end{equation*}
 is a homotopy pullback diagram, which follows since $e$ is
 $h'$-Cartesian.
\end{proof}

\begin{lem}
 \label{fiberprodcatmorphc}
 Let $f\colon\mc{C}\rightarrow\mc{E}$, $g\colon\mc{D}\rightarrow\mc{E}$
 be functors of $\infty$-categories, and assume that $g$ is a
 categorical fibration. Let $(C_0,D_0)$, $(C_1,D_1)$ be vertices of
 $\mc{C}\times_{\mc{E}}\mc{D}$. For $i=0,1$, put
 $E_i:=f(C_i)=g(D_i)$. Then
 \begin{equation*}
  \xymatrix{
   \mr{Map}_{\mc{C}\times_{\mc{E}}\mc{D}}
   \bigl((C_0,D_0),(C_1,D_1)\bigr)\ar[r]\ar[d]&
   \mr{Map}_{\mc{C}}(C_0,C_1)\ar[d]\\
  \mr{Map}_{\mc{D}}(D_0,D_1)\ar[r]&
   \mr{Map}_{\mc{E}}(E_0,E_1)
   }
 \end{equation*}
 is a homotopy Cartesian diagram.
\end{lem}
\begin{proof}
 Since $g$ is a categorical fibration,
 $\mr{Fun}(\Delta^1,\mc{D})\rightarrow\mr{Fun}(\Delta^1,\mc{E})$ is a
 categorical fibration as well by \cite[2.2.5.4]{HTT}.
 Thus,
 \begin{equation*}
  \mr{Fun}(\Delta^1,\mc{P})
   \cong
   \mr{Fum}(\Delta^1,\mc{C})
   \times_{\mr{Fun}(\Delta^1,\mc{E})}
   \mr{Fun}(\Delta^1,\mc{D})
   \simeq
   \mr{Fum}(\Delta^1,\mc{C})
   \times^{\mr{cat}}_{\mr{Fun}(\Delta^1,\mc{E})}
   \mr{Fun}(\Delta^1,\mc{D})
 \end{equation*}
 in $\Spc$.
 Now, we have
 $\mr{Map}_{\mc{C}}(C_0,C_1)\simeq
 \mr{Hom}_{\mc{C}}(C_0,C_1):=
 \mr{Fun}(\Delta^1,\mc{C})\times
 _{\mr{Fun}(\partial\Delta^1,\mc{C})}\{(C_0,C_1)\}$ by
 \cite[2.2.4.1, 4.2.1.8]{HTT}.
 The product is in fact a product in $\Cat_\infty$ since
 $\mr{Fun}(\Delta^1,\mc{C})\rightarrow
 \mr{Fun}(\partial\Delta^1,\mc{C})$ is a categorical fibration by
 \cite[3.1.4.3]{HTT}.
 Thus the square in question is a Cartesian square in $\Cat_\infty$.
 By Lemma \ref{propspccatinf}, the claim follows.
\end{proof}

\begin{lem}
 \label{critsubseg}
 Let $\mc{C}_i$ {\normalfont(}$i=0,1,2${\normalfont)}, $\mc{D}$ be $\infty$-categories
 and $\mc{C}'_i$, $\mc{D}'$ be its subcategories.
 Assume we are given a homotopy commutative diagram
 $(\Delta^1)^{3}\rightarrow\Cat_\infty$
 \begin{equation*}
  \xymatrix@R=15pt{
  &\mc{D}'\ar'[d][dd]\ar[ld]_{f'}\ar[rr]^-{g'}&&\mc{C}'_2\ar[dd]\ar[dl]\\
  \mc{C}'_1\ar[rr]\ar[dd]&&\mc{C}'_0\ar[dd]&\\
  &\mc{D}\ar'[r]^{g}[rr]\ar[dl]_{f}&&\mc{C}_2\ar[dl]\\
  \mc{C}_1\ar[rr]&&\mc{C}_0,&
  }
 \end{equation*}
 where the vertical arrows are inclusions. This induces the diagram
 \begin{equation*}
  \xymatrix@C=50pt{
   \mc{D}'\ar[r]^-{F':=f'\times g'}\ar[d]_{G}&
   \mc{C}'_1\times^{\mr{cat}}_{\mc{C}'_0}\mc{C}'_2
   =:\mc{C}'
   \ar[d]^{H}\\
  \mc{D}\ar[r]^-{F:=f\times g}&
   \mc{C}_1\times^{\mr{cat}}_{\mc{C}_0}\mc{C}_2
   =:\mc{C}
   }
 \end{equation*}
 where vertical arrows are the canonical functors.
 Assume the following:
 \begin{enumerate}
  \item The functor $F$ is a categorical equivalence;

  \item\label{critsubseg-2}
       The canonical functors $\mc{C}'_i\rightarrow\mc{C}_i$,
       $\mc{D}'\rightarrow\mc{D}$ are categorical fibrations;
       
  \item\label{critsubseg-3}
       An object $d\in\mc{D}$ is an object of $\mc{D}'$ precisely when
       $f(d)$ is in $\mc{C}'_1$ and $g(d)$ is in $\mc{C}'_2$;

  \item\label{critsubseg-4}
       A map $a\colon d\rightarrow d'$ in $\mc{D}$ such that
       $d,d'\in\mc{D}'$ is a map in $\mc{D}'$ precisely when $f(a)$ is
       in $\mc{C}'_1$ and $g(a)$ is in $\mc{C}'_2$.
 \end{enumerate}
 Then $F'$ is a categorical equivalence as well.
\end{lem}
\begin{proof}
 Let $\mc{A}$ be a (ordinary) category, and $\mc{B}$ be its
 subcategory.
 The functor $\mr{N}\mc{B}\rightarrow\mr{N}\mc{A}$ is a monomorphism in
 $\Cat_\infty$ (in the sense of \cite[5.5.6.13]{HTT} which coincides
 with the one in \cite[A.1]{AFR}) if $\mc{B}\rightarrow\mc{A}$ is
 an isofibration ({\it i.e.}\ if $b\in\mc{B}$ and $f\colon
 a\xrightarrow{\sim}b$ in $\mc{A}$, then $a$ and $f$ belongs to
 $\mc{B}$) by \cite[A.6]{AFR}.
 For any $\infty$-category $\mc{E}$, the functor
 $\mc{E}\rightarrow\mr{N}\mr{h}\mc{E}$ is a categorical fibration,
 so by \cite[5.5.6.12]{HTT},
 the functors $\mc{D}'\rightarrow\mc{D}$,
 $\mc{C}'_i\rightarrow\mc{C}_i$ are monomorphisms.
 By \cite[A.5]{AFR}, the functor
 $\mc{C}'\rightarrow\mc{C}$ is a monomorphism as well.
 By \cite[A.4]{AFR}, $F'$ is a monomorphism.

 Let us show that $F'$ is essentially surjective.
 Let $p_i\colon\mc{C}'_i\rightarrow\mc{C}'_0$ be the given map.
 Consider a triple $(C_1,C_2,\alpha)$ where $C_i\in\mc{C}'_i$ and
 $\alpha\colon p_1(C_1)\xrightarrow{\sim}p_2(C_2)$.
 This induces a functor
 $\Delta^0\rightarrow\mc{C}'_1\times^{\mr{cat}}_{\mc{C}'_0}\mc{C}'_2$ (up to
 contractible choices), and defines an object of the fiber product.
 Denote the associated object by $\mc{C}'(C_1,C_2,\alpha)$.
 Any object of the fiber product is equivalent
 to an object associated to a triple of the form above, because
 $\mc{C}'_1\rightarrow\mc{C}'_0$ can be factored into categorically
 equivalence followed by categorical fibration.
 Since $F$ is a categorical equivalence, there exists an object $D$ such
 that $F(D)\simeq H(\mc{C}'(C_1,C_2,\alpha))$.
 By \ref{critsubseg-3}, $D$ belongs to $\mc{D}'$  if $f(D)$ and $g(D)$
 belongs to $\mc{C}'_1$ and $\mc{C}'_2$ respectively. Since $f(D)\simeq
 C_1$ and $g(D)\simeq C_2$, combining with \ref{critsubseg-2},
 $D\in\mc{D}'$. Since $\mc{C}'\rightarrow\mc{C}$ is a monomorphism, the
 functor $\mc{C}'^{\simeq}\rightarrow\mc{C}^{\simeq}$ is a monomorphism
 by \cite[A.6]{AFR} and thus fully faithful (cf.\ \cite[A.1]{AFR}).
 This implies that $F'(D)$ and $\mc{C}'(C_1,C_2,\alpha)$ are equivalent,
 and $F'$ is essentially surjective.

 It remains to show the full faithfulness. For a simplicial set, recall
 $\mr{Map}_{\Cat_\infty}(K,-)\simeq\mr{Fun}(K,-)^{\simeq}$.
 If $\mc{A}\rightarrow\mc{B}$ is a categorical fibration, then
 $\mr{Fun}(K,\mc{A})\rightarrow\mr{Fun}(K,\mc{B})$ is a categorical
 fibration for any simplicial set $K$ by \cite[2.2.5.4]{HTT}, and so is
 the map $\mr{Map}(K,\mc{A})\rightarrow\mr{Map}(K,\mc{B})$ by Lemma
 \ref{propspccatinf}.
 Note that $\mr{Map}_{\Cat_\infty}$ is a model for the mapping space of
 $\Cat_\infty$ by the definition of Boardman-Vogt weak equivalence
 (cf.\ \cite[A.3.2.1]{HTT}). Thus, by dual of \cite[at the beginning of
 \S5.5.2]{HTT}, we have the following diagram of spaces induced by
 taking $\mr{Map}_{\Cat_\infty}(\Delta^1,-)$ to the diagram in the
 statement of the lemma
 \begin{equation*}
  \xymatrix@C=50pt{
   \mr{Map}(\Delta^1,\mc{D}')\ar[r]\ar[d]_{\widetilde{G}}&
   \mr{Map}(\Delta^1,\mc{C}'_1)
   \times^{\mr{cat}}_{\mr{Map}(\Delta^1,\mc{C}'_0)}
   \mr{Map}(\Delta^1,\mc{C}'_2)\ar[d]\\
  \mr{Map}(\Delta^1,\mc{D})\ar[r]&
   \mr{Map}(\Delta^1,\mc{C}_1)
   \times^{\mr{cat}}_{\mr{Map}(\Delta^1,\mc{C}_0)}
   \mr{Map}(\Delta^1,\mc{C}_2).
   }
 \end{equation*}
 The map $\widetilde{G}$ is a monomorphism by \cite[A.6]{AFR}, thus fully
 faithful (cf.\ \cite[A.1]{AFR}). This is the same if we replace $\mc{D}$,
 $\mc{D}'$ by $\mc{C}_i$, $\mc{C}'_i$ respectively.
 The lower horizontal functor is a categorical equivalence,
 and the upper horizontal functor is fully faithful. Repeating the
 argument of the essential surjectivity of $F'$, the upper
 horizontal functor is essentially surjective,
 and thus categorical equivalence as required.
\end{proof}

\subsection{}
The following lemma should be well-known to experts, but since we could
not find a reference, we write here for record.

\begin{lem*}
 Let $\mc{C}$ be an $\infty$-category.
 Consider a diagram
 $F\colon K:=K_1\times K_2\rightarrow\mc{C}$ where $K_1$, $K_2$ are
 simplicial sets.
 For any simplicial subsets $K'\subset K$, assume that the functor
 $F|_{K'}$ admits a limit.
 Let $F_1\colon K_1\rightarrow\mr{Fun}(K_2,\mc{C})$ and $F_2\colon
 K_2\rightarrow\mr{Fun}(K_1,\mc{C})$ be functors induced by $F$.
 Then we have a canonical equivalence
 \begin{equation*}
  \invlim_{K_2}(\invlim_{K_1} F_1)
   \simeq
   \invlim_K F
   \simeq
   \invlim_{K_1}(\invlim_{K_2} F_2).
 \end{equation*}
\end{lem*}
\begin{proof}
 Let us show the first equivalence. By taking the opposite category, we
 show the equivalence for colimits instead of limits.
 By \cite[4.2.3.15]{HTT}, there exists a (left) cofinal map
 $\mr{N}(\mc{I})\rightarrow K_2$ from an partially ordered set
 $\mc{I}$. In view of \cite[4.1.1.13]{HTT}, we may replace $K_2$ by
 $\mr{N}(\mc{I})$. For each $I\in\mc{I}$, let $K_I:=\{I\}\times K_1$ and
 we have the functor $G\colon\mr{N}(\mc{I})\rightarrow(\sSet)_{/K}$
 sending $I$ to $K_I$.
 In \cite[4.2.3.1]{HTT}, the simplicial set $K_G$ is defined.
 In view of \cite[4.2.3.9]{HTT}, the hypotheses of \cite[4.2.3.8]{HTT}
 is satisfied. Now, by construction, we have the evident inclusion
 $K_G\rightarrow K\diamond_{\mr{N}(\mc{I})}\mr{N}(\mc{I})$.
 By using \cite[4.2.2.7]{HTT}, $F$ admits an extension
 $\widetilde{F}\colon
 K\diamond_{\mr{N}(\mc{I})}\mr{N}(\mc{I})\rightarrow\mc{C}$.
 Since $\widetilde{F}|_{K_G}$ satisfies the hypotheses of
 \cite[4.2.3.4]{HTT}, and we invoke \cite[4.2.3.10]{HTT} to conclude.
\end{proof}

\begin{cor}
 \label{basechform}
 Let $\mc{C}$ be an $\infty$-category, and consider a diagram
 $F\colon(\Lambda^2_2)^{\triangleright}\rightarrow\mc{C}$ and a map
 $t\rightarrow F(\infty)$, where $\infty$ is the cone point. Then we
 have the canonical equivalence
 $(F(0)\times_{F(2)}F(1))\times_{F(\infty)} t\simeq
 (F(0)\times_{F(\infty)} t)\times_{(F(2)\times_{F(\infty)} t)}
 (F(1)\times_{F(\infty)} t)$.
\end{cor}
\begin{proof}
 Let us construct a functor
 $\widetilde{F}\colon\Lambda^2_2\times\Lambda^2_2\rightarrow\mc{C}$ as
 follows. Let
 $D:=(\Lambda^2_2)^{\triangleright}\coprod_{\infty,\{*\},[1]}\Delta^1$.
 We have the functor $F'\colon D\rightarrow\mc{C}$ sending $\Delta^1$ to
 $t\rightarrow F(\infty)$. Let
 $i\colon D\rightarrow\Lambda^2_2\times\Lambda^2_2$ be the
 inclusion. Let $\widetilde{F}$ to be a right Kan extension of $F$ along
 $i$. Now, the claim follows by applying the lemma.
\end{proof}

\subsection{}
\label{cartfibcocartnot}
Recall that $\Cat_\infty=\mr{N}((\sSet^+)^{\circ})$.
Using this specific model of the $\infty$-category of $\infty$-categories,
let $\Cart_\infty$ be the full subcategory of $\mr{Fun}(\Delta^1,\Cat_\infty)$ spanned by maps of marked simplicial sets
$\mc{C}^{\natural}\rightarrow\mc{D}^{\natural}$ whose underlying map $\mc{C}\rightarrow\mc{D}$ is a Cartesian fibration.
By evaluating at $1\in\Delta^1$, we have the map $\mr{Fun}(\Delta^1,\Cat_\infty)\rightarrow\Cat_\infty$,
which is in fact a Cartesian fibration since $\Cat_\infty$ admits limits.
This induces the Cartesian fibration $\theta\colon\Cart_\infty\rightarrow\Cat_\infty$.
Indeed, since $\Cart_\infty\hookrightarrow\mr{Fun}(\Delta^1,\Cat_\infty)$ is a full subcategory, it is an inner fibration.
Thus, $\theta$ is an inner fibration.
Moreover, Cartesian fibrations are stable under base change.
Since a Cartesian fibration is a categorical fibration by \cite[3.3.1.7]{HTT},
the pullback in $\Cat_\infty$ of a Cartesian fibration can be computed as the pullback of the simplicial set,
and $\Cart_\infty$ is stable under the base change in $\Cat_\infty$.
Thus, $\theta$ is a Cartesian fibration.
For an $\infty$-category $\mc{C}$, we denote
$\Cart_\infty\times_{\Cat_\infty}\{\mc{C}\}$ by $\Cart(\mc{C})$.
Note that since $\theta$ is a categorical fibration, given a categorical
equivalence $\mc{C}\xrightarrow{\sim}\mc{C}'$, the base change functor
$\Cart(\mc{C}')\rightarrow\Cart(\mc{C})$ is a categorical equivalence
as well by \cite[3.3.1.3]{HTT} applied to the case where $S$ is the
category with two objects and one isomorphism and $T$ is an inclusion
from $\Delta^0$ to an object of $S$.
Dually, we put $\coCart_\infty$ to be the full subcategory of
$\mr{Fun}(\Delta^1,\Cat_\infty)$ spanned by coCartesian fibrations, and
define $\coCart(\mc{C})$ to be the fiber.

Let $\coCart^{\mr{str}}_\infty$ be the subcategory of $\coCart_\infty$
consisting of simplices $\Delta^n\rightarrow\coCart_\infty$ such that
all the edges are of the form
\begin{equation*}
 \xymatrix{
  \mc{D}\ar[r]^-{r}\ar[d]_{q}&\mc{D}'\ar[d]^{p}\\
  \mc{C}\ar[r]&\mc{C}'
  }
\end{equation*}
such that $r$ sends $q$-coCartesian edges to $p$-coCartesian edges.
The functor $\coCart^{\mr{str}}_\infty\rightarrow\Cat_\infty$ is
Cartesian as well. The fiber over $\mc{C}\in\Cat_\infty$ is denoted by
$\coCart^{\mr{str}}(\mc{C})$. We have the equivalences
\begin{equation*}
 \mr{Fun}(\mc{C},\Cat_\infty)
  \xrightarrow[\mr{Un}_{\mc{C}}]{\sim}
  \mr{N}((\sSet^+)_{/\mc{C}}^{\circ})
  \xrightarrow{\sim}
  \coCart^{\mr{str}}(\mc{C}),
\end{equation*}
where the first one is the unstraightening functor.
The category $(\sSet^+)_{/\mc{C}}^{\circ}$ is endowed with coCartesian
model structure, and let us construct the second equivalence.
The functor of simplicial categories
$(\sSet^+)_{/\mc{C}}^{\circ}\rightarrow
((\sSet^+)_{/*})^{\circ}_{/\mc{C}}$ sending $X$ to
$X\times_{\mc{C}^{\sharp}}\mc{C}^{\natural}\rightarrow
\mc{C}^{\natural}$, where $\mc{C}$ is considered to be a coCartesian
fibered over $\Delta^0$, induces the functor
$\mr{N}((\sSet^+)_{/\mc{C}}^{\circ})\rightarrow
(\Cat_\infty)_{/\mc{C}}\simeq(\Cat_\infty)^{/\mc{C}}$,
using \cite[6.1.3.13, 4.2.1.5]{HTT}.
By definition, $\Cart^{\mr{str}}(\mc{C})$ is a subcategory of
$(\Cat_\infty)^{/\mc{C}}$, and the above functor factors through
$\Cart^{\mr{str}}(\mc{C})$, which is the desired functor.
This functor is essentially surjective by definition. It remains to show
that it is fully faithful. Let $\mc{D}$, $\mc{D}'$ be coCartesian
fibrations over $\mc{C}$. Let
$\mr{Fun}_{\mc{C}}(\mc{D},\mc{D}')^{\mr{coCart}}$ be the full
subcategory of $\mr{Fun}_{\mc{C}}(\mc{D},\mc{D}')$ spanned by functors
preserving coCartesian edges. Then we have
\begin{equation*}
 \mr{Map}_{\mc{C}}^{\sharp}(\mc{D}^{\natural},\mc{D}'^{\natural})
  \cong
  \mr{Map}_{\mc{C}}^{\flat}(\mc{D}^{\natural},\mc{D}'^{\natural})
  ^{\simeq}
  \simeq
  \bigl(\mr{Fun}_{\mc{C}}(\mc{D},\mc{D}')^{\mr{coCart}}\bigr)^{\simeq},
\end{equation*}
where the first isomorphism follows by \cite[3.1.3.1]{HTT}.
In view of \cite[3.1.4.4]{HTT}, we get the claim.

\begin{lem}
 \label{cartsegco}
 Let $\mc{D}\simeq\mc{C}_1\coprod_{\mc{C}_0}\mc{C}_2$
 be a pushout in $\Cat_\infty$.
 Then we have an equivalence
 $\alpha\colon\Cart(\mc{D})\simeq\Cart(\mc{C}_1)
 \times_{\Cart(\mc{C}_0)}\Cart(\mc{C}_1)$
 in $\Cat_\infty$.
\end{lem}
\begin{proof}
 The author learned the proof from Lysenko's
 notes\footnote{See
 \href
 {https://lysenko.perso.math.cnrs.fr/notes/comments_Gaitsgory_Lurie_Tamagawa.pdf}
 {lysenko.perso.math.cnrs.fr/notes/comments\_Gaitsgory\_Lurie\_Tamagawa.pdf}.
 }.
 First, for any $\infty$-category $\mc{C}$, we have
 \begin{equation*}
  \mr{Map}_{\Cat_\infty}(\Delta^0,\Cart(\mc{C}))
   \simeq
   \Cart^{\mr{str}}(\mc{C})^{\simeq}
   \simeq
   \mr{Map}_{\Cat_\infty}(\mc{C},\Cat_\infty),
 \end{equation*}
 where the second equivalence is the straightening/unstraightening
 equivalence. Since the pushout is in $\Cat_\infty$, we have
 \begin{align*}
  \mr{Map}_{\Cat_\infty}(\mc{D},\Cat_\infty)
  \simeq
  \mr{Map}_{\Cat_\infty}(\mc{C}_1,\Cat_\infty)
  \times^{\mr{cat}}_{\mr{Map}_{\Cat_\infty}(\mc{C}_0,\Cat_\infty)}
  \mr{Map}_{\Cat_\infty}(\mc{C}_2,\Cat_\infty)
 \end{align*}
 by \cite[at the beginning of \S5.5.2]{HTT}.
 Combining these equivalences, $\mr{Map}_{\Cat_\infty}(\Delta^0,\alpha)$
 is an equivalence.
 It remains to show that $\mr{Map}_{\Cat_\infty}(\Delta^1,\alpha)$ is an
 equivalence. By \cite[Ch.12, 2.1.3]{GR}, we have an equivalence
 \begin{equation*}
  \mr{Map}_{\Cat_\infty}(\Delta^1,\Cart(\mc{C}))
   \simeq
   \mr{Map}_{\Cat_\infty}(\mc{C}^{\mr{op}},\coCart(\Delta^1))
 \end{equation*}
 for any $\infty$-category $\mc{C}$. Since $(-)^{\mr{op}}$ is an
 auto-equivalence of $\Cat_\infty$, we have
 $\mc{D}^{\mr{op}}\simeq\mc{C}_1^{\mr{op}}\coprod_{\mc{C}_0^{\mr{op}}}
 \mc{C}_2^{\mr{op}}$, we have
 \begin{align*}
  \mr{Map}_{\Cat_\infty}&(\mc{D}^{\mr{op}},\coCart(\Delta^1))\\
   &\simeq
   \mr{Map}_{\Cat_\infty}(\mc{C}_1^{\mr{op}},\coCart(\Delta^1))
   \times^{\mr{cat}}
   _{\mr{Map}_{\Cat_\infty}(\mc{C}_0^{\mr{op}},\coCart(\Delta^1))}
   \mr{Map}_{\Cat_\infty}(\mc{C}_1^{\mr{op}},\coCart(\Delta^1)).
 \end{align*}
 Combining these equivalences, $\mr{Map}(\Delta^1,\alpha)$ is an
 equivalence as required.
\end{proof}

\subsection{}
\label{twisarrcat}
Let $\mc{C}$ be an $\infty$-category.
We define the $\infty$-category of arrows to be
$\mr{Ar}\mc{C}:=\mr{Fun}(\Delta^1,\mc{C})$.
On the other hand, the $\infty$-category of
{\em twisted arrows} denoted by $\mr{Tw}\mc{C}$ is studied extensively
in \cite[\S 5.2.1]{HA}.
Informally, this is the category of morphisms
$c\rightarrow c'$ in $\mc{C}$, with a map $(c_0\rightarrow
c'_0)\rightarrow(c_1\rightarrow c'_1)$ given by a diagram
$\Delta^3\rightarrow\mc{C}$ depicted as
\begin{equation*}
 \xymatrix{
  c_0\ar[d]\ar[r]&c_1\ar[d]\\
 c'_1&c'_1.\ar[l]}
\end{equation*}
We put $\mr{Tw}^{\mr{op}}\mc{C}:=(\mr{Tw}\mc{C})^{\mr{op}}$.
For a functor $F\colon K\rightarrow\mc{C}$ from a simplicial set $K$, we
denote by
$\mr{Tw}^{\mr{op}}_F\mc{C}:=\mr{Tw}^{\mr{op}}\mc{C}\times_{\mc{C}}K$.

\subsection{}
\label{operadsrecall}
We use two types of operads in this paper: $\infty$-operads and planar
$\infty$-operads in the sense of \cite{HA}.
In this paper, after \cite{GH}, we call {\em operad} what Lurie calls
planar $\infty$-operad, and {\em symmetric operad} what Lurie calls
$\infty$-operad. We only recall $\infty$-operads very briefly.

Recall that a function $a\colon[n]\rightarrow[m]$ is said to be
{\em inert} if there exists $i\in[m]$ such that $a(j)=i+j$. An
{\em active map} is a map $a$ such that $a(0)=0$, $a(n)=m$. The
corresponding maps in $\mbf{\Delta}^{\mr{op}}$ are also called
inert and active maps.
A {\em generalized $\infty$-operad} is an inner fibration
$f\colon\mc{C}^{\circledast}\rightarrow\mbf{\Delta}^{\mr{op}}$
satisfying the following three conditions: 1.\
for any $X\in\mc{C}^{\circledast}$ and an inert edge $f(x)\rightarrow y$
in $\mbf{\Delta}^{\mr{op}}$, there exists a $f$-coCartesian edge
$X\rightarrow Y$ lifting the inert edge; 2.\ the induced map
\begin{equation*}
 \mc{C}^{\circledast}_{[n]}\rightarrow
  \mc{C}^{\circledast}_{\{0,1\}}
  \times^{\mr{cat}}_{\mc{C}^\circledast_{\{1\}}}
  \mc{C}^{\circledast}_{\{1,2\}}
  \times^{\mr{cat}}_{\mc{C}^\circledast_{\{2\}}}
  \dots
  \times^{\mr{cat}}_{\mc{C}^\circledast_{\{n-1\}}}
  \mc{C}^{\circledast}_{\{n-1,n\}}
\end{equation*}
is a categorical equivalence (Segal condition); 3.\ for any
$C\in\mc{C}^{\circledast}_{[n]}$, we have a map from $C$ to the diagram
\begin{equation*}
  \xymatrix@R=7pt{
  C_{\{0,1\}}\ar[rd]&&C_{\{1,2\}}\ar[ld]\ar[rd]&&\dots&&
  C_{\{n-1,n\}}\ar[ld]\\
 &C_{\{1\}}&&C_{\{2\}}&\dots&C_{\{n-1\}}&
  }
\end{equation*}
which exhibits $C$ as a $\pi$-limit.
Here the functor
$\mc{C}^{\circledast}_{\{i,i+1\}}\rightarrow\mc{C}^{\circledast}_{\{j\}}$
($j=i,i+1$) is induced by the assumption that $f$ is coCartesian over
inert edges in $\mbf{\Delta}^{\mr{op}}$.
An $f$-coCartesian edges in $\mc{C}^{\circledast}$ over an inert map in
$\mbf{\Delta}^{\mr{op}}$ are called {\em inert edges}.
A generalized $\infty$-operad is an {\em $\infty$-operad}
if $\mc{C}^{\circledast}_{[0]}$ is contractible.
A map of generalized $\infty$-operads from $\mc{O}^{\circledast}$ to $\mc{C}^{\circledast}$ is a functor
$\mc{O}^{\circledast}\rightarrow\mc{C}^{\circledast}$ over $\mbf{\Delta}^{\mr{op}}$ which preserves inert edges.
The $\infty$-category of maps of generalized $\infty$-operads from $\mc{O}^{\circledast}$ to $\mc{C}^{\circledast}$
is denoted by $\mr{Alg}_{\mc{O}}(\mc{C})$.
Just as $\Cat_\infty$, $\infty$-operads form an $\infty$-category.
The $\infty$-category of (generalized) $\infty$-operads is denoted by
$\mr{Op}_{\infty}^{\mr{ns},(\mr{gen})}$ (cf.\ \cite[\S3.2]{GH}).

We generally use $\mc{C}^{\otimes}$ for (generalized) symmetric
$\infty$-operads and $\mc{C}^{\circledast}$ for (generalized)
$\infty$-operads.

\begin{dfn*}
 A map $\mc{M}^{\circledast}\rightarrow\mc{N}^{\circledast}$ of
 generalized $\infty$-operads is said to be {\em base preserving} if the
 induced map
 $\mc{M}^{\circledast}_{[0]}\rightarrow\mc{N}^{\circledast}_{[0]}$ is a
 categorical equivalence.
\end{dfn*}

\subsection{}
\label{algmonoidfuncto}
We have the bifunctor of symmetric $\infty$-operads
$\mr{N}\Fin\times\mr{N}\Fin\rightarrow\mr{N}\Fin$.
As in the proof of \cite[3.2.4.3]{HA}, we have the left Quillen
bifunctor $(\sSet^+)_{/\mf{P}}\times(\sSet^+)_{/\mf{P}}\rightarrow
(\sSet^+)_{/\mf{P}}$, where $\mf{P}$ is the categorical pattern
defining the $\infty$-category of symmetric $\infty$-operads
$\mr{Op}_\infty$ (cf.\ \cite[proof of 2.1.4.6]{HA}),
which is identical to $\odot$ in \cite[2.2.5.5]{HA}.
Thus, if we fix a fibration of symmetric $\infty$-operads
$\mc{C}^{\otimes}\rightarrow\mr{N}\Fin$, we have a functor
$\mr{Op}_\infty^{\mr{op}}\rightarrow\mr{Op}_\infty$
sending $\mc{O}^{\otimes}$ to $\mr{Alg}_{\mc{O}}(\mc{C})^{\otimes}$.
In particular, considering the embedding
$\Cat_\infty\rightarrow\mr{Op}_\infty$ (cf.\ \cite[2.1.4.11]{HA}),
we have the functor
$\mr{Fun}(-,\mc{C}^{\otimes})\colon\Cat_\infty^{\mr{op}}
\rightarrow\mr{Op}_\infty$ sending $\mc{D}$ to
$\mr{Alg}_{\mc{D}}(\mc{C})^{\otimes}\simeq
\mr{Fun}(\mc{D},\mc{C}^{\otimes})$.
Let $\mr{Op}_\infty^{\mr{co},\mr{pres}}$ be the subcategory of
$\mr{Op}_\infty$ spanned by coCartesian fibration
$\mc{O}^{\otimes}\rightarrow\mr{N}\Fin$ which comes from
$\mr{CAlg}(\PrL)$, and colimit preserving morphisms which preserve
coCartesian edges.
If $\mc{C}^{\otimes}\rightarrow\mr{N}\Fin$ is a coCartesian fibration
coming from $\mr{CAlg}(\PrL)$, then \cite[3.2.4.3]{HA} further
implies that the functor $\mr{Fun}(-,\mc{C}^{\otimes})$ factors through
$\Cat_\infty\rightarrow\mr{Op}_{\infty}^{\mr{co},\mr{pres}}
\simeq\mr{CAlg}(\PrL)$.

\subsection{}
\label{inftwocatintro}
In this paper, we follow \cite{GR} for the terminology for
$(\infty,2)$-categories. In particular, we employ complete Segal
$\infty$-category model for $(\infty,2)$-category.
Before recalling the definition of $(\infty,2)$-category, let us recall
the relation between $\infty$-categories and complete Segal spaces.
First, we have the pair of adjoint functors
\begin{equation*}
 \xymatrix{
  \mr{Fun}(\mbf{\Delta}^{\mr{op}},\Spc)
  \ar@<.5ex>[r]^-{\mr{JT}}&
  \Cat_\infty
  \ar@<.5ex>[l]^-{\mr{Seq}_\bullet}.
  }
\end{equation*}
Indeed, by \cite[4.11]{JT} taking \cite[1.5.1]{H} into account, we
have an equivalence $\Cat_\infty\simeq\mc{CSS}$, where $\mc{CSS}$
denotes the $\infty$-category of complete Segal spaces. By definition of
the complete Segal space model structure, $\mc{CSS}$ is a localization
of $\mr{Fun}(\mbf{\Delta}^{\mr{op}},\Spc)$, and we get the adjoint
functors above. The construction shows that, for an $\infty$-category
$\mc{C}$, the adjunction
$\mr{JT}(\mr{Seq}_\bullet(\mc{C}))\rightarrow\mc{C}$ is a categorical
equivalence.
For $\mc{C}\in\Cat_\infty$, we can compute $\mr{Seq}_\bullet(\mc{C})$ as
follows. Let $\Delta^{\bullet}\colon\mbf{\Delta}\rightarrow\Cat_\infty$
be the cosimplicial object such that $\Delta^{\bullet}([n]):=\Delta^n$.
For an $\infty$-category $\mc{C}$, we have
$\mr{Seq}_\bullet(\mc{C}):=\theta\circ\mr{Fun}(\Delta^{\bullet},\mc{C})$,
where $\theta$ is the functor in Lemma \ref{propspccatinf}, by
\cite[4.10]{JT}. Explicitly,
$\mr{Seq}_n(\mc{C})\simeq\mr{Fun}(\Delta^n,\mc{C})^{\simeq}
\simeq\mr{Map}_{\Cat_\infty}(\Delta^n,\mc{C})$.

The observation above shows that we may think of an $\infty$-category as an
object of $\mr{Fun}(\mbf{\Delta}^{\mr{op}},\Spc)$ which is a complete
Segal space.
In our treatment, following \cite{GR}, we upgrade this
picture, and use complete Segal
$\infty$-category model for a model of $(\infty,2)$-category.
An {\em $(\infty,2)$-category} $\mbf{C}$ is a functor
$\mc{C}_\bullet\colon\mbf{\Delta}^{\mr{op}}\rightarrow
\Cat_\infty$ satisfying the following conditions:
\begin{itemize}
 \item The $\infty$-category $\mc{C}_0$ is a space;

 \item (Segal condition) The functor
       $\mc{C}_n\rightarrow\mc{C}_1\times^{\mr{cat}}_{\mc{C}_0}\mc{C}_1
       \times^{\mr{cat}}_{\mc{C}_0}\dots\times^{\mr{cat}}
       _{\mc{C}_0}\mc{C}_1$ induced by inert maps
       (namely, the maps $[1]\cong\{i,i+1\}\subset[n]$ in $\mbf{\Delta}$) is an
       equivalence for $n\geq1$;
       
 \item (completeness) There exists an $\infty$-category $\mc{C}$ such
       that $\mr{Seq}_\bullet(\mc{C})\simeq\theta\circ\mc{C}_\bullet$.
\end{itemize}
By definition the composition $\mbf{\Delta}^{\mr{op}}
\xrightarrow{\mc{C}_\bullet}\Cat_\infty\xrightarrow{\theta}\Spc$ is a
complete Segal space, and yields an $\infty$-category. This
$\infty$-category is called the {\em underlying $\infty$-category} of
$\mbf{C}$.

\begin{ex*}
 Let $p\colon\mc{A}^{\circledast}\rightarrow\mbf{\Delta}^{\mr{op}}$ be a
 monoidal $\infty$-category. By straightening, this coCartesian
 fibration corresponds to a functor
 $\mbf{\Delta}^{\mr{op}}\rightarrow\Cat_\infty$.
 Since $p$ is a an $\infty$-operad, it satisfies the Segal condition,
 and $\mc{A}_0$ is a contractible Kan complex.
 Unfortunately, this Segal $\infty$-category may not be complete.
 By \cite[1.2.13]{LG}, we can localize the Segal $\infty$-category into
 a complete Segal $\infty$-category. This complete Segal space is called
 the classifying $(\infty,2)$-category of $\mc{A}^{\circledast}$ denoted
 by $\mbf{B}\mc{A}^{\circledast}$.
\end{ex*}

\subsection{}
We recall the $(\infty,2)$-category of correspondences used in
\cite{GR}. Let $\mc{C}$ be a category\footnote{We may assume $\mc{C}$ to
be an $\infty$-category, but for simplicity, we assumed this. For details
see \cite{GR}.}.
We need 3 classes of morphisms in $\mc{C}$ denoted by
$\mathit{vert}$, $\mathit{horiz}$, $\mathit{adm}$ satisfying certain
axioms (cf. \cite[Ch.7, 1.1.1]{GR}). To define the $(\infty,2)$-category
$\mbf{Corr}^{\mathit{adm}}_{\mathit{vert};\mathit{horiz}}(\mc{C})$, we
should define its associated Segal space
$\mr{Seq}_\bullet(\mbf{Corr}^{\mathit{adm}}
_{\mathit{vert};\mathit{horiz}}(\mc{C}))$. For $n\geq0$, 
$\mr{Seq}_n(\mbf{Corr}^{\mathit{adm}}_
{\mathit{vert};\mathit{horiz}}(\mc{C}))$ is the category of diagrams of
the form
\begin{equation*}
 \xymatrix{
  X_{n0}\ar[r]\ar[d]\ar@{}[rd]|\square&
  \dots\ar[r]\ar@{}[rd]|\square&
  X_{20}\ar[d]\ar[r]\ar@{}[rd]|\square&X_{10}\ar[d]\ar[r]
  &X_{00}.\\
 X_{n1}\ar[d]\ar[r]&\dots\ar[r]&
  X_{21}\ar[r]\ar[d]&X_{11}&\\
 X_{n2}\ar[r]\ar[d]&\dots\ar[r]&
  X_{22}&&\\
 \vdots\ar[d]&\dots&&&&\\
 X_{nn},&&&&
  }
\end{equation*}
where horizontal arrows are in the class $\mathit{horiz}$ and the
vertical arrows are in the class $\mathit{vert}$. A morphism in
$\mr{Seq}_n(\mbf{Corr}^{\mathit{adm}}
_{\mathit{vert};\mathit{horiz}}(\mc{C}))$ is a morphism of diagrams
$X_{\bullet\bullet}\rightarrow Y_{\bullet\bullet}$ such that each
morphism $X_{ij}\rightarrow Y_{ij}$ is in $\mathit{adm}$ and
$X_{kk}\rightarrow Y_{kk}$ is an equivalence. 
We may check that this is an $(\infty,2)$-category, and even an ordinary
$2$-category (cf.\ \cite[Ch.7]{GR}).

\section{Dualizing coCartesian fibrations}
\label{dualconst}
Let $f\colon X\rightarrow S$ be a Cartesian fibration. Via
straightening/unstraightening construction, there exists a coCartesian
fibration $f'\colon X'\rightarrow S^{\mr{op}}$ with the same
straightening as $f$.
The existence of such coCartesian fibration readily
follows from straightening/unstraightening theorem, but the
construction is far from explicit.
As far as the author knows, there are two models for $f'$. One is in
\cite[14.4.2]{SAG}, and the other is in \cite{BGN}.
In this section, we construct yet another model of $f'$ at least when
$S$ is an $\infty$-category.
This model naturally appears in a construction in \S\ref{constfun}.

\subsection{}
Let $\mc{C}$ be an $\infty$-category.
Using the notation of \ref{cartfibcocartnot}, we have the auto-functor
\begin{equation*}
 \mb{D}\colon\coCart^{\mr{str}}(\mc{C}^{\mr{op}})
  \simeq
  \mr{Fun}(\mc{C}^{\mr{op}},\Cat_\infty)
  \simeq
  \Cart^{\mr{str}}(\mc{C}).
\end{equation*}
When $\mc{C}=\Delta^0$, the functor $\mb{D}$ is equivalent to the
identity functor. Let $f\colon\mc{D}\rightarrow\mc{C}^{\mr{op}}$ be in
$\coCart^{\mr{str}}(\mc{C}^{\mr{op}})$.
We have a Cartesian fibration $\mb{D}(f)\rightarrow\mc{C}$.
Then by the functoriality of straightening/unstraightening functor, we
have an equivalence $\mc{D}_{v}\cong\mb{D}(\mc{D})_v$ for each object
$v\in\mc{C}$.
In the following we sometimes denote $\mb{D}(f)$ by $\mb{D}(\mc{D})$ or
$\mb{D}_{\mc{C}}(\mc{D})$ if no confusion may arise. Note that, by
construction,
$\mb{D}(f)^{\mr{op}}\cong\mb{D}^{-1}(f^{\mr{op}})$.

\subsection{}
Our goal of this section is to compare some Cartesian fibration with
$\mb{D}(f)$. For a preparation, we give a criterion to detect
$\mb{D}(f)$. A diagram of $\infty$-categories
$\mc{C}\leftarrow\mc{M}\rightarrow\mc{D}$ is said to be a
{\em weak pairing} if it is an object of $\mr{CPair}$ (cf.\
\cite[5.2.1.14, 5.2.1.15]{HA}). In other words, weak pairing is a
diagram which is equivalent to a pairing (cf.\ \cite[5.2.1.5]{HA}),
namely a diagram such that the induced map
$\mc{M}\rightarrow\mc{C}\times\mc{D}$ is equivalent to a right
fibration. We often say $\mc{M}\rightarrow\mc{C}\times\mc{D}$ is a weak
pairing without referring to the diagram.
A weak pairing is said to be {\em perfect} if it is contained in the
subcategory $\mr{CPair}^{\mr{perf}}$ (cf.\ \cite[5.2.1.20]{HA}), namely
a paring which is equivalent to the pairing
$\mr{Tw}\mc{C}\rightarrow\mc{C}\times\mc{C}^{\mr{op}}$ for some
$\infty$-category $\mc{C}$.
Definition \cite[5.2.1.8]{HA} makes sense also for weak pairings, so we
may talk about left universality {\it etc.}
Let $\lambda\colon\mc{M}\rightarrow\mc{C}\times\mc{D}$ be a weak pairing
and take an equivalence
\begin{equation*}
 \xymatrix{
  \mc{M}\ar[r]^-{\sim}_-{\gamma}\ar[d]_{\lambda}&
  \mc{M}'\ar[d]^{\lambda'}\\
 \mc{C}\times\mc{D}\ar[r]^-{\sim}_{\alpha\times\beta}&
  \mc{C}'\times\mc{D}'
  }
\end{equation*}
where $\lambda'$ is a pairing. Then $M\in\mc{M}$, such that
$\lambda(M)=(C,D)$, is left universal if and only if $\gamma(M)$ is left
universal because
$\mc{M}\times_{\mc{C}}\{C\}\xrightarrow{\sim}
\mc{M}'\times_{\mc{C'}}\{\alpha(C)\}$ and \cite[1.2.12.2]{HTT}.

\begin{lem*}
 \label{conddual}
 Let $S$ be a simplicial set, and consider the following diagram
 \begin{equation*}
  \xymatrix{
   \mc{M}
   \ar[rd]_{p}\ar[rr]^-{\nu}&&
   \mc{C}\times_S\mc{D}
   \ar[ld]^{q=f\times g}\\
  &S&
   }
 \end{equation*}
 Assume that $p$, $f$, $g$ are Cartesian fibrations, and $\nu$ sends
 $p$-Cartesian edges to $q$-Cartesian edges.
 If the following conditions are satisfied, then
 $\mc{D}\cong\mb{D}(f^{\mr{op}})$.

 \begin{itemize}
  \item For any vertex $s\in S$, $\nu_s:=\nu\times_S s$ is a perfect
	weak pairing
	{\normalfont(}resp.\ $\nu_s^{\mr{op}}$ is a perfect weak pairing{\normalfont)};
	
  \item For any $p$-Cartesian edge $x\rightarrow y$ in $\mc{M}$,
	if $y$ is right universal
	{\normalfont(}resp.\ right universal with respect to $\nu_{p(y)}^{\mr{op}}${\normalfont)}, then so is $x$.
 \end{itemize}
\end{lem*}
\begin{proof}
 First, let us show the non-resp claim.
 Recall that the $\infty$-category $\mr{CPair}$ is a full subcategory of
 $\mr{Fun}(\Lambda^2_0,\Cat_\infty)$. Thus, $\nu$ corresponds to a
 functor $S\rightarrow\mr{CPair}$ by straightening.
 By the second condition, this functor induces a functor
 $M\colon S\rightarrow\mr{CPair}^\mr{R}$
 (cf.\ \cite[5.2.1.16]{HA} for the notation).
 By \cite[5.2.1.19]{HA}, the functor
 $\phi\colon\mr{Pair}^{\mr{R}}\rightarrow\Cat_\infty$ sending
 $\nu$ to $\mc{C}$ admits a right adjoint $\mr{Tw}^{\mr{Pair}}$ such
 that $\mr{Tw}^{\mr{Pair}}\mc{C}\simeq
 (\mr{Tw}\mc{C}\rightarrow\mc{C}\times\mc{C}^{\mr{op}})$.
 Thus, we
 have the natural transformation
 $\mr{id}\rightarrow\mr{Tw}^{\mr{Pair}}\circ\phi$.
 Thus, we have the natural transformation
 $M\rightarrow\mr{Tw}^{\mr{Pair}}\circ\phi\circ M$.
 Put $\mr{Tw}(f):=\mr{Un}_S(\mr{Tw}^{\mr{Pair}}\circ\mr{St}_S(f))$
 and recall that
 $\mb{D}(f^{\mr{op}})\simeq\mr{Un}_S(\chi\circ\mr{St}_S(f))$,
 where $\chi$ is the unique non-trivial automorphism of $\Cat_\infty$,
 by \cite[14.4.2.4]{SAG}.
 By unstraightening, the natural transformation induces a diagram of
 Cartesian fibrations over $S$
 \begin{equation*}
  \xymatrix{
   \mc{M}\ar[r]\ar[d]&\mr{Tw}(f)\ar[d]\\
  \mc{C}\times_S\mc{D}\ar[r]&\mc{C}\times_S\mb{D}(f^{\mr{op}})
   }
 \end{equation*}
 where horizontal functors send $p$-Cartesian edges to Cartesian
 edges of $\mr{Tw}(f)$. Invoking \cite[3.3.1.5]{HTT}, horizontal
 functors are equivalences if and only if they are equivalences for each
 fibers of $S$. Since the construction is functorial with respect to
 $S$, the perfectness of $\nu_s$ implies that the horizontal
 functors are in fact equivalences, which implies that
 $\mc{D}\simeq\mb{D}(f^{\mr{op}})$ as required.
 Finally, let us show the resp claim. Consider the following
 diagram
 \begin{equation*}
  \xymatrix{
   \mb{D}(p^{\mr{op}})
   \ar[rd]_{p'}\ar[rr]^-{\mb{D}(\nu^{\mr{op}})}&&
   \mb{D}(q^{\mr{op}})
   \ar[ld]^{q'}\\
  &S.&
   }
 \end{equation*}
 Then $p'$, $q'$ are Cartesian fibrations by construction,
 $\mb{D}(\nu^{\mr{op}})_s\simeq\nu_s^{\mr{op}}$, and
 $\mb{D}(q^{\mr{op}})\simeq\mb{D}(f^{\mr{op}})\times_S\mb{D}(g^{\mr{op}})$.
 Since a $p$-Cartesian edge $x\rightarrow y$ yields a $p'$-Cartesian edge
 $x'\rightarrow y'$ by construction of $\mb{D}$, we may apply the
 non-resp claim, which implies that
 $\mb{D}(g^{\mr{op}})\simeq\mb{D}(f'^{\mr{op}})\simeq\mc{C}$, where
 $f'\colon\mb{D}(f^{\mr{op}})\rightarrow S$ is the Cartesian fibration.
 Thus, taking $\mb{D}^{-1}$, we get the claim.
\end{proof}

\begin{lem}
 \label{carttwar}
 Let $f\colon\mc{A}\rightarrow\mc{B}$ be an inner fibration of $\infty$-categories,
 and let $g\colon\mr{Tw}^{\mr{op}}\mc{A}\rightarrow\mr{Tw}^{\mr{op}}\mc{B}$ be the induced map.
 Let $e\colon\Delta^1\rightarrow\mr{Tw}^{\mr{op}}\mc{A}$ be an edge,
 and let $\widetilde{e}\colon\Delta^3\rightarrow\mc{A}$ be the associated map defining $e$ depicted as follows:
 \begin{equation*}
  \xymatrix{
   \widetilde{e}(1)\ar[d]&\widetilde{e}(0)
   \ar[l]_{\alpha}\ar[d]\\
  \widetilde{e}(2)\ar[r]^-{\beta}&\widetilde{e}(3).
   }
 \end{equation*}
 \begin{enumerate}
  \item\label{carttwar-1}
       The map $g$ is an inner fibration. If, moreover, $f$ is a
       categorical fibration, so is $g$.
	
  \item\label{carttwar-2}
       Assume that $\alpha:=\widetilde{e}(\Delta^{\{0,1\}})$ is an
       $f$-coCartesian edge and
       $\beta:=\widetilde{e}(\Delta^{\{2,3\}})$ is an $f$-Cartesian
       edge. Then $e$ is a $g$-Cartesian edge.
	 
  \item\label{carttwar-3}
       Assume that $\alpha$ is an $f$-Cartesian edge and
       $\beta$ is an $f$-coCartesian edge. Then $e$ is a
       $g$-coCartesian edge.
 \end{enumerate}
 \end{lem}
\begin{proof}
 Let us show the first claim.
 It suffices to show that the induced map
 $\mr{Tw}\mc{A}\rightarrow
 \mr{Tw}\mc{B}\times_{(\mc{B}\times\mc{B}^{\mr{op}})}
 (\mc{A}\times\mc{A}^{\mr{op}})$ is
 a right fibration. Indeed, we need to show the right lifting property
 of the map with respect to the inclusion
 $\Lambda^n_i\rightarrow\Delta^n$ for $0<i\leq n$. Unwinding the
 definition, it suffices to solve the right lifting problem of $f$
 with respect to $K\hookrightarrow\Delta^{2n+1}$, where $K$ is the
 same simplicial subset of $\Delta^{2n+1}$ appearing in
 the proof of \cite[5.2.1.3]{HA}. Since $K\hookrightarrow\Delta^{2n+1}$
 is shown to be an inner anodyne in {\em ibid.}, the claim follows.

 Let us prove the second claim. It amounts to solving the lifting
 problem on the left for \ref{carttwar-2} and right for
 \ref{carttwar-3}:
  \begin{equation*}
   \xymatrix{
    \Delta^{\{n-1,n\}}
    \ar@{^{(}->}[r]\ar@/^{3ex}/[rr]^{e}&
    \Lambda^n_n\ar[d]\ar[r]&
    \mr{Tw}^{\mr{op}}\mc{A}\ar[d]^{g}\\
   &\Delta^n\ar[r]\ar@{-->}[ur]&\mr{Tw}^{\mr{op}}\mc{B},
    }
    \qquad
    \xymatrix{
    \Delta^{\{0,1\}}
    \ar@{^{(}->}[r]\ar@/^{3ex}/[rr]^{e}&
    \Lambda^0_n\ar[d]\ar[r]&
    \mr{Tw}^{\mr{op}}\mc{A}\ar[d]^{g}\\
   &\Delta^n\ar[r]\ar@{-->}[ur]&\mr{Tw}^{\mr{op}}\mc{B}.
    }
  \end{equation*}
 We first treat \ref{carttwar-2}.
 Unwinding the definition, it suffices to solve the following lifting
 problem of marked simplicial sets:
 \begin{equation}
  \label{mainliftprob}
  \xymatrix@C=30pt{
   E
   \ar@{^{(}->}[r]\ar@/^{3ex}/[rr]^-{\psi}&
   K
   \ar[r]\ar@{^{(}->}[d]_{\varphi}&
   \mc{A}
   \ar[d]^{f}\\
  &\Delta^{2n+1}
   \ar[r]\ar@{-->}[ur]&\mc{B}.
   }
 \end{equation}
 Here, $E:=\Delta^{\{0,1\}}\cup\Delta^{\{2n,2n-1\}}$, the edge
 $\psi(\Delta^{\{0,1\}})$ (resp.\ $\psi(\Delta^{\{2n,2n+1\}})$) is an
 $f$-coCartesian (resp.\ $f$-Cartesian) edge, and $K$ is the
 union of the simplicial subsets $\Delta^I\subset\Delta^{2n+1}$ where
 $I=[2n+1]\setminus\{i,2n+1-i\}$ for $0<i\leq n$.
 
 Let $\Sigma$ be the simplicial subset $\Delta^{[2n+1]\setminus\{2n\}}$
 of $\Delta^{2n+1}$. Put $K_1:=K\cup\Sigma$. It suffices to check the
 following two claims
 \begin{enumerate}
  \item The map $f$ has right lifting property with respect to
	$K\hookrightarrow K_1$;
  \item The map $f$ has right lifting property with respect to
	$K_1\hookrightarrow\Delta^{2n+1}$.
 \end{enumerate}
 Let us show the first claim. Note that any simplex of $\Sigma$ which is
 not in $K$ contains $\{1\}$ as a vertex.
 Thus, we can divide the simplices of $K_1$ which do not belong to $K$
 into the following two classes:
 \begin{itemize}
  \item $A_k$ is the set of simplices which have $N$ vertices, where
	$N\leq k+2$, and contain the vertices $\{0,1\}$;
	
  \item $B_k$ is the set of simplices which have $k+1$ vertices and
	contain $\{1\}$ but do not contain $\{0\}$. For $\sigma\in A_k$,
	we denote by $\sigma'\in B_k$ the simplex obtained by deleting
	the vertex $\{0\}$.
 \end{itemize}
 Let $K_1^{(k)}:=K\cup\bigcup_{\sigma\in A_k}\sigma$, so
 $K_1^{(2n-1)}=K_1$. Then any element of $B_k$ is a simplex of
 $K_1^{(k)}$ because for any $\Delta^I\in B_k$,
 $\Delta^{I\sqcup\{0\}}\in A_k$. On the other hand no element of $B_{k+1}$
 belongs to $K_1^{(k)}$ because of the dimension reason.
 For $k<n-1$, we have $A_k=B_k=\emptyset$ because for any
 $I\subset[2,2n-1]$ with $\#I=k$, we
 can find $2\leq i\leq n+1$ such that $\{i,2n+1-i\}\cap I=\emptyset$.
 This implies that $K_1^{(n-2)}=K$.
 On the other hand, for $2n-1\geq k\geq n-1$, $A_k$, $B_k$ are
 non-empty.

 We solve the lifting problem with respect to $K\hookrightarrow
 K_1^{(k)}$ inductively.
 We assume we have a map $K_1^{(k)}\rightarrow\mc{A}$ solving the
 lifting problem.
 We choose a total ordering $\sigma_1<\dots<\sigma_a$ of the simplices
 in $A_{k+1}$ which do not belong to $K_1^{(k)}$,
 and we have the sequence $K_1^{(k)}=:L_0\subset L_1\subset\dots\subset
 L_a=:K_1^{(k+1)}$ where $L_i=L_{i-1}\cup\sigma_i$. Fix $i$ and put
 $L':=L_{i-1}$, $L:=L_i$.
 Let us solve the lifting problem with respect to
 $L'\hookrightarrow L$.
 For this, write $\sigma_i=\Delta^I$ with $I\subset[0,2n+1]$.
 Recall that $\{0,1\}\subset I$ by definition.
 Now, for $j\in I\setminus\{0,1\}$, $\Delta^{I\setminus\{j\}}$ belongs
 to $A_k$ or $K$, and thus $\Delta^{I\setminus\{j\}}\subset L'$.
 This holds also for $j=1$ since $\Delta^{I\setminus\{1\}}$ belongs to
 $K$. On the other hand, $\Delta^{I\setminus\{0\}}$ is not a simplex of
 $L'$ because $\Delta^{I\setminus\{0\}}\in B_k$ and
 $\sigma'_1,\dots,\sigma'_a$ are all different to each other.
 This implies that $L'\hookrightarrow
 L'\coprod_{\Lambda^{k+2}_0}\Delta^I=L$.
 Consider the following diagram:
 \begin{equation*}
  \xymatrix{
   \Delta^{\{0,1\}}\ar[r]\ar@/^3ex/[rrr]^-{\psi}&
   \Lambda^{k+2}_0\ar[r]\ar[d]&L'\ar[r]\ar[d]&
   \mc{A}\ar[d]^{f}\\&
  \Delta^{I}\ar[r]&L\ar[r]&\mc{B}.
   }
 \end{equation*}
 Since $\psi(\Delta^{\{0,1\}})$ is an $f$-coCartesian edge, we have a
 map $\Delta^I\rightarrow\mc{A}$ making the diagram commutative. Thus,
 we have a lifting $L\rightarrow\mc{A}$, and the first claim follows.

 Let us show the second claim. The idea of the proof is essentially the
 same as the first claim. If a simplex of $\Delta^{2n+1}$ does not
 contain the vertex $\{2n\}$, then it is contained in $K_1$.
 Thus, we can divide the simplices of $\Delta^{2n+1}$ which do not
 belong to $K_1$ into the following two
 classes:
 \begin{itemize}
  \item $C_k$ is the set of simplices which have $N$ vertices, where
	$N\leq k+2$, and contain the vertices $\{2n,2n+1\}$;
  \item $D_k$ is the set of simplices which have $k+1$ vertices and
	contain $\{2n\}$ but do not contain $\{2n+1\}$.
 \end{itemize}
 Let $K_2^{(k)}:=K_1\cup\bigcup_{\sigma\in C_k}\sigma$, so
 $K_2^{(2n-1)}=\Delta^{2n+1}$. As in the previous case, any element of
 $D_k$ is a simplex of $K^{(k)}_2$, and any element of $D_{k+1}$ does
 not belong to $K^{(k)}_2$. The sets $C_k$, $D_k$ are non-empty if and
 only if $2n-1\geq k\geq n-1$. We solve the lifting problem
 with respect to $K_1\hookrightarrow K_2^{(k)}$ inductively.
 Assume we have a map $K_2^{(k)}\rightarrow\mc{A}$ solving the
 problem. We choose a total ordering $\tau_1<\dots<\tau_b$ of the
 simplices in $C_{k+1}$ which do not belong to $K_2^{(k)}$,
 and we form a sequence $K_2^{(k)}=:M_0\subset\dots\subset
 M_b=:K_2^{(k+1)}$ where
 $M_{i}=M_{i-1}\cup\tau_i$. Fix $i$ and put $M':=M_{i-1}$, $M:=M_i$,
 $\tau:=\tau_i=\Delta^J$.
 We solve the lifting problem with respect to $M'\hookrightarrow M$.
 Since $\{2n,2n+1\}\subset J$, for any $j\in J\setminus\{2n,2n+1\}$,
 $\Delta^{J\setminus\{j\}}$ belongs to $C_k$ or $K_1$, and thus
 contained in $M'$. The same holds for $j=2n$ because it is contained in
 $K_1$. Finally, $\Delta^{I\setminus\{2n+1\}}$ is not contained in $M'$,
 which implies that $M'\hookrightarrow
 M'\coprod_{\Lambda^{k+2}_{k+2}}\Delta^J=M$. Since
 $\psi(\Delta^{\{2n,2n+1\}})$ is a $f$-Cartesian edge, the lifting
 problem is solved.

 Let us prove \ref{carttwar-3}. For this, it suffices to check the
 lifting problem (\ref{mainliftprob}) where
 $E:=\Delta^{\{n-1,n\}}\cup\Delta^{\{n+1,n+2\}}$, the edge
 $\psi(\Delta^{\{n+1,n+2\}})$ (resp.\ $\psi(\Delta^{\{n-1,n\}})$) is an
 $f$-coCartesian (resp.\ $f$-Cartesian) edge,
 and $K$ is the union of the simplicial subsets
 $\Delta^I\subset\Delta^{2n+1}$ where $I=[2n+1]\setminus\{i,2n+1-i\}$
 for $0\leq i<n$. The proof is essentially the same as \ref{carttwar-2},
 so we only indicate the difference.
 Put $\sigma:=\Delta^{[2n+1]\setminus\{n-1\}}$. The definition of $A_k$
 is replaced by the set of simplices which contain vertices
 $\{n+1,n+2\}$, $B_k$ is the one which contain $n+2$ but not $n+1$.
 We proceed as before, and define $L'\hookrightarrow L$ similarly.
 Small difference from the previous argument is that, in this case,
 $L=L'\coprod_{\Lambda^{k+2}_i}\Delta^{I}$ for $0\leq i<k+2$
 so that $\Delta^{\{i,i+1\}}$ in $\Lambda^{k+2}_i$ is mapped to the edge
 $\Delta^{\{n+1,n+2\}}$ in $\Delta^I$. The lifting problem can be solved
 for $i=0$ by the assumption that the edge $\psi(\Delta^{\{n+1,n+2\}})$
 is coCartesian, and for $0<i<k+2$ since $f$ is an inner fibration.
 The later part also works similarly, and we omit the detail.
\end{proof}

\subsection{}
 Let $\mc{C}$ be an $\infty$-category, and let
 $(\Phi,\Theta)\colon\mr{Tw}^{\mr{op}}
 \mc{C}\rightarrow\mc{C}^{\mr{op}}\times\mc{C}$ be the canonical
 functor.
 Recall the notation $\mr{Tw}^{\mr{op}}_K\mc{C}$ from
 \ref{twisarrcat}. In particular, for a vertex
 $v\colon\Delta^0\rightarrow\mc{C}$ and an edge
 $\phi\colon\Delta^1\rightarrow\mc{C}$, we have the $\infty$-categories
 $\mr{Tw}_v^{\mr{op}}\mc{C}$ and $\mr{Tw}_{\phi}^{\mr{op}}\mc{C}$.
 Let $X\rightarrow\mr{Tw}^{\mr{op}}\mc{C}$ be a map simplicial
 sets. By definition, vertices of $\Theta_*(X)\rightarrow\mc{C}$
 (see \ref{defpushpullsim} for the notation)
 over $v$ correspond to functors $\mr{Tw}_v^{\mr{op}}\mc{C}\rightarrow
 X$ over $\mr{Tw}^{\mr{op}}\mc{C}$.
 
\begin{dfn*}
 Let $f\colon\mc{D}\rightarrow\mc{C}^{\mr{op}}$ be a
 coCartesian fibration.
 We define $\mc{D}^{\vee}$ to be the full
 subcategory of $\Theta_*\Phi^*(\mc{D})$ (see \ref{defpushpullsim} for the notation)
 spanned by the vertices $G\colon\mr{Tw}^{\mr{op}}_v\mc{C}\rightarrow\mc{D}$
 over $\mc{C}^{\mr{op}}$ for some $v\in C$
 such that $G$ sends edges of $\mr{Tw}^{\mr{op}}_v\mc{C}$ to
 $f$-coCartesian edges in $\mc{D}$.
\end{dfn*}

Our goal of this section is the following theorem.

\begin{thm}
 \label{dualthm}
 Let $\mc{C}$ be an $\infty$-category and let
 $f\colon\mc{D}\rightarrow\mc{C}^{\mr{op}}$ be a coCartesian
 fibration. Then the map $f^\vee\colon\mc{D}^\vee\rightarrow\mc{C}$ is a
 Cartesian fibration, and this is equivalent to $\mb{D}(f)$.
\end{thm}
\begin{proof}
 Consider the following diagram:
 \begin{equation*}
  \xymatrix{
   \mc{D}^{\mr{op}}\ar[d]&
   \mr{Tw}^{\mr{op}}\mc{D}^{\mr{op}}
   \ar[d]\ar[r]^-{\Phi_\mc{D}}\ar[l]_-{\Theta_\mc{D}}&
   \mc{D}\ar[d]^{f}\\
  \mc{C}&
   \mr{Tw}^{\mr{op}}\mc{C}
   \ar[r]^-{\Phi}\ar[l]_-{\Theta}&
   \mc{C}^{\mr{op}}.
   }
 \end{equation*}
 Consider the subcategory $\widetilde{\mc{D}^{\mr{op}}}$ of
 $\Theta_*\Theta^*(\mc{D}^{\mr{op}})$
 spanned by the vertices
 $\mr{Tw}^{\mr{op}}_v\mc{C}\rightarrow\mc{D}^{\mr{op}}$ over
 $\mc{C}$ such that any edge of $\mr{Tw}^{\mr{op}}_v\mc{C}$ is sent to
 equivalences. We will later show that this category is equivalent to
 $\mc{D}^{\mr{op}}$.
 We define $\mc{M}$ to be the full subcategory of
 $\Theta_*(\mr{Tw}^{\mr{op}}\mc{D}^\mr{op})$
 spanned by the vertices $G\colon\mr{Tw}^{\mr{op}}_{v}\mc{C}\rightarrow
 \mr{Tw}^{\mr{op}}\mc{D}^\mr{op}$ over $\mr{Tw}^{\mr{op}}\mc{C}$
 satisfying the following conditions:
 \begin{itemize}
  \item The composition $\Phi_{\mc{D}}\circ G$ sends edges of
	$\mr{Tw}^{\mr{op}}_v\mc{C}$ to $f$-coCartesian edges;
	
  \item The composition $\Theta_{\mc{D}}\circ G$ sends edges of
	$\mr{Tw}^{\mr{op}}_v\mc{C}$ to equivalences.
 \end{itemize}
 By the first condition, the natural map
 $\mr{Tw}^{\mr{op}}\mc{D}^{\mr{op}}\rightarrow\Phi^*\mc{D}$ induces the
 map $\mc{M}\rightarrow\mc{D}^\vee$, and by the second condition,
 $\mr{Tw}^{\mr{op}}\mc{D}^{\mr{op}}\rightarrow\Theta^*\mc{D}^{\mr{op}}$
 induces the map $\mc{M}\rightarrow\widetilde{\mc{D}^{\mr{op}}}$.
 Thus, we have the map
 $\nu\colon\mc{M}\rightarrow\mc{D}^\vee\times_{\mc{C}}
 \widetilde{\mc{D}^{\mr{op}}}$.

 Let us show that $\mc{M}$, $\mc{D}^\vee$, and
 $\widetilde{\mc{D}^{\mr{op}}}$ are Cartesian fibrations over $\mc{C}$.
 Since the verifications are similar, and that for $\mc{M}$
 is much more complicated than the other two, we concentrate on this.
 Let $\pi\colon\mc{M}\rightarrow\mc{C}$ be the map.
 Take a vertex $m\in\mc{M}$ and an edge $\phi\colon v\rightarrow
 w:=\pi(m)$ in $\mc{C}$. We wish to take a right Kan extension as
 follows:
 \begin{equation*}
  \xymatrix@C=40pt{
   \mr{Tw}^{\mr{op}}_w\mc{C}
   \ar[r]^-{m}\ar@{^{(}->}[d]&
   \mr{Tw}^{\mr{op}}\mc{D}^{\mr{op}}
   \ar[d]^{p}\\
  \mr{Tw}^{\mr{op}}_\phi\mc{C}
   \ar[r]\ar@{-->}[ur]&
   \mr{Tw}^{\mr{op}}\mc{C}.
   }
 \end{equation*}
 In order to apply \cite[4.3.2.15]{HTT} to check the existence, take a
 vertex $C:=(v'\rightarrow v)$ of $\mr{Tw}^{\mr{op}}_\phi\mc{C}$.
 Then $\bigl(\mr{Tw}^{\mr{op}}_w\mc{C}\bigr)_{C/}$ has an initial object
 $C\rightarrow(v'\rightarrow w)$ which can be depicted as
 \begin{equation*}
  \xymatrix{
   v'\ar[d]_{C}&v'\ar[l]_-{=}\ar[d]\\
  v\ar[r]^-{\phi}&w.
   }
 \end{equation*}
 Choose a following diagram
 $D:=\Lambda^3_2\coprod_{\Delta^{\{1,3\}}}\Delta^{\{1,2',3\}}
 \rightarrow\mc{D}^{\mr{op}}$ of the following form:
 \begin{equation*}
  \xymatrix@C=50pt{
   {}_{v}D(1)
   \ar[dd]_{\mlq\mlq\phi^*m(w\rightarrow w)\mrq\mrq}
   ^{\ccirc{3}}
   \ar@{=>}[rr]&&
   {}_{w}D(2)
   \ar[dd]^{m(w\rightarrow w)}_{\ccirc{1}}\\
  &{}_{v'}D(0)
   \ar[rd]_(.4){\simeq m(v'\rightarrow w)}^{\ccirc{2}}
   \ar@{=>}[ul]\ar@{=>}[ur]
   \ar@{-->}[dl]&\\
  {}_vD(2')
   \ar@{=>}[rr]&&
   {}_wD(3).
   }
 \end{equation*}
 Here, ``$\Rightarrow$'' are $f^{\mr{op}}$-Cartesian edges and the big
 outer square is a Cartesian pullback square over $v\rightarrow w$.
 The left subscripts indicate the image of the object in $\mc{C}$
 ({\it e.g.}\ ${}_vD(1)$ is over $v$). The object $m(w\rightarrow w)$ is
 {\it a priori} an object of $\mr{Tw}^{\mr{op}}\mc{D}^{\mr{op}}$, but
 this determines an edge in $\mc{D}^{\mr{op}}$ which yields the edge
 $\ccirc{1}$. The same procedure yields an edge $m(v'\rightarrow w)$,
 and $\ccirc{2}$ is an edge equivalent to this edge. We can take such an
 edge because $\Theta_{\mc{D}}\circ m$ sends edges of
 $\mr{Tw}^{\mr{op}}_w\mc{C}$ to equivalences in $\mc{D}^{\mr{op}}$. The
 edge $\ccirc{3}$ is an edge that should be equivalent to
 $\phi^*m(w\rightarrow w)$ when $\pi$ is shown to be Cartesian.
 
 Since $D\hookrightarrow\Delta^{3}\coprod_{\Delta^{\{0,1,3\}}}
 \Delta^{\{0,1,2',3\}}$ is an inner anodyne and $f^{\mr{op}}$ is an
 inner fibration, we can complete the dashed arrow so that the diagram
 is commutative and the image in $\mc{C}$ is compatible with the map
 $C\rightarrow(v'\rightarrow w)$.
 The diagram $\Delta^{\{0,2',3\}}$ can be considered as a map from
 $(D(0)\rightarrow D(2'))$ to $(D(0)\rightarrow D(3))$ in
 $\mr{Tw}^{\mr{op}}\mc{D}^{\mr{op}}$ over
 $C\rightarrow(v'\rightarrow w)$.
 It suffices to show, by \cite[4.3.1.4]{HTT}, that this edge in
 $\mr{Tw}^{\mr{op}}\mc{D}^{\mr{op}}$ is a $p$-Cartesian edge.
 This follows by Lemma \ref{carttwar}, taking \cite[2.4.1.5]{HTT} into
 account. Applying \cite[B.4.8]{HA}, we know that this edge in $\mc{M}$
 is a $\pi$-Cartesian edge.
 Furthermore, by construction, the map $\nu$ sends Cartesian edges to
 Cartesian edges.
 
 Now, let us check that $\nu$ satisfies the conditions in Lemma
 \ref{conddual}. For this, let us analyze the fibers of $\mc{M}$ over
 $\mc{C}$. Fix a vertex $v\in\mc{C}$. Objects of $\mc{M}_v$ correspond
 to functors $\mr{Fun}_{\mr{Tw}^{\mr{op}}_{\mc{C}}}
 \bigl(\mr{Tw}_v^{\mr{op}}\mc{C},
 \mr{Tw}^{\mr{op}}\mc{D}^{\mr{op}}\bigr)$ satisfying some conditions.
 The map $i\colon\{*\}\rightarrow\mr{Tw}^{\mr{op}}_v\mc{C}$ sending the
 unique object to the object $v\rightarrow v$ yields the map
 $i^*\colon\mc{M}_v\rightarrow\mr{Tw}^{\mr{op}}\mc{D}_v^{\mr{op}}$.
 We show that this is a categorical equivalence.
 Consider the following diagram:
 \begin{equation*}
  \xymatrix@C=40pt{
   \{*\}
   \ar[r]^-{F}\ar@{^{(}->}[d]_{i}&
   \mr{Tw}^{\mr{op}}\mc{D}^{\mr{op}}\ar[d]^{p}\\
  \mr{Tw}^{\mr{op}}_v\mc{C}
   \ar[r]\ar@{-->}[ur]&
   \mr{Tw}^{\mr{op}}\mc{C}.
   }
 \end{equation*}
 Let us show that for any functor $F$, there exists a left Kan extension.
 Take $C=(w\rightarrow v)$ in
 $\mr{Tw}^{\mr{op}}_v\mc{C}$.
 Then $\{*\}_{/C}$ has an initial object
 $\{*\}\rightarrow(w\rightarrow v)$.
 By \cite[4.3.2.15, 4.3.1.4]{HTT}, it suffices to check that the map
 in $\mr{Tw}^{\mr{op}}\mc{D}^{\mr{op}}$ corresponding to the diagram
 \begin{equation*}
  \xymatrix{
   d\ar[d]_{F(*)}&d_w\ar[l]_{\alpha}\ar[d]\\
  d'\ar[r]^-{\sim}&d''}
 \end{equation*}
 in $\mc{D}^{\mr{op}}$, where $\alpha$ is an
 $f^{\mr{op}}$-Cartesian edge over $v\leftarrow w$,
 is a $p$-coCartesian edge. This follows by Lemma \ref{carttwar}.
 By construction, the left Kan extension can be regarded as an object of
 $\mc{M}_v$.
 Invoking \cite[4.3.2.17]{HTT}, $i^*$ admits a left adjoint
 $i_!\colon\mr{Tw}^{\mr{op}}\mc{D}_v^{\mr{op}}\rightarrow\mc{M}_v$. By
 the characterization of left Kan extension functor
 \cite[4.3.2.16]{HTT} and the definition of $\mc{M}$,
 $i_!$ is essentially surjective. Since
 $\mr{id}\xrightarrow{\sim}i^*i_!$, $i_!$ is fully faithful,
 thus, $i_!$ is a categorical equivalence. This implies that $i^*$ is
 also a categorical equivalence because it is so on the level of
 homotopy categories.

 On the other hand, the canonical map
 $\mc{D}^{\mr{op}}\rightarrow\Theta_*\Theta^*\mc{D}^{\mr{op}}$ induces a
 map $\iota\colon\mc{D}^{\mr{op}}\rightarrow
 \widetilde{\mc{D}^{\mr{op}}}$. This map is in fact an
 equivalence. Indeed, since
 $\widetilde{\mc{D}^{\mr{op}}}\rightarrow\mc{C}$ is a Cartesian
 fibration and $\iota$ sends Cartesian edges to Cartesian
 edges, it suffices to check that the fibers are equivalence by
 \cite[3.3.1.5]{HTT}. In order to see the equivalence,
 we may proceed as the proof of the equivalence $i_!$. Likewise, we have
 a canonical equivalence $(\mc{D}^\vee)_v\xrightarrow{\sim}\mc{D}_v$.

 By construction, we have the following commutative diagram of
 $\infty$-categories:
 \begin{equation*}
  \xymatrix@C=50pt{
   \mc{M}_v
   \ar[r]^-{i^*}\ar[d]_{\nu_v}&
   \mr{Tw}^{\mr{op}}\mc{D}_v^{\mr{op}}
   \ar[d]^{\Phi\times\Theta}\\
  (\mc{D}^\vee\times_{\mc{C}}\widetilde{\mc{D}^{\mr{op}}})_v
   \ar[r]&
   \mc{D}_v\times\mc{D}_v^{\mr{op}}.
   }
 \end{equation*}
 Here, the horizontal maps are equivalence. This implies that
 $\nu_v^{\mr{op}}$ is in fact a perfect weak pairing. By the description
 of Cartesian edges in $\mc{M}$, the preservation also holds, and the
 conditions of Lemma \ref{conddual} are satisfied. Thus, we have
 $\mc{D}^\vee\cong\mb{D}((\widetilde{\mc{D}^{\mr{op}}})^{\mr{op}})
 \xleftarrow[\iota]{\sim}\mb{D}(\mc{D})$.
\end{proof}

\section{Stable $R$-linear categories}
\label{stablecat}
We construct the $(\infty,2)$-category of stable $R$-linear categories
for an $\mb{E}_\infty$-ring $R$ ({\em e.g.}\ ordinary commutative ring,
which is more precisely a discrete $\mb{E}_\infty$-ring).
This has already been outlined in \cite[Ch.1, 8.3]{GR}, and the only
contribution of ours is to make the construction rigorous.

\subsection{}
\label{dfnofbcstar}
First, we recall the construction of \cite[4.1]{GH}.
Let $i\colon\Delta^0\rightarrow\mbf{\Delta}^{\mr{op}}$ the map classifying $[0]\in\mbf{\Delta}^{\mr{op}}$.
We may take the right Kan extension functor
$\mr{RKE}_i\colon\Cat_\infty\simeq\mr{Fun}(\Delta^0,\Cat_\infty)\rightarrow\mr{Fun}(\mbf{\Delta}^{\mr{op}},\Cat_\infty)$.
By \cite[4.3.2.17]{HTT}, for a functor $\mc{D}_\bullet\colon\mbf{\Delta}^{\mr{op}}\rightarrow\Cat_\infty$,
we have an equivalence $\mr{Fun}(\mc{D}_\bullet,\mr{RKE}_i(\mc{C}))\simeq\mr{Fun}(\mc{D}_0,\mc{C})$.
If we are given a functor $\mc{E}\rightarrow\mc{D}_0$, we denote
$\mc{D}_\bullet\times_{\mr{RKE}_i(\mc{D}_0)}\mr{RKE}_i(\mc{E})$, where the fiber product
is taken in $\mr{Fun}(\mbf{\Delta}^{\mr{op}},\Cat_\infty)$, by
$\mc{D}_\bullet\basech\mc{E}$. If $\mc{E}\rightarrow\mc{D}_0$ is a
categorical fibration, then $\mc{D}_\bullet\basech\mc{E}$ can be
computed termwise by \cite[5.1.2.3]{HTT}.
Unwinding the construction, we have $(\mc{D}_{\bullet}\basech\mc{E})([n])\simeq\mc{D}_{\bullet}([n])\times_{\mc{D}_0^{\times(n+1)}}\mc{E}^{\times(n+1)}$.
This construction can be regarded as the ``base change'' of $\mc{D}_{\bullet}$.

We can also have coCartesian fibration version of the above construction.
Let $\Gamma'$ be the category with objects $([n],i)$ where $[n]\in\mbf{\Delta}$ and $i\in[n]$,
and a morphism $([n],i)\rightarrow([n'],i')$ is an order-preserving function $\alpha\colon[n']\rightarrow[n]$ such that $\alpha(i')=i$.
Then the evident functor $\gamma'\colon\Gamma'\rightarrow\mbf{\Delta}^{\mr{op}}$ is a Cartesian fibration.
For an $\infty$-category $\mc{C}$ and using the notation of \ref{defpushpullsim}, let $\mc{C}^{\times}:=\gamma'_*(\Gamma'\times\mc{C})$.
Then by \ref{constsimplcat} (or by direct computation), $\mc{C}^{\times}$ is a model of the unstraightening of $\mr{RKE}_i(\mc{C})$.
For a coCartesian fibration $\mc{X}\rightarrow\mbf{\Delta}^{\mr{op}}$ and a map $\mc{C}\rightarrow\mc{X}_0$,
we put $\mc{X}\basech\mc{C}:=\mc{X}\times_{\mc{X}_0^{\times}}\mc{C}^{\times}$ in $\coCart(\mbf{\Delta}^{\mr{op}})$.
We also have a version for Cartesian fibration over $\mbf{\Delta}$.
All of these constructions are compatible via straightening/unstraightening constructions.

Finally, by \cite[4.1.3]{GH},
$\mc{C}^{\times}\rightarrow\mbf{\Delta}^{\mr{op}}$ is a generalized
$\infty$-operad. If $\mc{X}$ is a generalized $\infty$-operad,
$\mc{X}\basech\mc{C}$ is a generalized $\infty$-operads as well.
Given a map of generalized $\infty$-operads
$\mc{C}^{\circledast}\rightarrow\mc{D}^{\circledast}$ and a functor
$\mc{C}_0\rightarrow\mc{E}$ of $\infty$-categories over $\mc{D}_0$, we
have the induced map
$\mc{C}^{\circledast}\rightarrow\mc{D}^{\circledast}\basech\mc{E}$.

\subsection{}
Let $F\colon\mc{C}\rightarrow\Cat_\infty$ be a functor.
Applying the construction of \ref{constsimplcat} for
$\mc{D}=\Cat_\infty$, we have a functor
$Y_F:=\mr{Fun}(F,\Cat_\infty)\colon\mc{C}^{\mr{op}}
\rightarrow\widehat{\Cat}_\infty$, where $\widehat{\Cat}_\infty$ is the
$\infty$-category of (not necessarily small) $\infty$-categories (cf.\
\cite[3.0.0.5]{HTT}).
Recall that $Y_F$ is the functor sending $c\in\mc{C}$ to
$\mr{Fun}(F(c),\Cat_\infty)$.
On the other hand, recall from \ref{cartfibcocartnot} that we have the
Cartesian fibration $\Cart^{\mr{str}}_\infty\rightarrow\Cat_\infty$.
This induces the Cartesian fibration
$F'\colon\Cart^{\mr{str}}_\infty\times_{\Cat_\infty,F}\mc{C}
\rightarrow\mc{C}$. We define
$Y'_F:=\mr{St}(F')\colon\mc{C}^{\mr{op}}\rightarrow
\widehat{\Cat}_\infty$ (cf.\ \ref{cartfibcocartnot}).
The following lemma enables us to identify these two constructions:

\begin{lem*}
 \label{comptwoconst}
 We have a canonical equivalence $Y_F\simeq Y'_F$ of functors.
\end{lem*}
\begin{proof}
 For $i\in\{0,1\}$, let
 \begin{equation*}
  G_i\colon\mc{C}^{\mr{op}}\xrightarrow{F^{\mr{op}}}
   \Cat_\infty^{\mr{op}}
   \xrightarrow{\chi}
   \mr{Fun}(\Delta^1,\widehat{\Cat}_\infty)
   \xrightarrow{\{i\}\rightarrow\Delta^1}
   \widehat{\Cat}_\infty,
 \end{equation*}
 where $\chi$ is the map defined in \cite[A.32]{GHN}.
 Informally, $\chi$ is the functor sending $\mc{C}\in\Cat_\infty$ to the
 unstraightening equivalence
 $\mr{Fun}(\mc{C}^{\mr{op}},\Cat_\infty)\xrightarrow{\sim}
 \Cart^{\mr{str}}(\mc{C})$.
 Because unstraightening is an equivalence, we have $G_0\simeq G_1$.
 By construction, $Y_F$ is equivalent to $G_0$, thus it remains to show
 that $G_1\simeq Y'_F$. It suffices to show the equivalence for
 $\mc{C}=\Cat_\infty$.

 For a relative category $(\mc{C},W)$, we denote by $\mc{L}(\mc{C},W)$
 the $\infty$-localization (cf.\ \cite[1.1.2]{H}).
 Consider a Cartesian fibration
 $r\colon(\mc{M},W_{\mc{M}})\rightarrow(\mc{C},W_{\mc{C}})$ of relative
 categories in the sense of \cite[2.1.1]{H}.
 We note that this condition is slightly different from the relative
 Grothendieck fibration compatible with $W_{\mc{C}}$ in the sense of
 \cite[A.28]{GHN}, since \cite[A.28]{GHN} requires that {\em all} the
 $r$-Cartesian morphisms are in $W_{\mc{M}}$ whereas \cite{H} asks only
 for $r$-Cartesian morphisms lifting morphisms in $W_{\mc{C}}$ but
 $W_{\mc{C}}$ needs to be saturated (cf.\ \cite[1.1.2]{H}). However
 the construction of \cite[A.30]{GHN} can be carried out for Hinich's
 one\footnote{We can also make use of Hinich's construction
 \cite[2.2.2]{H} instead of \cite[A.30]{GHN}, which is very similar in
 spirit.}
 as well. Namely, the functor $r$ corresponds to a normal pseudo-functor
 $\mr{St}(r)\colon\mc{C}\rightarrow\mr{RelCat}_{(2,1)}$, and yields an
 $\infty$-functor
 $\mr{St}(r)_\infty\colon\mc{L}(\mc{C},W_{\mc{C}})\rightarrow
 \Cat_\infty$ by \cite[A.25]{GHN}.
 Since the straightening/unstraightening construction of Lurie is
 compatible with Grothendieck construction, we have the
 commutative diagram of $\infty$-categories
 \begin{equation*}
 \xymatrix{
  \mr{N}\mc{M}\ar[r]\ar[d]&\mr{Un}(\mr{St}(r)_\infty)\ar[d]\\
 \mr{N}\mc{C}\ar[r]&\mc{L}(\mc{C},W_{\mc{C}}).
  }
 \end{equation*}
 This diagram induces a map
 $\mc{L}(\mc{M},W_{\mc{M}})\rightarrow\mr{Un}(\mr{St}(r)_\infty)$ over
 $\mc{L}(\mc{C},W_{\mc{C}})$. This map is nothing but the functor
 $\theta$ in \cite[2.2.2 (34)]{H}, which is proved to be categorical
 equivalence in \cite{H}.

 We let $q\colon X\rightarrow\sSet$ be
 the pullback of the Grothendieck fibration
 $\mbf{E}\rightarrow\sSet\times\Delta^1$, defined in \cite[A.31]{GHN},
 by the map $\sSet\rightarrow\sSet\times\Delta^1$ defined by
 $\{0\}\rightarrow\Delta^1$.
 Explicitly, $X$ is the category\footnote{In the 2nd line of the proof
 of \cite[A.31]{GHN}, they say that $Y\rightarrow S^{\sharp}$ is a
 fibrant map in $\sSet^+$. We think this is a typo, and this should be
 replaced by ``a fibrant map in $(\sSet^+)_{/S}$''.
 }
 whose fiber over $S\in\sSet$ is $(\sSet^+)_{/S}^{\circ}$. Objects of
 $(\sSet^+)_{/S}^{\circ}$ can be written as $A^{\natural}\rightarrow
 S^\sharp$ where $A\rightarrow S$ is a Cartesian fibration by
 \cite[3.1.4.1]{HTT}. Given Cartesian fibrations $A\rightarrow S$ and
 $B\rightarrow T$, a map $f$ from $B^{\natural}\rightarrow T^{\sharp}$
 to $A^{\natural}\rightarrow S^{\sharp}$ in $X$ over $T\rightarrow S$ in
 $\sSet$ is the map of marked simplicial sets $B^{\natural}\rightarrow
 A^{\natural}$ compatible with $T\rightarrow S$.
 We slightly modify the marking of $X$ from \cite{GHN}:
 the map $f$ in $X$ is marked if $B\rightarrow A$ and $T\rightarrow S$
 are categorical equivalences. Note that when $S=T$ and $f$ is marked,
 the map $B\rightarrow A$ is a Cartesian equivalence by
 \cite[3.3.1.5]{HTT}, so the relative category of the fiber is
 $((\sSet^+)_{/S},W_S)$ where $W_S$ is the categorical equivalence.
 Moreover, given a categorical equivalence $X\rightarrow Y$ between
 Cartesian fibrations over $S$ and a map $T\rightarrow S$, the base
 change $X\times_S T\rightarrow Y\times_S T$ is a categorical
 equivalence by \cite[3.3.1.5]{HTT}. Combining
 with \cite[3.3.1.3]{HTT}, all the conditions of \cite[2.1.1]{H} are
 satisfied except for the saturatedness of the marking.

 We denote by $[n]$ the category whose nerve is $\Delta^n$.
 Next, we consider the relative category $((\sSet)^{[1]},W')$ where a map
 $(f,g)\colon(X\rightarrow Y)\rightarrow (X'\rightarrow Y')$ of
 $(\sSet)^{[1]}$ is in $W'$ precisely if $f,g\in W_{\mr{J}}$,
 where $W_{\mr{J}}$ is the collection of Joyal equivalent maps of
 $\sSet$.
 Then we have a map $X\rightarrow((\sSet)^{[1]},W')$ of relative
 categories which induces the map
 $X\rightarrow\mc{L}((\sSet)^{[1]},W')\simeq
 \mr{Fun}(\Delta^1,\Cat_\infty)$.
 Here, the equivalence follows by composing the equivalences
 \begin{equation*}
  \mc{L}((\sSet)^{[1]},W')
   \simeq
   \mc{L}((\sSet,W_{\mr{J}})^{[1]})
   \simeq
   \mc{L}((\sSet^+,W^+)^{[1]})
   \simeq
   \mr{Fun}(\Delta^1,\Cat_\infty),
 \end{equation*}
 where $W^+$ denotes the collection of Cartesian equivalences of
 marked simplicial sets, the middle two model categories are endowed
 with projective model structures, the middle equivalence follows by
 \cite[1.5.1]{H}, and the last equivalence follows by
 \cite[4.2.4.4]{HTT}.
 This implies that the
 marking of $X$ is saturated (cf.\ \cite[1.1.2]{H}) because maps in $X$
 is marked precisely when its image in $\mr{Fun}(\Delta^1,\Cat_\infty)$
 are equivalence. Thus, $q$ is a Cartesian fibration of relative
 categories. Consider the following diagram:
 \begin{equation*}
  \xymatrix@C=10pt{
   \mr{Un}(\mr{St}(q)_\infty)\ar[dr]_{\mr{Un}(G_1)}&
   \mc{L}(X)\ar[r]\ar[d]^{\mc{L}(q)}\ar[l]_-{\sim}&
   \mc{L}(\sSet^{[1]})\ar[d]\ar@{-}[r]^-{\sim}&
   \mr{Fun}(\Delta^1,\Cat_\infty)\ar[d]&
   \Cart_\infty\ar[l]\ar[dl]^-{\mr{Un}(Y'_F)}\\&
  \mc{L}(\sSet)\ar@{=}[r]&
   \mc{L}(\sSet)\ar@{-}[r]^-{\sim}&
   \Cat_\infty.&
   }
 \end{equation*}
 In view of the above observation and functoriality, this diagram is
 commutative, and all the vertical maps are Cartesian fibrations.
 The map
 $\mr{Un}(\mr{St}(q)_\infty)\rightarrow\mr{Fun}(\Delta^1,\Cat_\infty)$
 preserves Cartesian edges since
 $\mc{L}(X)\rightarrow\mc{L}(\sSet^{[1]})$ preserves Cartesian edges by
 the construction in \cite[2.2.2]{H}.
 Finally, since $\mr{Un}(G_1)$
 and $\mr{Un}(Y'_F)$ are equivalent over each fiber of $\Cat_\infty$, we
 get $G_1\simeq Y'_F$ by \cite[3.1.3.5]{HTT} as required.
\end{proof}

\subsection{}
Before constructing the $(\infty,2)$-category of $\mc{A}$-linear
categories, we recall the definition of the $(\infty,2)$-category of
$\infty$-categories $\mbf{Cat}_\infty$ since the construction is a
prototype of the construction of $\twoLinCat$. Recall the functor
$\theta\colon\widehat{\Cat}_\infty\rightarrow\widehat{\Spc}$ from
\ref{maxminKancpx} associating an to $\infty$-category $\mc{C}$ the
maximum Kan complex $\mc{C}^{\simeq}$.
Let $\Delta^{\bullet}\colon\mbf{\Delta}\rightarrow\Cat_\infty$ be the
evident functor sending $[n]$ to $\Delta^n$.
We have an equivalence $\mr{Seq}_{\bullet}(\Cat_\infty)\simeq\theta\circ
Y_{\Delta^{\bullet}}\colon\mbf{\Delta}^{\mr{op}}\rightarrow
\widehat{\Spc}$, where $\mr{Seq}_{\bullet}$ is the functor defined in
\ref{inftwocatintro}, by definition. We upgrade this construction by
letting
$\mr{Seq}_{\bullet}(\mbf{Cat}_\infty):=Y'_{\Delta^{\bullet}}\colon
\mbf{\Delta}^{\mr{op}}\rightarrow\widehat{\Cat}_\infty$.

\begin{prop*}[{\cite[Ch.10, 2.4.2]{GR}}]
 The simplicial $\infty$-category $\mr{Seq}_{\bullet}(\mbf{Cat}_\infty)$
 defines an $(\infty,2)$-category such that the underlying
 $\infty$-category is $\Cat_\infty$.
\end{prop*}
\begin{proof}
 The Segal condition holds by \ref{cartsegco}. We need to show the
 completeness. For this, it suffices to show that the associated Segal
 space $\mr{Seq}_{\bullet}(\mbf{Cat}_\infty)^{\simeq}$ is complete.
 By Lemma \ref{comptwoconst}, this Segal space is naturally equivalent
 to $\mr{Seq}_{\bullet}(\Cat_\infty)$, thus complete.
\end{proof}

\subsection{}
Now, we move to the definition of the $(\infty,2)$-category of
$\mc{A}$-linear stable categories.

\begin{dfn*}
 Let $\mbf{\Delta}_+$ be the augmented simplex category.
 For a simplicial set $S$, we define $\mr{RM}_S$ to be the simplicial
 subset of $S^{\triangleright}\times\mbf{\Delta}^{\mr{op}}_+$ spanned by
 all vertex but $(\infty,[-1])$, where $\infty\in S^{\triangleright}$ is
 the cone point. For $s\in S$, the vertex $(s,[n])$ is denoted by
 $(0_s,\underbrace{1,\dots,1}_{n+1})$ for $n\geq-1$, and $(\infty,[n])$
 is denoted by $(\underbrace{1,\dots,1}_{n+1})$.
\end{dfn*}

If $S$ is an $\infty$-category, $\mr{RM}_S$ is an $\infty$-category as
well. The construction of $\mr{RM}_S$ is functorial with respect to
$S$.

\begin{lem}
 \label{baRMpro}
 \begin{enumerate}
  \item\label{baRMpro-1}
       We have a canonical isomorphism $\mr{RM}_{\Delta^0}\cong\mr{RM}$ where $\mr{RM}$ is the
       category defined in {\normalfont\cite[7.1.4]{GH}}\footnote{
       We think that in \cite[7.1.1]{GH}, we should use $\mr{Simp}(\Delta^1)^{\mr{op}}$ instead of $\mr{Simp}(\Delta^1)$.}.
	
  \item\label{baRMpro-3}
       If $S$ is an $\infty$-category, then the map
       $\mr{RM}_S\rightarrow\mr{RM}$ is a coCartesian
       fibration\footnote{In \cite[7.1.4]{GH}, it is said that $\mr{RM}$
       is a double $\infty$-category, which implies that the map
       $\mr{RM}\rightarrow\mbf{\Delta}^{\mr{op}}$ is a coCartesian
       fibration. Unlike $\mr{BM}$, which is indeed a double
       $\infty$-category, we think that $\mr{RM}$ is not.
       Indeed, since there is no map from $(0)\in\mr{RM}$, the map
       $[0]\rightarrow[1]$ in $\mbf{\Delta}^{\mr{op}}$ cannot be lifted
       to a map from $(0)\in\mr{RM}$.
       However, only the fact that $\mr{RM}$ is a
       generalized $\infty$-operad is used in \cite{GH}.
       This can be checked as follows (or direct computation):
       The conditions (i), (ii) of \cite[2.2.6]{GH} are easy to check.
       The condition (iii) follows since $\mr{BM}$ is a generalized
       $\infty$-operad and the embedding $\mr{RM}\rightarrow\mr{BM}$ is
       fully faithful and preserves inert edges.}
       of generalized $\infty$-operads.

  \item\label{baRMpro-2}
       Let $\mf{a}\,(\cong\mbf{\Delta}^{\mr{op}})$ be the fiber over $[1]\in\Delta^1$ of the Cartesian
       fibration $\mr{RM}\rightarrow\Delta^1$ sending $(a_0,\dots,a_n)$ to $[a_0]$.
       The map of generalized $\infty$-operads $(S\times\mr{RM})\coprod_{S\times\mf{a}}\mf{a}\rightarrow\mr{RM}_S$
       is an equivalence (in $\mr{Op}^{\mr{ns},\mr{gen}}_{\infty}$).
       In particular, if $T\rightarrow S$ is a cofibration of simplicial sets and $T\rightarrow S'$ is a
       map, the map $\mr{RM}_S\coprod_{\mr{RM}_T}\mr{RM}_{S'}\rightarrow\mr{RM}_{S\coprod_T S'}$ is a categorical equivalence.
 \end{enumerate}
\end{lem}
\begin{proof}
 To see \ref{baRMpro-1}, we have the functor
 $\mr{RM}\rightarrow\Delta^1\times\mbf{\Delta}^{\mr{op}}_+$ by sending
 $(a_0,\dots,a_n)$ ($a_i\in\{0,1\}$) to $(0,[n-1])$ if $a_0=0$ and to
 $(1,[n])$ if $a_0=1$. It is easy to check that this induces an
 isomorphism we need. Via this identification, we see that the notation
 $(0_s,1\dots,1)$ is compatible with that of $\mr{RM}$. Note that the
 map $\mr{RM}\rightarrow\Delta^1$ is a Cartesian fibration.
 Let us check \ref{baRMpro-3}.
 We have isomorphisms of simplicial sets
 \begin{equation}
  \label{compRMope1}\tag{$\star$}
  \mr{RM}_S
   \cong
   (S^{\triangleright}\times\mbf{\Delta}^{\mr{op}}_+)
   \times_{(\Delta^1\times\mbf{\Delta}^{\mr{op}}_+)}
   \mr{RM}
   \cong
   S^{\triangleright}\times_{\Delta^1}\mr{RM}.
 \end{equation}
 Since
 $S^{\triangleright}\times\mbf{\Delta}^{\mr{op}}_+\rightarrow
 (\Delta^0)^{\triangleright}\times\mbf{\Delta}^{\mr{op}}_+$ is a
 coCartesian fibration, and since the map $\mr{RM}_S\rightarrow\mr{RM}$
 is the base change of this map by the isomorphisms above,
 it is coCartesian as well. In order to show that $\mr{RM}_S$ is a
 generalized $\infty$-operad, we only need to check the Segal
 condition by (non-symmetric analogue of) \cite[2.1.2.12]{HA}.
 The verification is straightforward.

 Finally, let us prove \ref{baRMpro-2}.
 We use the theory of categorical patterns \cite[\S B]{HA}.
 Let $\mf{P}$ be the categorical pattern
 $(\mc{E}_{\mr{int}},\mr{all},\{\mc{G}^{\mbf{\Delta}}_{[n]/}
 \rightarrow\mbf{\Delta}^{\mr{op}}\}_{n})$ where $\mc{E}_{\mr{int}}$ is
 the set of inert maps and $\mc{G}^{\mbf{\Delta}}_{[n]/}$ is the
 simplicial set defined in \cite[2.3.1]{GH}.
 The associated $\infty$-category is
 $\mr{Op}^{\mr{ns},\mr{gen}}_{\infty}$ by \cite[3.2.9]{GH}.
 Since $S\times\mf{a}\rightarrow S\times\mr{RM}$ is a cofibration in
 $(\sSet^+)_{/\mf{P}}$, the pushout is a homotopy pushout.
 For a generalized $\infty$-operad $\mc{O}^{\circledast}$, let
 $\overline{\mc{O}^{\circledast}}$ be the marked simplicial set
 $(\mc{O}^{\circledast},\mc{E}_{\mc{O}})$ where $\mc{E}_{\mc{O}}$ the
 set of inert edges.
 We have an isomorphism of simplicial sets
 \begin{equation}
  \label{compRMope2}\tag{$\star\star$}
  (S^{\flat}\times\overline{\mr{RM}})
   \coprod_{S^{\flat}\times\overline{\mf{a}}}\overline{\mf{a}}
   \cong
   \bigl(
   (S^{\flat}\times(\Delta^1)^{\sharp})
   \coprod_{S^{\flat}\times\{1\}^{\flat}}\{1\}^{\flat}
   \bigr)
   \times_{(\Delta^1)^{\sharp}}\overline{\mr{RM}}
   \cong
   (S\diamond\Delta^0,\mc{E})\times_{(\Delta^1)^{\sharp}}
   \overline{\mr{RM}},
 \end{equation}
 where $\mc{E}$ is the marking induced by
 $(S^{\flat}\times(\Delta^1)^{\sharp})
 \coprod_{S^{\flat}\times\{1\}^{\flat}}\{1\}^{\flat}$,
 and
 the first isomorphism holds since for any (marked simplicial)
 sets $B,C,D,A'$ over $A$ we have
 $(B\times_AA')\coprod_{(C\times_AA')}(D\times_AA')\cong
 (B\coprod_C D)\times_A A'$.
 We wish to apply \cite[B.4.2]{HA} to the following diagram of marked
 simplicial sets
 \begin{equation*}
  (\Delta^1)^{\sharp}\xleftarrow{\pi}
   \overline{\mr{RM}}
   \xrightarrow{\pi'}
   \overline{\mbf{\Delta}^{\mr{op}}}.
 \end{equation*}
 We consider the categorical pattern
 $\mf{Q}:=(\mr{all},\mr{all},\emptyset)$ on $\Delta^1$ and the
 categorical pattern $\mf{P}$ on $\mbf{\Delta}^{\mr{op}}$. Note that
 $(\sSet^+)_{/\mf{Q}}$ is the coCartesian model structure over
 $\Delta^1$ by \cite[B.0.28]{HA}.
 Then all the conditions of \cite[B.4.2]{HA} are satisfied:
 (5), (6), (8) hold since, in our situation,
 $A=\emptyset$, (1), (4) hold since $\mr{RM}\rightarrow\Delta^1$ is a
 Cartesian fibration, (3) holds since we are taking the set of all
 2-simplices, (2), (7) are easy to check. Thus, the functor
 $\pi'_!\circ\pi^*$ is a left Quillen functor. In particular, it
 preserves weak equivalences because any object is cofibrant.
 By presentations (\ref{compRMope1}) and (\ref{compRMope2}), it remains
 to show that the map $(S\diamond\Delta^0,\mc{E})\rightarrow
 (S^{\triangleright},\mc{E}')$, where $\mc{E}'$ is the union of
 degenerate edges and the edges lying over the unique non-degenerate
 edge of $\Delta^1$, is a coCartesian equivalence.
 By a similar argument to \cite[4.2.1.2]{HTT}, we are reduced to
 checking the equivalence in the cases where $S=\Delta^0,\Delta^1$. For
 $S=\Delta^0$, it is in fact an isomorphism, and for $S=\Delta^1$, we
 can construct a simplicial homotopy. Since $(\sSet^+)_{/\Delta^1}$ is a
 simplicial model category by \cite[3.1.4.4]{HTT}, simplicially
 homotopic objects are weakly equivalent, and the equivalence follows.
 The second claim of \ref{baRMpro-2} readily follows from the first
 one.
\end{proof}

\begin{dfn}
 \label{dfnofRModandoth}
 Let $\mc{C}$ be an $\infty$-category.
 \begin{enumerate}
  \item We define $\RMod_\mc{C}$ to be the full subcategory of
	\begin{equation*}
	 \mr{Fun}(\Delta^1,\coCart^{\mr{str}}(\mr{RM}))
	  \times_{\mr{ev}_{\{1\}},\coCart^{\mr{str}}(\mr{RM})}
	  \{\mr{RM}_{\mc{C}}\rightarrow\mr{RM}\},
	\end{equation*}
	where $\mr{ev}_{\{1\}}$ denotes the map evaluating at $1\in\Delta^1$,
	spanned by (homotopy commutative) diagram of $\infty$-categories
	\begin{equation*}
	 \xymatrix{
	  \mc{M}^{\circledast}\ar[r]^-{r}\ar[d]&
	  \mr{RM}_{\mc{C}}\ar[d]\\
	 \mr{RM}\ar@{=}[r]&\mr{RM}
	  }
	\end{equation*}
	such that $r$ is a base preserving (cf.\ Definition
	\ref{operadsrecall}) coCartesian fibration of
	generalized $\infty$-operads.

  \item Let $\rho\colon\mbf{\Delta}\xrightarrow{\Delta^{\bullet}}
	\Cat_\infty\rightarrow\coCart^{\mr{str}}(\mr{RM})$, where the
	second functor sends $\mc{C}$ to
	$\mr{RM}_{\mc{C}}\rightarrow\mr{RM}$.
	Consider the following diagram
	\begin{equation}
	 \label{dfnofRModandoth-diag}
	 \xymatrix{
	  \bp\RMod_{\mbf{\Delta}}
	  \ar@{^{(}->}[r]&
	  \mc{X}\ar[r]\ar[d]\ar@{}[rd]|\square&
	  \mr{Fun}(\Delta^1,\coCart^{\mr{str}}(\mr{RM}))\ar[d]^{\mr{ev}_{\{1\}}}\\
	 &\mbf{\Delta}\ar[r]^-{\rho}&\coCart^{\mr{str}}(\mr{RM}).
	  }
	\end{equation}
	We define $\bp\RMod_{\mbf{\Delta}}$ to be the full subcategory
	of $\mc{X}$ spanned by objects in $\RMod_{\Delta^n}$ over
	$[n]\in\mbf{\Delta}$.

  \item We put
	$\mr{alg}\colon\bp\RMod_{\mbf{\Delta}}\xrightarrow
	{\Delta^{\{0\}}\rightarrow\Delta^1}
	\coCart^{\mr{str}}(\mr{RM})\rightarrow\coCart^{\mr{str}}(\mf{a})$,
	where the second functor is induced by the base change by the
	inclusion $\mf{a}\rightarrow\mr{RM}$.
 \end{enumerate}
\end{dfn}

Let $\mc{C}$ be an $\infty$-category, $c\in\mc{C}$ be an object, and $i_c\colon\{c\}\hookrightarrow\mc{C}$ be the canonical functor.
For $(\mc{M}^{\circledast}\rightarrow\mr{RM}_{\mc{C}})\in\RMod_{\mc{C}}$ the induced map
$\mc{M}^{\circledast}\times_{\mr{RM}_{\mc{C}},i_c}\mr{RM}_{c}\rightarrow\mr{RM}_c\simeq\mr{RM}$ is a pseudo-enriched $\infty$-category in the
sense of \cite[7.2.5]{GH} (cf.\ \cite[7.2.8]{GH}).
For a coCartesian fibration $\mc{M}^{\circledast}\rightarrow\mr{RM}$ of generalized $\infty$-operads,
we sometimes denote $\mc{M}^{\circledast}_{(0,1)}$ by $\mc{M}$, and call it the {\em underlying $\infty$-category}.
Let $\mc{A}:=\mc{M}^{\circledast}\times_{\mr{RM}}\mf{a}$.
For an object $(0,1,\dots,1)\in\mr{RM}$ over $[n+1]\in\mbf{\Delta}^{\mr{op}}$,
we have an equivalence $\mc{M}^{\circledast}_{(0,1\dots,1)}\simeq\mc{M}\times\mc{A}^{\times n}$.
With this identification, an object $X$ of $\mc{M}^{\circledast}_{(0,1,\dots,1)}$ can be written as $X=(M_0,A_1,\dots,A_n)$.
This object is often denoted by $M_0\boxtimes A_1\boxtimes\dots\boxtimes A_n$.

\begin{rem*}
 \begin{enumerate}
  \item The back-prime $\bp(-)$ is put to indicate that the object is
	Cartesian over $\mbf{\Delta}$. When we take $\mb{D}$ of
	\S\ref{dualconst}, we erase the back-prime to indicate that it
	is a coCartesian fibration.

  \item The reason we employed {\em right} module rather than {\em left}
	module is the same as \cite[7.2.13]{GH}. However, in our
	application, we restrict our attention to modules over
	$\mb{E}_\infty$-ring, in which case the $\infty$-category of
	right and left modules can be identified (cf.\
	\cite[D.1.2.5]{SAG}).
 \end{enumerate}
\end{rem*}

\begin{lem}
 \label{basicproprmodsimp}
 \begin{enumerate}
  \item Let $\mc{C}$ be an $\infty$-category, and consider the diagram
	\begin{equation}
	 \label{basicproprmodsimp-diag}\tag{$\star$}
	 \xymatrix{
	  \mc{M}^{\circledast}
	  \ar[dr]_{p}\ar[rr]^-{r}&&
	  \mc{N}^{\circledast}\ar[ld]^{q}\\
	 &\mr{RM}_{\mc{C}}\ar[r]_-{s}&\mr{RM}
	  }
	\end{equation}
	where $p$ and $q$ are coCartesian fibrations of generalized
	$\infty$-operads. Furthermore, assume that for each $x\in\mc{C}$, the
	pullback diagram is in $\coCart^{\mr{str}}(\mr{RM})$.
	Then $r$ sends $(s\circ	p)$-coCartesian edge to $(s\circ
	q)$-coCartesian edge.
	 
  \item\label{basicproprmodsimp-2}
       The map $\alpha\colon\bp\RMod_{\mbf{\Delta}}\rightarrow\mbf{\Delta}$ is a Cartesian fibration,
       and satisfies the Segal condition.
       Moreover, the map $\mr{alg}$ sends an $\alpha$-Cartesian edge to an equivalent edge.
 \end{enumerate}
 \end{lem}
\begin{proof}
 First note that $s$ is a coCartesian fibration by Lemma \ref{baRMpro}.
 Let $e$ be an $(s\circ p)$-coCartesian edge. We wish to show that $r(e)$
 is an $(s\circ q)$-coCartesian edge. Note that we are allowed to
 replace $e$ by an edge equivalent to it, since being a coCartesian edge
 is preserved by equivalence. Since $s\circ q$ is a coCartesian
 fibration, it suffices to show that $r(e)$ is a {\em locally} $(s\circ
 q)$-coCartesian edge by \cite[2.4.2.8]{HTT}. Since $e$ is an
 $(s\circ p)$-coCartesian edge, $p(e)$ is an $s$-coCartesian edge.
 This implies that, by replacing $e$ by its equivalent edge, we may
 assume that there exists $x\in\mc{C}$ such that $p(e)$ sits
 inside $\mr{RM}_{x}$ in $\mr{RM}_{\mc{C}}$.
 Thus, it suffices to show that
 \begin{equation*}
  \mc{M}^{\circledast}\times_{\mr{RM}_{\mc{C}}}
   \mr{RM}_{x}
   \rightarrow
   \mc{N}^{\circledast}\times_{\mr{RM}_{\mc{C}}}
   \mr{RM}_{x}
 \end{equation*}
 preserves coCartesian edges over $\mr{RM}_{x}$.
 This follows by assumption.

 Let us show the second claim. We first show that it is a Cartesian
 fibration. By \cite[2.3.2.5]{HTT} and the fact that a fully faithful
 inclusion is an inner fibration, the map is an inner fibration.
 Because any base preserving coCartesian fibration of generalized $\infty$-operad
 is stable by base change of generalized $\infty$-operad, we get the claim.
 By construction, the claim for $\mr{alg}$ follows as well.
 
 We are left to show the Segal condition.
 Let $\RMod_{\mc{C}}^{\sim}$ be the subcategory of
 $\coCart(\mr{RM}_{\mc{C}})$ spanned by simplices
 $\Delta^n\rightarrow\coCart(\mr{RM}_{\mc{C}})$ all of whose vertices
 $\mc{M}^{\circledast}\rightarrow\mr{RM}_{\mc{C}}$ are base preserving
 coCartesian fibration of generalized $\infty$-operads,
 and all of whose edges are of the form (\ref{basicproprmodsimp-diag})
 such that the base change to $\mr{RM}_x$ for any $x\in\mc{C}$ is
 in $\coCart^{\mr{str}}(\mr{RM}_x)$. Then the evident map
 $\theta\colon\RMod_{\mc{C}}\rightarrow\RMod_{\mc{C}}^{\sim}$ is a
 trivial fibration. Indeed, let
 $D:=(\Delta^1\times\Delta^1)\coprod_{\Delta\times\{1\}}\{*\}$, and
 $D':=\Delta^1\times\{0\}\cup\{1\}\times\Delta^1$.
 Then $D'^{\flat}\rightarrow D^{\flat}$ is a
 (Cartesian) marked anodyne.
 This implies that the map
 $\theta'\colon\mr{Map}^{\flat}(D^{\flat},\Cat_\infty^{\natural})
 \rightarrow\mr{Map}^{\flat}(D'^{\flat},\Cat_\infty^{\natural})$ is a
 trivial fibration by \cite[3.1.2.3]{HTT}.
 This map is isomorphic to
 $\mr{Fun}(D,\Cat_\infty)\rightarrow\mr{Fun}(D',\Cat_\infty)$.
 We have the inclusion
 \begin{equation*}
  \RMod_{\mc{C}}^{\sim}\hookrightarrow
   \mr{Fun}(\Delta^1,\Cat_\infty)\times_{\{1\},\Cat_\infty}
   \{\mr{RM}_{\mc{C}}\}
   \hookrightarrow
   \mr{Fun}(D',\Cat_\infty)
 \end{equation*}
 where the second map sends $F\colon\mc{X}\rightarrow\mr{RM}_{\mc{C}}$
 to
 $\mc{X}\xrightarrow{F}\mr{RM}_{\mc{C}}\rightarrow\mr{RM}$. Similarly,
 $\RMod_{\mc{C}}$ can be viewed as a subcategory of
 $\mr{Fun}(D,\Cat_\infty)$.
 In view of the first claim, $\theta$ is a base change of $\theta'$,
 thus $\theta$ is a trivial fibration as well.

 Thus, in order to check the Segal condition for $\RMod$, it
 suffices to show that the canonical functor
 \begin{equation*}
  \RMod_{\Delta^{m}}^{\sim}\rightarrow
   \RMod_{\Delta^{\{0,\dots,n\}}}^{\sim}
   \times^{\mr{cat}}_{\RMod_{\Delta^{\{n\}}}^{\sim}}
   \RMod_{\Delta^{\{n,\dots,m\}}}^{\sim}
 \end{equation*}
 is a categorical equivalence. Since $\RMod_S^{\sim}$ is a
 subcategory of $\coCart(\mr{RM}_S)$, we only need to check
 the conditions of Lemma \ref{critsubseg}.
 The Segal map is an equivalence for $\coCart(\mr{RM}_S)$ by Lemma
 \ref{baRMpro}.\ref{baRMpro-2} and Lemma \ref{cartsegco}.
 For the rest, it suffices to show the following assertions:
 \begin{enumerate}
  \item An object $\mc{M}^{\circledast}\in\coCart(\mr{RM}_{\Delta^n})$
	is in $\RMod_{\Delta^n}^{\sim}$ if and only if the restriction
	$\iota_i^*\mc{M}^{\circledast}$ belongs to
	$\RMod_{\Delta^{\{i\}}}^{\sim}$ for any $i$.
	Here, $\iota_i\colon\mr{RM}_{\Delta^{\{i\}}}\rightarrow
	\mr{RM}_{\Delta^n}$ is the canonical functor;
	
  \item Given $f\colon\mc{M}\rightarrow\mc{N}$ in
	$\coCart(\mr{RM}_{\Delta^n})$ such that
	$\mc{M},\mc{N}\in\RMod_{\Delta^n}^{\sim}$, $f$ is a map in
	$\RMod_{\Delta^n}^{\sim}$ if and only if $\iota_i^*(f)$ is in
	$\RMod_{\Delta^{\{i\}}}^{\sim}$ for any $i$.
 \end{enumerate}
 The second assertion follows by the definition of
 $\RMod_{\mc{C}}^{\sim}$.
 Let us show the first assertion. The coCartesian fibration
 $p\colon\mc{M}^{\circledast}\rightarrow\mr{RM}_{\Delta^n}$ is in
 $\RMod_{\Delta^n}^{\sim}$ if and only if, the map
 $\mc{M}^{\circledast}\rightarrow\mbf{\Delta}^{\mr{op}}$ exhibits
 $\mc{M}^{\circledast}$ as a generalized $\infty$-operad, and the
 induced map
 $\mc{M}^{\circledast}_{[0]}\rightarrow\mr{RM}_{\Delta^n,[0]}$ is an
 equivalence. We have the induced coCartesian
 fibration
 $p_{[0]}\colon\mc{M}^{\circledast}_{[0]}\rightarrow
 \mr{RM}_{\Delta^n,[0]}$. This is an equivalence if and only if it is so
 after pulling-back by map
 $\mr{RM}_{\Delta^{\{i\}},[0]}\rightarrow\mr{RM}_{\Delta^n,[0]}$ for any
 $i$ by \cite[3.3.1.5]{HTT}. Thus the equivalence is equivalent to the
 equivalence of $(\iota^*_i\mc{M}^{\circledast})_{[0]}\rightarrow
 \mr{RM}_{\Delta^{\{i\}},[0]}$ for any $i$.
 Now, we may assume that $p_{[0]}$ is an equivalence. In view of (an
 analogue of) \cite[2.1.2.12]{HA}, the map
 $\mc{M}^{\circledast}\rightarrow\mbf{\Delta}^{\mr{op}}$ is a
 generalized $\infty$-operad if and only if the map $\pi$ below induced
 by inert maps
 \begin{align*}
  \mc{M}^{\circledast}_{[m]}
  &\xrightarrow{\pi}
  \mc{M}^{\circledast}_{\{0,1\}}
  \times^{\mr{cat}}_{\mc{M}^{\circledast}_{\{1\}}}
  \mc{M}^{\circledast}_{\{1,2\}}
  \times^{\mr{cat}}\dots
  \times^{\mr{cat}}_{\mc{M}^{\circledast}_{\{m-1\}}}
  \mc{M}^{\circledast}_{\{m-1,m\}}\\
  &\xrightarrow{\alpha}
  \mc{M}^{\circledast}_{\{0,1\}}
  \times^{\mr{cat}}_{\mr{RM}_{\Delta^n,\{1\}}}
  \mc{M}^{\circledast}_{\{1,2\}}
  \times^{\mr{cat}}\dots
  \times^{\mr{cat}}_{\mr{RM}_{\Delta^n,\{m-1\}}}
  \mc{M}^{\circledast}_{\{m-1,m\}}
 \end{align*}
 is an equivalence. Since $\alpha$ is an equivalence, it suffices to
 show that $\alpha\circ\pi$ is an equivalence if and only it is so after
 pullback by $\iota_i$ for any $i$. Consider the following diagram:
 \begin{equation*}
  \xymatrix{
   \mc{M}^{\circledast}_{[m]}
   \ar[r]^-{\alpha}\ar[d]&
   \mc{M}^{\circledast}_{\{0,1\}}
   \times^{\mr{cat}}_{\mr{RM}_{\Delta^n,\{1\}}}
   \mc{M}^{\circledast}_{\{1,2\}}
   \times^{\mr{cat}}\dots
   \times^{\mr{cat}}_{\mr{RM}_{\Delta^n,\{m-1\}}}
   \mc{M}^{\circledast}_{\{m-1,m\}}
   \ar[d]^{\beta}\\
  \mr{RM}_{\Delta^n,[m]}
   \ar[r]^-{\gamma}&
   \mr{RM}_{\Delta^n,\{0,1\}}
   \times^{\mr{cat}}_{\mr{RM}_{\Delta^n,\{1\}}}
   \mr{RM}_{\Delta^n,\{1,2\}}\times\dots
   \times^{\mr{cat}}_{\mr{RM}_{\Delta^n,\{m-1\}}}
   \mr{RM}_{\Delta^n,\{m-1,m\}}
   }
 \end{equation*}
 Since
 $\sigma^i_!\colon\mr{RM}_{S,[1]}\rightarrow\mr{RM}_{S,[0]}$ is a
 coCartesian fibration for $i=0,1$, the fiber product of the target of
 $\alpha$ can be computed by fiber products in the category of
 simplicial sets. With respect to this model of the fiber product,
 $\beta$ is a coCartesian fibration. By direct computation, $\gamma$ is
 an isomorphism of simplicial sets. Thus, by \cite[3.3.1.5]{HTT} again,
 the equivalence of $\alpha$ is equivalent to the equivalence of
 $\alpha$ over each vertices of $\mr{RM}_{\Delta^n,[m]}$. Thus,
 $\mc{M}^{\circledast}$ is a generalized $\infty$-operad if and only if it is so
 after pullback by $\iota_i$.
\end{proof}

\subsection{}
\label{modelofrmod}
Let $\mc{C}$ be an $\infty$-category and $\mc{A}^{\circledast}\rightarrow\mbf{\Delta}^{\mr{op}}$ be a monoidal $\infty$-category,
in other words a coCartesian fibration of $\infty$-operads.
First, we put
$\RMod_{\mc{C},\mc{A}}:=\RMod_{\mc{C}}\times^{\mr{cat}}_{\coCart^{\mr{str}}(\mf{a})}\{\mc{A}^{\circledast}\}$.
Since
$\mr{alg}\colon\bp\RMod_{\mbf{\Delta}}\rightarrow\coCart^{\mr{str}}(\mf{a})\times\mbf{\Delta}$ is a functor in
$\Cart^{\mr{str}}(\mbf{\Delta})$, we may define
\begin{equation*}
 \bp\RMod_{\mbf{\Delta},\mc{A}}:=
  \bp\RMod_{\mbf{\Delta}}\times^{\mr{cat}}
  _{\coCart^{\mr{str}}(\mf{a})\times\mbf{\Delta}}
  (\{\mc{A}^{\circledast}\}\times\mbf{\Delta}).
\end{equation*}
in $\Cart^{\mr{str}}(\mbf{\Delta})$ by Lemma \ref{basicproprmodsimp}.\ref{basicproprmodsimp-2}.
Using the notation of \ref{dfnofbcstar}, we further put
\begin{equation*}
 \bp\RMod^{\circledast}:=
  \bp\RMod_{\mbf{\Delta}}\basech(\RMod_{\Delta^0})^{\simeq},
  \qquad
  \bp\RMod_{\mc{A}}^{\circledast}:=
  \bp\RMod_{\mbf{\Delta},\mc{A}}\basech
  (\RMod_{\Delta^0,\mc{A}})^{\simeq}.
\end{equation*}
We denote by $\RMod^{\circledast}$, $\RMod^{\circledast}_{\mc{A}}$ the dual coCartesian fibration of
$\bp\RMod^{\circledast}$, $\bp\RMod^{\circledast}_{\mc{A}}$.
Note that since $\bp\RMod_{\mbf{\Delta}}$ satisfies the Segal condition by Lemma \ref{basicproprmodsimp}, we have an equivalence
$\bp\RMod_{\mc{A}}^{\circledast}\simeq\bp\RMod_{\mbf{\Delta}}\basech(\RMod_{\Delta^0,\mc{A}})^{\simeq}$.

\begin{rem*}
 In the definition of $\RMod_{\mc{C},\mc{A}}$, we used the fiber product
 in $\Cat_\infty$, which is determined only up to contractible choices.
 If we need to fix a specific model for the fiber product, we may use
 $\RMod_S\times_{\coCart^{\mr{str}}(\mf{a})}
 (\coCart^{\mr{str}}(\mf{a})_{\mc{A}/})^{\mr{init}}$. Here,
 $\mc{C}^{\mr{init}}$ denotes the full subcategory spanned by initial
 objects, which is a contractible Kan complex by \cite[1.2.12.9]{HTT}.
 The fiber product is a fiber product in $\Cat_\infty$ by
 \cite[2.1.2.2]{HTT}. In other words, an object of $\RMod_{S,\mc{A}}$ is
 a pair of an object $\mc{M}^{\circledast}\rightarrow\mr{RM}_S$ in
 $\Cat_\infty^{\mr{RM}_S}$ and an equivalence
 $\mc{A}^{\circledast}\xrightarrow{\sim}
 \mc{M}^{\circledast}\times_{\mr{RM}_S}\mf{a}$.
\end{rem*}

\begin{prop}
 \label{inftwoRMod}
 The map
 $\rho\colon\bp\RMod^{\circledast}_{\mc{A}}\rightarrow\mbf{\Delta}$ is a
 complete Segal space, and defines an $(\infty,2)$-category whose
 underlying $\infty$-category is categorically equivalent to
 $\RMod^{\mb{A}_\infty}_{\mc{A}}(\Cat_\infty)$
 {\normalfont(}cf.\ {\normalfont\cite[4.2.2.10]{HA}} for the notation{\normalfont)}.
\end{prop}
\begin{proof}
 In view of Lemma \ref{basicproprmodsimp}, it remains to show the
 completeness and compute the underlying $\infty$-category. We have
 \begin{equation*}
  \RMod_{\Delta^n}^{\simeq}
   \simeq
   (\coCart(\mr{RM}_{\Delta^n})^{\mr{bp}})^\simeq
   \simeq
   (\coCart^{\mr{str}}(\mr{RM}_{\Delta^n})^{\mr{bp}})^{\simeq}
   \simeq
   \mr{Alg}_{\mr{RM}_{\Delta^n}}(\Cat_\infty)^{\simeq}.
 \end{equation*}
 Here $\coCart^{-}(\mr{RM}_{\Delta^n})^{\mr{bp}}$ denotes the full
 subcategory of $\coCart^{-}(\mr{RM}_{\Delta^n})$ spanned by vertices
 $\mc{M}^{\circledast}\rightarrow\mr{RM}_{\Delta^n}$ which is base
 preserving coCartesian fibration of generalized $\infty$-operads.
 On the other hand, we have
 \begin{align}
  \notag
  \mr{Alg}_{\mr{RM}_S}(\Cat_\infty)
  \times^{\mr{cat}}_{\mr{Alg}(\Cat_\infty)}\{\mc{A}\}
  &\simeq
  \bigl(\mr{Alg}_{\mr{RM}}(\Cat_\infty)^{S}
  \times^{\mr{cat}}_{\mr{Alg}(\mr{Cat}_\infty)^{S}}
  \mr{Alg}(\Cat_\infty)\bigr)
  \times^{\mr{cat}}_{\mr{Alg}(\Cat_\infty)}\{\mc{A}\}
  \\\label{compalgRM}
  &\simeq
  \bigl(\mr{RMod}^{\mb{A}_\infty}(\Cat_\infty)^{S}
  \times^{\mr{cat}}_{\mr{Alg}(\mr{Cat}_\infty)^{S}}
  \mr{Alg}(\Cat_\infty)\bigr)
  \times^{\mr{cat}}_{\mr{Alg}(\Cat_\infty)}\{\mc{A}\}
  \\\notag
  &\simeq
  \mr{Fun}\bigl(S,\mr{RMod}^{\mb{A}_\infty}_{\mc{A}}
  (\Cat_\infty)\bigr),
 \end{align}
 where the first equivalence follows by Lemma
 \ref{baRMpro}.\ref{baRMpro-2}, and the second by \cite[7.1.9]{GH}. By
 Lemma \ref{comptwoconst}, the composition of functors
 $\mbf{\Delta}^{\mr{op}}\xrightarrow{\rho'}\Cat_\infty
 \xrightarrow{\kappa}\mr{Spc}$, where $\rho'$ is the functor
 associated with the Cartesian fibration $\rho$, is equivalent to
 $\mr{Seq}_\bullet(\mr{RMod}^{\mb{A}_\infty}_{\mc{A}}(\Cat_\infty))$.
 Thus, the proposition follows.
\end{proof}

\subsection{}
\label{presmoncatint}
The monoidal $\infty$-category $\mc{A}^{\circledast}$ is said to be {\em presentable} if it comes from an object of $\mr{Alg}(\PrL)$.
In other words, this is equivalent to saying that $\mc{A}$ is presentable and the tensor product $\otimes\colon\mc{A}\times\mc{A}\rightarrow\mc{A}$
preserve small colimits separately in each variable (cf.\ \cite[4.8.1.15]{HA}).
The $\infty$-category $\RMod_{\mc{C}}$ makes sense if we replace the $\infty$-category of small $\infty$-categories $\Cat_\infty$
by the $\infty$-category of possibly large $\infty$-categories $\widehat{\Cat}_\infty$.
We denote by $\widehat{\RMod}_{\mc{C}}$ by the $\infty$-category defined by this replacement.

\begin{dfn*}
 For an $\infty$-category $\mc{C}$, let $\LinCat_{\mc{C},-}$ be the {\em full} subcategory of $\widehat{\RMod}_{\mc{C}}$ spanned by
 coCartesian fibrations $p\colon\mc{M}^{\circledast}\rightarrow\mr{RM}_\mc{C}$ which is at the same time presentable
 (namely each fiber is presentable and the functor associated to each edge of $\mr{RM}_{\mc{C}}$ commutes with colimits, cf.\ \cite[5.5.3.2]{HTT}).
 We put $\LinCat_{\mc{C},\mc{A}}:=\LinCat_{\mc{C},-}\times^{\mr{cat}}_{\Cat_\infty^{\mbf{\Delta}^{\mr{op}}}}\{\mc{A}^{\circledast}\}$.
 Let $\bp\LinCat_{\mbf{\Delta}}$ be the full subcategory of $\bp\RMod_{\mbf{\Delta}}$ spanned by vertices $\LinCat_{\Delta^n,-}$ over
 $[n]\in\mbf{\Delta}^{\mr{op}}$.
 We define $\bp\LinCat^{\circledast}_{\mc{A}}:=\bp\LinCat_{\mbf{\Delta}}\basech\LinCat_{\Delta^0,\mc{A}}^{\simeq}$.
\end{dfn*}

\begin{rem*}
 \begin{enumerate}
  \item One may wonder why $\LinCat_{\mc{C},-}$ is a {\em full} subcategory of $\RMod_{\mc{C}}$.
	Indeed, the subcategory $\RMod^{\mb{A}_\infty}_{\mc{A}}(\PrL)$ of $\RMod^{\mb{A}_\infty}_{\mc{A}}(\widehat{\Cat}_\infty)$ is {\em not} full.
	However, we will use $\LinCat_{\mc{C},-}$ to construct the $(\infty,2)$-category $\twoLinCat_{\mc{A}}$, and the underlying
	$\infty$-category of $\twoLinCat_{\mc{A}}$ does not coincide with $\LinCat_{\Delta^0,\mc{A}}$.

  \item From now on, we ignore the size of sets,
	and in particular, we often denote $\widehat{\Cat}_\infty$ and $\widehat{\RMod}$ by $\Cat_\infty$ and $\RMod$ unless there may be a confusion.
 \end{enumerate}
\end{rem*}

\begin{prop}
 \label{inftwolincat}
 Let $\mc{A}$ be a presentable monoidal $\infty$-category.
 The simplicial $\infty$-category $\bp\LinCat_{\mc{A}}^{\circledast}$ is
 an $(\infty,2)$-category whose associated $\infty$-category is
 equivalent to $\RMod^{\mb{A}_\infty}_{\mc{A}}(\PrL)$.
\end{prop}
\begin{proof}
 Let us show the Segal condition.
 For this, it suffices to check the Segal condition for $\bp\LinCat_{\mbf{\Delta}}$.
 By Lemma \ref{critsubseg}, in view of Proposition \ref{inftwoRMod}, we only need to show the following claim:
 An object $\mc{M}\in\RMod_{\Delta^n}$ belongs to $\LinCat_{\Delta^n,-}$ if and only if $\kappa_i^*\mc{M}$ belongs to
 $\LinCat_{\Delta^{\{i,i+1\}},-}$ for any $0\leq i<n$.
 Here $\kappa_i\colon\mr{RM}_{\Delta^{\{i,i+1\}}}\rightarrow\mr{RM}_{\Delta^n}$.
 The verification is straightforward.

 Now, we need to show that it is complete.
 We have $\LinCat_{\Delta^n,-}^{\simeq}\simeq\mr{Alg}_{\mr{RM}_{\Delta^n}}(\PrL)^{\simeq}$.
 Similarly to the computation (\ref{compalgRM}), we have
 \begin{equation*}
  \mr{Alg}_{\mr{RM}_S}(\PrL)\times_{\mr{Alg}(\PrL)}
   \{\mc{A}\}
   \simeq
   \mr{Fun}(S,\RMod^{\mb{A}_\infty}_{\mc{A}}(\PrL)).
 \end{equation*}
 As in Proposition \ref{inftwoRMod}, use Lemma \ref{comptwoconst} to show that this equivalence induces the equivalence between the
 underlying Segal space of $\bp\LinCat^{\circledast}$ and $\mr{Seq}_\bullet(\RMod^{\mb{A}_\infty}_{\mc{A}}(\PrL))$ to conclude.
\end{proof}

\begin{dfn}
 \label{dfnoftwolincat}
 We define $\twoLinCat_{\mc{A}}$ to be the $(\infty,2)$-category such
 that the equivalence
 $\mr{Seq}_{\bullet}(\twoLinCat_{\mc{A}}^{\mr{2-op}})\simeq
 \mr{St}(\bp\LinCat^{\circledast}_{\mc{A}})$, where $\mr{St}$ denotes
 the straightening functor, holds.
 If $R$ is an $\mb{E}_2$-ring, then $\RMod_R$ can naturally be
 considered as an object in $\mr{Alg}(\PrL)$ by \cite[7.1.2.6]{HA}.
 In this case, we denote
 $\twoLinCat_{\RMod_R}$ by $\twoLinCat_R$. The underlying
 $\infty$-category of $\twoLinCat_{\mc{A}}$ and $\twoLinCat_R$ are
 denoted by $\LinCat_{\mc{A}}$ and $\LinCat_R$.
 Note that, by Proposition \ref{inftwolincat}, $\LinCat_R$ coincides with $\mr{LinCat}^{\mr{St}}_R$ in \cite[D.1.5.1]{SAG}.
 For a $1$-morphism $F\colon \mc{C}\rightarrow\mc{D}$ in $\twoLinCat_{\mc{A}}$,
 the corresponding monoidal functor of generalized $\mr{RM}$-operads is
 denoted by $F^{\circledast}\colon\mc{C}^{\circledast}\rightarrow\mc{D}^{\circledast}$.
\end{dfn}

\begin{rem*}
 \begin{enumerate}
  \item We are taking $(-)^{\mr{2-op}}$ in order to have the forgetful
	functor $\twoLinCat_{\mc{A}}\rightarrow\mbf{Cat}_\infty$. See
	\cite[Ch.10, 2.4.5]{GR}. In \cite[Ch.1, 8.3.1]{GR}, they used
	mixture of Cartesian and coCartesian fibrations to define
	$\mr{Seq}_n(\twoLinCat_{\mc{A}})$. We did not employ this
	approach in order to avoid too much complications.
	
  \item The $(\infty,2)$-category $\mbf{Pres}$ appeared in Introduction
	is by definition $\twoLinCat_{\mb{S}}$, where $\mb{S}$ is the
	sphere spectrum. The underlying $\infty$-category is
	$\PrL_{\mr{st}}$, the full subcategory of $\PrL$ spanned by
	stable presentable $\infty$-categories, by \cite[4.8.2.18]{HA}.
 \end{enumerate}
\end{rem*}

\subsection{}
The following lemma is an useful criterion to detect adjoint maps in $\twoLinCat_R$.

\begin{lem*}
 \label{adjcritlincat}
 Let $F\colon\mc{C}\rightarrow\mc{D}$ be a $1$-morphism in $\twoLinCat_{\mc{A}}$.
 Then the following are equivalent:
 \begin{enumerate}
  \item\label{adjcritlincat-1}
       The functor $F$ admits a left {\normalfont(}resp.\ right{\normalfont)} adjoint in
       $\twoLinCat_{\mc{A}}$ in the sense of {\normalfont\cite[Ch.12, 1.1.3]{GR}};
	
  \item\label{adjcritlincat-2}
       There exists a monoidal functor $G^{\circledast}\colon\mc{D}^{\circledast}\rightarrow\mc{C}^{\circledast}$
       which is left {\normalfont(}resp.\ right{\normalfont)} adjoint to $F^{\circledast}$ relative to $\mr{RM}$ in the sense of
       {\normalfont\cite[7.3.2.2]{HA}} and $G^{\circledast}_{(01)}$ commutes with small colimits.

  \setcounter{adjprop}{\theenumi}
 \end{enumerate}
 Moreover, if $R$ is an $\mb{E}_2$-ring, and $\mc{A}^{\circledast}=\mr{LMod}_R$,
 then the above two conditions are equivalent to
 \begin{enumerate}
  \setcounter{enumi}{\theadjprop}
  \item\label{adjcritlincat-3}
       $F^{\circledast}_{(01)}$ admits a left adjoint
	{\normalfont(}resp.\ right adjoint which commutes with small colimits{\normalfont)}.
 \end{enumerate}
\end{lem*}
\begin{proof}
 Let us show the equivalence of \ref{adjcritlincat-1} and
 \ref{adjcritlincat-2}.
 We only show the non-resp claim,
 since a proof for right adjoints can be obtained simply by replacing
 left by right.
 Let us show \ref{adjcritlincat-1} to \ref{adjcritlincat-2}.
 Since $F$ admits a left adjoint, there exists a coCartesian fibration
 $\mc{M}^{\circledast}_G\rightarrow\mr{RM}_{\Delta^1}$, a unit map
 $\alpha\colon\mc{D}^{\circledast}\times_{\mr{RM}}\mr{RM}_{\Delta^1}
 \rightarrow\mc{M}^{\circledast}_{F\circ G}$, a counit map
 $\mc{M}^{\circledast}_{G\circ F}\rightarrow\mc{C}^{\circledast}
 \times_{\mr{RM}}\mr{RM}_{\Delta^1}$ satisfying some conditions.
 By taking the base change by the canonical map
 $S\times\mr{RM}\rightarrow\mr{RM}_S$ in Lemma \ref{baRMpro}, the data
 yields a pair of adjoint functors relative to $\mr{RM}$. Let us show
 \ref{adjcritlincat-2} to \ref{adjcritlincat-1}.
 Combining Lemma \ref{baRMpro} and (dual version of) Lemma
 \ref{cartsegco}, we have an equivalence
 $\coCart(S\times\mr{RM})\times^{\mr{cat}}_{\coCart(S\times\mf{a})}
 \coCart(\mf{a})\simeq\coCart(\mr{RM}_S)$. First, let us construct a
 functor $G\colon\mc{D}\rightarrow\mc{C}$ in $\twoLinCat_{\mc{A}}$.
 Since $G^{\circledast}$ is assumed monoidal, the left adjoint
 $G^{\circledast}$ yields an object in $\widetilde{G}$ in
 $\coCart(\Delta^1\times\mr{RM})$. Since
 $\widetilde{G}|_{\Delta^1\times\mf{a}}$ is a left adjoint to the
 equivalence $F^{\circledast}|_{\mf{a}}$, the restriction of
 $\widetilde{G}$ is equivalence as well. This implies that
 $\widetilde{G}$ induces an object of
 $\coCart(S\times\mr{RM})\times^{\mr{cat}}_{\coCart(S\times\mf{a})}
 \coCart(\mf{a})$, and the equivalence above yields a an object in
 $\coCart(\mr{RM}_{\Delta^1})$. Since $G^{\circledast}_{(01)}$ commutes
 with small colimits, $\widetilde{G}$ yields a $1$-morphism
 $\mc{D}\rightarrow\mc{C}$ in $\twoLinCat_{\mc{A}}$.
 Let us construct a unit
 map $\alpha$ similarly.
 Because $F^{\circledast}$ admits a left adjoint relative to
 $\mr{RM}$, we have a unit map
 $\widetilde{\alpha}\colon\mc{D}^{\circledast}\times\Delta^1
 \rightarrow\mc{M}^{\circledast}_{F\circ
 G}\times_{\mr{RM}_{\Delta^1}}(\mr{RM}\times\Delta^1)$ over
 $\Delta^1\times\mr{RM}$
 The restriction of $\widetilde{\alpha}$ to $\Delta^1\times\mf{a}$ is a
 unit map of the adjunction of
 $F^{\circledast}\times_{\mr{RM}}{\mf{a}}$. Since
 $F^{\circledast}\times_{\mr{RM}}\mf{a}$ is the identity,
 $\widetilde{\alpha}\times_{\Delta^1\times\mr{RM}}(\Delta^1\times\mf{a})$
 is an equivalence. Consequently, $\widetilde{\alpha}$ yields a map in
 $\coCart(S\times\mr{RM})\times_{\coCart(S\times\mf{a})}\coCart(\mf{a})$,
 which induces a desired unit map $\alpha$ using the equivalence above.
 Similarly, we construct a counit map. In order to show that these maps
 actually gives an adjoint pair $(G,F)$, we need to show that certain
 compositions of maps given by unit and counit maps are
 equivalences. The relative adjunction on $\mr{RM}$ yields corresponding
 relations in $\coCart(S\times\mr{RM})$, so in order to show the
 relations in $\coCart(\mr{RM}_S)$, we use the above equivalence again.

 Let us show the equivalence between \ref{adjcritlincat-2} and
 \ref{adjcritlincat-3}.
 The \ref{adjcritlincat-2} to \ref{adjcritlincat-3}
 direction is obvious, so we will show the other direction. 
 First, consider the case where $F^{\circledast}_{(01)}$ admits a right
 adjoint $G^{\circledast}_{(01)}$.
 In this case, by \cite[7.3.2.9]{HA}, we have a right adjoint
 $G^{\circledast}$ relative to $\mr{RM}$. This functor is observed to be
 monoidal when $G^{\circledast}_{(01)}$ commutes with small colimits in
 \cite[D.1.5.3]{SAG}, thus the claim follows. Next, assume that
 $F^{\circledast}_{(01)}$ admits a left adjoint, denoted by $G$.
 We wish to check the conditions of \cite[7.3.2.11]{HA}. Similarly to
 the proof of \cite[7.3.2.7]{HA}, the condition (1) follows from our
 assumption, we only need to check (2).
 For this, it suffices to show that the induced map
 $\phi_{M,C}\colon G(M)\otimes_{\mc{C}}C\rightarrow
 G(M\otimes_{\mc{D}}C)$ for any $C\in\mc{A}\simeq\mr{LMod}_R$ and
 $M\in\mc{D}$ is an equivalence.
 Since $G$ admits a right adjoint, $G$ commutes with small colimits.
 By \cite[7.2.4.2]{HA}, $C$ can be written as a small filtered
 colimit of perfect $R$-modules, and it suffices to show the equivalence
 when $C$ is a perfect $R$-module. Thus, it suffices to show the
 following two assertions:
 \begin{itemize}
  \item The map $\phi_{M,R^n[m]}$ is an equivalence for any integers
	$n\geq0$, $m$;

  \item If $\phi_{M,C}$ is an equivalence, then $\phi_{M,C'}$ is an
	equivalence for any retract $C'$ of $C$.
 \end{itemize}
 Since the formation of $\phi$ commutes with pushouts, we have
 $\phi_{M,C[m]}\simeq\phi_{M,C}[m]$, which implies the first assertion.
 For the second assertion, let $I\colon\Delta^2\rightarrow\mc{A}$ be a
 diagram such that $I(0)=I(2)=C'$, $I(1)=C$ and
 $I(\Delta^{\{0,2\}})=\mr{id}$. Then this induces a diagram
 $J:=\mr{cof}(\phi_{M,I})\colon\Delta^2\rightarrow\mc{D}$ such that
 $J(0)=J(2)=\mr{cof}(\phi_{M,C'})$, $J(1)=\mr{cof}(\phi_{M,C})$ and
 $J(\Delta^{\{0,2\}})=\mr{id}$. Since $\phi_{M,C}$ is assumed to be an
 equivalence, $\mr{cof}(\phi_{M,C})\simeq 0$.
 Since initial objects can be detected in the homotopy category, a
 retract of $0$ is $0$. Thus $\phi_{M,C'}\simeq 0$ as required.
\end{proof}

\subsection{}
\label{prmodfunct}
Let $R$ be an $\mb{E}_\infty$-ring. Then $\LinCat_R$ is equipped with
canonical symmetric monoidal structure by \cite[D.2.3.3]{SAG}.
Before concluding this section, we make some construction in terms of
$\mr{CAlg}(\LinCat_R)$, which is used to construct a motivic theory
associated to an algebra object in \S\ref{secExam}.

Let $\mc{K}$ be a collection of small simplicial sets.
Let $\mr{Mon}^{\mr{pr}}_{\mr{Assoc}}(\Cat_\infty)^{\otimes}$ be the
full subcategory of
$\mr{Mon}^{\mc{K}}_{\mr{Assoc}}(\Cat_\infty)^{\otimes}$ (cf.\
\cite[4.8.5.14]{HA}) spanned by vertices
$(\mc{C}_1^{\otimes},\dots,\mc{C}_n^{\otimes})$ such that
$\mc{C}_i$ is presentable for any $i$.
Arguing similarly to the proof of \cite[4.8.5.16 (1)]{HA},
$\mr{Mon}^{\mr{pr}}_{\mr{Assoc}}(\Cat_\infty)^{\otimes}$ is a symmetric
$\infty$-operad.
In \cite[4.8.5.10]{HA}, the full subcategories $\Prcat^{\mr{Alg}}$,
$\Prcat^{\mr{Mod}}$ of $\Cat_\infty^{\mr{Alg}}(\mc{K})$ and
$\Cat_\infty^{\mr{Mod}}(\mc{K})$ are introduced.
Informally, $\Prcat^{\mr{Alg}}$ is the $\infty$-category of pairs
$(\mc{C}^{\otimes},A)$ where
$\mc{C}^{\otimes}\in\mr{Mon}^{\mr{pr}}_{\mr{Assoc}}(\Cat_\infty)$ and
$A\in\mr{Alg}(\mc{C})$, and $\Prcat^{\mr{Mod}}$ is the $\infty$-category
of pairs $(\mc{C}^{\otimes},\mc{M})$ where $\mc{M}$ is an
$\infty$-category left-tensored over $\mc{C}^{\otimes}$ with some
suitable presentability.
We can promote these
$\infty$-categories to symmetric monoidal $\infty$-categories, similarly
to \cite[4.8.5.14]{HA} as follows. The $\infty$-category
$\Prcat^{\mr{Alg},\otimes}$ is simply
\begin{equation*}
 \mr{Mon}^{\mr{pr}}_{\mr{Assoc}}(\Cat_\infty)^{\otimes}
  \times_{\mr{Mon}_{\mr{Assoc}}^{\mc{K}}(\Cat_\infty)^{\otimes}}
  \Cat_\infty^{\mr{Alg}}(\mc{K})^{\otimes}.
\end{equation*}
We define $\Prcat^{\mr{Mod},\otimes}$ to be the full subcategory of
$\Cat_\infty^{\mr{Mon}}(\mc{K})^{\otimes}$ spanned by the objects of the
form
$((\mc{C}_1^{\otimes},\mc{M}_1),\dots,(\mc{C}_n^{\otimes},\mc{M}_n))$
such that $\mc{C}_i$ and $\mc{M}_i$ are presentable for any $i$.
We have the following diagram
\begin{equation*}
 \xymatrix@C=30pt{
  \Prcat^{\mr{Alg},\otimes}
  \ar[rr]^-{\Theta}\ar[rd]_-{\phi}&&
  \Prcat^{\mr{Mod},\otimes}\ar[dl]^-{\psi}\\
 &\mr{Mon}^{\mr{pr}}_{\mr{Assoc}}(\Cat_\infty)^{\otimes}.&
  }
\end{equation*}
The maps $\phi$ is a coCartesian fibration by \cite[4.8.5.16]{HA}.
Now, the symmetric monoidal structure of $\PrL$ is closed by
\cite[4.8.1.18]{HA}. This implies that the tensor product of $\PrL$
commutes with all small colimits separately in each variable.
Thus, the argument of \cite[4.8.5.1]{HA} can be applied to show that the
functor
$\Prcat^{\mr{Mod}}\rightarrow
\mr{Mon}^{\mr{pr}}_{\mr{Assoc}}(\Cat_\infty)$ is a coCartesian
fibration. Similarly to the proof of \cite[4.8.5.16]{HA}, the map
$\psi$ is also a coCartesian fibration.
Since $\mr{Mon}_{\mr{Assoc}}^{\mc{K}}(\Cat_\infty)^{\otimes}$ can be
identified with $\mr{Alg}(\Cat_\infty(\mc{K}))^{\otimes}$ (cf.\ proof of
\cite[4.8.5.16]{HA}), we can identify
$\mr{Mon}^{\mr{pr}}_{\mr{Assoc}}(\Cat_\infty)^{\otimes}$ with
$\mr{Alg}(\PrL)^{\otimes}$.

The monoidal category of spectra $\mr{Sp}^{\otimes}=:\mbf{1}$ is an
initial object of the $\infty$-category $\mr{CAlg}(\PrL)$.
Thus, we have a map $\mbf{1}\rightarrow\mr{Mod}_R$ by
\cite[3.2.1.9]{HA}. By \cite[3.4.3.4]{HA}, we have a map of symmetric
$\infty$-operads
$\LinCat_R^{\otimes}:=\mr{Mod}_{\mr{Mod}_R}(\PrL)^{\otimes}
\rightarrow\mr{Mod}_{\mbf{1}}(\PrL)^{\otimes}\simeq\PrLsym$,
where the last equivalence is by \cite[3.4.2.1]{HA}.
On the other hand, we have the bifunctor (cf.\ \cite[2.2.5.3]{HA}) of
symmetric $\infty$-operads
$\mr{Comm}^{\otimes}\times\mr{Assoc}^{\otimes}\rightarrow
\mr{Comm}^{\otimes}$ (cf.\ \cite[3.2.4.4]{HA}).
For any symmetric $\infty$-operad $\mc{C}^{\otimes}$, this induces the
map
\begin{equation*}
 \mr{CAlg}(\mc{C})\rightarrow
  \mr{Alg}_{\mr{Comm}\otimes\mr{Assoc}}(\mc{C})
  \simeq
  \mr{CAlg}(\mr{Alg}(\mc{C})),
\end{equation*}
where the last map follows by definition. Thus, this induces the functor
\begin{equation*}
 \mr{CAlg}(\LinCat_R)
  \rightarrow
  \mr{CAlg}(\PrL)
  \rightarrow
  \mr{CAlg}(\mr{Alg}(\PrL)).
\end{equation*}
Taking the adjoint, we have
$\mr{Comm}^{\otimes}\times\mr{CAlg}(\LinCat_R)\rightarrow
\mr{Alg}(\PrL)^{\otimes}\simeq\mr{Mon}^{\mr{pr}}
_{\mr{Assoc}}(\Cat_\infty)^{\otimes}$.
We take the base change of $\Prcat^{\mr{Alg},\otimes}$ and
$\Prcat^{\mr{Mod},\otimes}$ by this map, which we denote by
$\Prcat^{\mr{Alg},\otimes}_{\mc{L}}$ and
$\Prcat^{\mr{Mod},\otimes}_{\mc{L}}$.

Let us construct a functor
$\Prcat^{\mr{Mod},\otimes}_{\mc{L}}\rightarrow\LinCat_R^{\otimes}$ over
$\mr{Comm}^{\otimes}$.
Informally, this functor sends $(\mc{C}^{\otimes},\mc{M})$ to $\mc{M}$
considered as an object of $\LinCat_R$ via the structural map
$\mr{Mod}_R\rightarrow\mc{C}^{\otimes}$.
Let $\mc{C}$ be an $\infty$-category with initial object
$\emptyset\in\mc{C}$. The map $\mc{C}^{\emptyset/}\rightarrow\mc{C}$ is
a trivial fibration by \cite[4.2.1.6]{HTT}, we way take a quasi-inverse
$\mc{C}\rightarrow\mc{C}^{\emptyset/}$. This induces a map
$I_{\emptyset}\colon\Delta^1\times\mc{C}\rightarrow
\Delta^0\diamond\mc{C}\rightarrow\mc{C}$
sending $(0,c)$ to $\emptyset$ and $(1,c)$ to $c$.
The object $\mr{Mod}_R$ in $\mr{CAlg}(\LinCat_R)$ is an initial object
(cf.\ \cite[3.2.1.9, 3.4.4.7]{HA}).
Thus, applying the above observation, we have the diagram
\begin{equation*}
 \xymatrix@C=40pt{
  \Delta^{\{1\}}\times\Prcat^{\mr{Mod},\otimes}
  \ar@{=}[r]\ar@{^{(}->}[d]&
  \Prcat^{\mr{Mod},\otimes}\ar[d]^{\psi_{\mc{L}}}\\
 \Delta^1\times\Prcat^{\mr{Mod},\otimes}
  \ar[r]^-{I}\ar@{--}[ur]^{\widetilde{\rho}}&
  \mr{Comm}^{\otimes}\times\mr{CAlg}(\LinCat_R)
  }
\end{equation*}
where $I$ is the map induced by $I_{\mr{Mod}_R}$.
We wish to take the right Kan extension of the above diagram.
For the existence, in view of \cite[4.3.2.15]{HTT}, it suffices to check
the following:
\begin{quote}
 Let
 $F\colon
 ((\mc{C}_1^{\otimes},\mc{M}_1),\dots,(\mc{C}_m^{\otimes},\mc{M}_m))
 \rightarrow
 ((\mc{D}_1^{\otimes},\mc{N}_1),\dots,(\mc{D}_m^{\otimes},\mc{N}_m))$ be
 the map covering the identity $\left<m\right>\rightarrow\left<m\right>$
 in $\Fin$ such that the map $\mc{M}_i\rightarrow\mc{N}_i$ is an
 equivalence for any $i$. Then $F$ is a $\psi_{\mc{L}}$-Cartesian edge.
\end{quote}
Since $\psi_{\mc{L}}$ is a coCartesian fibration since $\psi$ is,
it suffices to check that the edge is locally $\psi_{\mc{L}}$-Cartesian
by \cite[5.2.2.4]{HTT}. This is, indeed, locally Cartesian since it is
inner fibration and \cite[4.2.3.2]{HA}.
By using this right Kan extension, we have
\begin{equation*}
 \rho\colon
  \Prcat^{\mr{Mod},\otimes}=
  \Delta^{\{0\}}\times\Prcat^{\mr{Mod},\otimes}
  \xrightarrow{\widetilde{\rho}}
  \psi_{\mc{L}}^{-1}(\mr{Comm}^{\otimes}\times\{\mr{Mod}_R\})
  \simeq
  \LinCat_R^{\otimes}.
\end{equation*}
For the last equivalence, see \cite[4.8.5.19]{HA}.
By construction, this is a map of symmetric $\infty$-operads.
Summing up, we have the following diagram
\begin{equation*}
 \xymatrix@C=30pt{
  \Prcat^{\mr{Alg},\otimes}_{\mc{L}}
  \ar[r]^-{\Theta}\ar[rd]&
  \Prcat^{\mr{Mod},\otimes}_{\mc{L}}\ar[d]\ar[r]&
  \LinCat_R^{\otimes}\times\mr{CAlg}(\LinCat_R)\ar[dl]\\
 &\mr{Comm}^{\otimes}\times\mr{CAlg}(\LinCat_R)\ar[r]^-{f}&
  \mr{CAlg}(\LinCat_R).
  }
\end{equation*}
Let $\mc{C}^{\otimes}\rightarrow\mc{O}^{\otimes}\times S$ be a
coCartesian $S$-family of $\mc{O}$-monoidal $\infty$-categories. Let
$p\colon\mc{O}^{\otimes}\times S\rightarrow S$ be the projection. We
denote by $p_*^{\otimes}(\mc{C}^{\otimes})$ the full subcategory of
$p_*(\mc{C}^{\otimes})$ spanned by vertices corresponding to
$\mr{Alg}_{\mc{O}}(\mc{C}_s)$.
Since $p_*^{\otimes}(\mc{C}^{\otimes})$ contains all the coCartesian
edges connecting vertices in it, the map
$p_*^{\otimes}(\mc{C}^{\otimes})\rightarrow S$ is a coCartesian
fibration as well.
Let $\mc{C}^{\otimes}\rightarrow\mc{C}'^{\otimes}$ be a map over
$\mc{O}^{\otimes}\times S$ of $S$-family of $\mc{O}$-monoidal
$\infty$-categories such that for each vertex $s\in S$, the induced map
$\mc{C}^{\otimes}_s\rightarrow\mc{C}'^{\otimes}_s$ is a map of symmetric
$\infty$-operads.
Then we have $p_*^{\otimes}(\mc{C}^{\otimes})\rightarrow
p_*^{\otimes}(\mc{C}'^{\otimes})$.
This observation being applied to the above diagram yields a map
\begin{align*}
 \Xi\colon
 \Prcat^{\mr{CAlg}}_{\mc{L}}:=
  p_*^{\otimes}(\Prcat^{\mr{Alg},\otimes}_{\mc{L}})
  \rightarrow
  p_*^{\otimes}&(\LinCat_R^{\otimes}\times\mr{CAlg}(\LinCat_R))\\
  &\simeq
  \mr{CAlg}(\LinCat_R)\times\mr{CAlg}(\LinCat_R)
\end{align*}
over $\mr{CAlg}(\LinCat_R)$, where the last $\infty$-category is
considered over $\mr{CAlg}(\LinCat_R)$ by the second projection.
Note that the fiber of $\Prcat^{\mr{CAlg}}_{\mc{L}}$ over
$\mc{C}^{\otimes}\in\mr{CAlg}(\LinCat_R)$ is $\mr{CAlg}(\mr{Alg}(\mc{C}))$.
By the proof of \cite[4.8.5.21]{HA},
$\mr{Comm}^{\otimes}\otimes\mr{Assoc}^{\otimes}\simeq\mr{Comm}^{\otimes}$,
and in particular, $\mr{CAlg}(\mr{Alg}(\mc{C}))\simeq\mr{CAlg}(\mc{C})$.
Thus, the fiber $\Xi_{\mc{C}}$ of $\Xi$ over $\mc{C}^{\otimes}$ sends
$A\in\mr{CAlg}(\mc{C})$ to $(\mr{Mod}_A(\mc{C}),\mc{C}^{\otimes})$.
What we have done so far can be summarized as follows:

\begin{lem*}
 Let $R$ be an $\mb{E}_\infty$-ring. Then we have the following diagram
 of coCartesian fibrations
 \begin{equation*}
  \xymatrix{
   \Prcat^{\mr{CAlg}}_{\mc{L}}
   \ar[rd]_{\phi^{\mr{CAlg}}}\ar[rr]^-{\Xi}&&
   \mr{CAlg}(\LinCat_R)\times\mr{CAlg}(\LinCat_R)\ar[ld]^-{\mr{pr}_2}\\
  &\mr{CAlg}(\LinCat_R).&
   }
 \end{equation*}
 Here, the fiber of $\phi^{\mr{CAlg}}$ over
 $\mc{C}^{\otimes}\in\mr{CAlg}(\LinCat_R)$ is $\mr{CAlg}(\mc{C})$, and
 $\Xi$ sends $A\in\mr{CAlg}(\mc{C})$ to
 $(\mr{Mod}_A(\mc{C}),\mc{C}^{\otimes})$.
\end{lem*}

\section{Construction of the bivariant $(\infty,2)$-functor}
\label{constfun}
Assume we are given an $(\infty,2)$-functor
$F\colon\mbf{C}\rightarrow\twoLinCat_{\mc{A}}$.
For each $X\in\mbf{C}$,
assume we are given an object $I_X\in F(X)$. We do not ask any
compatibilities of $I_X$ with $F$.
For example, if $F(X)$ is a symmetric monoidal, then
$I_X$ may be taken as a unit object. Then we may consider the assignment
to each $1$-morphism $f\colon X\rightarrow Y$ in $\mbf{C}$ the object
$\mr{Mor}_{F(Y)}(F(f)(I_X),I_Y)$, where $\mr{Mor}_{\mc{M}}$ is the
morphism object of $\mc{M}\in\LinCat_{\mc{A}}$.
It is natural to ask the functoriality of this construction.
This will be encoded in the non-unital right-lax functor
$\mbf{C}\dashrightarrow\mbf{B}\mc{A}^{\circledast}$, which will be
constructed in this section.
The construction is a ``family version'' of \cite[4.7.1]{HA} and
\cite[7.3, 7.4]{GH}.

\subsection{}
\label{setupmodspM}
Let us construct a universal $\mc{A}$-module over $\bp\RMod_{\mbf{\Delta},\mc{A}}$.
As in Remark \ref{modelofrmod}, we fix a model of $\bp\RMod_{\mbf{\Delta},\mc{A}}$.
We also fix a model for $\RMod_{\Delta^n,\mc{A}}$ so that $\RMod_{\Delta^n,\mc{A}}=\bp\RMod_{\mbf{\Delta},\mc{A}}\times_{\mbf{\Delta}}\{[n]\}$.
Consider the composition of functors
\begin{equation*}
 \bp\RMod_{\mbf{\Delta},\mc{A}}
  \rightarrow
  \bp\RMod_{\mbf{\Delta}}
  \rightarrow
  \mr{Fun}(\Delta^1,\coCart^{\mr{str}}(\mr{RM}))
  \xrightarrow[\mr{St}]{\sim}
  \mr{Fun}(\Delta^1,\mr{Fun}(\mr{RM},\Cat_\infty)),
\end{equation*}
where the second functor is the composition of the upper horizontal functors of (\ref{dfnofRModandoth-diag}).
This induces the functor
$\mr{RM}\times\bp\RMod_{\mbf{\Delta},\mc{A}}\rightarrow\mr{Fun}(\Delta^1,\Cat_\infty)$.
By taking the unstraightening, we have the commutative diagram
\begin{equation*}
 \label{funddiagmod}
 \xymatrix{
  \mc{M}^{\mr{univ},\circledast}\ar[rd]_{f}\ar[rr]^-{h}&&
  \mr{RM}^{\mr{univ}}\ar[ld]^{g}\\
 &\mr{RM}\times\bp\RMod_{\mbf{\Delta},\mc{A}}&
  }
\end{equation*}
such that $f$, $g$ are coCartesian fibrations and $h$ preserves coCartesian edges.
Furthermore, if we replace $\mc{M}^{\mr{univ},\circledast}$ by equivalence, we may assume that $h$ is a categorical fibration.
By construction, it is equipped with an equivalence
$\mc{A}^{\circledast}\xrightarrow{\sim}\mc{M}^{\mr{univ},\circledast}\times_{\mr{RM}}\mf{a}$.
We put
$\mc{M}^{\mr{univ},\circledast}_{[n]}:=\mc{M}^{\mr{univ},\circledast}\times_{\mbf{\Delta}}\{[n]\}$.
We often abbreviate $\bp\RMod_{\mbf{\Delta},\mc{A}}$ simply by $\bp\RMod$.

By construction, we have the equivalence
$\mr{RM}^{\mr{univ}}_{[n]}\simeq\mr{RM}_{\Delta^n}
\times\bp\RMod_{\Delta^n}$.
This induces a map
$\mr{RM}^{\mr{univ}}_{[n]}\rightarrow\mr{RM}_{\Delta^n}$.
Note that since all the equivalence
in $\mr{RM}_{\Delta^n}$ are degenerate edges, in other words,
$\mr{RM}_{\Delta^n}$ is {\em gaunt}, any functor
$\mc{C}\rightarrow\mr{RM}_{\Delta^n}$ from an $\infty$-category is a
categorical fibration by \cite[2.3.1.5, 2.4.6.5]{HTT}.

\begin{lem}
 \label{compofcartRModM}
 Let $\phi\colon[m]\rightarrow[n]$ be a map in $\mbf{\Delta}$.
 We have the following pullback diagram in $\Cat_\infty$:
 \begin{equation*}
  \xymatrix{
   \mc{M}^{\mr{univ},\circledast}_{[n]}
   \times_{\mr{RM}_{\Delta^n}}\mr{RM}_{\Delta^m}
   \ar[r]\ar[d]&
   \mc{M}^{\mr{univ},\circledast}_{[m]}\ar[d]\\
   \RMod_{\Delta^n}\ar[r]^-{\phi_*}&\RMod_{\Delta^m}.
    }
 \end{equation*}
\end{lem}
\begin{proof}
 Let $\mc{X}$ be the subcategory of
 $F\in\mr{Fun}(\Delta^1\times\Delta^1,\Cat_\infty)$ spanned by the
 vertices of the form
 \begin{equation*}
  \xymatrix{
   \mc{N}_0^{\circledast}\ar[d]\ar[r]&\mc{N}_1^{\circledast}\ar[d]\\
  \mr{RM}_{\Delta^m}\ar[r]^-{\phi_*}&\mr{RM}_{\Delta^n},
   }\qquad
   \xymatrix{
   F(0,0)\ar[r]\ar[d]&F(0,1)\ar[d]\\
  F(1,0)\ar[r]&F(1,1).
   }
 \end{equation*}
 such that
 $\mc{N}_0^{\circledast}\rightarrow
 \mr{RM}_{\Delta^m}\in\RMod_{\Delta^m}^\sim$,
 using the notation of the proof of Lemma \ref{basicproprmodsimp},
 and $\mc{N}_1^{\circledast}\rightarrow\mr{RM}_{\Delta^n}
 \in\RMod_{\Delta^n}^\sim$, and the square is a pullback square.
 Maps are those which induce maps in $\RMod_{\Delta^m}^{\sim}$ and
 $\RMod_{\Delta^n}^{\sim}$.
 We have the functor $\mc{X}\rightarrow\RMod_{\Delta^n}^{\sim}$.
 This is a trivial fibration by \cite[4.3.2.15]{HTT}.
 Let
 $\iota_i\colon\Delta^1\times\{i\}\rightarrow\Delta^1\times\Delta^1$.
 We have the following commutative diagram
 \begin{equation*}
  \xymatrix@C=30pt{
   \Delta^1\times\RMod_{\Delta^n}^{\sim}
   \ar@/_10pt/[rrd]&
   \Delta^1\times\mc{X}
   \ar@{^{(}->}[r]^-{\iota_1\times\mr{id}}
   \ar[l]_-{\mr{id}\times\iota_1^*}^-{\sim}&
   (\Delta^1\times\Delta^1)\times\mc{X}\ar[d]^{\tau}&
   \Delta^1\times\mc{X}
   \ar@{_{(}->}[l]_-{\iota_0\times\mr{id}}
   \ar[r]^-{\mr{id}\times\iota_0^*}&
   \Delta^1\times\RMod_{\Delta^m}^{\sim}
   \ar@/^10pt/[lld]\\
  &&\Cat_\infty.&&
   }
 \end{equation*}
 Unstraightening $\tau$, we get the pullback diagram of the form
 \begin{equation*}
  \xymatrix{
   \mc{N}_0^{\mr{univ},\circledast}\ar[r]\ar[d]_{\alpha}&
   \mc{N}_1^{\mr{univ},\circledast}\ar[d]^{\beta}\\
  \mr{RM}_{\Delta^m}\times\mc{X}\ar[r]&
   \mr{RM}_{\Delta^n}\times\mc{X}.
   }
 \end{equation*}
 The commutative diagram above identifies
 $\alpha$ with
 $\mc{M}^{\mr{univ},\circledast}_{[m]}\times_{\RMod_{\Delta^m}}
 \mc{X}\rightarrow\mr{RM}_{\Delta^m}\times\mc{X}$ and
 $\beta$ with $\mc{M}^{\mr{univ},\circledast}_{[n]}
 \rightarrow\mr{RM}_{\Delta^n}\times\RMod_{\Delta^n}$, and the lemma
 follows.
\end{proof}

\subsection{}
Recall that $\mr{Tw}^{\mr{op}}\mbf{\Delta}$ is the category
consisting of maps $[k]\rightarrow[n]$ in $\mbf{\Delta}$, and a morphism
$([k]\rightarrow[n])\rightarrow([k']\rightarrow[n'])$ is a commutative
diagram
\begin{equation*}
 \xymatrix{
  [k]\ar[d]&[k']\ar[l]\ar[d]\\
 [n]\ar[r]&[n']}
\end{equation*}
in $\mbf{\Delta}$.
We have the functor
$(\Phi,\Theta)\colon\mr{Tw}^{\mr{op}}\mbf{\Delta}\rightarrow\mbf{\Delta}^{\mr{op}}\times\mbf{\Delta}$
sending $([k]\rightarrow[n])$ to $([k],[n])$.
The functor $\Theta$ is a coCartesian fibration.
We denote by $\mr{Tw}^{\mr{op}}_{[n]}\mbf{\Delta}$ the fiber $\Theta^{-1}([n])$.
We define $\mr{Tw}^{\mr{op}}\mbf{\Delta}'$ by the full subcategory of $\mr{Tw}^{\mr{op}}\mbf{\Delta}$ spanned by the maps
$\sigma^i\colon[0]\rightarrow[n]$.
The category $\mr{Tw}^{\mr{op}}\mbf{\Delta}$ is an analogue of $\mr{Po}^{\mr{op}}$ in \cite[4.7.1]{HA},
but larger since Lurie considers inert maps whereas we consider all the maps in $\mbf{\Delta}$.
The inclusion $\Delta^1\times\mbf{\Delta}^{\mr{op}}\subset\Delta^1\times\mbf{\Delta}^{\mr{op}}_+$ factors through $\mr{RM}$,
and defines the map
$\overline{\chi}\colon\Delta^1\times\mbf{\Delta}^{\mr{op}}\rightarrow\mr{RM}$.
The notation is compatible with \cite[7.3.6]{GH}.

We have the commutative diagram of functors
\begin{equation*}
 \xymatrix@C=90pt{
  &\mr{Tw}^{\mr{op}}\mbf{\Delta}'
  \ar[d]_-{\{0\}}
  \ar@/_10pt/[dl]_{\Theta'}
  \ar@/^10pt/[rd]^-{\pi'}&\\
  \mbf{\Delta}&
  \Delta^1\times\mr{Tw}^{\mr{op}}\mbf{\Delta}
  \ar[l]_-{\Theta\circ\mr{pr}_2}^-{=:\overline{\Theta}}
  \ar[r]^-{\bigl(\overline{\chi}\circ(\mr{id}\times\Phi),
  \Theta\circ\mr{pr}_2\bigr)}
  _-{=:\overline{\pi}}&
  \mr{RM}\times\mbf{\Delta}\\
 &\mr{Tw}^{\mr{op}}\mbf{\Delta}
  \ar[u]^{\{1\}\times\mr{id}}_{=:i}\ar@/^10pt/[ul]^{\Theta}&
  }
\end{equation*}
Via the structural map $\mr{RM}\times\bp\RMod_{\mc{A}}\rightarrow
\mr{RM}\times\mbf{\Delta}$, the map
$\mc{M}^{\mr{univ},\circledast}\rightarrow\mr{RM}^{\mr{univ}}$
can be considered over $\mr{RM}\times\mbf{\Delta}$.
Thus, using the notation of \ref{defpushpullsim}, the diagram induces a diagram as follows:
\begin{equation}
 \label{commadjstrdiag}
 \xymatrix{
  \Theta'_*\pi'^*(\mc{M}^{\mr{univ},\circledast})
  \ar[d]&
  \overline{\Theta}_*\overline{\pi}^*(\mc{M}^{\mr{univ},\circledast})
  \ar[l]_-{\iota}\ar[d]\ar[r]^-{\alpha}&
  \Theta_*(\overline{\pi}\circ i)^*(\mc{M}^{\mr{univ},\circledast})\\
 \Theta'_*\pi'^*(\mr{RM}^{\mr{univ}})&
  \overline{\Theta}_*\overline{\pi}^*(\mr{RM}^{\mr{univ}})
  \ar[l]&
  }
\end{equation}

\subsection{}
\label{analeasydualap}
Let us analyze $\Theta_*(\overline{\pi}\circ
i)^*\mc{M}^{\mr{univ},\circledast}$ first.
By definition, the composition $\overline{\pi}\circ i$ is equal to the
composition
$\mr{Tw}^{\mr{op}}\mbf{\Delta}
\xrightarrow{\Phi\times\Theta}
\mbf{\Delta}^{\mr{op}}\times\mbf{\Delta}
\xrightarrow{\mf{a}\times\mr{id}}
\mr{RM}\times\mbf{\Delta}$.
We also have the canonical equivalence
$(\mf{a}\times\mr{id})^*\mc{M}^{\mr{univ},\circledast}\simeq
\mc{A}^{\circledast}\times\bp\RMod$.

\begin{dfn*}
 A vertex of $\Theta_*(\overline{\pi}\circ i)^*
 \mc{M}^{\mr{univ},\circledast}$ corresponds to a map
 $f\colon\mr{Tw}^{\mr{op}}_{[n]}\mbf{\Delta}\rightarrow
 \mc{A}^{\circledast}\times\bp\RMod$ over
 $\mr{Tw}^{\mr{op}}\mbf{\Delta}\rightarrow\mbf{\Delta}^{\mr{op}}
 \times\mbf{\Delta}$.
 We consider the full subcategory $\bp\mc{B}$
 of $\Theta_*(\overline{\pi}\circ i)^*
 \mc{M}^{\mr{univ},\circledast}$ spanned by the vertices satisfying
 \begin{enumerate}
  \item The composition $\mr{pr}_1\circ
	f\colon\mr{Tw}^{\mr{op}}_{[n]}\mbf{\Delta}
	\rightarrow\mc{A}^{\circledast}$
	sends an edge of $\mr{Tw}^{\mr{op}}_{[n]}$ to a coCartesian edge
	of $\mc{A}^{\circledast}$ over $\mbf{\Delta}^{\mr{op}}$;
       
  \item The composition $\mr{pr}_2\circ
	f\colon\mr{Tw}^{\mr{op}}_{[n]}\mbf{\Delta}\rightarrow\bp\RMod$
	sends any edge of $\mr{Tw}^{\mr{op}}_{[n]}\mbf{\Delta}$ to an
	equivalent edge in $\bp\RMod$.
 \end{enumerate}
\end{dfn*}

\begin{lem*}
 We have the canonical equivalence
 $(\mc{A}^{\circledast})^\vee\times_{\mbf{\Delta}}
 \bp\RMod\xrightarrow{\sim}\bp\mc{B}$ over $\mbf{\Delta}$.
\end{lem*}
\begin{proof}
 Recall that the functor $\Theta_*\colon(\sSet)_{/\mr{Tw}^{\mr{op}}\mbf{\Delta}}\rightarrow(\sSet)_{/\mbf{\Delta}}$
 is a right adjoint functor, and commutes with pullbacks.
 Thus, we have the isomorphisms
 \begin{equation*}
  \Theta_*(\overline{\pi}\circ i)^*\mc{M}^{\mr{univ},\circledast}
   \cong
   \Theta_*\bigl(\Phi^*\mc{A}^{\circledast}
   \times_{\mr{Tw}^{\mr{op}}\mbf{\Delta}}\Theta^*\bp\RMod
   \bigr)
   \cong
   \Theta_*\Phi^*\mc{A}^{\circledast}
   \times_{\mbf{\Delta}}\Theta_*\Theta^*\bp\RMod.
 \end{equation*}
 Thus, we have the canonical functor
 $(\mc{A}^{\circledast})^\vee\times_{\mbf{\Delta}}\bp\RMod
 \rightarrow
 \Theta_*(\overline{\pi}\circ i)^*\mc{M}^{\mr{univ},\circledast}$.
 This induces the functor in the statement of the lemma. We must show
 that it is an equivalence. Since the functor preserves Cartesian edges
 over $\mbf{\Delta}$, it suffices to show the equivalence for each
 $[n]\in\mbf{\Delta}$ by \cite[3.3.1.5]{HTT}.
 For a simplicial set $K$ and an $\infty$-category $\mc{C}$, let
 $\mr{Fun}(K,\mc{C})^{\mr{equiv}}$ be the full subcategory of
 $\mr{Fun}(K,\mc{C})$ spanned by functors sending any edge of $K$ to an
 equivalent edge in $\mc{C}$.
 We only need to show that the constant functor
 $c\colon\RMod_{\Delta^n}\rightarrow\mr{Fun}(\mr{Tw}^{\mr{op}}_{[n]}
 \mbf{\Delta},\mr{RMod}_{\Delta^n})^{\mr{equiv}}$ is a categorical
 equivalence. Consider the following left Kan extension diagram:
 \begin{equation*}
  \xymatrix{
   \{[n]\rightarrow[n]\}\ar[r]\ar[d]_{i}&\RMod_{\Delta^n}\ar[d]\\
  \mr{Tw}^{\mr{op}}_{[n]}\mbf{\Delta}\ar[r]\ar@{-->}[ru]&\{*\}.
   }
 \end{equation*}
 Invoking \cite[4.3.2.15]{HTT}, the restriction by $i$ is a trivial
 fibration, which gives a quasi-inverse to $c$.
\end{proof}

\begin{dfn}
 \label{strremarkoncatprod}
 Let $\alpha\colon\mr{Tw}^{\mr{op}}_{[n]}\mbf{\Delta}'\rightarrow
 \mr{RM}_{\Delta^n}$ be the functor defined by sending
 $\phi\colon[0]\rightarrow[n]$ to $(0_{\phi(0)},1)$.
 \begin{enumerate}
  \item A vertex of $\Theta'_*\pi'^*(\mr{RM}^{\mr{univ}})$ corresponds
	to a map $f\colon\mr{Tw}^{\mr{op}}_{[n]}\mbf{\Delta}'\rightarrow
	\mr{RM}^{\mr{univ}}_{(0,1),[n]}\cong
	(\mr{RM}_{\Delta^n})_{(0,1)}\times\mr{RMod}_{\Delta^n}$, where
	$(\mr{RM}_{\Delta^n})_{(0,1)}:=\mr{RM}_{\Delta^n}
	\times_{\mr{RM}}\{(0,1)\}\cong\Delta^n$, for some $n$.
	We consider the full subcategory $\bp\mr{Pre}\Str$ of
	$\Theta'_*\pi'^*(\mr{RM}^{\mr{univ}})$ spanned by the vertices
	satisfying the following conditions:
	\begin{enumerate}
	 \item The composition $\mr{pr}_1\circ
	       f\colon\mr{Tw}^{\mr{op}}_{[n]}\mbf{\Delta}'
	       \rightarrow(\mr{RM}_{\Delta^n})_{(0,1)}$ is equal to
	       $\alpha$;
	 
	 \item The composition $\mr{pr}_2\circ
	       f\colon\mr{Tw}^{\mr{op}}_{[n]}\mbf{\Delta}'
	       \rightarrow\mr{RMod}_{\Delta^n}$ is constant.
	\end{enumerate}
	 
  \item  We put
	 $\bp\Str:=\Theta'_*\pi'^*(\mc{M}^{\mr{univ},\circledast})
	 \times_{\Theta'_*\pi'^*(\mr{RM}^{\mr{univ}})}\bp\mr{Pre}\Str$.
 \end{enumerate}
\end{dfn}

\begin{rem*}
 \begin{enumerate}
  \item Informally, an object of $\bp\Str$ consists of $\mc{M}^{\circledast}\rightarrow\mr{RM}_{\Delta^n}$ in $\mr{RMod}_{\Delta^n}$ and objects
	$M_0,\dots,M_n\in\mc{M}_{(0,1)}$ such that $M_i$ is an object over $(0_i,1)\in\mr{RM}_{\Delta^n}$.
	This construction is a variant of the enriched string construction appeared in \cite[4.7.1]{HA} and \cite[7.3]{GH}.
	Even though an object of $\bp\Str$ is not a ``string'', we decided to call it a string
	because we decorate $\bp\Str$ to define the $\infty$-category of enriched strings in Definition \ref{dfnenstr},
	which is an analogy of the one in \cite{HA} or \cite{GH}.

  \item \label{remstrcatfibpro-1}
        We have the following diagram:
	\begin{equation*}
	 \xymatrix{
	  \bp\Str_{[n]}\ar@{-->}[rr]\ar@{-->}[dd]\ar@{-->}[rd]&&
	  \RMod_{\Delta^n}\ar[d]\\
	 &\mr{Fun}(\mr{Tw}^{\mr{op}}_{[n]}\mbf{\Delta}',
	  \mc{M}^{\mr{univ},\circledast}_{[n]})\ar[r]^-{b}\ar[d]&
	  \mr{Fun}(\mr{Tw}^{\mr{op}}_{[n]}\mbf{\Delta}',
	  \RMod_{\Delta^n})\\
	 \{*\}\ar[r]^-{a}&
	  \mr{Fun}(\mr{Tw}^{\mr{op}}_{[n]}\mbf{\Delta}',\mr{RM}_{\Delta^n}).
	  }
	\end{equation*}
	Here the dashed functors exhibits $\bp\Str$ as a limit of the
	diagram in $\sSet$. The map $a$ is a categorical fibration by
	\cite[2.4.6.5]{HTT} since $\mr{RM}_{\Delta^n}$ is gaunt,
	and $b$ is also a categorical fibration since
	$\mc{M}^{\mr{univ},\circledast}\rightarrow\mr{RM}
	\times\bp\RMod_{\mc{A}}$ is a coCartesian fibration. This
	implies that the limit in $\sSet$ is actually a limit in
	$\Cat_\infty$. In particular, the category $\bp\Str$ does not
	depend on the choice of a ``model'' of
	$\mc{M}^{\mr{univ},\circledast}$ up to equivalences, and the
	functor $\bp\Str_{[n]}\rightarrow\RMod_{\Delta^n}$ is a
	categorical fibration.
  
  \item We have the following diagram
	\begin{equation*}
	 \xymatrix{
	  \bp\Str\ar@{-->}[rr]\ar@{^{(}->}[d]&&\bp\RMod\ar[d]\\
	 \Theta'_*\pi'^*(\mr{RM}^{\mr{univ}})\ar[r]^-{c}&
	  \Theta'_*\pi'^*(\mr{RM}\times\bp\RMod)\ar[r]^-{\cong}&
	  \Theta'_*\Theta'^*(\bp\RMod).
	  }
	\end{equation*}
	By \cite[B.4.5]{HTT}, the map $c$ is a categorical
	fibration, and the fiber over $[n]\in\mbf{\Delta}$ coincides
	with $b$ above. By definition of $\bp\Str$, we have the
	dashed arrows making the diagram commutative whose fiber over
	$[n]$ is map of \ref{remstrcatfibpro-1}.
 \end{enumerate}
\end{rem*}

\begin{dfn}
 \label{dfnenstr}
 A vertex of
 $\overline{\Theta}_*\overline{\pi}^*(\mc{M}^{\mr{univ},\circledast})$
 consists of maps $g\colon\Delta^1\times
 \mr{Tw}^{\mr{op}}_{[n]}\mbf{\Delta}\rightarrow
 \mc{M}^{\mr{univ},\circledast}\times_{\bp\mr{RMod}}\mr{RMod}_{\Delta^n}$
 over $\mr{RM}$ for some $n$. We define $\bp\Str^{\mr{en},+}$ to be the
 full subcategory of
 $\overline{\Theta}_*\overline{\pi}^*(\mc{M}^{\mr{univ},\circledast})$
 spanned by the vertices satisfying the following conditions, which are
 analogous to the conditions of enriched strings
 (cf.\ \cite[7.3.7]{GH}):
 \begin{enumerate}
  \item\label{dfnenstr-1}
       The vertex sits over $\bp\Str$;
       
  \item\label{dfnenstr-2}
       Given a map
       $\phi\colon([k]\rightarrow[n])\rightarrow([l]\rightarrow[n])$ in
       $\mr{Tw}^{\mr{op}}_{[n]}\mbf{\Delta}$ such that $\phi(0)=0$, the
       edge $g(0,\phi)$ in $\mc{M}^{\mr{univ},\circledast}$ is a
       coCartesian edge with respect to the coCartesian fibration
       $s\colon\mc{M}^{\mr{univ},\circledast}\rightarrow\mr{RM}$;

  \item\label{dfnenstr-3}
       The map
       $g\colon\Delta^1\times\mr{Tw}^{\mr{op}}_{[n]}\mbf{\Delta}
       \rightarrow\mc{M}^{\mr{univ},\circledast}$
       is a $s$-left Kan extension of
       $g|_{\{0\}\times\mr{Tw}^{\mr{op}}_{[n]}}$;

  \item\label{dfnenstr-4}
       Given a map
       $\phi\colon([k]\rightarrow[n])\rightarrow([l]\rightarrow[n])$ in
       $\mr{Tw}^{\mr{op}}_{[n]}\mbf{\Delta}$, the edge
       $g(1,\phi)$ in $\mc{M}^{\mr{univ},\circledast}$ is a coCartesian
       edge with respect to the coCartesian fibration
       $s$.
 \end{enumerate}
\end{dfn}

\begin{rem*}
 \begin{enumerate}
  \item Given a map $\phi$ in $\mr{RM}$, let
	$\mr{Fun}_{\phi}(\Delta^1,\mc{M}^{\mr{univ},\circledast})
	^{\mr{coCart}}$ be the full subcategory of
	$\mr{Fun}(\Delta^1,\mc{M}^{\mr{univ},\circledast})$
	spanned by coCartesian edges over $\phi$.
	The canonical inclusion $\mr{Fun}_{\phi}
	(\Delta^1,\mc{M}^{\mr{univ},\circledast})^{\mr{coCart}}
	\rightarrow\mr{Fun}(\Delta^1,\mc{M}^{\mr{univ},\circledast})$
	is a categorical fibration because $\mr{RM}$ is gaunt and
	coCartesian edges are preserved by equivalence.
	The conditions except for \ref{dfnenstr-1} can be rephrased as
	certain edges in $\mc{M}^{\circledast}$ are coCartesian edges
	over some specified maps $\phi$ in $\mr{RM}$.
	Thus, we have the pullback diagrams in $\sSet$:
	\begin{equation*}
	 \xymatrix{
	  \mc{X}\ar[r]\ar[d]&
	  \mr{Fun}(\Delta^1\times\mr{Tw}^{\mr{op}}_{[n]}\mbf{\Delta},
	  \mc{M}^{\mr{univ},\circledast}_{[n]})\ar[d]^{a}\\
	 \bp\Str_{[n]}\ar[r]&
	  \mr{Fun}(\mr{Tw}^{\mr{op}}_{[n]}\mbf{\Delta}',
	  \mc{M}^{\mr{univ},\circledast}_{[n]}),
	  }\qquad
	  \xymatrix{
	  \bp\Str^{\mr{en},+}_{[n]}
	  \ar[r]\ar[d]&
	  \prod\mr{Fun}_{\phi}
	  (\Delta^1,\mc{M}^{\mr{univ},\circledast})
	  ^{\mr{coCart}}\ar[d]^{b}\\
	 \mc{X}\ar[r]&\prod\mr{Fun}(\Delta^1,\mc{M}^{\mr{univ},
	  \circledast}).
	  }
	\end{equation*}
	The functor $a$ is a categorical fibration by
	\cite[2.2.5.4]{HTT}, and $b$ is also a categorical fibration by
	observation above. This implies that the functor
	$\bp\Str^{\mr{en},+}_{[n]}\rightarrow\bp\Str_{[n]}$ is a
	categorical fibration, and $\bp\Str^{\mr{en},+}$ does not depend
	on the model of $\mc{M}^{\mr{univ},\circledast}$.
       
  \item Let us have a closer look at objects of $\bp\Str^{\mr{en},+}$.
	Vertices of $\bp\Str^{\mr{en},+}$ correspond to functors
	$g\colon\Delta^1\times
	\mr{Tw}^{\mr{op}}_{[n]}\mbf{\Delta}\rightarrow
	\mc{M}^{\mr{univ},\circledast}\times_{\bp\mr{RMod}}
	\mr{RMod}_{\Delta^n}$ for some $n$.
	Because $g$ is over $\bp\Str$, the map
	\begin{equation*}
	 \mr{Tw}^{\mr{op}}_{[n]}\mbf{\Delta}'\xrightarrow{\{0\}}
	  \Delta^1\times\mr{Tw}^{\mr{op}}_{[n]}\mbf{\Delta}
	  \xrightarrow{g}
	  \mc{M}^{\mr{univ},\circledast}\times_{\bp\mr{RMod}}
	  \mr{RMod}_{\Delta^n}
	  \xrightarrow{\mr{pr}_2}
	  \mr{RMod}_{\Delta^n}
	\end{equation*}
	is constant, and determines a coCartesian fibration of
	generalized $\infty$-operads
	$\mc{M}^{\circledast}\rightarrow\mr{RM}_{\Delta^n}$.
	Let $\phi_k\colon[n-k]\rightarrow[n]$ be an inert function
	considered as in $\mr{Tw}_{[n]}^{\mr{op}}\mbf{\Delta}$ such that
	$\phi(0)=k$. Then we can write $g(0,\phi_k)\simeq M_{k}\boxtimes
	A_{k+1}\boxtimes A_{k+2}\boxtimes\dots\boxtimes A_n$ with
	$M_k\in\mc{M}^{\circledast}_{(0,1)}$, $A_i\in\mc{A}$.
	For a careful reader, we note that $g(0,\phi_k)$ may not lie in
	$\mc{M}^{\circledast}$, but a generalized $\infty$-operad {\em
	equivalent} to $\mc{M}^{\circledast}$.
	Condition \ref{dfnenstr-3}, \ref{dfnenstr-4} implies that $A_i$
	are the same even if we change $k$.
	
	Now, let $\phi\colon[m]\rightarrow[n]$ a function considered as
	an object of $\mr{Tw}_{[n]}^{\mr{op}}\mbf{\Delta}$.
	Put $a^{\phi}_0=0$, and for $j>0$, define $a_j^{\phi}$
	inductively as follows: $a_{j+1}^{\phi}$ is the minimum number
	$a_j^{\phi}<k\leq m$ such that $\phi(a_j^{\phi})\neq\phi(k)$.
	This can be depicted as follows:
	\begin{equation*}
	 \xymatrix@R=8pt{
	  0=a^{\phi}_0\ar[rd]&\dots\ar[d]&a^{\phi}_1-1\ar[ld]&
	  a^{\phi}_1\ar[rd]&\dots\ar[d]&a^{\phi}_2-1\ar[ld]&
	  a^{\phi}_2\ar[rd]&
	  \dots\ar[d]\\
	 &\phi(a^{\phi}_0)&&&\phi(a^{\phi}_1)&&&\phi(a^{\phi}_2)
	  }
	\end{equation*}
	Then condition \ref{dfnenstr-2} implies that
	$g(0,\phi)=M_{\phi(0)}\boxtimes B_1\boxtimes B_2\dots\boxtimes
	B_m$
	where
	$B_{a_j^{\phi}}=A_{\phi(a_{j-1}^{\phi})+1}\otimes\dots\otimes
	A_{\phi(a_j^{\phi})}$ for $j>0$, and $B_i=\mbf{1}$ otherwise.

	As a summary, a vertex consists of a sequence
	\begin{equation*}
	 M_{0}\boxtimes	A_{1}\boxtimes A_{2}\boxtimes\dots\boxtimes A_n
	  \rightarrow
	  M_1\boxtimes	A_{2}\boxtimes\dots\boxtimes A_n
	  \rightarrow\dots\rightarrow
	  M_n
	\end{equation*}
	with lots of ``redundant information'' determined from the above
	sequence in an essentially unique manner.
 \end{enumerate}
\end{rem*}

\begin{lem}
 \label{basicpropofstr}
 Let $p\colon\mc{M}^{\mr{univ},\circledast}\rightarrow
 \mr{RM}\times\mbf{\Delta}$, $q\colon\bp\Str\rightarrow\mbf{\Delta}$,
 $r\colon\bp\Str^{\mr{en},+}\rightarrow\mbf{\Delta}$, and
 $\alpha\colon\bp\RMod\rightarrow\mbf{\Delta}$ be the canonical maps. We
 recall maps $f$, $g$, $h$ from {\normalfont\ref{setupmodspM}}.
 \begin{enumerate}
  \item\label{basicpropofstr-1}
       Let $x$ be a vertex of $\mc{M}^{\mr{univ},\circledast}$.
       Assume we are given a map $\phi\colon[m]\rightarrow[n]$ in
       $\mbf{\Delta}$, and let $e_{\phi}$ be a Cartesian edge in
       $\mr{RM}\times\bp\RMod$ over $\mbf{\Delta}$ with endpoint
       $f(x)$. Assume that $e_{\phi}$ has a $g$-coCartesian lifting
       $e'_{\phi}$ in $\mr{RM}^{\mr{univ}}$ with endpoint $h(x)$.
       Then there exists a $h$-Cartesian edge $y\rightarrow x$ in
       $\mc{M}^{\mr{univ},\circledast}$ which lifts $e'_{\phi}$.
       
  \item The maps $q$, $r$ are Cartesian fibrations.
	
  \item\label{basicpropofstr-3}
       The map $\bp\Str^{\mr{en},+}\rightarrow\bp\Str$ sends
       $q$-Cartesian edges to $r$-Cartesian edges, the map
       $\bp\Str\rightarrow\bp\RMod$ from Remark {\normalfont\ref{strremarkoncatprod}}
       sends $q$-Cartesian edges to $\alpha$-Cartesian edges, and
       $\bp\Str^{\mr{en},+}\rightarrow\bp\RMod$ sends $r$-Cartesian
       edges to $\alpha$-Cartesian edges.
	
  \item\label{basicpropofstr-5}
       The maps $\bp\Str_{[n]}\rightarrow\bp\RMod_{\Delta^n}$,
       $\bp\Str^{\mr{en},+}_{[n]}\rightarrow\bp\RMod_{\Delta^n}$
       are coCartesian fibrations, and the map
       $\bp\Str^{\mr{en},+}_{[n]}\rightarrow\bp\Str_{[n]}$ preserves
       coCartesian edges.
 \end{enumerate}
\end{lem}
\begin{proof}
 Let us show \ref{basicpropofstr-1}.
 Since $f$, $g$, $h$ are categorical fibrations, we may replace an edge
 we wish to lift by an edge equivalent to it.
 In view of Lemma \ref{cartedgedetelem} applied to the diagram
 (\ref{funddiagmod}), we may replace the diagram by the diagram of
 fibers over $e_\phi$.
 Write $f(x)=(X,\mc{M}^{\circledast})$ where $X\in\mr{RM}$ and
 $\mc{M}^{\circledast}\rightarrow\mr{RM}_{\Delta^n}$ be in
 $\RMod_{\Delta^n}$. By replacing $e_{\phi}$ by an edge equivalent to
 it, we may assume that $e_\phi$ is a morphism of the form
 $(\mr{id}_X,\iota)\colon
 (X,\mc{M}^{\circledast}\times_{\mr{RM}_{\Delta^n}}\mr{RM}_{\Delta^m})
 \rightarrow(X,\mc{M}^{\circledast})$, where $\iota$ is the canonical
 functor. We have the coCartesian fibration
 $\mc{M}^{\circledast}_X\rightarrow(\mr{RM}_{\Delta^n})_X$,
 and let $i\in\Delta^n$ be the vertex over which $x$ is
 lying. The fiber $(\mr{RM}_{\Delta^n})_X$ is equivalent to $\Delta^n$
 or $\{*\}$ depending on whether $X$ begins with $0$ or $1$.
 We have the $g$-coCartesian lifting $e'_{\phi}$ if and only if there
 exists $j\in[m]$ such that $\phi(j)=i$ in the case where $X$ begins
 with $0$ and it always have a lifting when $X$ begins with $1$.
 By replacing $e'_{\phi}$ by an edge equivalent to it, we may assume
 that $e'_{\phi}$ is of the form
 $\bigl((0_j,1\dots,1),\mc{M}^{\circledast}\times
 _{\mr{RM}_{\Delta^n}}\mr{RM}_{\Delta^m}\bigr)\rightarrow
 \bigl((0_i,1,\dots,1),\mc{M}^{\circledast}\bigr)$ when $X$ begins
 with $0$. If $X$ begins with $1$, erase $0_j$, $0_i$ from the map
 above. From now on, we only treat the case where $X$ begins with $0$
 since the other case is similarly and easier to check.
 Let $\Gamma^\vee_\phi:=
 \Gamma^\vee\times_{\mbf{\Delta},\phi}\Delta^1$ using the notation of
 \ref{gammaanddual}.
 Then the fiber of $h$ over $e'_{\phi}$ is Cartesian equivalent to the
 projection
 \begin{equation*}
  F\colon(\Delta^1\times\mc{M}^{\circledast}_X)
   \times_{\Delta^1\times\Delta^n}
   \Gamma^\vee_\phi\rightarrow\Gamma^\vee_\phi.
 \end{equation*}
 Since $\mc{M}^{\circledast}_X\rightarrow\Delta^n$ is a coCartesian
 fibration, this functor is coCartesian fibration as well.
 Invoking \cite[5.2.2.4]{HTT} (or Lemma \ref{cartedgedetelem}), we may
 replace $F$ by $F\times_{\Gamma^\vee_\phi,e'_\phi}\Delta^1$.
 Because $e'_\phi$ is a coCartesian edge in $\Gamma^\vee_\phi$ over
 $\Delta^1$, $F\times_{\Gamma^\vee_\phi,e'_\phi}\Delta^1$ is equivalent
 to $\Delta^1\times(\mc{M}^{\circledast}_X\times_{\Delta^n}
 \Delta^{\{i\}})\rightarrow\Delta^1$, and thus, we have a Cartesian lift.

 Let us show that $r$ is a Cartesian fibration.
 The argument works similarly, or even simpler, for $q$, by replacing
 $\mr{Tw}^{\mr{op}}\mbf{\Delta}'$ by $\mr{Tw}^{\mr{op}}\mbf{\Delta}$, so
 we omit. Take a vertex of
 $\bp\Str^{\mr{en},+}$ corresponding to a map
 $g\colon\Delta^1\times\mr{Tw}^{\mr{op}}_{[a]}
 \mbf{\Delta}\rightarrow\mc{M}^{\mr{univ},\circledast}$,
 and take a map $\phi\colon[b]\rightarrow[a]$.
 We define an edge $\phi^*(g)\rightarrow g$ over $\phi$ as follows.
 Since $\Theta$ is a coCartesian fibration, we have a functor
 $\phi_*\colon\mr{Tw}^{\mr{op}}_{[b]}\mbf{\Delta}\rightarrow
 \mr{Tw}^{\mr{op}}_{[a]}\mbf{\Delta}$. We wish to take a $p$-right
 Kan extension $e\colon\Delta^1\times\mr{Tw}^{\mr{op}}_{\phi}
 \mbf{\Delta}\rightarrow\mc{M}^{\mr{univ},\circledast}$
 of $g$ along $\Delta^1\times\mr{Tw}^{\mr{op}}_{[a]}\mbf{\Delta}
 \hookrightarrow\Delta^1\times\mr{Tw}^{\mr{op}}_{\phi}\mbf{\Delta}$, and
 define the corresponding edge in $\bp\Str^{\mr{en},+}$
 as the desired edge.
 Let us check the existence of the extension. By \cite[4.3.2.15]{HTT},
 for $v:=(i,v'\colon[k]\rightarrow[b])\in
 \Delta^1\times\mr{Tw}^{\mr{op}}_{[b]}
 \mbf{\Delta}$, we only need to check that the diagram on the left
 below
 \begin{equation*}
  \xymatrix@C=50pt{
   (\Delta^1\times\mr{Tw}^{\mr{op}}_{[a]}\mbf{\Delta})_{v/}
   \ar[r]^-{g}\ar[d]&
   \mc{M}^{\mr{univ},\circledast}\ar[d]^{h}\\
  (\Delta^1\times\mr{Tw}^{\mr{op}}_{[a]}
   \mbf{\Delta})_{v/}^{\triangleleft}
   \ar[r]\ar@{-->}[ur]&
   \mr{RM}^{\mr{univ}}
   }\qquad
   \xymatrix@C=50pt{
   \Delta^0\ar[r]^-{g(I)}\ar[d]&
   \mc{M}^{\mr{univ},\circledast}\ar[d]^{h}\\
  (\Delta^0)^{\triangleleft}
   \ar[r]\ar@{-->}[ur]^-{e'}&
   \mr{RM}^{\mr{univ}}
   }
 \end{equation*}
 has a $p$-limit. The object $I:=((i,v')\rightarrow(i,\phi_*(v')))$
 in $\Delta^1\times\mr{Tw}^{\mr{op}}_{\phi}\mbf{\Delta}$ is an initial
 object in
 $(\Delta^1\times\mr{Tw}^{\mr{op}}_{[a]}\mbf{\Delta})_{v/}$. By
 \cite[4.3.1.7]{HTT}, we need to check that the induced diagram
 on the right above extends to a $p$-limit diagram.
 Thus we are reduced to checking the existence of a $p$-Cartesian edge
 $e'((\Delta^0)^{\triangleleft})$ by \cite[4.3.1.4]{HTT}, where $g(I)$
 is over $((0_{\phi
 v'(0)},\underbrace{1,\dots,1}_{k+1}),[a])\in
 \mr{RM}_{\Delta^a}\times\mbf{\Delta}$, and the cone point is sent to
 $((0_{v'(0)},\underbrace{1,\dots,1}_{k+1}),[b])
 \in\mr{RM}_{\Delta^b}\times\mbf{\Delta}$. 
 Thus, the existence of a $h$-Cartesian edge is exactly the content of
 \ref{basicpropofstr-1}.

 It remains to check that the edge $\phi^*(g)\rightarrow g$ is a
 $q$-Cartesian edge. In view of \cite[2.4.1.4]{HTT},
 given any map $\Delta^n\rightarrow\mbf{\Delta}$, we need to solve the
 lifting problem
 \begin{equation*}
  \xymatrix@C=70pt{ 
   \Lambda^n_n\times_{\mbf{\Delta}}\mr{Tw}^{\mr{op}}\mbf{\Delta}'
   \ar[r]^-{f}\ar[d]&
   \mc{M}^{\mr{univ},\circledast}
   \ar[d]\\
  \Delta^n\times_{\mbf{\Delta}}\mr{Tw}^{\mr{op}}\mbf{\Delta}'
   \ar[r]\ar@{-->}[ur]&
   \mr{RM}^{\mr{univ}}
   }
 \end{equation*}
 where
 $f|_{\Delta^{\{n-1,n\}}\times_{\mbf{\Delta}}\mr{Tw}^{\mr{op}}\mbf{\Delta}'}$
 is the edge $e$. Apply \cite[B.4.8]{HA} with
 $\mc{C}=\Delta^n\times_{\mbf{\Delta}}\mr{Tw}^{\mr{op}}\mbf{\Delta}'$,
 $\mc{C}^0:=\{n\}\times_{\Delta^n}\mc{C}$. We need to check that
 $e\colon\mr{Tw}^{\mr{op}}_{\phi}\mbf{\Delta}'\rightarrow
 \mc{M}^{\mr{univ},\circledast}$ is a $h$-right Kan extension of $g$,
 which follows from the construction.

 The claim \ref{basicpropofstr-3} follows by concrete description of
 Cartesian edges.
 For claim \ref{basicpropofstr-5}, consider the following maps
 \begin{equation*}
  \{*\}
   \xleftarrow{t}
   \Delta^1\times\mr{Tw}^{\mr{op}}_{[n]}\mbf{\Delta}
   \xrightarrow{u}
   \mr{RM}\times\mbf{\Delta},
 \end{equation*}
 where $u$ is the map induced by $\overline{\pi}$.
 Invoking \cite[3.2.2.12]{HTT}, the map
 $t_*u^*(\mc{M}^{\mr{univ},\circledast})\rightarrow
 t_*u^*(\mr{RM}\times\RMod_{\Delta^n})$ is a coCartesian fibration. Let $\mc{X}$
 be the full subcategory of $\mr{RM}^{\mr{univ}}$ spanned by functors
 $\Delta^1\times\mr{Tw}^{\mr{op}}_{[n]}\mbf{\Delta}\rightarrow
 \RMod_{\Delta^n}$ such that all the edges are sent to equivalent edges
 and the restriction to $\mr{Tw}^{\mr{op}}_{[n]}\mbf{\Delta}'$ is
 constant. Then $\bp\Str_{[n]}^{\mr{en},+}$ is a full subcategory of the
 pullback of $t_*u^*(\mc{M}^{\mr{univ},\circledast})$ by $\mc{X}$.
 The concrete description of coCartesian edges of \cite[3.2.2.12]{HTT}
 allows us to show that the map
 $\bp\Str_{[n]}^{\mr{en},+}\rightarrow\mc{X}$ is a coCartesian
 fibration. By arguing similarly to the last half of the proof of Lemma
 \ref{analeasydualap}, $\mc{X}\rightarrow\RMod_{\Delta^n}$ is a trivial
 fibration, thus the claim follows for $\bp\Str_{[n]}^{\mr{en},+}$.
 The claim for $\bp\Str_{[n]}$ can be shown similarly. The preservation
 of coCartesian edges follows by the concrete description.
\end{proof}

\subsection{}
We apply dualizing construction of \S\ref{dualconst} to the Cartesian
fibrations $\bp\Str$ and $\bp\Str^{\mr{en},+}$ over
$\mbf{\Delta}$, and induce coCartesian fibrations
\begin{equation*}
 \Str:=\mb{D}^{-1}_{\mbf{\Delta}}
  (\bp\Str)\rightarrow\mbf{\Delta}^{\mr{op}},\quad
  \Str^{\mr{en},+}:=\mb{D}^{-1}_{\mbf{\Delta}}(\bp\Str^{\mr{en},+})
  \rightarrow\mbf{\Delta}^{\mr{op}}.
\end{equation*}
Recall that
\begin{equation*}
 \mb{D}^{-1}_{\mbf{\Delta}}(\bp\mc{B})
  \simeq
  \mc{A}^{\circledast}\times\RMod
\end{equation*}
by Theorem \ref{dualthm}.
Using the notation of \ref{dfnofbcstar}, we put $\Str^{\sim}:=\Str\basech\Str^{\simeq}_{[0]}$.
Since $\Str\rightarrow\mbf{\Delta}^{\mr{op}}$ is a coCartesian
fibration, so is $\Str^{\sim}$. We put
$\Str^{\mr{en},+,\sim}:=\Str^{\mr{en},+}\basech
(\Str^{\mr{en},+}_{[0]})^{\simeq}$.
By the functoriality of $\mb{D}^{-1}$, we have the following commutative
diagram:
\begin{equation*}
 \xymatrix{
  \Str^{\mr{en},+,\sim}\ar[r]^-{\alpha}\ar[d]_{\iota}&
  \mc{A}^{\circledast}\times_{\mbf{\Delta}^{\mr{op}}}
  \RMod\ar[r]^-{\mr{pr}_1}&
  \mc{A}^{\circledast}\ar[d]\\
 \Str^\sim\ar[r]&\RMod\ar[r]&
  \mbf{\Delta}^{\mr{op}}.
  }
\end{equation*}
Here $\iota$ and $\alpha$ are the functors induced by taking dual
of the corresponding maps in (\ref{commadjstrdiag}).
We wish to take the left Kan extension of $\alpha$ along $\iota$.

\begin{lem*}
 \label{genopestr}
 \begin{enumerate}
  \item\label{genopestr-1}
       The map
       $\Str^{\mr{en},+,\sim}_{[0]}\rightarrow\Str^{\sim}_{[0]}$ is a
       categorical equivalence of Kan complexes.
  
  \item\label{genopestr-2}
       The coCartesian fibrations
	$\Str^{\sim}\rightarrow\mbf{\Delta}^{\mr{op}}$ and
	$\Str^{\mr{en},+,\sim}\rightarrow\mbf{\Delta}^{\mr{op}}$ are
       generalized $\infty$-operads.
       
  \item\label{genopestr-3}
       The functor $\Str^{\sim}\rightarrow\RMod$ induces a categorical equivalence
       $\Str^{\sim}\rightarrow\RMod\basech\Str_{[0]}^{\sim}\simeq\RMod^{\circledast}\basech\Str_{[0]}^{\sim}$
       {\normalfont(}see {\normalfont\ref{modelofrmod}} for the notation{\normalfont)}.
 \end{enumerate}
\end{lem*}
\begin{proof}
 Let us show \ref{genopestr-1}. By definition,
 $\Str^{\mr{en},+,\sim}_{[0]}$ and $\Str^{\sim}_{[0]}$ are Kan
 complexes. By Lemma \ref{propspccatinf}, it is enough to show that the
 map $\Str^{\mr{en},+}_{[0]}\rightarrow\Str_{[0]}$ is a categorical
 equivalence. For this, it suffices to show that the
 map $\bp\Str^{\mr{en},+}_{[0]}\rightarrow\bp\Str_{[0]}$ is a trivial
 fibration. Since $\bp\Str^{\mr{en},+}_{[0]}$ is a full subcategory of
 $\mr{Fun}_{\mr{RM}}(\Delta^1\times\mr{Tw}^{\mr{op}}_{[0]}
 \mbf{\Delta},\mc{M}^{\circledast}_{[0]})$ and $\bp\Str_{[0]}$ has a
 similar description, we invoke \cite[4.3.2.15]{HTT}.
 
 Let us show \ref{genopestr-2}. Since these are coCartesian fibrations,
 by \cite[2.1.2.12]{HA}, we only need to check the Segal condition.
 We only treat the case $\Str^{\mr{en},+,\sim}$ as the verification is
 similar, and this case is more complicated.
 It suffices to show that $\Str^{\mr{en},+}$ is a generalized $\infty$-operad.
 Since the Segal condition is stable under taking dual,
 it suffices to show that the map
 \begin{equation*}
  \bp\Str^{\mr{en},+}_{[n]}\rightarrow
   \bp\Str^{\mr{en},+}_{\{0,1\}}
   \times^{\mr{cat}}_{\bp\Str^{\mr{en},+}_{\{1\}}}
   \bp\Str^{\mr{en},+}_{\{1,2\}}
   \dots
   \times^{\mr{cat}}_{\bp\Str^{\mr{en},+}_{\{n-1\}}}
   \bp\Str^{\mr{en},+}_{\{n-1,n\}}
 \end{equation*}
 is a categorical equivalence. The argument is similar to that of
 \cite[4.7.1.13]{HA}.
 Let $I\subset[n]$ be a subset such that $m:=\#I$. Let $\mc{X}_I$ be the
 full subcategory of
 $\mr{Fun}(\Delta^1\times\mr{Tw}^{\mr{op}}_{[m]}\mbf{\Delta},
 \mc{M}^{\mr{univ},\circledast}_{[n]})$ spanned by functors which can be
 lifted to a functor $\Delta^1\times\mr{Tw}^{\mr{op}}_{[m]}\mbf{\Delta}
 \rightarrow
 \mc{M}^{\mr{univ},\circledast}_{[n]}\times_{\mr{RM}_{\Delta^n}}
 \mr{RM}_{\Delta^I}$ such that the composition with
 $\mc{M}^{\mr{univ},\circledast}_{[n]}\times_{\mr{RM}_{\Delta^n}}
 \mr{RM}_{\Delta^I}\rightarrow\mc{M}^{\mr{univ},\circledast}_{[m]}$
 belongs to $\bp\Str^{\mr{en},+}_{[m]}$.
 By Lemma \ref{compofcartRModM}, we have
 $\mc{X}_{I}\simeq\bp\Str^{\mr{en},+}_{I}\times_{\RMod_{I}}
 \RMod_{\Delta^n}$. Since $\bp\RMod_{\mbf{\Delta}}$ satisfies the Segal
 condition, it suffices to show that the map
 $\bp\Str^{\mr{en},+}\rightarrow\mc{X}_{\{0,1\}}
 \times^{\mr{cat}}_{\mc{X}_{\{1\}}}
 \mc{X}_{\{1,2\}}\times^{\mr{cat}}\dots
 \times^{\mr{cat}}\mc{X}_{\{n-1,n\}}$ is a categorical equivalence.

 Let $\mr{Tw}^{\mr{op}}_{[n]}\mbf{\Delta}^0$ be the
 full subcategory of $\mr{Tw}^{\mr{op}}_{[n]}\mbf{\Delta}$ spanned by
 maps $[m]\rightarrow[n]$ which factors through an inert morphism of the
 form $\rho^i\colon[1]\rightarrow[n]$.
 Let $p\colon\mc{M}^{\mr{univ},\circledast}_{[n]}\rightarrow\mr{RM}$,
 which is a coCartesian fibration. We let $\mc{X}$ be the
 full subcategory of $\mr{Fun}_{\mr{RM}}
 (\Delta^1\times\mr{Tw}^{\mr{op}}_{[n]}\mbf{\Delta}^0,
 \mc{M}^{\mr{univ},\circledast}_{[n]})$ spanned by functors such that
 \begin{enumerate}
  \item We have the inclusion
	$\mr{Tw}^{\mr{op}}_{[n]}\mbf{\Delta}'\rightarrow
	\mr{Tw}^{\mr{op}}_{[n]}\mbf{\Delta}^0
	\xrightarrow{\{0\}\times\mr{id}}
	\Delta^1\times\mr{Tw}^{\mr{op}}_{[n]}\mbf{\Delta}^0$.
	Given a vertex, if we restrict the functor along this inclusion,
	the functor belongs to $\bp\Str$;
	
  \item Let $\phi\colon[1]\rightarrow[n]$ be an inert map. Since
	$\Theta$ is a coCartesian fibration, we have the map
	$\phi_*\colon\mr{Tw}^{\mr{op}}_{[1]}\mbf{\Delta}\rightarrow
	\mr{Tw}^{\mr{op}}_{[n]}\mbf{\Delta}$.
	By the first condition, given a vertex, the functor
	$\Delta^1\times\mr{Tw}^{\mr{op}}_{[1]}\mbf{\Delta}\rightarrow
	\mc{M}^{\mr{univ},\circledast}_{[n]}$ induced by $\phi_*$
	factors through $\mc{M}^{\mr{univ},\circledast}_{[1]}$.
	Then, this functor belongs to $\bp\Str^{\mr{en},+}_{[1]}$.
 \end{enumerate}
 This $\mc{X}$ is a model for the product and we must show that the map
 $\bp\Str^{\mr{en},+}\rightarrow\mc{X}$ is a categorical equivalence.

 Now, let $\mr{Tw}^{\mr{op}}_{[n]}\mr{Inc}^1$ (resp.\
 $\mr{Tw}^{\mr{op}}_{[n]}\mr{Inc}^0$) be the full subcategory of
 $\mr{Tw}^{\mr{op}}_{[n]}\mbf{\Delta}$ spanned by inert
 maps $[k]\rightarrow[n]$ (resp.\ inert maps $[k]\rightarrow[n]$ where
 $k=0,1$). For $i=0,1$, let $\mc{Y}_i$ be the full subcategory of
 $\mr{Fun}_{\mr{RM}}(\Delta^1\times\mr{Tw}_{[n]}^{\mr{op}}\mr{Inc}^i,
 \mc{M}^{\mr{univ},\circledast}_{[n]})$
 spanned by functors which satisfies the conditions of Definition
 \ref{dfnenstr} if we replace $\mr{Tw}^{\mr{op}}_{[n]}\mbf{\Delta}$
 by $\mr{Tw}^{\mr{op}}_{[n]}\mr{Inc}^i$.
 We have the following commutative diagram of simplicial sets on the
 left induced by the commutative diagram on the left:
 \begin{equation*}
  \xymatrix{
   \bp\Str^{\mr{en},+}\ar[r]^-{\theta_2}\ar[d]&
   \mc{Y}_1\ar[d]^{\tau}\\
  \mc{X}\ar[r]^-{\theta_1}&\mc{Y}_0,
   }
   \hspace{5em}
   \xymatrix{
   \mr{Tw}^{\mr{op}}_{[n]}\mbf{\Delta}&
   \mr{Tw}^{\mr{op}}_{[n]}\mr{Inc}^1\ar[l]\\
  \mr{Tw}^{\mr{op}}_{[n]}\mbf{\Delta}^0\ar[u]&
   \mr{Tw}^{\mr{op}}_{[n]}\mr{Inc}^0.\ar[u]\ar[l]
   }
 \end{equation*}
 It suffices to show that $\theta_1$, $\theta_2$, $\tau$ are categorical
 equivalences. The verification for $\tau$ is similar to the
 proof of \cite[4.7.1.13]{HA}, so we omit.
 Let us check that $\theta_2$ is a trivial fibration.
 The verification of $\theta_1$ is similar, so we omit.
 The strategy is similar to \cite[4.7.1.13]{HA}.
 In view of \cite[4.3.2.15]{HTT}, it suffices to show the following two
 assertions:
\begin{enumerate}
 \item For any $G\in\mc{Y}_1$, $p$-left Kan extension of $G$ along the
       inclusion $\mr{Tw}^{\mr{op}}_{[n]}\mr{Inc}^1\hookrightarrow
       \mr{Tw}^{\mr{op}}_{[n]}\mbf{\Delta}$ exists;

 \item Any
       $F\in\mr{Fun}(\Delta^1\times\mr{Tw}^{\mr{op}}_{[n]}\mbf{\Delta})$
       is in $\mc{X}_2$ if and only if
       $G:=F|_{\Delta^1\times\mr{Tw}^{\mr{op}}_{[n]}\mr{Inc}^1}$ is in
       $\mc{X}_1$ and $F$ is a $p$-left Kan extension of $G$.
\end{enumerate}
 The verification is standard: Fix an object $C:=(a,[k]\rightarrow[n])$
 in $\Delta^1\times\mr{Tw}^{\mr{op}}_{[n]}\mbf{\Delta}$ and we wish to
 show the existence of the $p$-colimit of the diagram
 \begin{equation*}
  (\Delta^1\times\mr{Tw}^{\mr{op}}_{[n]}\mr{Inc}^1)
   \times_{\Delta^1\times\mr{Tw}^{\mr{op}}_{[n]}\mbf{\Delta}}
   (\Delta^1\times\mr{Tw}^{\mr{op}}_{[n]}\mbf{\Delta})_{/C}
   \rightarrow
   \mc{M}^{\mr{univ},\circledast}_{[n]}
 \end{equation*}
 This category
 has an initial object. More precisely if we write
 $C=(a,\phi\colon[k]\rightarrow[n])$, there exists a unique inert map
 $\psi\colon[k']\hookrightarrow[n]$ such that $\psi(0)=\phi(0)$ and
 $\psi(k')=\phi(k)$. The initial object is $(a,\psi)\rightarrow C$.
 Since $p$ is coCartesian, we get the existence by \cite[4.3.1.4]{HTT}.
 This construction also tells us that
 $F$ is a $p$-left Kan extension if and only if the induced map
 $F(a,[k]\rightarrow[n])\rightarrow F(a,[k']\rightarrow[n])$ is a
 $p$-coCartesian edge. Thus, we also have the second assertion.
 
 Finally, let us show \ref{genopestr-3}. Since we have
 \begin{equation*}
  \Str^{\sim}\simeq\Str\basech\Str_{[0]}^{\simeq},
   \qquad
   \RMod\basech\Str_{[0]}^{\sim}\simeq
   (\RMod\basech\Str_{[0]})\basech\Str_{[0]}^{\simeq},
 \end{equation*}
 it suffices to show the map $\Str\rightarrow\RMod\basech\Str_{[0]}$ is
 an equivalence. This is equivalent to showing the induced functor
 $\bp\Str\rightarrow\bp\RMod\basech\Str_{[0]}$ is an equivalence.
 Since both $\bp\Str$ and $\bp\RMod^{\circledast}\basech\Str_{[0]}$ are
 Cartesian fibrations over $\mbf{\Delta}$ and preserves coCartesian
 edges by Lemma \ref{basicpropofstr},
 it suffices to check the equivalence for each fiber over $\mbf{\Delta}$
 by \cite[3.3.1.5]{HTT}. We choose the following commutative diagram
 (which is possible up to contractible space of choices)
 \begin{equation*}
  \xymatrix@C=50pt{
   \mc{M}^{\mr{univ}}_{[n]}\ar[r]^-{\prod\rho^i_!}\ar[d]&
   \prod_{i\in[n]}\mc{M}^{\mr{univ}}_{[0]}\ar[d]\\
  \RMod_{[n]}\ar[r]^-{\prod\rho^i_!}&\prod_{i\in[n]}\RMod_{[0]}.
   }
 \end{equation*}
 This diagram induces the map $\bp\Str_{[n]}\rightarrow\RMod_{\Delta^n}
 \times^{\mr{cat}}_{(\RMod_{\Delta^0})^{(\times(n+1))}}
 \bp\Str_{[0]}^{\times(n+1)}$.
 It is reduced to showing that this is an equivalence.
 We have an isomorphism
 $\mr{Tw}_{[n]}^{\mr{op}}\mbf{\Delta}'\cong\coprod_{i\in[n]}\{i\}$ sending
 $\phi\colon[0]\rightarrow[n]$ to $\phi(0)$.
 We consider the maps
 $\{i\}\rightarrow\mr{RM}_{\Delta^{\{i\}}}\rightarrow\mr{RM}_{\Delta^n}$
 where the first map sends to $(0,1)$.
 Unwinding the definition, this isomorphism induces the equivalences
 \begin{align*}
  \bp\Str_{[n]}&\cong
  \mr{Fun}_{\mr{RM}_{\Delta^n}}
  (\coprod_{i\in[n]}\{i\},\mc{M}^{\mr{univ},\circledast}_{[n]})
  \times_{\mr{Fun}(\coprod\{i\},\RMod_{\Delta^n}),\alpha}
  \RMod_{\Delta^n}\\
  &\simeq
  \prod_{i\in[n]}\mr{Fun}_{\mr{RM}_{\Delta^n}}
  (\{i\},\mc{M}^{\mr{univ},\circledast}_{[n]})
  \times^{\mr{cat}}_{\prod\mr{Fun}(\{i\},\RMod_{\Delta^n}),\Delta}
  \RMod_{\Delta^n},
 \end{align*}
 where $\alpha$ is induced by the unique map
 $\coprod\{i\}\rightarrow\{*\}$, $\Delta$ is the diagonal map.
 The second equivalence follows from Remark \ref{strremarkoncatprod}.
 Using this, we may compute
 \begin{align*}
  \RMod_{\Delta^n}
  &\times^{\mr{cat}}
  _{(\RMod_{\Delta^0})^{\times(n+1)}}
  (\bp\Str_{[0]}^\sim)^{\times(n+1)}\\
  &\simeq
  \RMod_{\Delta^n}
  \times^{\mr{cat}}
  _{\prod\RMod_{\Delta^{\{i\}}}}
  \prod\mr{Fun}_{\mr{RM}_{\Delta^{\{i\}}}}
  (\{i\},\mc{M}^{\mr{univ},\circledast}_{\{i\}}).
 \end{align*}
 Finally, we have
 \begin{align*}
  \mr{Fun}_{\mr{RM}_{\Delta^n}}(\{i\},\mc{M}^{\mr{univ}}_{[n]})
  &\cong
  \mr{Fun}_{\mr{RM}_{\{i\}}}((0_i,1),\mc{M}^{\mr{univ}}_{[n]}
  \times^{\mr{cat}}_{\mr{RM}_{\Delta^n}}\mr{RM}_{\{i\}})\\
  &\rightarrow
  \mr{Fun}_{\mr{RM}_{[0]}}((0,1),\mc{M}^{\mr{univ}}_{[0]}
  \times^{\mr{cat}}_{\RMod_{\Delta^{\{i\}}}}\RMod_{\Delta^n})\\
  &\cong
  \mr{Fun}_{\mr{RM}_{[0]}}((0,1),\mc{M}^{\mr{univ}}_{[0]})
  \times^{\mr{cat}}_{\RMod_{\Delta^{\{i\}}}}\RMod_{\Delta^n}
 \end{align*}
 The middle map is a categorical equivalence by Lemma
 \ref{compofcartRModM}. Combining these three equivalences, we have the
 desired equivalence.
\end{proof}

\begin{lem}
 \label{exisintmorobj}
 Let $\mc{M}^{\circledast}\rightarrow\mr{RM}_{\Delta^1}$ be an object of $\LinCat_{\Delta^1}$.
 Let $\mc{M}^{\circledast}_i\rightarrow\mr{RM}$ be the pullback by $\mr{RM}_{\Delta^{\{i\}}}\rightarrow\mr{RM}_{\Delta^1}$ for $i=0,1$.
 Assume we are given $M_i\in\mc{M}_i$.
 Then there exists an object $\Mor_{\mc{M}}(M_0,M_1)$ equipped with a map
 $M_0\boxtimes\Mor_{\mc{M}}(M_0,M_1)\rightarrow M_1$ over the active map in $\mc{M}^{\circledast}$ having the universal property that for any
 $A\in\mc{A}$, the induced map
 \begin{equation*}
  \mr{Map}_{\mc{A}}(A,\Mor_{\mc{M}}(M_0,M_1))\rightarrow
   \mr{Map}_{\mc{M}}(M_0\otimes_{\mc{M}}A,M_1)
 \end{equation*}
 is a homotopy equivalence.
 If $F^{\circledast}\colon\mc{M}^{\circledast}_0\rightarrow\mc{M}_1^{\circledast}$ is the monoidal functor of generalized
 $\infty$-operads classified by $\mc{M}^{\circledast}$
 {\normalfont(}{\em i.e.}\ the functor obtained by applying the construction of {\normalfont\cite[5.2.1]{HTT}} to the map
 $\mc{M}^{\circledast}\rightarrow\mr{RM}_{\Delta^1}\rightarrow\Delta^1${\normalfont)},
 then $\Mor_{\mc{M}}(M_0,M_1)\simeq\Mor_{\mc{M}_1}(F(M_0),M_1)$, where $\Mor_{\mc{M}_1}$ is the morphism object
 {\normalfont(}cf.\ {\normalfont\cite[4.2.1.33]{HA})}.
\end{lem}
\begin{proof}
 Consider the functor
 \begin{equation*}
  \mr{Map}(M_0\otimes_{\mc{M}}(-),M_1)\colon
  \mc{A}^{\mr{op}}\xrightarrow{(M_0\otimes,M_1)}
   \mc{M}^{\mr{op}}\times\mc{M}
   \xrightarrow{\mr{Map}}
   \Spc.
 \end{equation*}
 It suffices to show that this functor is equivalent to
 $\mr{Map}_{\mc{M}_1}(F(M_0)\otimes_{\mc{M}_1}(-),M_1)$.
 Using \cite[5.2.1.4]{HTT}, choose a functor
 $G\colon\mc{M}_0\times\Delta^1\rightarrow\mc{M}$ associated to
 $\mc{M}$. Let $\iota_i\colon\mc{M}_i\rightarrow\mc{M}$ be the canonical
 functor, and $F':=\iota_1\circ F$. Then $G$ determines a map of
 functors $\iota_0\rightarrow F'$. This induces the map of functors
 \begin{equation*}
  \mr{Map}_{\mc{M}}\bigl(F'(-),(-)\bigr)\rightarrow
   \mr{Map}_{\mc{M}}\bigl(\iota_0(-),(-)\bigr)\colon
   \mc{M}_0^{\mr{op}}\times\mc{M}\rightarrow\Spc.
 \end{equation*}
 This induces an equivalence
 $\mr{Map}_{\mc{M}_1}\bigl(F(-),(-)\bigr)\xrightarrow{\sim}
 \mr{Map}_{\mc{M}}\bigl(\iota_0(-),(-)\bigr)$.
 Thus, we have
 \begin{equation*}
  \mr{Map}_{\mc{M}}\bigl(M_0\otimes_{\mc{M}}(-),M_1\bigr)
  \xleftarrow{\sim}
  \mr{Map}_{\mc{M}_1}\bigl(F(M_0\otimes_{\mc{M}}(-)),
  M_1\bigr)
  \simeq
  \mr{Map}_{\mc{M}_1}\bigl(F(M_0)\otimes_{\mc{M}_1}(-),M_1\bigr),
 \end{equation*}
 and we get the desired equivalence.
\end{proof}

\subsection{}
To proceed, we need to restrict our attention to $\LinCat$ in $\RMod$.
Let $\mc{A}^{\circledast}$ be a presentable monoidal $\infty$-category (cf.\ \ref{presmoncatint}).
We put ${}_{\mc{L}}\Str^{\sim}:=\Str^{\sim}\times_{\RMod^{\circledast}}\LinCat^{\circledast}_{\mc{A}}$,
${}_{\mc{L}}\Str^{\mr{en},+,\sim}:=\Str^{\mr{en},+,\sim}\times_{\RMod^{\circledast}}\LinCat^{\circledast}_{\mc{A}}$
where the fiber products are taken in $\mr{Op}^{\mr{ns},\mr{gen}}_{\infty}$.
Assume we are given a diagram
\begin{equation*}
 \xymatrix{
  \mc{A}\ar[r]\ar[d]_f&\mc{O}\ar[d]^{p}\\
 \mc{B}\ar[r]\ar@{-->}[ur]&\mbf{\Delta}^{\mr{op}}.
  }
\end{equation*}
Assume that $f$ is a map of generalized $\infty$-operads.
An {\em operadic $p$-left Kan extension of this diagram} consists of a factorization
$\mc{X}:=(\mc{A}\times\Delta^1)\coprod_{\mc{A}\times\{1\},f}\mc{B}\xrightarrow{h'}\mc{M}\xrightarrow{h''}\mbf{\Delta}^{\mr{op}}\times\Delta^1$
where $h'$ is an inner anodyne and $h''$ is a $\Delta^1$-family of generalized $\infty$-operads (cf.\ \cite[A.3.1]{GH}),
and an operadic $p$-left Kan extension of $\mc{M}\times_{\Delta^1}\{0\}\simeq\mc{A}\rightarrow\mc{O}$ along the
inclusion $\mc{M}\times_{\Delta^1}\{0\}\hookrightarrow\mc{M}$ (cf.\ \cite[A.3.3]{GH}).
These data give rise to the dashed arrow above rendering the diagram commutative.
Indeed, the arrow is the composition of functors $\mc{B}\cong\mc{X}\times_{\Delta^1}\{1\}\rightarrow\mc{M}\times_{\Delta^1}\{1\}\subset\mc{M}\rightarrow\mc{O}$.

\begin{prop*}
 Let $\mc{A}^{\circledast}$ be a presentable monoidal $\infty$-category, and consider
 the following diagram:
 \begin{equation*}
  \xymatrix{
   {}_{\mc{L}}\Str^{\mr{en},+,\sim}
   \ar[r]\ar[d]&
   \mc{A}^{\circledast}\ar[d]^{p}\\
  {}_{\mc{L}}\Str^{\sim}
   \ar[r]\ar@{-->}[ru]^{\mc{H}}&
   \mbf{\Delta}^{\mr{op}}.
   }
 \end{equation*}
 Then the diagram admits an operadic $p$-left Kan extension $\mc{H}$.
 For $\mc{M}^{\circledast}\rightarrow\mr{RM}_{\Delta^1}$ in
 $\LinCat_{\Delta^1}$ and an object $(M_0,M_1)$ of
 ${}_{\mc{L}}\Str^{\sim}$ over $\mc{M}^{\circledast}$, we have
 $\mc{H}(M_0,M_1)\simeq\mr{Mor}_{\mc{M}}(M_0,M_1)$.
\end{prop*}
\begin{proof}
 By small object argument, we may take a factorization of the map
 \begin{equation*}
  ({}_{\mc{L}}\Str^{\mr{en},+,\sim}\times\Delta^1)\coprod_{{}_{\mc{L}}
   \Str^{\mr{en},+,\sim}\times\{1\}}{}_{\mc{L}}\Str^{\sim}
   \rightarrow
   \mbf{\Delta}^{\mr{op}}\times\Delta^1
 \end{equation*}
 into inner anodyne followed by an inner fibration $X\rightarrow\mbf{\Delta}^{\mr{op}}\times\Delta^1$.
 Since $\mbf{\Delta}^{\mr{op}}\times\Delta^1$ is (nerve of) a category, the latter inner fibration is in fact a categorical fibration,
 and thus it exhibits $X$ as a $\Delta^1$-family of generalized $\infty$-operads.
 We abbreviate ${}_{\mc{L}}(-)$ as $(-)$ to ease the notation. For a generalized $\infty$-operad $\mc{O}^{\circledast}$,
 we denote by $\mc{O}^{\circledast}_{\mr{act}}$ the subcategory of $\mc{O}^{\circledast}$ consisting of active maps.
 Let $M\in\Str^{\sim}_{[1]}$.  Invoking \cite[A.3.4]{GH}, it suffices to
 show that the diagram
 \begin{equation*}
  \xymatrix@C=50pt{
   \Str^{\mr{en},+,\sim}\times_{\Str^{\sim}}
   (\Str^{\sim}_{\mr{act}})_{/M}
   \ar[r]\ar[d]&
   \mc{A}^{\circledast}\ar[d]^{\pi}\\
  (\Str^{\mr{en},+,\sim}\times_{\Str^{\sim}}
   (\Str^{\sim}_{\mr{act}})_{/M})
   ^{\triangleright}
   \ar[r]\ar@{-->}[ur]&
   \mbf{\Delta}^{\mr{op}}
   }
 \end{equation*}
 extends to an operadic $\pi$-colimit diagram.
 Since $\Str^{\mr{en},+,\sim}\rightarrow\Str^{\sim}$ preserves
 coCartesian edges over $\mbf{\Delta}^{\mr{op}}$ by Lemma
 \ref{basicpropofstr}.\ref{basicpropofstr-3}, the map
 $\Str^{\mr{en},+,\sim}\times_{\Str^{\sim}}
 (\Str^{\sim}_{\mr{act}})_{/M}
 \rightarrow(\mbf{\Delta}^{\mr{op}}_{\mr{act}})_{/[1]}$
 is a coCartesian fibration using \cite[2.4.3.2]{HTT}.
 This implies that the inclusion
 $\Str^{\mr{en},+,\sim}_{[1]}\times_{\Str^{\sim}_{[1]}}
 (\Str^{\sim}_{[1]})_{/M}\rightarrow
 \Str^{\mr{en},+,\sim}\times_{\Str^{\sim}}
 (\Str^{\sim}_{\mr{act}})_{/M}$
 is left cofinal by \cite[4.1.2.15]{HTT}.
 Thus, by (non-symmetric analogue of) \cite[3.1.1.4]{HA},
 it suffices to show the existence of the operadic colimit indexed by $\Str^{\mr{en},+,\sim}_{[1]}\times_{\Str^{\sim}_{[1]}}(\Str^{\sim}_{[1]})_{/M}$.
 
 Let $\mc{M}^{\circledast}\rightarrow\mr{RM}_{\Delta^1}$ be the generalized $\infty$-operad over which $M$ is defined.
 We put $\Str_{\mc{M}}^{(\mr{en},+,)\sim}:=\bp\Str^{(\mr{en},+,)\sim}_{[1]}\times_{\RMod_{\Delta^1}}\{\mc{M}\}$.
 By Lemma \ref{basicpropofstr}, $\bp\Str^{\mr{en},+,\sim}_{[1]}\rightarrow\bp\Str^{\sim}_{[1]}$ is a
 map between coCartesian fibration over $\bp\RMod^{\circledast}_{[1]}$ which preserves coCartesian edges.
 Thus, by the same argument as above, we are reduced to showing the existence of operadic colimit indexed by
 $\Str^{\mr{en},+,\sim}_{\mc{M}}\times_{\Str^{\sim}_{\mc{M}}}(\Str^{\sim}_{\mc{M}})_{/M}$.
 Let $F_M\colon\mr{Tw}^{\mr{op}}_{[1]}\mbf{\Delta}'\rightarrow\mc{M}^{\circledast}$ be the functor corresponding to $M$.
 Put $M_i:=F_M(\{i\}\rightarrow[1])$ in $\mc{M}^{\circledast}_{[0]}\times_{\mr{RM}_{\Delta^1}}\mr{RM}_{\Delta^{\{i\}}}$.
 Unwinding the definition, existence of $\Mor_{\mc{M}}(M_0,M_1)$ is equivalent to the existence of an initial object
 of $\Str^{\mr{en},+,\sim}_{\mc{M}}\times_{\Str^{\sim}_{\mc{M}}}(\Str^{\sim}_{\mc{M}})_{/M}$.
 Thus an initial object exist by Lemma \ref{exisintmorobj}, and in particular,
 the (ordinary) colimit indexed by $\Str^{\mr{en},+,\sim}_{\mc{M}}\times_{\Str^{\sim}_{\mc{M}}}(\Str^{\sim}_{\mc{M}})_{/M}$ exists
 whose value at the cone point is nothing but $\Mor_{\mc{M}}(M_0,M_1)$.
 Since $\mc{A}^{\circledast}$ is compatible with small colimits (cf.\ \cite[3.1.23]{GH}),
 \cite[A.2.7]{GH} implies that the colimit is in fact an operadic colimit as required.
\end{proof}

\subsection{}
\label{laxfuncdef}
Let us carry out one of the main constructions of this paper.
Let $\mbf{C}$ be an $(\infty,2)$-category, and assume we are given a
$2$-functor
$\mbf{D}\colon\mbf{C}\rightarrow\twoLinCat_{\mc{A}}^{\mr{2-op}}$ of
$(\infty,2)$-categories.
Let $D\colon\mc{C}\rightarrow\LinCat_{\mc{A}}$ be the
functor of $\infty$-categories associated to $\mbf{D}$, and assume we are given the
following commutative diagram
\begin{equation*}
 \xymatrix@C=40pt{
  \mc{C}^{\simeq}\ar[r]^-{M}\ar[d]&
  {}_{\mc{L}}\Str^{\sim}_{[0]}\ar[d]\\
 \mc{C}\ar[r]^-{D}&\LinCat_{\mc{A}}.
  }
\end{equation*}
Recall that ${}_{\mc{L}}\Str^{\sim}_{[0]}$ is, informally, the $\infty$-category of pairs $(\mc{M},X)$
where $\mc{M}\in\LinCat_{\mc{A}}$ and $X\in\mc{M}$, and that the right vertical functor sends $(\mc{M},X)$ to $\mc{M}$.
Recall also that giving the $2$-functor $\mbf{D}$ is equivalent to giving a monoidal
functor $D^{\circledast}\colon\mc{C}^{\circledast}\rightarrow
\LinCat^{\circledast}_{\mc{A}}$ of the underlying generalized $\infty$-operads.
We have the functor
$\mc{C}^{\circledast}_{[0]}\simeq\mc{C}^{\simeq}\xrightarrow{M}
{}_{\mc{L}}\Str^{\sim}_{[0]}$, also denoted by $M$.
Using this, we have the map of
generalized $\infty$-operads
\begin{equation*}
 \mc{H}_M\colon
  \mc{C}^{\circledast}
  \xrightarrow{D^{\circledast}\basech M}
  \LinCat^{\circledast}\basech
  {}_{\mc{L}}\Str_{[0]}^{\sim}
  \xleftarrow{\sim}
  {}_{\mc{L}}\Str^{\sim}
  \xrightarrow{\mc{H}}
  \mc{A}^{\circledast},
\end{equation*}
where the equivalence follows by Lemma \ref{genopestr}.
Let us describe this informally.
For a $1$-morphism $f\colon X\rightarrow Y$ in $\mbf{C}$, we have the
map $\mbf{D}(f)\colon\mbf{D}(X)\rightarrow\mbf{D}(Y)$. The functor $M$
defines $M(X)\in\mbf{D}(X)$, $M(Y)\in\mbf{D}(Y)$. Then $\mc{H}_M$ sends
$f$ to $\Mor_{\mbf{D}(f)}(M(X),M(Y))$.

In our application, it is not hard to construct $M$.
Assume that $\mc{C}$ admits an initial object $\emptyset$.
Fix an object $I\in DF(\emptyset)$. Consider the diagram
\begin{equation*}
 \xymatrix@C=40pt{
  \{\emptyset\}\ar[r]^-{I}\ar[d]&
  {}_{\mc{L}}^{\backprime}\Str_{[0]}\ar[d]^{p}\\
 \mc{C}\ar[r]_-{D\circ F}\ar@{-->}[ur]^-{M_I}&
  \LinCat_{\mc{A},\Delta^0}.
  }
\end{equation*}
The functor $p$ is a categorical fibration by Remark \ref{strremarkoncatprod}.
Since $p$ is equivalent to the base change of
$\mc{M}^{\mr{univ},\circledast}_{[0],(0,1)}(\simeq\bp\Str_{[0]})
\rightarrow\RMod_{\Delta^0}$ which is a coCartesian fibration,
$p$ is a coCartesian fibration by \cite[2.4.4.3]{HTT}.
Thus we may take a $p$-left Kan extension of $I$. This extension is
denoted by $M_I$.
By the above construction, we have the map of generalized
$\infty$-operads
$\mc{H}_{M_I}\colon\mc{C}^{\circledast}\rightarrow
\mc{A}^{\circledast}$
associated to $I$.

\begin{rem*}
 \begin{enumerate}
  \item Recall the construction of the classifying $(\infty,2)$-category of Example \ref{inftwocatintro}.
	The map $\mc{H}_M$ of generalized $\infty$-operads induces a non-unital	right-lax functor of $(\infty,2)$-categories
	(cf.\ \cite[Ch.10, 3.1.3]{GR})
	\begin{equation*}
	 \mbf{H}_{M}\colon\mbf{C}\dashrightarrow\mbf{B}
	  \mc{A}^{\circledast}.
	\end{equation*}
	Let $f\colon X\rightarrow Y$ be a $1$-morphism in $\mbf{C}$, and let $D(f)\colon D(X)\rightarrow D(Y)$ be the associated
	$1$-morphism in $\twoLinCat_{\mc{A}}$.
	Then the functor $M$ defines $M(X)\in D(X)$ for each $X$, and
	\begin{equation*}
	 \mbf{H}_M(f)\simeq\Mor_{D(Y)}\bigl(D(f)(M(X)),M(Y)\bigr)
	\end{equation*}
	by viewing $1$-morphisms in $\mbf{B}\mc{A}^{\circledast}$ as
	objects of $\mc{A}$. This interpretation of $\mc{H}_M$ is more
	conceptual, but {\it a priori} discards some information from
	$\mc{H}_M$ when we take the localization to pass from
	$\mc{A}^{\circledast}$ to $\mbf{B}\mc{A}^{\circledast}$.
	We believe that $\mbf{H}_M$ is more essential than $\mc{H}_M$,
	and the construction in the next section, for which we use
	$\mc{H}_M$ rather than $\mbf{H}_M$ crucially,
	should be able to be carried out within the realm of
	$(\infty,2)$-categories.
       
  \item In \S\ref{secExam}, we apply this construction to
	Gaitsgory-Rozenblyum's 6-functor formalism. Then
	$\mc{H}_{M_I}(f)$ becomes the corresponding bivariant homology
	theory in the sense of Fulton-MacPherson \cite[\S2]{FM}. The
	$2$-functor $\mc{H}_{M_I}$ is supposed to encode all the axioms
	of the theory.
	However, it is still not satisfactory because treating
	$(\infty,2)$-categories is not as easy as treating
	$\infty$-categories. In \S\ref{funtobitheo}, we will extract a
	functor between $\infty$-categories which is much easier to
	handle, yet retains some important features of bivariant
	homology theory.
 \end{enumerate}
\end{rem*}

\subsection{}
\label{exploflacon}
Assume that the functor $M\colon\mc{C}^{\simeq}\rightarrow\Str_{[0]}$
can be lifted to a functor
$\widetilde{M}\colon\mc{C}'\rightarrow\Str_{[0]}$ compatible with
$D\colon\mc{C}\rightarrow\LinCat_{\mc{A}}$.
The construction above yields a functor $\mc{H}_M|_{\mc{C}'}$
sending a sequence
$C_0\xrightarrow{f_1}C_1\rightarrow\dots\xrightarrow{f_n} C_n$ in
$\mc{C}$ to
$\Mor_{D(f_1)}(M(C_0),M(C_1))\boxtimes\dots\boxtimes
\Mor_{D(f_n)}(M(C_{n-1}),M(C_n))$ in
$\mc{A}^{\circledast}_{[n]}$. On the other hand, we also have a functor
$\mbf{1}_{\widetilde{M}}$ sending the sequence to
$\mbf{1}_{\mc{A}}\boxtimes\dots\boxtimes\mbf{1}_{\mc{A}}$, where we take
$n$-times product and $\mbf{1}_{\mc{A}}$ is a unit-object of $\mc{A}$.
On the other hand, we have a map $\widetilde{M}(f_i)\colon
M(C_{i-1})\rightarrow M(C_i)$. This yields a map
$\mbf{1}_{\mc{A}}\rightarrow\mr{Mor}_{D(f_i)}(M(C_{i-1}),M(C_i))$.
Thus it is natural to expect for a map of functors
$\mbf{1}_{\widetilde{M}}\rightarrow\mc{H}_M|_{\mc{C}'}$,
which we will construct in the rest of this section.
This map will be used in the next section.

Recall the notation of \ref{gammaanddual}.
Let $\alpha\colon\Gamma^\vee\rightarrow\mr{Tw}^{\mr{op}}\mbf{\Delta}$
be the unique functor over $\mbf{\Delta}$ sending $([n],i)$ to
$a_i\colon[n-i]\rightarrow[n]$ in $\mr{Tw}^{\mr{op}}\mbf{\Delta}$ such
that $a_i(0)=i$.
Let $z\colon\Gamma^\vee\rightarrow\mr{RM}\times\mbf{\Delta}$ be the map
sending $([n],i)$ to $(0,1)\in\mr{RM}$ over $\mbf{\Delta}$.
Put
\begin{equation*}
 D:=(\Delta^{[1]}\times\Gamma^{\vee})
  \coprod_{\{0\}\times\Gamma^\vee}
  (\Delta^{[1]}\times\mr{Tw}^{\mr{op}}\mbf{\Delta}).
\end{equation*}
The maps $\Gamma^\vee\rightarrow\mbf{\Delta}$ and
$\Theta\colon\mr{Tw}^{\mr{op}}\mbf{\Delta}\rightarrow\mbf{\Delta}$
induce the map $\Theta_D\colon D\rightarrow\mbf{\Delta}$.
We may check easily that $D$ is (nerve of) a category, and
$\Theta_D$ is a coCartesian fibration.
We also have a map of simplicial sets
$\pi_D\colon D\rightarrow\mr{RM}\times\mbf{\Delta}$
such that the restriction to
$\Delta^{[1]}\times\mr{Tw}^{\mr{op}}\mbf{\Delta}$ is $\overline{\pi}$,
and restriction to $\{1\}\times\Gamma^{\vee}$ is $z$ and
$\Delta^1\times([n],i)$ is the unique active map.

\begin{dfn*}
 \label{plplstrdfn}
 \begin{enumerate}
  \item \label{plplstrdfn-1}
	Let $\bp\Str^{\mr{en},++}$ be the full subcategory of
	$\Theta_{D,*}\pi_D^*(\mc{M}^{\mr{univ},\circledast})$ spanned by
	functors $F\colon D_{[n]}:=D\times_{\mbf{\Delta},\Theta_D}
	\{[n]\}\rightarrow\mc{M}^{\mr{univ},\circledast}$ satisfying the
	following conditions:
	\begin{enumerate}
	 \item \label{plplstrdfn-a}
	       The restriction
	       $F|_{\Delta^1\times\mr{Tw}_{[n]}^{\mr{op}}\mbf{\Delta}}$
	       belongs to $\bp\Str^{\mr{en},+}_{[n]}$;

	 \item \label{plplstrdfn-b}
	       The functor $F$ is an $p$-left Kan extension of
	       $F|_{\Delta^1\times\mr{Tw}_{[n]}^{\mr{op}}\mbf{\Delta}}$
	       where $p\colon\mc{M}^{\mr{univ},\circledast}\rightarrow
	       \mr{RM}\times\mbf{\Delta}$.
	\end{enumerate}
	
  \item \label{plplstrdfn-2}
	Let $\bp\gamma(\mc{M}^{\mr{univ},\circledast})
	\rightarrow\mbf{\Delta}$ be the full subcategory of
	$\gamma^\vee_*z^*(\mc{M}^{\mr{univ},\circledast})\cong
	\gamma^\vee_*\gamma^{\vee *}
	(\mc{M}^{\mr{univ},\circledast}_{(0,1)})$ spanned by the
	functors
	$\Gamma^\vee_{[n]}\cong\Delta^n\rightarrow
	\mc{M}^{\mr{univ},\circledast}_{[n],(0,1)}$ such that
	the composition
	$\Delta^n\rightarrow\mc{M}^{\mr{univ},\circledast}_{[n],(0,1)}
	\rightarrow\mr{RM}^{\mr{univ}}_{[n],(0,1)}\simeq\Delta^n\times
	\RMod_{\Delta^n}$ is of the form $\mr{id}\times m$ where
	$m\colon\Delta^n\rightarrow\RMod_{\Delta^n}$
	factors through $\RMod_{\Delta^n}^{\simeq}$.
 \end{enumerate}
\end{dfn*}

The functors $\{1\}\times\Gamma^\vee\rightarrow D\leftarrow
\Delta^1\times\mr{Tw}^{\mr{op}}\mbf{\Delta}$ induce the diagram
\begin{equation*}
  \bp\gamma(\mc{M}^{\mr{univ},\circledast})
  \leftarrow
  \bp\Str^{\mr{en},++}
  \xrightarrow{\alpha}
  \bp\Str^{\mr{en},+}.
\end{equation*}

\begin{lem}
 \label{compoftwoconsseqstr}
 \begin{enumerate}
  \item The map
	$\bp\Str^{\mr{en},++}\rightarrow\mbf{\Delta}$ is a Cartesian
	fibration, $\alpha$ preserves Cartesian edges, and
	$\alpha_{[n]}$ is a trivial fibration. In particular, $\alpha$
	is a Cartesian equivalence.
	As usual, we put $\Str^{\mr{en},++}:=
	\mb{D}^{-1}_{\mbf{\Delta}}(\bp\Str^{\mr{en},++})$.
	
  \item The map $\bp\gamma(\mc{M}^{\mr{univ},\circledast})
	\rightarrow\mbf{\Delta}$
        is a Cartesian fibration. Moreover, we have a canonical
        equivalence
	$\gamma^\vee_*(\mc{M}^{\mr{univ}}_{[0],(0,1)}\times\Gamma^\vee)
	^{\simeq}_{/\mbf{\Delta}}\simeq
	\bp\gamma(\mc{M}^{\mr{univ},\circledast})^{\simeq}_{/\mbf{\Delta}}$
	of right fibrations over $\mbf{\Delta}$
	{\normalfont(}recall {\normalfont\ref{maxminKancpx}} for the notation{\normalfont)}.
	We put $\gamma(\mc{M}^{\mr{univ},\circledast}):=
	\mb{D}^{-1}_{\mbf{\Delta}}
	(\bp\gamma(\mc{M}^{\mr{univ},\circledast}))$
	as usual.
 \end{enumerate}
\end{lem}
\begin{proof}
 Let us show the first claim. We can check that $\bp\Str^{\mr{en},++}$
 is a Cartesian fibration by exactly the same argument as Lemma
 \ref{basicpropofstr}. By description of coCartesian edges, we see that
 $\alpha$ preserves Cartesian edges. The fiber $\alpha_{[n]}$ is trivial
 fibration by \cite[4.3.2.15]{HTT}.
 
 Let us show the second assertion. The first claim is a straightforward
 application of \cite[3.2.2.12]{HTT},
 so let us check the second claim. To ease the notations, we abbreviate
 $\gamma^\vee_*$, $\gamma^{\vee*}$ by $\gamma_*$, $\gamma^*$.
 Let $F\colon\mc{C}\rightarrow\mc{D}$
 be a coCartesian fibration of $\infty$-categories. We may consider the
 following diagram:
 \begin{equation*}
  \xymatrix{
   \mc{C}_F
   \ar[d]_{F'}\ar[r]\ar@{}[rd]|\square&
   \mc{C}\ar[d]^{F}\\
  \gamma^*\gamma_*(\mc{D}\times\Gamma^\vee)
   \ar[r]\ar[d]&
   \mc{D}\\
  \gamma_*(\mc{D}\times\Gamma^\vee).&
   }
 \end{equation*}
 All the vertical arrows are coCartesian fibrations.
 We define $\gamma(\mc{C}_F)$ to be the full subcategory of
 $\gamma_*\gamma^*\mc{C}_F$ spanned by vertices
 $\Delta^n\rightarrow\mc{C}_{F,[n]}$ such that the composition
 $\Delta^n\rightarrow\mc{C}_{F,[n]}\rightarrow
 \gamma^*\gamma_*(\mc{D}\times\Gamma^\vee)_{[n]}\cong
 \Delta^n\times\mr{Fun}(\Delta^n,\mc{D})$ is of the form
 $\mr{id}\times m$ where $m$ factors through
 $\mr{Fun}(\Delta^n,\mc{D})^{\simeq}$.
 We claim that the composition
 \begin{equation*}
  \gamma(\mc{C}_F)\subset\gamma_*\gamma^*\mc{C}_F
   \rightarrow
   \gamma_*\gamma^*(\mc{C}\times\mbf{\Delta})
   \cong
   \gamma_*(\mc{C}\times\Gamma^\vee)
 \end{equation*}
 is a categorical equivalence between Cartesian fibrations.
 Indeed, using \cite[3.2.2.12]{HTT}, we can check that
 this is a map between Cartesian fibrations that preserves Cartesian
 edges. The induced map between fibers over $[n]\in\mbf{\Delta}$ can be
 computed explicitly, and the equivalence follows.
 Moreover, the equivalence induces the equivalence
 \begin{equation*}
  \gamma\bigl(\mc{C}_F\times_{\gamma_*(\mc{D}\times\Gamma^\vee)}
   \gamma_*(\mc{D}\times\Gamma^\vee)^{\simeq}_{/\mbf{\Delta}}\bigr)
   \xrightarrow{\sim}
   \gamma_*(\mc{C}\times\Gamma^\vee)^{\simeq}_{/\mbf{\Delta}}.
 \end{equation*}
 Now let
 $\bp\RMod_{\mbf{\Delta}}^{\circledast,\mr{str}}\subset
 \bp\RMod_{\mbf{\Delta}}^{\circledast}$ be the subcategory spanned by
 all objects of $\bp\RMod^{\circledast}_{\mc{A}}$ and morphisms
 $\mc{M}^{\circledast}\rightarrow\mc{N}^{\circledast}$ over
 $[n]\rightarrow[m]$ which sends coCartesian edge over
 $\mr{RM}_{\Delta^n}$ to coCartesian edge over $\mr{RM}_{\Delta^m}$.
 Let $\RMod_{\Delta^n}^{\mr{str}}$ be the fiber over $[n]$.
 We apply the observation above to the coCartesian fibration
 $F\colon\mc{M}^{\mr{univ},\circledast}_{[0],(0,1)}
 \times_{\RMod_{\Delta^0}}\RMod_{\Delta^0}^{\mr{str}}\rightarrow
 \RMod^{\mr{str}}_{\Delta^0}$.
 The unstraightening of $F'$ is the composition
 \begin{equation*}
   \Gamma^\vee\times_{\mbf{\Delta}}\Phi^{\mr{Cart}}
   (\gamma,\RMod^{\mr{str}}_{\Delta^0})
   \simeq
   \gamma^*\gamma_*(\RMod^{\mr{str}}_{\Delta^0}\times\Gamma^\vee)
   \rightarrow
   \RMod^{\mr{str}}_{\Delta^0}
   \xrightarrow{\mr{St}(F)}
   \Cat_\infty.
 \end{equation*}
 By Proposition \ref{inftwoRMod}, we have
 $\Gamma^\vee\times_{\mbf{\Delta}}
 \bp\RMod^{\circledast,\mr{str}}_{\mbf{\Delta}}
 \simeq
 \Gamma^\vee\times_{\mbf{\Delta}}\Phi^{\mr{Cart}}
 (\gamma,\RMod^{\mr{str}}_{\Delta^0})$.
 Unwinding the definition, the unstraightening of the composition
 $\Gamma^\vee\times_{\mbf{\Delta}}
 \bp\RMod^{\circledast,\mr{str}}_{\mbf{\Delta}}\rightarrow\Cat_\infty$
 can be identified with
 $\mc{M}^{\mr{univ},\circledast}_{(0,1)}\rightarrow
 \mr{RM}^{\mr{univ}}_{(0,1)}$ base changed to
 $\bp\RMod_{\mbf{\Delta}}^{\circledast,\mr{str}}$.
 Thus, the claim follows by the observation above.
\end{proof}

\begin{dfn}
 Let $p\colon\mc{A}^{\circledast}\rightarrow\mbf{\Delta}^{\mr{op}}$ be a
 monoidal $\infty$-category.
 \begin{enumerate}
  \item An edge $e\colon\Delta^1\rightarrow\mc{A}^{\circledast}$ is said
	to exhibit $e(1)$ as a {\em unit object} if $e$ is
	$p$-coCartesian edge and $p(0)=[0]$.
	
  \item Let $\mc{A}^{\circledast}_{\mbf{1}}$ be the full subcategory of
	$\mr{Fun}(\Delta^1,\mc{A}^{\circledast})$ spanned by unit
	objects.
 \end{enumerate}
\end{dfn}

The inclusion
$\mc{A}^{\circledast}_{\mbf{1}}\rightarrow\mc{A}^{\circledast}$ is a
categorical fibration. Moreover, the map
$\mc{A}^{\circledast}_{\mbf{1}}\rightarrow\mbf{\Delta}^{\mr{op}}$ is a
trivial fibration by a similar argument to \cite[3.2.1.4]{HA}.
We put $\Str^{\mr{en},\mbf{1}+}:=\Str^{\mr{en},++}
\times_{\mc{A}^{\circledast}}\mc{A}^{\circledast}_{\mbf{1}}$.

\begin{rem*}
 Informally objects of $\bp\Str^{\mr{en},++}$ consists of the data of
 $\bp\Str^{\mr{en},+}$, which contains a sequence in
 $\mc{M}^{\mr{univ},\circledast}_{[n]}$ of the form
 \begin{equation*}
  M_0\boxtimes A_1\boxtimes\dots\boxtimes A_n
   \rightarrow M_1\boxtimes A_2\boxtimes\dots\boxtimes A_n
   \rightarrow\dots\rightarrow M_n,
 \end{equation*}
 together with a sequence
 \begin{equation*}
  M_0\otimes A_1\otimes\dots\otimes A_n
   \rightarrow M_1\otimes A_2\otimes\dots\otimes A_n
   \rightarrow\dots\rightarrow M_n
 \end{equation*}
 in $\mc{M}^{\mr{univ}}_{[0],(0,1)}$.
 Since $\otimes$ is defined essentially uniquely, it is not surprising
 that $\alpha\colon\bp\Str^{\mr{en},++}\rightarrow\bp\Str^{\mr{en},+}$
 is a trivial fibration.
 Furthermore, $\Str^{\mr{en},\mbf{1}+}$ consists of data as above such
 that any $A_i$ is a unit object of $\mc{A}^{\circledast}$ for any $i$.
\end{rem*}

\subsection{}
The main feature of this construction can be seen from the following
lemma:

\begin{lem*}
 The composition
 $\Str^{\mr{en},\mbf{1}+}\rightarrow\Str^{\mr{en},++}
 \rightarrow
 \gamma(\mc{M}^{\mr{univ},\circledast})$
 is a categorical equivalence.
\end{lem*}
\begin{proof}
 Using the description of coCartesian edges in Lemma
 \ref{basicpropofstr}, we can check that the map is a functor
 between coCartesian fibrations that preserves coCartesian edges, so it
 suffices to show that the fiber over $[n]\in\mbf{\Delta}^{\mr{op}}$ is
 an equivalence. In this case,
 $\Gamma^\vee\times_{\mbf{\Delta}}\{[n]\}\cong\Delta^n$.
 We consider the following category $D^+$:
 Let $E$ be the category of two objects $-1$, $0$ such that
 $\mr{Hom}(-1,0)=\{a\}$, $\mr{Hom}(0,-1)=\{b\}$,
 $\mr{Hom}(-1,-1)=\{\mr{id}\}$, $\mr{Hom}(0,0)=\{\mr{id},a\circ b\}$.
 We consider $D':=(E\times\Delta^n)
 \coprod_{\{0\}\times\Delta^n}(\Delta^1\times\mr{Tw}^{\mr{op}}_{[n]}
 \mbf{\Delta})$, where the coproduct is taken in the
 category of small (ordinary) categories. There is a unique map
 $f\colon(-1,0)\rightarrow(1,[n]\rightarrow[n])$ where
 $0\in\Delta^n$. Using this morphism, we define
 $D^+:=D'\coprod_{f,\Delta^{\{0,2\}}}\Delta^2$ where the coproduct is
 taken in the category of small categories.
 The functor $b\colon\Delta^1\rightarrow E$ induces the
 faithful functor $D_{[n]}\rightarrow D'$.
 There is a unique extension of $D\rightarrow\mr{RM}$ to $D'$. Note that
 the morphism $a\times\mr{id}_i$ of $E\times\Delta^n$ is sent to the
 map $(0,1)\rightarrow(0,1,\dots,1)$. This map can further be extended
 to $D^+$ by putting $\Delta^{\{1\}}$ to $(1)$ in $\mr{RM}$.
 Let $C\subset D^+$ be the full subcategory consisting of objects in
 $\Delta^1\times\mr{Tw}^{\mr{op}}_{[n]}\mbf{\Delta}$ and $\Delta^2$.

 Now, let $\mc{S}_D$ be the full subcategory of
 $\mr{Fun}_{\mr{RM}}(D^+,\mc{M}^{\mr{univ},\circledast}_{[n]})$ which is
 spanned by functors $F$ such that $F|_{\{-1\}\times\Delta^n}$ belongs
 to $\bp\gamma(\mc{M}^{\mr{univ},\circledast})$ and which is a $p$-left
 Kan extension along the inclusion $\{-1\}\times\Delta^n\hookrightarrow
 D^+$, where
 $p\colon\mc{M}^{\mr{univ},\circledast}_{[n]}\rightarrow\mr{RM}$.
 We put $\bp\Str^{\mr{en},\mbf{1}}:=\bp\Str^{\mr{en},+}
 \times_{\mc{A}^{\circledast}}\mc{A}^{\circledast}_{\mbf{1}}$.
 We have the following diagram
 \begin{equation*}
  \xymatrix@C=40pt{
   \mc{S}_D\ar[r]_-{}\ar[rd]_{\tau}\ar@/^10pt/[rr]^-{\phi}&
   \bp\Str^{\mr{en},\mbf{1}+}_{[n]}\ar[r]_-{\psi}\ar[d]&
   \bp\Str^{\mr{en},\mbf{1}}_{[n]}\\
  &\bp\gamma(\mc{M}^{\mr{univ},\circledast})&
   }
 \end{equation*}
 It suffices to show that $\phi$, $\psi$, $\tau$ are categorical
 equivalences. Consider the following two left Kan extension diagrams:
 \begin{equation*}
  \xymatrix{
   \Delta^n\times\{-1\}
   \ar[r]\ar[d]&
   \mc{M}^{\mr{univ},\circledast}_{[n]}\ar[d]^{p}\\
  D^+\ar[r]\ar@{-->}[ur]&\mr{RM},
   }
   \qquad
   \xymatrix{
   C
   \ar[r]\ar[d]&
   \mc{M}^{\mr{univ},\circledast}_{[n]}\ar[d]^{p}\\
  D^+\ar[r]\ar@{-->}[ur]&\mr{RM}.
   }
 \end{equation*}
 Invoking \cite[4.3.2.15]{HTT}, these diagrams yields the trivial
 fibrations $\tau$, $\phi$. The map $\psi$ is a trivial fibration by
 Lemma \ref{compoftwoconsseqstr}.
\end{proof}

\subsection{}
\label{constonetohm}
This lemma yields a functor
\begin{equation*}
 \gamma(\mc{M}^{\mr{univ},\circledast})
  \xleftarrow{\sim}
  \Str^{\mr{en},\mbf{1}+}
  \rightarrow
  \Str^{\mr{en},++}
  \xrightarrow{\alpha}
  \Str^{\mr{en},+}
\end{equation*}
which is defined up to a contractible space of choices.
Let us apply this construction to the situation in \ref{laxfuncdef}.
Assume that we are given an $(\infty,2)$-functor $\mbf{D}$, and assume
that the functor $\mc{C}^{\simeq}\rightarrow\mc{C}$ factors as
$\mc{C}^{\simeq}\rightarrow\mc{C}'\rightarrow\mc{C}$.
We assume further that the map
$M\colon\mc{C}^{\simeq}\rightarrow\Str^{\sim}_{[0]}$ is lifted to
$\widetilde{M}\colon\mc{C}'\rightarrow\Str_{[0]}$ which is compatible
with $D$.
If $\mc{C}'$ admits an initial object, the functor $M_I$ in
\ref{laxfuncdef} satisfies this condition.
Since
$\Str_{[0]}\simeq\mc{M}^{\mr{univ},\circledast}_{[0],(0,1)}$,
we have the composition
\begin{equation*}
 s_M\colon
 \gamma_*(\mc{C}'\times\Gamma)^{\simeq}
  _{/\mbf{\Delta}^{\mr{op}}}
  \xrightarrow{\widetilde{M}}
  \gamma_*({}_{\mc{L}}\Str_{[0]}\times\Gamma)^{\simeq}
  _{/\mbf{\Delta}^{\mr{op}}}
  \rightarrow
  {}_{\mc{L}}\gamma(\mc{M}^{\mr{univ},\circledast})
  \rightarrow
  {}_{\mc{L}}\Str^{\mr{en},+},
\end{equation*}
where the second functor follows from Lemma \ref{compoftwoconsseqstr}.
Note that the last two functors are base changed from
$\RMod^{\circledast}$ to $\LinCat^{\circledast}$.
Thus, we have a diagram
\begin{equation*}
 \xymatrix{
  \gamma_*(\mc{C}'\times\Gamma)
  ^{\simeq}_{/\mbf{\Delta}^{\mr{op}}}
  \ar[r]^-{s_M}&
  {}_{\mc{L}}\Str^{\mr{en},+,\sim}
  \ar[r]^-{A}\ar[d]_{\iota}&
  \mc{A}^{\circledast}\ar[d]\\
 &{}_{\mc{L}}\Str^{\sim}
  \ar[r]\ar[ur]^{\mc{H}}&
  \mbf{\Delta}^{\mr{op}}.
  }
\end{equation*}
Recall that the operadic left Kan extension is equipped with a map
$A\rightarrow\mc{H}\circ\iota$.
Putting $\mbf{1}_{\widetilde{M}}:=A\circ s_M$,
$\mc{H}_{M}|_{\mc{C}'}:=\mc{H}\circ\iota\circ s_M$.
Then we have the map of functors
\begin{equation*}
 \mbf{1}_{\widetilde{M}}
  \rightarrow
  \mc{H}_M|_{\mc{C}'}
  \colon
  \Phi^{\mr{co}}(\gamma,\mc{C}')^{\simeq}_{/\mbf{\Delta}^{\mr{op}}}=
  \gamma_*(\mc{C}'\times\Gamma)
  ^{\simeq}_{/\mbf{\Delta}^{\mr{op}}}
  \rightarrow
  \mc{A}^{\circledast}
\end{equation*}
that we informally explained at the beginning of \ref{exploflacon}.

\section{Bivariant homology functor}
\label{funtobitheo}
In the last section, we ``extracted'' a functor which encodes axioms of
bivariant theory. However, the description is still very inexplicit.
Assume we are given a ``6-functor formalism'' for the category of
schemes $\mr{Sch}$.
Then there should be an associated cohomology theory
$\mr{H}^*\colon\mr{Sch}^{\mr{op}}\rightarrow\mr{Mod}_R$.
There should also be an associated Borel-Moore homology
$\mr{H}^{\mr{BM}}\colon\mr{Sch}^{\mr{prop}}\rightarrow\mr{Mod}_R$.
Here, $\mr{Sch}^{\mr{prop}}$ is the category of schemes and consider
only proper morphisms as morphisms.
The goal of this section is to construct a unified functor from which we
can retrieve these two theories easily, as well as explaining the
relations of these theories.

\subsection{}
\label{functorseup}
Let $\mr{Sch}$ be a category which admits finite limits.
In particular, it admits a final object. We fix a final object denoted
by $*$. Let $\mr{prop}$, $\mr{sep}$ be a class of morphisms of
$\mr{Sch}$ which satisfy conditions \cite[Ch.7, 1.1.1]{GR} if we put
$\mathit{adm}=\mr{prop}$, $\mathit{vert}=\mr{sep}$,
$\mathit{horiz}=\mr{all}$. We put
$\mbf{Corr}:=\mbf{Corr}(\mr{Sch})^{\mr{prop}}_{\mr{sep};\mr{all}}$.
Let $\mc{A}^{\circledast}$ be a presentable monoidal $\infty$-category,
and $\mc{A}:=\mc{A}^{\circledast}_{[1]}$.
We assume we are given the following diagram
\begin{equation*}
 \xymatrix@C=50pt{
  \mr{Sch}^{\mr{op}}
  \ar[r]^-{\mc{D}}\ar[d]&\LinCat_{\mc{A}}\ar[d]\\
 \mbf{Corr}\ar[r]^-{\mbf{D}}&
  \twoLinCat_{\mc{A}}^{2\mbox{-}\mr{op}}.
  }
\end{equation*}
We fix two objects $I,J\in\mc{D}(*)$. By \ref{laxfuncdef}, we have the
functors
$M_I,M_J\colon\mr{Sch}^{\mr{op}}\rightarrow{}_{\mc{L}}\Str_{[0]}$. The
image of $X\in\mr{Sch}^{\mr{op}}$ is denoted by $I_X$, $J_X$.
For $f\colon X\rightarrow Y$ in $\mr{Sch}$, we put
$f^*:=\mc{D}(f)$. Similarly, if $f\in\mr{sep}$, let $V_f$ be the
$1$-morphism $X\rightarrow Y$
\begin{equation*}
 \xymatrix{
  X\ar@{=}[r]\ar[d]_{f}&X\\Y.&}
\end{equation*}
in $\mbf{Corr}$. We put
$f_!:=\mbf{D}(V_f)\colon\mc{D}(X)\rightarrow\mc{D}(Y)$.

\begin{ex*}
 As the notation suggests,
 the main example we have in mind is the case where $\mr{Sch}$ is the
 category of schemes of finite type over a base scheme $S$,
 and $\mr{prop}$ is the class of proper morphisms, and $\mr{sep}$ is the
 class of separated morphisms.
\end{ex*}

\subsection{}
We denote by $\widetilde{\mr{Ar}}^{\mr{prop}}_{\mr{sep}}(\mr{Sch})$ the category whose objects consists of morphisms $X\rightarrow Y$ in $\mr{sep}$.
Assume we are given two objects $f_1\colon X_1\rightarrow Y_1$ and $f_0\colon X_0\rightarrow Y_0$.
Let $S(f_1,f_0)$ be the set of diagrams of the form below on the left
\begin{equation}
 \label{diagmortilar}
 \xymatrix{
  X_1\ar[d]_{f_1}&
  W_{10}\ar[d]_{\widetilde{f}_0}\ar[r]^-{\widetilde{g}}\ar[l]_{\alpha}\ar@{}[rd]|\square&
  X_0\ar[d]^{f_0}\\
 Y_1&Y_1\ar[r]^-{g}\ar@{=}[l]&
  Y_0,}
  \qquad
  \xymatrix{
  X_1\ar@{=}[d]&W_{10}\ar[r]^-{\widetilde{g}}\ar[l]_-{\alpha}\ar[d]^{\beta}&X_0\ar@{=}[d]\\
 X_1&W'_{10}\ar[r]^-{\widetilde{g}'}\ar[l]_-{\alpha'}&X_0,}
\end{equation}
where $\alpha$ belongs to $\mr{prop}$.
Take two elements of $S(f_1,f_0)$ defined by $W_{10}$ and $W'_{10}$.
We denote the corresponding morphisms of the diagram of $W'_{10}$ by putting $(-)^{\prime}$ ({\it e.g.}\ $\alpha'\colon W'_{10}\rightarrow X_1$).
The elements $W_{10}$ and $W'_{10}$ in $S(f_1,f_0)$ are said to be {\em equivalent}, denoted by $W_{10}\sim W'_{10}$,
if $g=g'$ and there exists a diagram of the form above on the right.
Note that, since $g=g'$, the morphism $\beta$ is determined uniquely and an isomorphism.
We define the set of morphisms from $f_1$ to $f_0$ by $S(f_1,f_0)/\sim$.
The composition of $(X_2\rightarrow Y_2)\rightarrow(X_1\rightarrow Y_1)\rightarrow (X_0\rightarrow Y_0)$ is
defined by the following diagram
\begin{equation*}
 \xymatrix@R=10pt{
 W_{10}\times_{Y_1}Y_2\ar[r]\ar[d]\ar@{}[rd]|\square&
 W_{10}\ar[r]\ar[d]\ar@{}[rddd]|\square&X_0\ar[ddd]\\
 W_{21}\ar[r]\ar[d]\ar@{}[rdd]|\square&X_1\ar[dd]&\\
 X_2\ar[d]&&\\
 Y_2\ar[r]&Y_1\ar[r]&Y_0.
  }
\end{equation*}
The functor
$\widetilde{\mr{Ar}}^{\mr{prop}}_{\mr{sep}}(\mr{Sch})\rightarrow\mr{Sch}$
sending $X\rightarrow Y$ to $Y$ is a Grothendieck fibration.

\begin{rem*}
 \label{straitencatarrprop}
 For $Y\in\mr{Sch}$, let $\mr{Sch}_{\mr{sep}/Y}^{\mr{prop}}$ be the subcategory
 of of $\mr{Sch}_{/Y}$ consisting of objects $f\colon X\rightarrow Y$ in
 $\mr{Sch}_{/Y}$ such that $f$ is in $\mr{sep}$, and a morphism
 $(X\rightarrow Y)\rightarrow(X'\rightarrow Y)$ in $\mr{Sch}_{/Y}$ depicted as
 \begin{equation*}
  \xymatrix{X\ar[d]\ar[r]^-{\alpha}&X'\ar[d]\\
  Y\ar@{=}[r]&Y}
 \end{equation*}
 belongs to $\mr{Sch}_{\mr{sep}/Y}^{\mr{prop}}$ if $\alpha$ is in $\mr{prop}$.
 Given a morphism $Y_1\rightarrow Y_0$, we can choose a base change functor
 $\mr{Sch}_{/Y_0}^{\mr{prop}}\rightarrow\mr{Sch}_{/Y_1}^{\mr{prop}}$.
 This yields a functor $\mr{Sch}^{\mr{op}}\rightarrow\Cat_{1}$ sending $Y$ to $(\mr{Sch}_{/Y}^{\mr{prop}})^{\mr{op}}$
 where $\Cat_1$ is the $(2,1)$-category of ($1$-)categories,
 in other words, a pseudo-functor.
 The Grothendieck fibration $\widetilde{\mr{Ar}}^{\mr{prop}}_{\mr{sep}}(\mr{Sch})\rightarrow\mr{Sch}$ is associated to this pseudo-functor.
\end{rem*}

Now, the goal of this section is to show the following theorem:

\begin{thm}
 \label{mainthmcons}
 Under the setting of {\normalfont\ref{functorseup}}, there exists a functor
 $\mr{H}\colon\widetilde{\mr{Ar}}^{\mr{prop}}_{\mr{sep}}(\mr{Sch})^{\mr{op}}\rightarrow\mc{A}$
 such that for $f\in\mr{sep}$, the object $\mr{H}(f)\in\mc{A}$ is
 equivalent to $\Mor_{\mbf{D}(V_f)}(I_X,J_Y)\simeq\Mor(f_!(I_X),J_Y)$.
 Assume we are given a morphism $m\colon f_1\rightarrow f_0$ given by
 the diagram {\normalfont(\ref{diagmortilar})} on the left.
 Then $\mr{H}(m)$ is equivalent to the composition of the following morphisms
 \begin{align*}
  \mr{Mor}(f_{0!}(I_{X_0}),J_{Y_0})
  &\rightarrow
  \mr{Mor}(g^*f_{0!}(I_{X_0}),g^*J_{Y_0})
  \simeq
  \mr{Mor}(\widetilde{f}_!\widetilde{g}^*(I_{X_0}),g^*J_{Y_0})\\
  &\simeq
  \mr{Mor}(f_{1!}\alpha_!\alpha^*(I_{X_1}),J_{Y_1})
  \rightarrow
  \mr{Mor}(f_{1!}(I_{X_1}),J_{Y_1}).
 \end{align*}
\end{thm}

\begin{rem*}
 \begin{enumerate}
  \item Consider the functor $i_0\colon\mr{Sch}\rightarrow
	\widetilde{\mr{Ar}}^{\mr{prop}}_{\mr{sep}}(\mr{Sch})$ sending
	$X$ to $X\rightarrow X$. Then $\mr{H}\circ i_0^{\mr{op}}\colon
	\mr{Sch}^{\mr{op}}\rightarrow\mc{A}$ is called the
	{\em cohomology theory}. On the other hand, let
	$\mr{Sch}^{\mr{prop}}_{\mr{sep}}$ be the subcategory of
	$\mr{Sch}$ consisting of objects $X$ such that the map
	$X\rightarrow*$ is in $\mr{sep}$ and morphisms $X\rightarrow Y$
	which are in $\mr{prop}$. Consider the functor
	$i_1\colon\mr{Sch}^{\mr{prop}}_{\mr{sep}}\rightarrow
	\widetilde{\mr{Ar}}^{\mr{prop}}_{\mr{sep}}(\mr{Sch})^{\mr{op}}$
	sending $X$ to $X\rightarrow*$. Then $\mr{H}\circ i_1$ is called
	the {\em Borel-Moore homology theory}.
	Finally, the functor $\mr{H}$ describes the relations
	between these theories.
	
  \item Even though we can unify cohomology and Borel-Moore homology
	theories, it is not completely satisfactory because we are not
	able to retrieve all the features of bivariant homology
	theory. In fact, bivariant homology theory has 3 operations:
	contravariant functoriality with respect to all the morphisms,
	covariant functoriality with respect to proper morphisms, and
	product structure.
	In our treatment, product structure is missing.
	We wonder if there is an upgraded version of the functor
	$\mr{H}$ so that all the axioms \cite[\S2]{FM} can be
	incorporated.
 \end{enumerate}
\end{rem*}

\subsection{}
Let $\infty$ be the cone point of $\mr{Sch}^{\triangleright}$.
We have the functor
$c\colon\mr{Sch}^{\triangleright}\rightarrow\mr{Sch}$
sending $\infty$ to $*$ such that the composition
$\mr{Sch}\rightarrow\mr{Sch}^{\triangleright}\rightarrow\mr{Sch}$ is
isomorphic to the identity. A morphism in $\mr{Sch}^{\triangleright}$ is
defined to be in $\mr{sep}$, $\mr{prop}$ if and only if the image by $c$
is in $\mr{sep}$, $\mr{prop}$. We have the functor
\begin{equation*}
 \mbf{Corr}(\mr{Sch}^{\triangleright})^{\mr{prop}}_{\mr{sep};\mr{all}}
  \rightarrow
  \mbf{Corr}(\mr{Sch})^{\mr{prop}}_{\mr{sep};\mr{all}}
  \xrightarrow{\mbf{D}}
  \twoLinCat_{\mc{A}}^{2\mbox{-}\mr{op}}.
\end{equation*}
Now, we have the functor
$M\colon(\mr{Sch}^{\triangleright})^{\simeq}\rightarrow\Str^{\sim}_{[0]}$
such that $M(\infty)=I$ and $M|_{\mr{Sch}}=J_{-}$.
We denote by $\mr{Corr}^{\circledast}$ the generalized $\infty$-operad
defining
$\mbf{Corr}(\mr{Sch}^{\triangleright})^{\mr{prop}}_{\mr{sep};\mr{all}}$.

Till \ref{insertone}, we will focus on constructing the following
sequence of functors between $\infty$-categories over
$\mbf{\Delta}^{\mr{op}}$:
\begin{equation}
 \label{seqfuncdefex}
  \mc{H}_{\top}\colon
  \Phi^{\mr{co}}
  (\Gamma\times\Delta^1,\mr{Sch}^{\mr{op}})^{\mr{prop}}
  \xrightarrow{\alpha}
  \overleftarrow{s}^*\mr{Corr}^{\circledast}
  \xleftarrow[\sim]{\beta}
  \Psi^{\star}\overleftarrow{s}^*\mr{Corr}^{\circledast}
  \xrightarrow{\Psi\mc{H}_M}
  \Psi(\overleftarrow{s}^*\mc{A}^{\circledast}).
\end{equation}
The undefined notations will be introduced later.
Since $\beta$ is a categorical equivalence, $\mc{H}_{\top}$ is defined
canonically up to a contractible space of choices.

If we are given a Cartesian fibration
$f\colon\mc{C}\rightarrow\mbf{\Delta}^{\mr{op}}$ and an
$\infty$-category $\mc{D}$, we introduced the notation
$\Phi^{\mr{co}}(f,\mc{D})$ in \ref{constsimplcat}.
In this section, this $\infty$-category is also denoted by
$\Phi^{\mr{co}}(\mc{C},\mc{D})$ especially when the structural map $f$
is clear. In particular, when we use this notation, the
$\infty$-category $\mc{C}$ is always considered over
$\mbf{\Delta}^{\mr{op}}$.

\begin{rem*}
 A vertex of $\Phi^{\mr{co}}(\Gamma\times\Delta^1,
 \mr{Sch}^{\mr{op}})^{\mr{prop}}$ over $[n]\in\mbf{\Delta}^{\mr{op}}$ is
 a  diagram in $\mr{Sch}$ of the form (\ref{diagarmap}). The functor
 $\mc{H}_{\top}$ sends this diagram to a sequence
 \begin{align*}
  \mr{Mor}(I_{X_0},J_{Y_0})&\boxtimes
  \mr{Mor}(J_{Y_0},J_{Y_1})\boxtimes
  \mr{Mor}(J_{Y_1},J_{Y_2})\boxtimes
  \dots\boxtimes
  \mr{Mod}(J_{Y_{n-1}},J_{Y_n})\\
  &\rightarrow
  \mr{Mor}(I_{X_1},J_{Y_1})\boxtimes
  \mr{Mor}(J_{Y_1},J_{Y_2})\boxtimes\dots\boxtimes
  \mr{Mod}(J_{Y_{n-1}},J_{Y_n})\\
  &\rightarrow
  \dots
  \rightarrow
  \mr{Mor}(I_{X_n},J_{Y_n})
 \end{align*}
 in $\mc{A}^{\circledast}$. Here, $\mr{Mor}(J_{Y_i},J_{Y_{i+1}})$ is
 taken over the functor $\mc{D}(f\colon Y_{i+1}\rightarrow
 Y_i)=:f^*\colon\mc{D}(Y_i)\rightarrow\mc{D}(Y_{i+1})$. 
 Since we have the morphism
 $\mbf{1}_{\mc{A}}\rightarrow\mr{Mor}(J_{Y_i},J_{Y_{i+1}})$
 corresponding to $f^*J_{Y_i}\simeq
 J_{Y_{i+1}}\xrightarrow{\mr{id}}J_{Y_{i+1}}$ the above sequence yields
 a sequence in $\mc{A}^{\circledast}$
 \begin{align*}
  \mr{Mor}(I_{X_0},J_{Y_0})\boxtimes\mbf{1}\boxtimes\dots\boxtimes\mbf{1}
   \rightarrow
  \mr{Mor}(I_{X_1},J_{Y_1})\boxtimes\mbf{1}\boxtimes\dots\boxtimes\mbf{1}
  \rightarrow\dots\rightarrow
  \mr{Mor}(I_{X_n},J_{Y_n}).
 \end{align*}
 This functor $\mc{H}_{\top,\mbf{1}}$ will be constructed in
 \ref{insertone}. Taking the tensor product in $\mc{A}$, this sequence
 yields
 \begin{equation*}
  \mr{Mor}(I_{X_0},J_{Y_0})
   \rightarrow
  \mr{Mor}(I_{X_1},J_{Y_1})
  \rightarrow\dots\rightarrow
  \mr{Mor}(I_{X_n},J_{Y_n}).
 \end{equation*}
 This is the functor $\mc{H}_{\bullet}$ which will be constructed in
 \ref{taketensordefH}.
\end{rem*}

\subsection{}
\label{oversandqcons}
Let $\overleftarrow{s}\colon\mbf{\Delta}^{\mr{op}}\rightarrow
\mbf{\Delta}^{\mr{op}}$ be the functor $(-)^{\triangleleft}$.
First, let us prepare some result on $\overleftarrow{s}$. Let
$F\colon\mbf{\Delta}^{\mr{op}}\times\Delta^1\rightarrow
\mbf{\Delta}^{\mr{op}}$ be a functor sending $([n],i)$ to $[n+1-i]$,
$F|_{\mbf{\Delta}^{\mr{op}}\times\Delta^{\{0\}}}=\overleftarrow{s}$,
$F|_{\mbf{\Delta}^{\mr{op}}\times\Delta^{\{1\}}}=\mr{id}$, and
$([n],0)\rightarrow([n],1)$ to the map
$d\colon[n]^{\triangleleft}\rightarrow[n]$ such that
$d|_{[n]}=\mr{id}$. This defines a natural transformation of functors
$\overleftarrow{S}\colon\overleftarrow{s}\rightarrow\mr{id}$.

Let $A\colon\mbf{\Delta}^{\mr{op}}\rightarrow\sSet^+$ be a functor such
that
\begin{quote}
 (*) For any vertex $[n]\in\mbf{\Delta}^{\mr{op}}$, $A([n])$ is an
 $\infty$-category, and for any inert map $f\colon[n]\rightarrow[m]$
 such that $m\in[m]$ is sent to $n\in[n]$, the functor between
 $\infty$-categories $A(f)$ is a categorical fibration.
\end{quote}
Since $A$ is assumed to be fibrant with respect to the projective model
structure, $\mc{A}:=\mr{N}^+_A\mbf{\Delta}^{\mr{op}}\rightarrow
\mbf{\Delta}^{\mr{op}}$ (cf.\ \cite[\S 3.2.5]{HTT} for the notation) is
a coCartesian fibration. The natural transformation
$\overleftarrow{S}$ induces a functor
$A\circ\overleftarrow{s}\rightarrow A$ in
$(\sSet^+)^{\mbf{\Delta}^{\mr{op}}}$.
By assumption (*), this morphism is a fibration with respect to the
projective model structure. Thus, invoking \cite[3.2.5.18]{HTT}, we
have a fibration of coCartesian fibrations
\begin{equation}
 \label{dfnoffunctmapqshift}
 q_{\mc{A}}\colon\overleftarrow{s}^*\mc{A}\simeq
  \mr{N}_{A\circ\overleftarrow{s}}^+(\mbf{\Delta}^{\mr{op}})
  \rightarrow
  \mr{N}^+_A(\mbf{\Delta}^{\mr{op}})=:\mc{A}
\end{equation}
where the first isomorphism of simplicial sets by (adjoint of)
\cite[3.2.5.14]{HTT}.
We often abbreviate $q_{\mc{A}}$ by $q$.
This map $q_{\mc{A}}$ is, in particular, a categorical fibration by
\cite[B.2.7]{HA}.

\begin{rem*}
 Let $\mc{A}\rightarrow\mbf{\Delta}^{\mr{op}}$ be a coCartesian
 fibration. Then $\mc{A}$ is coCartesian equivalent to
 $\mr{N}^+_A(\mbf{\Delta}^{\mr{op}})$ such that $A$ satisfies (*).
 Indeed, let $\mbf{\Delta}^{\mr{op}}\rightarrow\sSet^+$ be a functor,
 and put the Reedy model structure on
 $(\sSet^+)^{\mbf{\Delta}^{\mr{op}}}$ associated to the Cartesian model
 structure on $\sSet^+$.
 Recall that for a simplicial set $A$ and
 $X\in(\sSet^+)^{\mbf{\Delta}^{\mr{op}}}$, an object $\mr{hom}(A,X)$ in
 $\sSet^+$ is defined in \cite[4.1]{Dug}. Now, for any cofibration
 $A\rightarrow B$ of simplicial sets and a Reedy fibrant object $X$, the
 induced map $\mr{hom}(B,X)\rightarrow\mr{hom}(A,X)$ is a fibration in
 $\sSet^+$ by \cite[4.5]{Dug}. In particular, in view of
 \cite[4.2]{Dug}, the the map $X([n])\rightarrow X([m])$ is a fibration
 for any injective map $[m]\rightarrow[n]$.
 This implies that any Reedy fibrant object satisfies (*).
 Now, for $\mf{F}_{\overline{X}}(\mbf{\Delta}^{\mr{op}})$, where
 $\overline{X}$ is the object
 $(\mc{A}\rightarrow\mbf{\Delta}^{\mr{op}})^{\natural}$ in
 $(\sSet^+)_{/\mr{N}(\mbf{\Delta}^{\mr{op}})}$, there exists a objectwise
 weak equivalence, thus Reedy weak equivalence,
 $\mf{F}_{\overline{X}}(\mbf{\Delta}^{\mr{op}})\simeq
 A_\bullet$ such that $A_\bullet$ is a Reedy fibrant, thus the claim follows
\end{rem*}

\begin{lem}
 \label{shiftfunccarted}
 Let further assume that $\mc{A}$ is an $\infty$-operad.
 Let $e\colon x\rightarrow y$ be an edge in $\overleftarrow{s}^*\mc{A}$
 over $[n]$ such that $\sigma^0_!x\rightarrow\sigma^0_!y$
 is an equivalence. Then $e$ is a $q$-Cartesian edge.
 If $t\in\overleftarrow{s}^*\mc{A}$ is an object such that
 $\sigma^0_!(t)$ is a final object in $\mc{A}_{[1]}$, then $t$ is
 $q$-final.
\end{lem}
\begin{proof}
 Let $\mc{C}\rightarrow\mbf{\Delta}^{\mr{op}}$ be a map,
 $\phi\colon[a]\rightarrow[b]$ in $\mbf{\Delta}^{\mr{op}}$
 and $c\in\mc{C}_{[a]}$,
 $d\in\mc{C}_{[b]}$. Then $\mr{Map}^{\phi}(c,d)$ denotes the union of
 connected components of $\mr{Map}(c,d)$ lying over $\phi$.
 Let us show the first claim.
 Since $q$ is an inner fibration, it suffices to
 show by \cite[2.4.4.3]{HTT} that for any $\phi\colon[m]\rightarrow[n]$
 in $\mbf{\Delta}^{\mr{op}}$ and
 $z\in(\overleftarrow{s}^*\mc{A})_{[m]}$, the diagram
 \begin{equation*}
  \xymatrix{
   \mr{Map}^{\phi}_{\overleftarrow{s}^*\mc{A}}(z,x)\ar[r]\ar[d]&
   \mr{Map}^{\phi}_{\overleftarrow{s}^*\mc{A}}(z,y)\ar[d]\\
  \mr{Map}^{\phi}_{\mc{A}}(q(z),q(x))\ar[r]&
   \mr{Map}^{\phi}_{\mc{A}}(q(z),q(y))
   }
 \end{equation*}
 induced by $e$ is a homotopy Cartesian diagram.
 Since $A$ is an $\infty$-operad, for any $w\in\mc{A}_{[n]}$, we have homotopy
 equivalences
 \begin{align*}
  &\mr{Map}^{\phi}_{\overleftarrow{s}^*\mc{A}}(z,w)
  \simeq
  \mr{Map}_{(\overleftarrow{s}^*\mc{A})_{[0]}}((\sigma^0\circ\phi)_!(z),\sigma^0_!(w))
  \times
  \prod_{0<i\leq n}
  \mr{Map}_{\mc{A}_{[1]}}\bigl((\rho^i\circ\phi)_!(q(z)),
  \rho^i_!(q(w))\bigr)\\
  &\mr{Map}^{\phi}_{\mc{A}}(q(z),q(w))
  \simeq
  \prod_{0<i\leq n}
  \mr{Map}_{\mc{A}_{[1]}}\bigl((\rho^i\circ\phi)_!(q(z)),
  \rho^i_!(q(w))\bigr)
 \end{align*}
 where $\rho^i\colon[n]\rightarrow[1]$ is the inert map,
 which concludes the proof. Let us show the second
 claim. In view of \cite[4.3.1.13]{HTT}, it suffices to show that
 $(t,\mr{id})$ is a final object of
 $\overleftarrow{s}^*\mc{A}\times_{\mc{A}}\mc{A}_{/q(t)}$. For this, we
 must show that for any $z':=(z,q(z)\rightarrow q(t))$, the space
 $\mr{Map}(z',(t,\mr{id}))$ is weakly contractible.  Applying Lemma
 \ref{fiberprodcatmorphc}, the verification is similar to the first
 half.
\end{proof}

\begin{rem*}
 The edge $e$ is {\em not} $q$-coCartesian.
\end{rem*}

\subsection{}
\label{intgamanddiag}
Recall the notations from \ref{gammaanddual}.
Let $\phi\colon[n]\rightarrow[m]$ in $\mbf{\Delta}^{\mr{op}}$, and put
$D_\phi:=\Gamma\times_{\gamma,\mbf{\Delta}^{\mr{op}},\phi}\Delta^1$.
We put $D_i:=D_{\phi}\times_{\Delta^1}\{i\}$ for $i=0,1$.
Let us describe the category $D_\phi$ more explicitly. It can be
depicted as the following diagram:
\begin{equation*}
  \xymatrix@C=10pt{
  [n]\ar[d]_{\phi}&
  0\ar[r]&1\ar[r]&\dots\ar[r]&\phi(0)\ar[r]\ar[dlll]_{\alpha_0}\ar[r]&
  \dots\ar[r]&\phi(k)\ar[r]\ar[dll]_{\alpha_k}&
  \dots\ar[r]&\phi(m)\ar[r]\ar[dll]_{\alpha_m}&
  \dots\ar[r]&0\\
 [m]&
  0\ar[r]&1\ar[r]&\dots\ar[r]&k\ar[r]&\dots\ar[r]&m.}
\end{equation*}
Here $\alpha_k$ for $k\in[m]$ is the unique map from $([n],\phi(k))$ to
$([m],k)$.
We define a function $\phi'\colon[\phi(m)]\rightarrow[m]$
by $\phi'(i):=\min\bigl\{k\in[m]\mid i\leq\phi(k)\bigr\}$
for $i\in[n]$. The unique map from $([n],i)$ to $([m],\phi'(i))$ is
denoted by $\beta_i$. By construction, $\alpha_k=\beta_{\phi(k)}$.

Let us define another functor
$\delta\colon\Gamma\rightarrow\mbf{\Delta}^{\mr{op}}$. We put
$\delta([n],i):=[n-i]$. For a map $f\colon([n],i)\rightarrow([m],j)$
given by a map $\phi\colon[m]\rightarrow[n]$, we define
$\delta(f)\colon[n-i]\rightarrow[m-j]$ in $\mbf{\Delta}^{\mr{op}}$ to be
map corresponding to the function sending $k\in[m-j]$ to $\phi(k+j)-i$.
Note that if $\phi$ is an inert map, then so is $\delta(f)$ for any $f$
over $\phi$.

\begin{dfn}
 \label{dfnofphocoprop}
 Let
 $\Phi^{\mr{co}}(\Gamma\times\Delta^1,\mr{Sch}^{\mr{op}})^{\mr{prop}}$
 be the simplicial subset of $\Phi^{\mr{co}}(\Gamma\times\Delta^1,
 \mr{Sch}^{\mr{op}})$ consisting of simplices
 $\Delta^k\rightarrow\Phi^{\mr{co}}(\Gamma\times\Delta^1,
 \mr{Sch}^{\mr{op}})$ satisfying the following two conditions:
 \begin{itemize}
  \item For any vertex corresponding to a functor
	$f\colon\Gamma_{[n]}\times\Delta^1\rightarrow
	\mr{Sch}^{\mr{op}}$,
	the diagram $f|_{\Delta^{\{i,i+1\}}\times\Delta^1}$, considered
	as a square in $\mr{Sch}$, is a pullback square for any $0\leq
	i<n$;
	
  \item for any edge corresponding to a functor $f\colon
	D_{\phi}\times\Delta^1\rightarrow\mr{Sch}^{\mr{op}}$ over
	$\phi\colon[n]\rightarrow[m]$ in $\mbf{\Delta}^{\mr{op}}$, the
	morphism $f(\alpha_k,\{1\})$ is proper and $f(\alpha_k,\{0\})$
	is an equivalence for any $k\in[m]$.
 \end{itemize}
\end{dfn}

Now, let us construct a functor
$\alpha\colon\Phi^{\mr{co}}(\Gamma\times\Delta^1,\mr{Sch}^{\mr{op}})^{\mr{prop}}
\rightarrow
\overleftarrow{s}^*\mr{Corr}^{\circledast}$ of categories.
A vertex of
$\Phi^{\mr{co}}(\Gamma\times\Delta^1,\mr{Sch}^{\mr{op}})^{\mr{prop}}$
corresponds to a diagram
$F\colon(\Delta^n\times\Delta^1)^{\mr{op}}\rightarrow\mr{Sch}$ as
follows:
\begin{equation}
 \label{diagarmap}
 \xymatrix{
  X_n\ar[r]\ar[d]\ar@{}[rd]|\square&
  X_{n-1}\ar[r]\ar[d]\ar@{}[rd]|\square&
  \dots\ar[r]\ar@{}[rd]|\square&
  X_1\ar[d]\ar[r]\ar@{}[rd]|\square&X_0\ar[d]\\
 Y_n\ar[r]&Y_{n-1}\ar[r]&\dots\ar[r]&Y_1\ar[r]&Y_0.
  }
\end{equation}
Here, we put $X_i:=F(i,1)$ and $Y_i:=F(i,0)$.
We note that the functor $F$ is contravariant, which is why $(i,1)$
corresponds to $X_i$.
A morphism from $(X_i\rightarrow Y_i)$ to $(X'_i\rightarrow Y'_i)$ is a
morphism of diagrams such that the morphisms $X_i\rightarrow X'_i$ are
proper and morphisms $Y_i\rightarrow Y'_i$ are equivalences.
The functor $\alpha$ in (\ref{seqfuncdefex}) is defined to be the functor
sending the diagram above to the following object in
$\mr{Seq}_{n+1}\mr{Corr}(\mr{Sch}^{\triangleright})$
\begin{equation*}
 \xymatrix{
  X_n\ar[r]\ar[d]\ar@{}[rd]|\square&
  \dots\ar[r]\ar@{}[rd]|\square&
  X_2\ar[r]\ar@{}[rd]|\square\ar[d]&
  X_1\ar[d]\ar[r]\ar@{}[rd]|\square&X_0\ar[d]\ar[r]
  &{\infty}.\\
 Y_n\ar[d]^{=}\ar[r]&\dots\ar[r]&Y_2\ar[r]\ar[d]^{=}&
  Y_1\ar[r]\ar[d]^{=}&Y_0&\\
 Y_n\ar[r]\ar[d]^{=}&\dots\ar[r]&Y_2\ar[r]\ar[d]^{=}&Y_1&&\\
 Y_n\ar[r]\ar[d]^-{=}&\dots\ar[r]&Y_2&&&\\
 \vdots\ar[d]^{=}&\dots&&&&\\
 Y_n&&&&&
  }
\end{equation*}
By definition, we have the following commutative diagram of functors
\begin{equation}
 \label{dfnofphocoprop-diag2}
 \xymatrix@C=20pt{
  \Phi^{\mr{co}}(\Gamma\times\Delta^1,\mr{Sch}^{\mr{op}})^{\mr{prop}}
  \ar[rr]\ar[d]_{\Delta^{\{0\}}\rightarrow\Delta^1}&&
  \overleftarrow{s}^*\mr{Corr}^{\circledast}
  \ar[d]^{q_{\mr{Corr}}}\\
 \Phi^{\mr{co}}(\Gamma,\mr{Sch}^{\mr{op}})^{\simeq}
  _{/\mbf{\Delta}^{\mr{op}}}
  \ar[r]&(\mr{Corr}(\mr{Sch})^{\mr{prop}}
  _{\mr{sep};\mr{all}})^{\circledast}\ar[r]&
  \mr{Corr}^{\circledast}.
  }
\end{equation}

\subsection{}
\label{dfnofdeltandpsi}

Let $\mc{C}\rightarrow\mbf{\Delta}^{\mr{op}}$ be a
coCartesian fibration.
Let $\Psi\mc{C}:=\gamma_*\delta^*(\mc{C})$. By
\cite[3.2.2.12]{HTT}, the functor
$\Psi\mc{C}\rightarrow\mbf{\Delta}^{\mr{op}}$ is a coCartesian
fibration. A vertex of $\Psi\mc{C}$ over $[n]\in\mbf{\Delta}^{\mr{op}}$
corresponds to a functor $\Delta^n\rightarrow\mc{C}$ over
$\mbf{\Delta}^{\mr{op}}$ where
$\Delta^n\rightarrow\mbf{\Delta}^{\mr{op}}$ sends $i\in\Delta^n$ to
$[n-i]$ and $i\rightarrow i+1$ to $d^0\colon[n-i]\rightarrow[n-i-1]$.
Now, we define $\Psi^{\star}\mc{C}$ to be the full
subcategory of $\Psi\mc{C}$ spanned by the functors
$\phi\colon\Delta^n\rightarrow\mc{C}$ such
that the following condition holds:
\begin{itemize}
 \item For each $0\leq i<n$, the edge $\phi(i\rightarrow(i+1))$ in
       $\mc{C}$ is coCartesian over the map
       $[n-i]\rightarrow[n-i-1]$ in $\mbf{\Delta}^{\mr{op}}$ defined by
       the inert map $d^0\colon[n-i-1]\rightarrow[n-i]$.
\end{itemize}
Let $\mc{C}\rightarrow\mc{D}$ be a map of coCartesian fibrations over
$\mbf{\Delta}^{\mr{op}}$ which preserves coCartesian edges over the
inert map $d^0$.
Then the induced functor $\Psi\mc{C}\rightarrow\Psi\mc{D}$ induces
$\Psi^{\star}\mc{C}\rightarrow\Psi^{\star}\mc{D}$.

Now, let us describe coCartesian edges explicitly.
Recall the notations of \ref{intgamanddiag}.
Let $\phi\colon[m]\rightarrow[n]$ be a map in $\mbf{\Delta}^{\mr{op}}$.
Then the diagram $\delta\colon D_\phi\rightarrow\mbf{\Delta}^{\mr{op}}$
can be depicted as
\begin{equation*}
 \xymatrix@C=8pt{
  [n]\ar[d]_{\phi}&
  [n]\ar[r]&[n-1]\ar[r]&\dots\ar[r]&[n-\phi(0)]\ar[r]\ar[dlll]_-{\delta(\alpha_0)}\ar[r]&
  \dots\ar[r]&[n-\phi(k)]\ar[r]\ar[dll]_-{\delta(\alpha_k)}&
  \dots\ar[r]&[n-\phi(m)]\ar[r]\ar[dll]_-{\delta(\alpha_m)}&
  \dots\ar[r]&[0]\\
 [m]&
  [m]\ar[r]&[m-1]\ar[r]&\dots\ar[r]&[m-k]\ar[r]&\dots\ar[r]&[0].}
\end{equation*}
In the case $\phi$ is an inert map, the map $\delta(\alpha_k)$ is the
unique inert map which sends $0$ to $0$.
A vertex of $\Psi\mc{C}$ over $[n]$ (resp.\ $[m]$) is a functor
$D_0\rightarrow\mc{C}$ (resp.\ $D_1\rightarrow\mc{C}$) over
$\mbf{\Delta}^{\mr{op}}$, and an edge between these vertices is a
functor $F\colon D_\phi\rightarrow\mc{C}$.
The description of coCartesian edges of \cite[3.2.2.12]{HTT},
coCartesian edges over $\phi$ are exactly the functors $F$ such that
$F(\alpha_k)$ are coCartesian edges in $\mc{C}$ over
$\mbf{\Delta}^{\mr{op}}$. This description, in particular, implies that
the induced map $\Psi^\star\mc{C}\rightarrow\mbf{\Delta}^{\mr{op}}$
is a coCartesian fibration as well.

Now, consider the following diagram
\begin{equation*}
 \xymatrix@C=70pt{
  &\mbf{\Delta}^{\mr{op}}
  \ar[d]_{z}\ar@/^10pt/[rd]^-{\mr{id}}\ar@/_10pt/[ld]_-{\mr{id}}&\\
 \mbf{\Delta}^{\mr{op}}&
  \Gamma\ar[r]^-{\gamma}\ar[l]_-{\delta}&
  \mbf{\Delta}^{\mr{op}}
  }
\end{equation*}
where $z$ is the functor sending $[n]$ to $([n],0)$. The morphism of
functors
$\gamma_*\delta^*\rightarrow\gamma_*z_*z^*\delta^*\simeq\mr{id}$ induces
the functor $\Psi\mc{C}\rightarrow\mc{C}$
over $\mbf{\Delta}^{\mr{op}}$. Thus, we have a functor
$G\colon\Psi^{\star}\mc{C}\rightarrow
\mc{C}$. Our desired functor $\beta$ in (\ref{seqfuncdefex}) is the
functor $G$ in the case where
$\mc{C}=\overleftarrow{s}^*\mr{Corr}^{\circledast}$.

\begin{rem*}
 Let us informally describe objects of
 $\Psi(\overleftarrow{s}^*\mc{A}^{\circledast})$.
 An object of $\mc{A}^{\circledast}_{[n]}$
 is denoted by $M_1\boxtimes M_2\boxtimes\dots\boxtimes M_n$ where
 $M_i\in\mc{A}$ by identifying $\mc{A}^{\circledast}_{[n]}$ and
 $\mc{A}^n$. Then objects of
 $\Psi(\overleftarrow{s}^*\mc{A}^{\circledast})$ over
 $[n]\in\mbf{\Delta}^{\mr{op}}$ are diagrams of the form
 \begin{equation*}
  (M^0_{-\infty}\boxtimes M^0_1\boxtimes M^0_2
   \boxtimes\dots\boxtimes M^0_n)
   \rightarrow
   (M^1_{-\infty}\boxtimes M^1_2\boxtimes\dots\boxtimes M^1_n)
   \rightarrow\dots\rightarrow
   (M^n_{-\infty}),
 \end{equation*}
 where the map $(M^i_{-\infty}\boxtimes M^i_{i+1}\boxtimes\dots\boxtimes
 M^i_n)\rightarrow(M^{i+1}_{-\infty}\boxtimes M^{i+1}_{i+2}\boxtimes\dots\boxtimes
 M^{i+1}_n)$ consists of data $M^i_{-\infty}\otimes M^i_{i+1}\rightarrow
 M^{i+1}_{-\infty}$ and $M^i_j\rightarrow M^{i+1}_j$ for $j\geq i+2$.
 Objects of $\Psi\mc{A}^{\circledast}$ over $[n]$ are diagrams without
 the $M^i_{-\infty}$-factors. In order for it to belong to
 $\Psi^{\star}\mc{A}^{\circledast}$, the map $M^i_j\rightarrow
 M^{i+1}_j$ should be equivalences. Finally, $\mc{H}_{\top}$ sends the
 diagram (\ref{diagarmap}) to the diagram above such that
 $M^i_{-\infty}\simeq\Mor(I_{X_i},J_{Y_i})$,
 $M^i_j\simeq\Mor(J_{Y_j},J_{Y_j})$.
\end{rem*}

\begin{lem}
 \label{eapropofpsi}
 \begin{enumerate}
  \item Let $p\colon\mc{C}\rightarrow\mc{D}$ be a categorical fibration
	between coCartesian fibrations over $\mbf{\Delta}^{\mr{op}}$.
	Then $\Psi p\colon\Psi\mc{C}\rightarrow\Psi\mc{D}$ is a
	categorical fibration. Moreover, an edge $e\colon
	D_\phi\rightarrow\mc{C}$ of $\Psi\mc{C}$
	is $\Psi p$-Cartesian if $e(\beta_i)$ is $p$-Cartesian for
	$i\leq\phi(m)$ and $e([n],i)$ for $i>\phi(m)$ is $p$-final
	{\normalfont(}recall from {\normalfont\ref{intgamanddiag}} for the definition of $\beta_i${\normalfont)}.
	
  \item\label{eapropofpsi-2}
       The functor $G$ is a categorical equivalence.
 \end{enumerate}
\end{lem}
\begin{proof}
 Let us check the first assertion.
 The functor $\Psi p$ is a categorical fibration by \cite[B.4.5]{HA}.
 Consider the following diagram
 \begin{equation*}
  \xymatrix{
   D_1\ar[r]^-{v}\ar[d]_{i}&\mc{C}\ar[d]^{p}\\
  D_{\phi}\ar[r]^-{w}\ar[ur]^{e}&\mc{D}.}
 \end{equation*}
 In order to show that the edge $e$ is Cartesian, it suffices to show
 that the diagram above is a $p$-right Kan extension diagram by invoking
 \cite[B.4.8]{HA}, as usual.
 For this, we must show that the diagram
 $(D_\phi)_{([n],i)/}^{\triangleleft}\rightarrow\mc{C}$ is a $p$-limit
 diagram.
 Let $i\leq\phi(m)$. Then we have the map $\beta_i$ in $D_\phi$.
 This is an initial object of $(D_{\phi})_{([n],i)/}$. Thus, the
 assumption that $e(\beta_i)$ is a $p$-Cartesian edge implies that the
 diagram is $p$-limit.
 If $i>\phi(m)$, then $(D_\phi)_{([n],i)/}$ is empty.
 Then $e([n],i)$ must be a $p$-final object, which follows by
 assumption.

 Let us show the second claim. The functor $G$ is a functor between
 coCartesian fibrations over $\mbf{\Delta}^{\mr{op}}$, and moreover, it
 sends coCartesian edges to coCartesian edges. Thus, by
 \cite[3.3.1.5]{HTT}, it suffices to show that the fibers are trivial
 fibrations. This follows from \cite[4.3.2.15]{HTT}.
\end{proof}

\subsection{}
\label{insertone}
We currently have the following diagram
\begin{equation*}
 \xymatrix@C=40pt{
  \Phi^{\mr{co}}(\Gamma\times\Delta^1,\mr{Sch}^{\mr{op}})^{\mr{prop}}
  \ar[r]_-{\beta^{-1}\circ\alpha}
  \ar[d]_{\iota}^{=\{\Delta^{\{0\}}\hookrightarrow\Delta^1\}}
  \ar@/^15pt/[rr]^-{\mc{H}'_{\top}}&
  \Psi^\star\overleftarrow{s}^*\mr{Corr}^{\circledast}
  \ar[r]\ar[d]^{\Psi q_{\mr{Corr}}}&
  X\ar[r]^-{\rho}\ar[d]_{F'}\ar@{}[rd]|\square&
  \Psi(\overleftarrow{s}^*\mc{A}^{\circledast})
  \ar[d]_{F}^{=\Psi q_{\mc{A}}}\\
 \Phi^{\mr{co}}(\Gamma,\mr{Sch}^{\mr{op}})^{\simeq}
  _{/\mbf{\Delta}^{\mr{op}}}
  \ar[r]\ar@/_15pt/[rrd]_{\mc{H}}&
  \Psi^{\star}\mr{Corr}^{\circledast}
  \ar[r]&
  \Psi^{\star}\mc{A}^{\circledast}
  \ar[d]^{G}_{\sim}\ar[r]&
  \Psi\mc{A}^{\circledast}
  \\
 &&\mc{A}^{\circledast}.&
  }
\end{equation*}
Let $\mc{E}$ be the collection of the edges
$D_\phi\rightarrow\mc{A}^{\circledast}$ of $\Psi\mc{A}^{\circledast}$
such that $\phi$ is the identity. Let $\mc{E}'$ be the collection of the
$F$-Cartesian edges in $\Psi(\overleftarrow{s}^*\mc{A}^{\circledast})$
which sits over edges in $\mc{E}$. Now, the condition (A) of
\cite[3.1.1.6]{HTT} follows by Lemma \ref{eapropofpsi}, (B) follows by
definition, and (C) follows by combining Lemma \ref{shiftfunccarted} and
Lemma \ref{eapropofpsi}. Thus, invoking \cite[3.1.1.6]{HTT}, the
map $(\Psi(\overleftarrow{s}^*\mc{A}^{\circledast}),\mc{E}')\rightarrow
(\Psi\mc{A}^{\circledast},\mc{E})$ in $\sSet^+$ has the right lifting
property with respect to any marked anodyne.
We put the induced marking on $\Psi^{\star}\mc{A}^{\circledast}$ from
$\Psi\mc{A}^{\circledast}$, and to $X$ by the pullback diagram.
The marked simplicial sets are denoted by
$\overline{\Psi^{\star}\mc{A}^{\circledast}}$ and $\overline{X}$
respectively. The map $F'\colon\overline{X}\rightarrow
\overline{\Psi^{\star}\mc{A}^{\circledast}}$ also has the right lifting
property with respect to any marked anodyne.

For $X,Y\in\sSet^+$, the marked simplicial set $X^Y$ is denoted by
$\mr{Fun}^+(Y,X)$.
Now, consider the following sequence of functors:
\begin{align*}
 \mr{Fun}^+(\Phi^{\mr{co}}(\Gamma\times
 \Delta^1,\mr{Sch}^{\mr{op}})^{\mr{prop},\flat},\overline{X})
 &\xrightarrow{a}
 \mr{Fun}^+(\Phi^{\mr{co}}(\Gamma\times
 \Delta^1,\mr{Sch}^{\mr{op}})^{\mr{prop},\flat},
 \overline{\Psi^{\star}\mc{A}^{\circledast}})\\
 \mr{Fun}(\Phi^{\mr{co}}(\Gamma\times
 \Delta^1,\mr{Sch}^{\mr{op}})^{\mr{prop}},\Psi^{\star}\mc{A}^{\circledast})
 &\xrightarrow[\sim]{b}
 \mr{Fun}(\Phi^{\mr{co}}(\Gamma\times
 \Delta^1,\mr{Sch}^{\mr{op}})^{\mr{prop}},\mc{A}^{\circledast}).
\end{align*}
By \cite[3.1.2.3]{HTT}, $a$ has the right lifting property with respect
to any marked anodyne.
On the other hand, by \cite[1.2.7.3]{HTT}, $b$ is a categorical
equivalence.

Now, by \ref{constonetohm}, we have the map $\mbf{1}\rightarrow\mc{H}$ in
the $\infty$-category
$\mr{Fun}(\Phi^{\mr{co}}(\Gamma,\mr{Sch}^{\mr{op}})^{\simeq},
\mc{A}^{\circledast})$.
By composing with $\iota$, this induces the map
$f\colon\mbf{1}\circ\iota\rightarrow(G\circ F')\circ\mc{H}'_{\top}$ in
$\mr{Fun}(\Phi^{\mr{co}}(\Gamma\times
\Delta^1,\mr{Sch}^{\mr{op}})^{\mr{prop}},\mc{A}^{\circledast})$.
Because $b$ is a categorical equivalence, we can take a map $f'$ such
that $b(f')\simeq f$.
For each object $S\in\Phi^{\mr{co}}(\Gamma\times
\Delta^1,\mr{Sch}^{\mr{op}})^{\mr{prop}}$, we have the edge $f(S)$ of
$\Phi^{\mr{co}}(\Gamma\times\Delta^1,\mr{Sch}^{\mr{op}})^{\mr{prop}}$. The
image of this edge $f(S)$ in $\mbf{\Delta}^{\mr{op}}$ is constant.
Since $G$ is a functor over
$\mbf{\Delta}^{\mr{op}}$, the image of $f'(S)$ in
$\mbf{\Delta}^{\mr{op}}$ is constant. Thus, by definition of $\mc{E}$,
$f'$ defines a marked edge of
$\mr{Fun}^+(\Phi^{\mr{co}}(\Gamma\times\Delta^1,\mr{Sch}^{\mr{op}})
^{\mr{prop},\flat},\overline{\Psi^{\star}\mc{A}^{\circledast}})$.
Since $a$ has the right lifting property with respect to any marked
anodyne, we may lift the marked edge $f'$ along $a$, and we get a
functor $\mc{H}'_{\top,\mbf{1}}\colon
\Phi^{\mr{co}}(\Gamma\times\Delta^1,\mr{Sch}^{\mr{op}})^{\mr{prop}}
\rightarrow X$ and a map
$\mc{H}'_{\top,\mbf{1}}\rightarrow\mc{H}'_{\top}$ whose composition with
$G\circ F'$ is equivalent to $f$. Finally, put
$\mc{H}_{\top,\mbf{1}}:=\rho\circ\mc{H}'_{\top,\mbf{1}}$.
By construction, the edge
$\mc{H}_{\top,\mbf{1}}(S)\rightarrow\mc{H}_{\top}(S)$ is $F$-Cartesian
for any $S\in\Phi^{\mr{co}}(\Gamma\times\Delta^1,\mr{Sch}^{\mr{op}})
^{\mr{prop}}$.

\subsection{}
\label{taketensordefH}
Let
$\delta^+\colon\Gamma\times\Delta^1\rightarrow\mbf{\Delta}^{\mr{op}}$ be
the functor whose restriction to $\Gamma\times\{0\}$ is $\delta$ and
$\delta^+(([n],i),1):=[0]$ such that
$\delta^+\bigl((([n],i),0)\rightarrow(([n],i),1)\bigr)$ is equal to the
map $[n-i]\rightarrow[0]$ in $\mbf{\Delta}^{\mr{op}}$ corresponding to
the function $[0]\rightarrow[n-i]$ sending $0$ to $n-i$.
We have the following diagram
\begin{equation*}
 \xymatrix@C=40pt{
  &&\Gamma\ar[d]^{\iota_0}
  \ar@/_10pt/[dl]_{\delta}\ar@/^10pt/[dr]^{\gamma}&\\
 \mbf{\Delta}^{\mr{op}}&
  \mbf{\Delta}^{\mr{op}}\ar[l]_-{\overleftarrow{s}}&
  \Gamma\times\Delta^1
  \ar[r]^-{\gamma^+}\ar[l]_{\delta^+}&
  \mbf{\Delta}^{\mr{op}}\\
 &&\Gamma\ar[u]_{\iota_1}
  \ar@/^10pt/[ull]^-{[1]}\ar@/_10pt/[ru]_-{\gamma}&
  }
\end{equation*}
where $\iota_i$ is the inclusion into $\Delta^{\{i\}}\subset\Delta^1$
and $\gamma^+:=\gamma\circ\mr{pr}_1$.
Let $\mc{C}\rightarrow\mbf{\Delta}^{\mr{op}}$ be a coCartesian
fibration.
We put $\Psi^+\mc{C}:=\gamma^+_*\circ
(\delta^+\circ\overleftarrow{s})^*(\mc{C})$.
We have the map
$\theta_0\colon\Psi^+\rightarrow\Psi\circ\overleftarrow{s}^*$ by using
the adjunction $\mr{id}\rightarrow\iota_{0,*}\iota_0^*$, and
$\theta_1\colon\Psi^+\rightarrow\Phi^{\mr{co}}(\Gamma,(-)_{[1]})$ by
using the adjunction $\mr{id}\rightarrow\iota_{1,*}\iota_1^*$.
We define a category $\Psi^{+,\star}\mc{C}$ by the full subcategory of
$\Psi^+\mc{C}$ spanned by vertices corresponding to the functors
$\phi\colon\Delta^n\times\Delta^1\simeq\Gamma_{[n]}\times\Delta^1
\rightarrow\mc{C}\times
_{\mbf{\Delta}^{\mr{op}},\overleftarrow{s}}\mbf{\Delta}^{\mr{op}}$ over
$\mbf{\Delta}^{\mr{op}}$ satisfying the following condition:
\begin{itemize}
 \item for each $i\in\Delta^n$, the edge $\phi(\{i\}\times\Delta^1)$ is
       a coCartesian edge over the unique active map
       $[n]^{\triangleleft}\rightarrow[0]^{\triangleleft}$ in
       $\mbf{\Delta}^{\mr{op}}$.
\end{itemize}
Similarly to the proof of Lemma \ref{eapropofpsi}.\ref{eapropofpsi-2},
the map
$\theta_0\colon\Psi^{+,\star}\mc{C}\rightarrow
\Psi(\overleftarrow{s}^*\mc{C})$ is a categorical equivalence.
We now define $\mc{H}_\bullet$ as in the following diagram:
\begin{equation*}
 \xymatrix{
  \Phi^{\mr{co}}(\Gamma\times\Delta^1,\mr{Sch}^{\mr{op}})^{\mr{prop}}
  \ar[r]^-{\mc{H}_{\top,\mbf{1}}}\ar@/_10pt/[drr]_-{\mc{H}_\bullet}&
  \Psi(\overleftarrow{s}^*\mc{A}^{\circledast})&
  \Psi^{+,\star}(\mc{A}^{\circledast})
  \ar[l]^-{\sim}_-{\theta_0}\ar[d]^{\theta_1}\\
 &&\Phi^{\mr{co}}(\Gamma,\mc{A}).
  }
\end{equation*}

\begin{lem*}
 \label{fundlemHprop}
 The functor $\mc{H}_\bullet$ preserves coCartesian edges.
\end{lem*}
\begin{proof}
 For a coCartesian fibration
 $p\colon\mc{C}\rightarrow\mbf{\Delta}^{\mr{op}}$,
 an edge in $\mc{C}$ is said to be an {\em inert edge} if it is
 $p$-coCartesian over an inert map in $\mbf{\Delta}^{\mr{op}}$.
 First, let us check that $\mc{H}_\bullet$ preserves inert edges.
 Preservation for $\alpha$ is easy to check, that for $\beta$ follows
 because it is a Cartesian equivalence.
 Let $[n]\rightarrow[m]$ be the inert map sending $0\in[m]$ to
 $0\in[n]$. A coCartesian edge over such a map is called a
 {\em $0$-inert map}.
 Let $\mc{C}\rightarrow\mc{D}$ be a functor between
 coCartesian fibrations over $\mbf{\Delta}^{\mr{op}}$ such that
 $0$-inert maps are preserved. Then, by the description of coCartesian
 edges in \ref{dfnofdeltandpsi}, the induced functor
 $\Psi\mc{C}\rightarrow\Psi\mc{D}$ preserves inert edges.
 In particular, for a map of coCartesian fibrations
 $\mc{C}'\rightarrow\mc{D}'$ over $\mbf{\Delta}^{\mr{op}}$ which
 preserves inert edges, the induced functor
 $\Psi\overleftarrow{s}^*\mc{C}'\rightarrow
 \Psi\overleftarrow{s}^*\mc{D}'$ preserves inert edges because
 $\overleftarrow{s}^*\mc{C}'\rightarrow\overleftarrow{s}^*\mc{D}'$
 preserves $0$-inert edges. Thus, $\mc{H}_{\top}$ preserves inert
 edges. Now, for any inert map $[n]\rightarrow[m]$ in
 $\mbf{\Delta}^{\mr{op}}$ sending $0\in[m]$ to $0\in[n]$, a map
 $X\rightarrow Y$ in $\overleftarrow{s}^*\mc{A}^{\circledast}$ is
 coCartesian over $\mbf{\Delta}^{\mr{op}}$ if and only if
 $\sigma^0_!X\rightarrow\sigma^0_!Y$ is an equivalence and
 $q_{\mc{A}}(X)\rightarrow q_{\mc{A}}(Y)$ is a coCartesian edge in
 $\mc{A}^{\circledast}$. The description of coCartesian edges in
 \ref{dfnofdeltandpsi} implies that $\mc{H}_{\top,\mbf{1}}$ preserves
 inert edge as well. Finally, in order to check that $\mc{H}_\bullet$
 preserves inert edge, we may describe coCartesian edges of
 $\Psi^{+,\star}\mc{A}^{\circledast}$ similarly to
 \ref{dfnofdeltandpsi}, and using the fact that
 $F\circ\mc{H}_{\top,\mbf{1}}$ is a lifting of $\mbf{1}$ along $G$.
 
 We have shown that $\mc{H}_\bullet$ preserves inert edges. Let us treat
 the general case.
 Let $\phi\colon[m]\rightarrow[n]$ be a map in
 $\mbf{\Delta}^{\mr{op}}$, and $e\colon v_0\rightarrow v_1$ be a
 coCartesian edge in $\Phi^{\mr{co}}(\dots)^{\mr{prop}}$ over $\phi$.
 We wish to show that $\mc{H}_\bullet(e)$ is a
 coCartesian edge. Let $\xi\colon\mc{H}_\bullet(v_0)\rightarrow w$ be a
 coCartesian edge over $\phi$. Then we have a map
 $D\colon\Delta^2\rightarrow\Phi^{\mr{co}}(\Delta,\mc{A})$ such that
 $D(\Delta^{\{0,1\}})=\xi$, $D(\Delta^{\{0,2\}})=\mc{H}_\bullet(e)$.
 Put $A:=D(\Delta^{\{1,2\}})$. We must show that $A$ is an equivalence.
 Let $f\colon v_1\rightarrow v_2$ be a coCartesian edge
 over $\sigma^i$. We have a diagram
 $(\Delta^1\times\Delta^1)^{\triangleleft}\rightarrow
 \Phi^{\mr{co}}(\Gamma,\mc{A})$ depicted as
 \begin{equation*}
  \xymatrix@C=50pt{
   \ar@{}[rd]|(.6){D}&w\ar@{~>}[r]\ar[d]^{A}&w'\ar[d]^{B}\\
  \mc{H}_\bullet(v_0)
   \ar[r]_-{\mc{H}_\bullet(e)}\ar@{~>}@/^10pt/[ur]^-{\xi}&
   \mc{H}_\bullet(v_1)
   \ar@{~>}[r]_-{\mc{H}_\bullet(f)}&
   \mc{H}_\bullet(v_2)\\
  [m]\ar[r]^-{\phi}&[n]\ar[r]^-{\sigma^i}&[0]
   }
 \end{equation*}
 The arrows $\rightsquigarrow$ mean that the edges are coCartesian.
 Note that $\mc{H}_\bullet(f)$ is coCartesian since $\sigma^i$ is an
 inert map. Let $a\colon X\rightarrow Y$ be a map in
 $\mr{Fun}(\Delta^n,\mc{A})$. Then it is an equivalence if and only if
 $\sigma^i_!(a)\in\mc{A}$ is an equivalence for any $i$.
 Thus, it suffices to show that $B$ is an
 equivalence. Now, $\phi\circ\sigma^i$ is an inert map, so a composition
 of $\mc{H}_\bullet(f)$ and $\mc{H}_\bullet(e)$ is coCartesian. Thus,
 $B$ is an equivalence by \cite[2.4.1.7, 2.4.1.5]{HTT}.
\end{proof}

\subsection{}
\label{integlconst}
We have constructed the functor $\mc{H}_\bullet$. By taking the
straightening functor, this is a functor between certain simplicial
objects in $\Cat_\infty$.
We need to extract a functor of $\infty$-categories
``associated to'' $\mc{H}_\bullet$.
In fact, a simplicial $\infty$-category $\mc{C}_\bullet$ has two
directions of morphisms. A morphism of $\mc{C}_0$, namely an object of
$\mr{Fun}(\Delta^1,\mc{C}_0)$, and an object of $\mc{C}_1$, namely an
object of $\mr{Fun}(\Delta^0,\mc{C}_1)$. We wish to ``integrate'' these
two types of morphisms. The functor $\mr{Int}$ we will construct in the
rest of this section enables us to do this.

Let $\mc{C}$ be an $\infty$-category. By \ref{constsimplcat}, the
simplicial object
$\Delta^{\bullet}\colon\mbf{\Delta}\rightarrow\Cat_\infty$
induces the functor
$M_\bullet\mc{C}:=\mr{Fun}(\Delta^{\bullet},\mc{C})\colon
\mbf{\Delta}^{\mr{op}}\rightarrow\Cat_\infty$.
Now, let $p\colon\mc{D}\rightarrow\mc{C}$ be a Cartesian
fibration.
An edge $\Delta^1\rightarrow\mc{D}$ is said to be
{\em $p$-equivalent} if the edge
$\Delta^1\rightarrow\mc{D}\rightarrow\mc{C}$ is an equivalence.
We define $\widetilde{M}_np$ to be the subcategory of
$M_n\mc{D}$ spanned by functors $\Delta^n\rightarrow\mc{D}$ such that
any induced edge $\Delta^1\rightarrow\Delta^n\rightarrow\mc{D}$ is
$p$-Cartesian, and morphisms $\Delta^n\times\Delta^1\rightarrow\mc{D}$
such that for any vertex $k$ of $\Delta^n$, the induced edge
$\Delta^1\xrightarrow{\{k\}\times\mr{id}}\Delta^n\times\Delta^1
\rightarrow\mc{D}$ is $p$-equivalent. Then $M_\bullet\mc{D}$ 
induces the functor
$\widetilde{M}_\bullet p\colon\mbf{\Delta}^{\mr{op}}\rightarrow
\Cat_\infty$.

\begin{prop*}
 \label{propinfcatext}
 There exists a functor
 $\mr{Int}\colon\mr{Fun}(\Delta^{\bullet},\Cat_\infty)
 \rightarrow\Cat_\infty$ such that the following holds.
 \begin{enumerate}
  \item\label{propinfcatext-1}
       Let $\mc{C}$ be an $\infty$-category. Then we have a canonical
       functor $D\colon\mr{Int}(M_\bullet\mc{C})\rightarrow\mc{C}$;
	
  \item\label{propinfcatext-2}
       If we are given a Cartesian fibration
       $p\colon\mc{D}\rightarrow\mc{C}$ of $\infty$-categories, the
       induced functor $\mr{Int}(\widetilde{M}_\bullet p)
       \rightarrow\mr{Int}(M_\bullet\mc{D})\xrightarrow{D}\mc{D}$
       is an equivalence.
	 
  \item\label{propinfcatext-3}
       For $\mc{C}_\bullet\in\mr{Fun}(\Delta^\bullet,\Cat_\infty)$, we
       have the functor
       $\alpha\colon\mc{C}_0^{\simeq}\rightarrow
       \mr{Int}(\mc{C}_\bullet)^{\simeq}$
       and for $x,y\in\mc{C}_0$, a functor
       \begin{equation*}
	\alpha_{x,y}\colon
	 \{x\}
	 \times^{\mr{cat}}_{\mc{C}_0^{\simeq},\{0\}}
	 \mr{Fun}(\Delta^1,\mc{C}_0)^{\simeq}
	 \times^{\mr{cat}}_{\{1\},\mc{C}_0^{\simeq},s}
	 \mc{C}_1^{\simeq}
	 \times^{\mr{cat}}_{t,\mc{C}_0^{\simeq}}
	 \{y\}
	 \rightarrow
	 \mr{Map}_{\mr{Int}(\mc{C}_\bullet)}(\alpha(x),\alpha(y)).
       \end{equation*}
       The maps $\alpha$ and $\alpha_{x,y}$ are functorial with respect
       to $\mc{C}_\bullet$.
       If $\mc{C}_\bullet=M_\bullet\mc{C}$ with an $\infty$-category
       $\mc{C}$, we have $\mc{C}_0\simeq\mc{C}$,
       $\mc{C}_1\simeq\mr{Fun}(\Delta^1,\mc{C})$ by definition.
       Under this identification, we have
       $D^{\simeq}\circ\alpha\simeq\mr{id}$ and the following diagram
       commutes:
       \begin{equation*}
	\xymatrix{
	 \{x\}\times^{\mr{cat}}_{\mc{C}^{\simeq},\{0\}}
	 \mr{Fun}(\Delta^1,\mc{C})^{\simeq}
	 \times^{\mr{cat}}_{\{1\},\mc{C}^{\simeq},\{0\}}
	 \mr{Fun}(\Delta^1,\mc{C})^{\simeq}
	 \times^{\mr{cat}}_{\{1\},\mc{C}^{\simeq}}
	 \{y\}
	 \ar[r]^-{\alpha_{x,y}}\ar[rd]&
	 \mr{Map}_{\mr{Int}(\mc{C}_\bullet)}(\alpha(x),\alpha(y))
	 \ar[d]^{D(\alpha(x),\alpha(y))}\\
	&\mr{Map}_{\mc{C}}(x,y).
	 }
       \end{equation*}
       where the diagonal map is the composition map.
 \end{enumerate}
\end{prop*}

\begin{rem*}
 Consider $\mc{C}_\bullet=\widetilde{M}_{\bullet}p$ for a Cartesian
 fibration $p\colon\mc{D}\rightarrow\mc{C}$.
 In this case, the fiber product of assertion \ref{propinfcatext-3} is
 the space of maps of the form
 \begin{equation*}
  \xymatrix{
   x\ar[d]_-{a}&\\
  \star\ar[r]^-{b}&y}
 \end{equation*}
 where $a$ is $p$-equivalent and $b$ is $p$-Cartesian.
\end{rem*}

For now, we complete the proof of the theorem assuming the above
proposition is known.

\begin{proof}[Proof of Theorem \ref{mainthmcons}]
 Let
 $p\colon\widetilde{\mr{Ar}}(\mr{Sch})^{\mr{prop}}_{\mr{sep}}
 \rightarrow\mr{Sch}$
 be the Cartesian fibration (constructed in Remark \ref{straitencatarrprop}).
 By Definition \ref{dfnofphocoprop}, we have $\widetilde{M}_\bullet
 p\simeq\mr{St}(\Phi^{\mr{co}}(\Gamma\times\Delta^1,
 \mr{Sch}^{\mr{op}})^{\mr{prop}})$.
 On the other hand, we have $\mr{St}(\Phi^{\mr{co}}(\Gamma,\mc{A}))
 \simeq M_{\bullet}\mc{A}$. By Lemma \ref{fundlemHprop},
 we can take the straightening of the functor $\mc{H}_\bullet$.
 Thus, we have the diagram
 \begin{equation*}
  \xymatrix@C=40pt{
   \mr{Int}(\mr{St}(\Phi^{\mr{co}}(\Gamma\times\Delta^1,
   \mr{Sch}^{\mr{op}})^{\mr{prop}}))
   \ar[d]_{\sim}\ar[r]^-{\mr{Int}\mr{St}(\mc{H}_\bullet)}&
   \mr{Int}(\mr{St}(\Phi^{\mr{co}}(\Gamma,\mc{A})))
   \ar@{-}[r]^-{\sim}&
   \mr{Int}(M_{\bullet}\mc{A})\ar[d]\\
  \widetilde{\mr{Ar}}(\mr{Sch})^{\mr{prop}}_{\mr{sep}}&&
   \mc{A}.
   }
 \end{equation*}
 The composition of these functors is the desired functor.
 The required properties follow by unwinding the construction and the
 description \ref{propinfcatext}.\ref{propinfcatext-3}.
\end{proof}

The rest of this section is devoted to proving the proposition.
We need some preparations.

\subsection{}
\label{introgamm}
Let $F\colon \mc{C}\rightarrow\mc{D}$ be a Cartesian fibration of
(ordinary) categories. Let us construct a category $\mc{T}$ as follows:
The objects consists of pairs $(c\rightarrow c',d)$ where $d\in\mc{D}$
and $c\rightarrow c'$ is a map in the fiber $\mc{C}_d$. A morphism
$(c_0\rightarrow c'_0,d_0)\rightarrow(c_1\rightarrow c'_1,d_1)$ is a
pair of a morphism $d_0\rightarrow d_1$ in $\mc{D}$ and a diagram
\begin{equation*}
 \xymatrix{
  c_0\ar[d]&c\ar@{~>}[r]^-{\alpha}\ar[l]_{\beta}&c_1\ar[d]\\
 c'_0\ar[rr]&&c'_1
  }
\end{equation*}
where $\alpha$ is a Cartesian edge over $d_0\rightarrow d_1$ and $\beta$
is a morphism in $\mc{C}_{d_0}$. We apply this construction to the case
where $F=\gamma\colon\Gamma\rightarrow\mbf{\Delta}^{\mr{op}}$,
and get a Cartesian fibration
$p_{\mc{T}}\colon\mc{T}\rightarrow\mbf{\Delta}^{\mr{op}}$ whose fiber
over $[n]\in\mbf{\Delta}^{\mr{op}}$ is $\mr{Tw}^{\mr{op}}\Delta^n$.

\begin{rem*}
 This is nothing but the unfurling construction of Barwick
 (cf.\ \cite[3.2]{BGN}),
 and the construction above can be generalized to a Cartesian fibration
 between $\infty$-categories. We restricted our attention to ordinary
 categories just to avoid too much complications.
\end{rem*}

Now, let us construct a functor
$F\colon\mc{T}\rightarrow\mbf{\Delta}^{\mr{op}}\times
\mbf{\Delta}^{\mr{op}}$.
Integers $0\leq a\leq b\leq n$ determine a map $a\rightarrow b$ in
$\Delta^n$. This map is denoted by $ab$. The object of $\mc{T}$ over
$[n]\in\mbf{\Delta}^{\mr{op}}$ defined by $ab$ is denoted by $(n,ab)$.
We put $F(n,ab):=([a],[n-b])$. Now, let
$f\colon(n,ab)\rightarrow(m,a'b')$ be a morphism.
Let $\phi_f\colon[m]\rightarrow[n]$ be a function corresponding to the
morphism $[n]\rightarrow[m]$ in $\mbf{\Delta}^{\mr{op}}$, and by
definition of morphisms in $\mc{T}$, $f$ can be written as a diagram
\begin{equation*}
 \xymatrix{
  a\ar[d]&\widetilde{a}\ar@{~>}[r]\ar[l]\ar[d]&a'\ar[d]\\
 b\ar[r]&\widetilde{b}\ar@{~>}[r]&b',
  }
\end{equation*}
in $\Gamma$ where $\rightsquigarrow$ are Cartesian edges over
$\mbf{\Delta}^{\mr{op}}$.
By construction, we have $\phi_f(a')=\widetilde{a}$,
$\phi_f(b')=\widetilde{b}$, $\widetilde{a}\leq a$, and
$b\leq\widetilde{b}$. Let $\psi_{\mr{a}}:=\phi_f|_{[a']}$,
which yields a function $[a']\rightarrow[\widetilde{a}]$ since
$\phi_f(a')=\widetilde{a}$.
We also define a function
$\psi_{\mr{b}}\colon[m-b']\rightarrow[n-\widetilde{b}]$
by $\psi_{\mr{b}}(i):=\phi_f(i+b')-\widetilde{b}$,
which is well-defined since $\phi_f(b')=\widetilde{b}$.
For $c\leq d$, let $d^0_{c,d}\colon[c]\rightarrow[d]$
be the function such that $d(i)=i+d-c$, namely the inert map such that
$d_{c,d}^0(c)=d$.
We define the map
$F(f)\colon([a],[n-b])\rightarrow([a'],[m-b'])$ by
$(\psi_{\mr{a}},d^0_{n-\widetilde{b},n-b}\circ\psi_{\mr{b}})$.

Now, assume we are given a functor
$\mbf{\Delta}^{\mr{op}}\times\mbf{\Delta}^{\mr{op}}\rightarrow\mc{C}$.
Assume that $\mc{C}$ admits finite limits.
Invoking \cite[4.3.3.7]{HTT}, we have
the right Kan extension functor
$p_{\mc{T},*}\colon\mr{Fun}(\mc{T},\mc{C})
\rightarrow
\mr{Fun}(\mbf{\Delta}^{\mr{op}},\mc{C})$, which is a right adjoint to
the restriction functor
$p_{\mc{T}}^*\colon\mr{Fun}(\mbf{\Delta}^{\mr{op}},\mc{C})\rightarrow
\mr{Fun}(\mc{T},\mc{C})$.
Note that for $f\colon\mc{T}\rightarrow\mc{C}$,
$p_{\mc{T},*}(f)([n])\simeq\invlim_{\mc{T}_{[n]}}(f)$ by
\cite[4.3.1.9]{HTT} since $p_{\mc{T}}$ is a Cartesian fibration.
Using this functor, we define
\begin{equation*}
 L\colon
 \mr{Fun}(\mbf{\Delta}^{\mr{op}}\times\mbf{\Delta}^{\mr{op}},\mc{C})
  \xrightarrow{\circ F}
  \mr{Fun}(\mc{T},\mc{C})
  \xrightarrow{p_{\mc{T},*}}
  \mr{Fun}(\mbf{\Delta}^{\mr{op}},\mc{C}).
\end{equation*}

Finally, let us construct the functor by the composition
\begin{align}
 \label{defintfunc}
 \mr{Int}\colon
 \mr{Fun}(\mbf{\Delta}^{\mr{op}},\Cat_\infty)
 &\xleftarrow{\sim}
 \mr{Fun}(\mbf{\Delta}^{\mr{op}},\mc{CSS})
 \rightarrow
 \mr{Fun}(\mbf{\Delta}^{\mr{op}},
 \mr{Fun}(\mbf{\Delta}^{\mr{op}},\Spc))\\
 \notag
 &\simeq
 \mr{Fun}(\mbf{\Delta}^{\mr{op}}\times\mbf{\Delta}^{\mr{op}},\Spc)
 \xrightarrow{L}
 \mr{Fun}(\mbf{\Delta}^{\mr{op}},\Spc)
 \xrightarrow{\mr{JT}}
 \Cat_\infty.
\end{align}
Here $\mr{JT}$ is the localization functor in \ref{inftwocatintro}.
We wish to show that this functor satisfies the conditions of
Proposition \ref{propinfcatext}.

\subsection{}
\label{defofsqaureet}
An object $i\rightarrow j$ of $\mr{Ar}(\Delta^n)$ is denoted by $(i;j)$.
For example the simplicial set $\mr{Ar}(\Delta^2)$ can be depicted as
\begin{equation*}
 \xymatrix{
  (0;0)\ar[d]&&\\
 (0;1)\ar[d]\ar[r]&(1;1)\ar[d]&\\
 (0;2)\ar[r]&(1;2)\ar[r]&(2;2).
  }
\end{equation*}
A functor $\sigma\colon\Delta^k\times\Delta^l\rightarrow
\mr{Ar}(\Delta^n)$ is said to be a {\em square} if
$\sigma(i,j)=(a+i;b+j)$ for some
integers $a$, $b$. {\em Small squares} are the squares such that
$k,l\leq 1$.
The map $\sigma$ is a monomorphism of simplicial sets, so squares can be
viewed as simplicial subsets of $\mr{Ar}(\Delta^n)$.
Let $\phi\colon\Delta^k\rightarrow\mr{Ar}(\Delta^n)$ be a
$k$-simplex. The ordered pair $(\phi(0),\phi(k))$ of vertices of
$\mr{Ar}(\Delta^n)$ is called the {\em terminal pair} of $\phi$.
Let $\sigma$ be a square. A simplicial subset $X$ of $\mr{Ar}(\Delta^n)$
is said to be {\em saturated} if a simplex $\phi$ with terminal pair $T$
is contained in $X$, then any simplex with terminal pair $T$ belongs to
$X$.

\begin{ex*}
 \begin{enumerate}
  \item\label{defofsqaureet-1}
       Let $X$ be the union of all the squares in
       $\mr{Ar}(\Delta^n)$. Then $X$ is saturated. Indeed, any square is
       saturated and the union of saturated simplicial sets is
       saturated, we get the claim. Let
       $\phi\colon\Delta^k\rightarrow\mr{Ar}(\Delta^n)$ be a
       $k$-simplex, and put $\phi(0)=(i_0;j_0)$, $\phi(k)=(i_k;j_k)$.
       We can check that $\phi$ belongs to $X$ if and only if $i_k\leq
       j_0$.
       
       Let us give an alternative description of $X$ which is useful
       for us. We also denote elements of $\mr{Tw}(\Delta^n)$ by
       $(a;b)$. For $(a;b)\in\mr{Tw}(\Delta^n)$, consider the map
       $\phi\colon\Delta^a\times\Delta^{n-b}\rightarrow\mr{Ar}(\Delta^n)$
       such that $\phi(i,j)=(i,b+j)$, which is in fact a square.
       This gives us a functor
       $\Delta^{\bullet}\times\Delta^{n-\bullet}\colon
       \mr{Tw}(\Delta^n)\rightarrow\sSet$ and the morphism of functors
       $\Delta^{\bullet}\times\Delta^{n-\bullet}\rightarrow X$
       where $X$ denotes the constant functor. Thus, we get a map
       $\indlim_{\mr{Tw}(\Delta^n)}\Delta^{\bullet}\times
       \Delta^{n-\bullet}\rightarrow X$. This map is an isomorphism of
       simplicial sets.
	
  \item\label{defofsqaureet-2}
       Let $X$ be the union of the small squares in
       $\mr{Ar}(\Delta^n)$. Then $X$ is saturated. Indeed, let
       $\phi\colon\Delta^k\rightarrow\mr{Ar}(\Delta^n)$ be a
       $k$-simplex, and put $\phi(0)=(i_0;j_0)$, $\phi(k)=(i_k;j_k)$.
       Then $\phi$ belongs to $X$ if and only if $i_k\leq j_0$,
       $i_k-i_0\leq 1$, and $j_k-j_0\leq 1$.
 \end{enumerate}
\end{ex*}

\begin{lem}
 \label{innanodtwar}
 Let $X$ be a simplicial subset of $\mr{Ar}(\Delta^n)$ which is
 saturated and any small square $\sigma$ belongs to $X$.
 Then the inclusion $X\rightarrow\mr{Ar}(\Delta^n)$ is an inner
 anodyne.
\end{lem}
\begin{proof}
 Given two vertices $a=(i;j)$, $b=(i';j')$ of $\mr{Ar}(\Delta^n)$,
 the distance of these points, denoted by $\mr{Dist}(a,b)$, is defined
 to be $\left|(i'-i)+(j'-j)\right|$ if either $i\leq i'$ and $j\leq j'$
 or $i\geq i'$ and $j\geq j'$, and $\infty$ otherwise.
 If $D(a,b)=\infty$, then there is no morphism in
 $\mr{Ar}(\Delta^n)$ from $a$ to $b$ or $b$ to $a$.
 We denote the distance by $D(a,b)$.
 Given a $k$-simplex $\phi\colon\Delta^k\rightarrow\mr{Ar}(\Delta^n)$,
 the length is defined to be the distance of the terminals, namely
 $\mr{Dist}(\phi(0),\phi(k))$.
 Let $X_k$ be the simplicial subset of $\mr{Ar}(\Delta^n)$ which is the
 union of $X$ and simplices of length $\leq k$. Since any small square
 belongs to $X$, we have $X_1=X$ and
 $X_{2n}=\mr{Ar}(\Delta^n)$. It suffices to check that
 $X_{k-1}\rightarrow X_k$ is an inner anodyne. When $k=2$, this follows
 by the assumption that any small square belongs to $X$, so we may
 assume that $k>2$.
 Let $S_k$ be the union of the empty set $\emptyset$ and the set of
 pairs $(a,b)$ of objects of $\mr{Ar}(\Delta^n)$ such that $D(a,b)=k$.
 Put a total order of $S_k$ such that $\emptyset$ is the minimum.
 Take $P\in S_k$. Let $Y_P$ be the
 union of $X_{k-1}$ and simplices whose terminals are $P'\in S_k$ for
 $P'\leq P$, especially, $P_{\emptyset}=X_{k-1}$.
 Let $P^+$ be the successor of $P$. It remains to
 show that $Y_P\rightarrow Y_{P^+}$ is an inner anodyne. If there exists
 a simplex of $Y_P$ with terminal pair in $P^+$,
 then $Y_P=Y_{P^+}$ since $X$ is assumed saturated.
 Thus, we may assume that no simplex of $Y_P$ has terminal pair $P^+$.
 Let $T$ be the finite set of simplices
 $\phi\colon\Delta^k\rightarrow\mr{Ar}(\Delta^n)$ with terminals $P^+$
 such that $\phi(l)\neq\phi(l+1)$ for any $l$.
 The last condition is equivalent to $\phi$ being non-degenerate.
 For any $\phi\in T$, the simplices
 $\phi|_{\Delta^{[k]\setminus\{k\}}}$,
 $\phi|_{\Delta^{[k]\setminus\{0\}}}$ belong to $X_{k-1}$.
 Now, for any subset $I\subset[k]$, $\phi|_{\Delta^{[k]\setminus I}}$
 belongs to $Y_P$ if and only if $I$ contains either $0$ or $k$.
 Indeed, ``if'' direction is clear by induction hypothesis.
 If $I$ does not contain both $0$, $k$,
 then the terminals of $\phi|_{\Delta^{[k]\setminus I}}$ is $P^+$ and
 this is not contained in $Y_P$, thus the claim.
 Let $P^+=((i,j),(i',j'))$. Let $\phi_0$ be the unique element of $T$ such
 that $\phi(l)=(i;j+l)$ for $l\leq j'-j$.
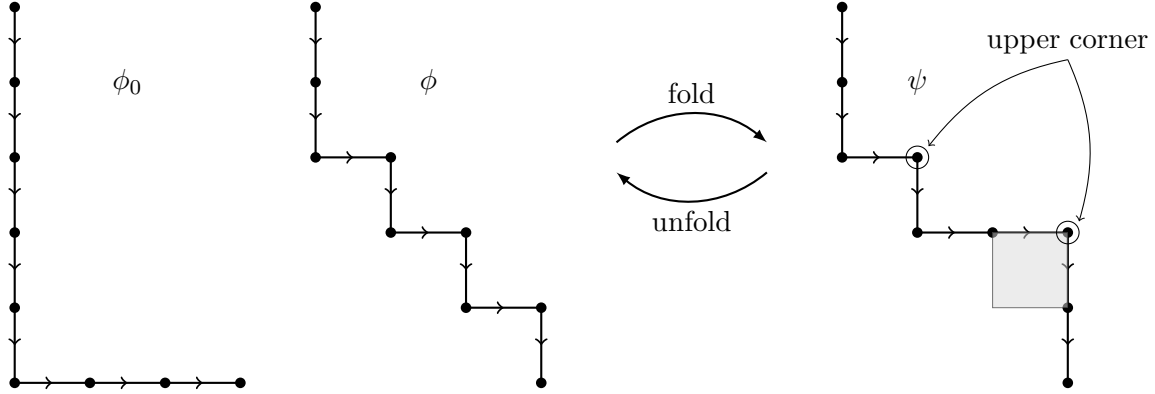
 \begin{center}
  \begin{tikzpicture}
   \fill (0,5) coordinate (O0) circle[radius=2pt];
   \fill (0,4) coordinate (O1) circle[radius=2pt];
   \fill (0,3) coordinate (O2) circle[radius=2pt];
   \fill (1,3) coordinate (O3) circle[radius=2pt];
   \fill (1,2) coordinate (O4) circle[radius=2pt];
   \fill (2,2) coordinate (O5) circle[radius=2pt];
   \fill (3,2) coordinate (O6) circle[radius=2pt];
   \fill (2,1) coordinate (O6');
   \fill (3,1) coordinate (O7) circle[radius=2pt];
   \fill (3,0) coordinate (O8) circle[radius=2pt];
   \fill (-7,5) coordinate (L0) circle[radius=2pt];
   \fill (-7,4) coordinate (L1) circle[radius=2pt];
   \fill (-7,3) coordinate (L2) circle[radius=2pt];
   \fill (-6,3) coordinate (L3) circle[radius=2pt];
   \fill (-6,2) coordinate (L4) circle[radius=2pt];
   \fill (-5,2) coordinate (L5) circle[radius=2pt];
   \fill (-5,1) coordinate (L6) circle[radius=2pt];
   \fill (-4,1) coordinate (L7) circle[radius=2pt];
   \fill (-4,0) coordinate (L8) circle[radius=2pt];
   \fill (-11,5) coordinate (M0) circle[radius=2pt];
   \fill (-11,4) coordinate (M1) circle[radius=2pt];
   \fill (-11,3) coordinate (M2) circle[radius=2pt];
   \fill (-11,2) coordinate (M3) circle[radius=2pt];
   \fill (-11,1) coordinate (M4) circle[radius=2pt];
   \fill (-11,0) coordinate (M5) circle[radius=2pt];
   \fill (-10,0) coordinate (M6) circle[radius=2pt];
   \fill (-9,0) coordinate (M7) circle[radius=2pt];
   \fill (-8,0) coordinate (M8) circle[radius=2pt];
   \foreach \a [evaluate=\a as \b using \a+1] in {0,...,7}
   {
   \draw[thick, decoration={markings,
   mark=at position 0.5 with \arrow{>}},
   postaction=decorate]
   (O\a)--(O\b);
   \draw[thick, decoration={markings,
   mark=at position 0.5 with \arrow{>}},
   postaction=decorate]
   (L\a)--(L\b);
   \draw[thick, decoration={markings,
   mark=at position 0.5 with \arrow{>}},
   postaction=decorate]
   (M\a)--(M\b);
   }
   \node at ($(M0)+(1.5,-1)$) {$\phi_0$};
   \node at ($(L0)+(1.5,-1)$) {$\phi$};
   \node at ($(O0)+(1,-1)$) {$\psi$};
   \draw [-latex,thick] (-3,3.2) to [bend left=40]
   node[midway,above] {fold} (-1,3.2);
   \draw [-latex,thick] (-1,2.8) to [bend left=40]
   node[midway,below] {unfold} (-3,2.8);
   \draw [fill=gray!30, opacity=.5]
   (O6')--(O7)--(O6)--(O5)--(O6')--cycle;
   \draw (O3) circle [radius=.15];
   \draw (O6) circle [radius=.15];
   \node (A) at ($(O1)+(3,.5)$) {upper corner};
   \draw [->] ($(A)+(0,-0.2)$) to [bend right=20]
   ($(O3)+(.15,.15)$);
   \draw [->] ($(A)+(0,-0.2)$) to [bend left=20]
   ($(O6)+(.15,.15)$);
  \end{tikzpicture}
 \end{center}
 An {\em upper corner} of $\phi\in T$ is $l\in[k]\setminus\{0,k\}$ such
 that $\phi(l-1)=(w-1;z)$, $\phi(l)=(w;z)$, and $\phi(l+1)=(w;z+1)$.
 For $\phi,\psi\in T$, we say that $\psi$ is obtained by {\em folding}
 $\phi$, or $\phi$ is obtained by {\em unfolding} $\psi$,
 if there exists $0<l<k$ such that $\phi(a)=\psi(a)$ for any
 $a\neq l$ and $\psi(l)=(w+1;z-1)$ if $\phi(l)=(w;z)$.
 Note that the $l$-vertex is an upper corner of $\psi$.
 We also note that the set of upper corners $U_\phi$ of $\phi\in T$
 completely determines $\phi$.
 Considering the number of foldings from $\phi_0$, we can put a total
 ordering on $T$ so that if $\psi$ is obtained by folding
 $\phi$ then $\phi<\psi$. Then $\phi_0$ is the minimum element in $T$.
 For $\phi\in T$, let $Z_\phi$ be the union of $Y_P$ and $\psi$ for
 $\psi\leq\phi$. Let $\phi'$ is the successor of $\phi$. It
 suffices to show that $Y_P\rightarrow Z_{\phi_0}$ and
 $Z_\phi\rightarrow Z_{\phi'}$ are inner anodynes. Since it is similar,
 we only check the latter case. Put
 $\Lambda^{U}:=\bigcup_{l\in U_{\phi'}\cup\{0,k\}}
 \Delta^{[k]\setminus\{l\}}$. We have the map
 $\Lambda^U\rightarrow\Delta^k\xrightarrow{\phi'}\mr{Ar}(\Delta^n)$.
 Since $k>2$,
 the inclusion $\Lambda^U\rightarrow\Delta^k$ is an inner anodyne by
 \cite[2.12 (iv)]{J}.
 Thus, it remains to show that
 $Z_{\phi'}=Z_\phi\sqcup_{\Lambda^U}\Delta^k$.
 For $l\in U_{\phi'}$, it is clear that
 $\phi'|_{\Delta^{[k]\setminus\{l\}}}$ is
 in $Z_{\phi}$ because the simplicial set $\psi\in T$ obtained from
 $\phi'$ by unfolding the corner $l$ satisfies $\psi\leq\phi$ by the
 choice of the ordering of $T$ and $\phi'|_{\Delta^{[k]\setminus\{l\}}}$
 is also a simplex of $\psi$. Let $\sigma$ be a simplicial subset of
 $\phi'$ which contains $U_{\phi'}\cup\{0,k\}$ as vertices. Assume
 $\sigma$ belongs to $Z_\phi$. Then there would exist $\psi\leq\phi$
 such that $\sigma$ is a simplex of $\psi$. However, since $\phi'$ is
 the minimum element in $T$ which contains all the vertices in
 $U_{\phi'}$, we should have $\psi\geq\phi'$. This is a contradiction,
 and we have $\sigma\not\in Z_\phi$, which completes the proof.
\end{proof}

\begin{cor}
 \label{compofcolcat}
 Let $\Delta^{\bullet,\bullet}\colon
 \mbf{\Delta}\times\mbf{\Delta}\rightarrow\Cat_\infty$ be the
 functor sending $([m],[n])$ to $\Delta^m\times\Delta^n$.
 Let $\mr{Ar}(\Delta^{\bullet})\colon\mbf{\Delta}
 \rightarrow\Cat_\infty$ be the functor
 sending $[n]\in\mbf{\Delta}$ to $\mr{Ar}(\Delta^n)$. Then we have a
 canonical equivalence of functors
 $L^{\mr{op}}(\Delta^{\bullet,\bullet})\simeq
 \mr{Ar}(\Delta^{\bullet})$.
\end{cor}
\begin{proof}
 First, let us construct a functor
 $L^{\mr{op}}(\Delta^{\bullet,\bullet})\rightarrow
 \mr{Ar}(\Delta^{\bullet})$.
 Let $A$ be the composition
 $\mc{T}^{\mr{op}}\xrightarrow{F^{\mr{op}}}
 \mbf{\Delta}\times\mbf{\Delta}
 \xrightarrow{\Delta^{\bullet,\bullet}}\Cat_\infty$ and
 $B$ be the composition
 $\mc{T}^{\mr{op}}\rightarrow\mbf{\Delta}
 \xrightarrow{\mr{Ar}}\Cat_\infty$. We must define a morphism
 $\theta\colon A\rightarrow B$ of functors because we have the
 adjunction $(p^{\mr{op}}_{\mc{T},!},p^{\mr{op},*}_{\mc{T}})$.
 Let $(m,a'b')\rightarrow(n,ab)$ be a map in $\mc{T}^{\mr{op}}$
 corresponding to a map $f$ in $\mc{T}^{\mr{op}}$ using the notation of
 \ref{introgamm}. Let
 $\theta(n,ab)\colon\Delta^{a}\times\Delta^{n-b}\rightarrow
 \mr{Ar}(\Delta^n)$ be the functor sending $(i,j)$ to $(i;j+b)$.
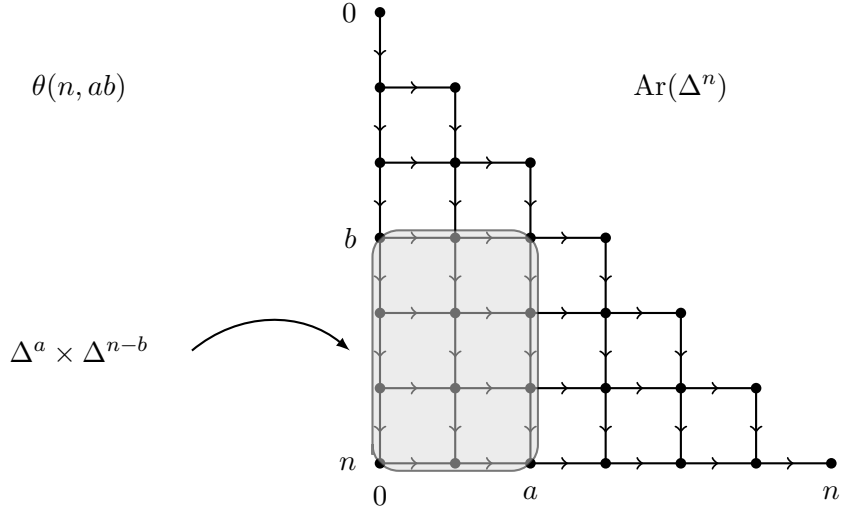
 \begin{center}
 \begin{tikzpicture}
  \foreach \x [evaluate=\x as \m using 6-\x] in {0,...,6}
  {
  \foreach \y in {0,...,\m}
  \coordinate (O\x\y) at (\x,\y);
  }
  \foreach \x [evaluate=\x as \m using 6-\x] in {0,...,6}
  {
  \foreach \y in {0,...,\m}
  \fill (O\x\y) circle [radius=2pt];
  }
  \foreach \x [evaluate=\x as \m using 5-\x] in {0,...,5}
  {
  \foreach \a [evaluate=\a as \b using \a+1] in {0,...,\m}
  \draw [thick, decoration={markings,
  mark=at position 0.5 with \arrow{<}},
  postaction=decorate] (O\x\a)--(O\x\b);
  }
  \foreach \y [evaluate=\y as \m using 5-\y] in {0,...,5}
  {
  \foreach \x [evaluate=\x as \w using (1+\x)*10+\y] in {0,...,\m}
  \draw [thick, decoration={markings,
  mark=at position 0.5 with \arrow{>}},
  postaction=decorate] (O\x\y)--(O\w);
  };
  \node [left=5] at (O03) {$b$};
  \node [left=5] at (O00) {$n$};
  \node [left=5] at (O06) {$0$};
  \node [below=5] at (O00) {$0$};
  \node [below=5] at (O20) {$a$};
  \node [below=5] at (O60) {$n$};
  \draw [thick, rounded corners=10pt,fill=gray!30, opacity=.5]
  (-0.1,-0.1)--(2.1,-0.1)--(2.1,3.1)--(-0.1,3.1)--(-0.1,-0.1)--cycle;
  \node at (4,5) {$\mr{Ar}(\Delta^n)$};
  \node at (-4,1.5) {$\Delta^a\times\Delta^{n-b}$};
  \node at (-4,5) {$\theta(n,ab)$};
  \draw [-latex,thick] (-2.5,1.5) to [bend left=40] (-0.4,1.5);
 \end{tikzpicture} 
\end{center}
 Note that since $\mr{Hom}(x,y)$ is either
 a singleton or an empty set for any $x,y\in\mr{Ar}(\Delta^n)$,
 $\theta(n,ab)$ is determined uniquely.
 The following diagram commutes, thus we have the morphism $\theta$:
 \begin{equation*}
  \xymatrix@C=50pt{
   \Delta^{a'}\times\Delta^{m-b'}
   \ar[r]^-{\psi_{\mr{a}}\times\psi_{\mr{b}}}
   \ar[d]_{\theta(m,a'b')}&
   \Delta^a\times\Delta^{n-b}\ar[d]^{\theta(n,ab)}\\
  \mr{Ar}(\Delta^m)\ar[r]^-{\phi_f}&
   \mr{Ar}(\Delta^n).
   }
 \end{equation*}
 Now, let us show that this functor is an equivalence. It suffices to
 check this for each $[n]\in\mbf{\Delta}$. Then
 $L^{\mr{op}}(\Delta^{\bullet,\bullet})([n])\simeq
 \indlim_{(a;b)\in\mr{Tw}\Delta^n}\Delta^a\times\Delta^{n-b}$ where the
 colimit is taken in $\Cat_\infty$.
 Endow $\mr{Tw}\Delta^n$ with the Reedy category structure by declaring
 $\deg(ab)=a+b$, the direct subcategory as the one spanned by maps of
 the form $(a;b)\rightarrow(a';b)$ for $a\leq b$,
 the inverse subcategory as the one spanned by maps of the form
 $(a;b)\rightarrow(a;b')$ for $b\geq b'$ (cf.\ \cite[15.1.2]{Hir}).
 Then one can check that for any functor of (ordinary) categories
 $F\colon\mr{Tw}\Delta^n\rightarrow\mc{C}$, the latching object
 $\mr{L}_{(a;b)}F\simeq F(a-1;b)$. Thus, the map
 $\mr{L}_{(a;b)}\Delta^{\bullet,n-\bullet}\rightarrow\Delta^{a,n-b}$ is
 a cofibration, and the functor
 $\Delta^{\bullet,n-\bullet}\colon\mr{Tw}\Delta^n\rightarrow
 \sSet^+$ is a Reedy cofibrant diagram.
 Moreover, by \cite[15.10.2]{Hir}, it has fibrant constant.
 Thus, by \cite[19.9.1]{Hir}, the (ordinary) colimit of
 $\Delta^{\bullet,n-\bullet}$ as a simplicial set is
 a homotopy colimit by \cite[4.2.4.1]{HTT}.
 Now the desired claim follows from Lemma \ref{innanodtwar} as well as
 Example \ref{defofsqaureet}.\ref{defofsqaureet-1}.
\end{proof}

\begin{cor}
 \label{corfactsystf}
 Let $\mc{C}$ be an $\infty$-category, and $(S_L,S_R)$ be a factorization system
 {\normalfont(}cf.\ {\normalfont\cite[5.2.8.8]{HTT})}.
 Let $\mr{Fun}(\mr{Ar}(\Delta^n),\mc{C})^{\prime}$
 be the full subcategory of $\mr{Fun}(\mr{Ar}(\Delta^n),\mc{C})$
 spanned by the functors $F$ such that $F((i;j)\rightarrow(i;j+1))$ is
 in $S_L$ and $F((i;j)\rightarrow(i+1;j))$ is in $S_R$. 
 We have the functor $\Delta^n\rightarrow\mr{Ar}(\Delta^n)$ sending
 $i$ to $(i;i)$. The induced functor
 $\mr{Fun}(\mr{Ar}(\Delta^n),\mc{C})^{\prime}\rightarrow
 \mr{Fun}(\Delta^n,\mc{C})$ is a trivial fibration.
\end{cor}
\begin{proof}
 Let $X$ be the union of small squares in $\mr{Ar}(\Delta^n)$.
 Then the inclusion $X\rightarrow\mr{Ar}(\Delta^n)$ is an inner anodyne
 by Lemma \ref{innanodtwar} and
 \ref{defofsqaureet}.\ref{defofsqaureet-2}. We can similarly
 define $\mr{Fun}(X,\mc{C})^{\prime}$ in $\mr{Fun(X,\mc{C})}$ similarly
 to $\mr{Fun}(\mr{Ar}(\Delta^n),\mc{C})^{\prime}$.
 Let us check that the functors
 $\mr{Fun}(\mr{Ar}(\Delta^n),\mc{C})^{\prime}\xrightarrow{\alpha}
 \mr{Fun}(X,\mc{C})^{\prime}\xrightarrow{\beta}
 \mr{Fun}(\Delta^n,\mc{C})$
 are trivial fibrations. The functor $\alpha$ is the pullback of the map
 $\mr{Fun}(\mr{Ar}(\Delta^n),\mc{C})\rightarrow\mr{Fun}(X,\mc{C})$.
 This functor is a trivial fibration by \cite[2.3.2.5]{HTT}.
 It remains to show that $\beta$ is a trivial fibration.

 For a vertex $v=(a;b)\in\mr{Ar}(\Delta^n)$, let $D_v$ be the
 $2$-simplex of $\mr{Ar}(\Delta^n)$ such such that the $0$-vertices are
 $(a;b)$, $(a;b+1)$, $(a+1;b+1)$. If some of the vertices are not
 defined, we put $D_v=\emptyset$. Similarly, let $U_v$ be the
 $2$-simplex such that the $0$-vertices are $(a;b)$, $(a+1;b)$,
 $(a+1;b+1)$, and $\emptyset$ if some of the vertices are not defined.
 We put $Y_0:=\Delta^n\hookrightarrow\mr{Ar}(\Delta^n)$, and for $i>0$,
 we define $X_i:=Y_{i-1}\cup\bigcup_{v\in Y_{i-1}}D_v$,
 $Y_i:=X_i\cup\bigcup_{v\in X_i}U_v$ inductively.
 
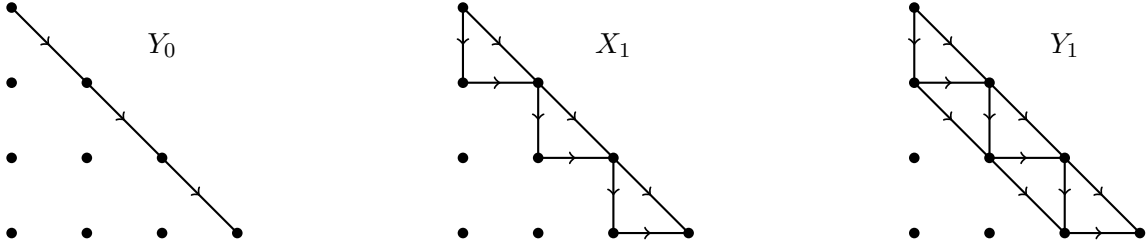
\begin{center}
  \begin{tikzpicture}
   \foreach \x [evaluate=\x as \m using 3-\x] in {0,...,3}
   {
   \foreach \y in {0,...,\m}
   \coordinate (O\x\y) at (\x,\y);
   }
   \foreach \x [evaluate=\x as \m using 3-\x] in {0,...,3}
   {
   \foreach \y in {0,...,\m}
   \fill (O\x\y) circle [radius=2pt];
   }
   \foreach \x [evaluate=\x as \m using 3-\x] in {0,...,3}
   {
   \foreach \y in {0,...,\m}
   \coordinate (P\x\y) at ($(\x,\y)+(6,0)$);
   }
   \foreach \x [evaluate=\x as \m using 3-\x] in {0,...,3}
   {
   \foreach \y in {0,...,\m}
   \fill (P\x\y) circle [radius=2pt];
   }
   \foreach \x [evaluate=\x as \m using 3-\x] in {0,...,3}
   {
   \foreach \y in {0,...,\m}
   \coordinate (Q\x\y) at ($(\x,\y)+(12,0)$);
   }
   \foreach \x [evaluate=\x as \m using 3-\x] in {0,...,3}
   {
   \foreach \y in {0,...,\m}
   \fill (Q\x\y) circle [radius=2pt];
   }
   \draw [thick, decoration={markings,mark=at position 0.5 with \arrow{>}},postaction=decorate] (O03)--(O12);
   \draw [thick, decoration={markings,mark=at position 0.5 with \arrow{>}},postaction=decorate] (O12)--(O21);
   \draw [thick, decoration={markings,mark=at position 0.5 with \arrow{>}},postaction=decorate] (O21)--(O30);
   \draw [thick, decoration={markings,mark=at position 0.5 with \arrow{>}},postaction=decorate] (P03)--(P12);
   \draw [thick, decoration={markings,mark=at position 0.5 with \arrow{>}},postaction=decorate] (P12)--(P21);
   \draw [thick, decoration={markings,mark=at position 0.5 with \arrow{>}},postaction=decorate] (P21)--(P30);
   \draw [thick, decoration={markings,mark=at position 0.5 with \arrow{>}},postaction=decorate] (P03)--(P02);
   \draw [thick, decoration={markings,mark=at position 0.5 with \arrow{>}},postaction=decorate] (P02)--(P12);
   \draw [thick, decoration={markings,mark=at position 0.5 with \arrow{>}},postaction=decorate] (P12)--(P11);
   \draw [thick, decoration={markings,mark=at position 0.5 with \arrow{>}},postaction=decorate] (P11)--(P21);
   \draw [thick, decoration={markings,mark=at position 0.5 with \arrow{>}},postaction=decorate] (P21)--(P20);
   \draw [thick, decoration={markings,mark=at position 0.5 with \arrow{>}},postaction=decorate] (P20)--(P30);
   \draw [thick, decoration={markings,mark=at position 0.5 with \arrow{>}},postaction=decorate] (Q03)--(Q12);
   \draw [thick, decoration={markings,mark=at position 0.5 with \arrow{>}},postaction=decorate] (Q12)--(Q21);
   \draw [thick, decoration={markings,mark=at position 0.5 with \arrow{>}},postaction=decorate] (Q21)--(Q30);
   \draw [thick, decoration={markings,mark=at position 0.5 with \arrow{>}},postaction=decorate] (Q03)--(Q02);
   \draw [thick, decoration={markings,mark=at position 0.5 with \arrow{>}},postaction=decorate] (Q02)--(Q12);
   \draw [thick, decoration={markings,mark=at position 0.5 with \arrow{>}},postaction=decorate] (Q12)--(Q11);
   \draw [thick, decoration={markings,mark=at position 0.5 with \arrow{>}},postaction=decorate] (Q11)--(Q21);
   \draw [thick, decoration={markings,mark=at position 0.5 with \arrow{>}},postaction=decorate] (Q21)--(Q20);
   \draw [thick, decoration={markings,mark=at position 0.5 with \arrow{>}},postaction=decorate] (Q20)--(Q30);
   \draw [thick, decoration={markings,mark=at position 0.5 with \arrow{>}},postaction=decorate] (Q02)--(Q11);
   \draw [thick, decoration={markings,mark=at position 0.5 with \arrow{>}},postaction=decorate] (Q11)--(Q20);
   \node at ($(O03)+(2,-0.5)$) {$Y_0$};
   \node at ($(P03)+(2,-0.5)$) {$X_1$};
   \node at ($(Q03)+(2,-0.5)$) {$Y_1$};
  \end{tikzpicture} 
 \end{center}
 We have $X_n=X$. We define $\mr{Fun}(X_i,\mc{C})^{\prime}$ and
 $\mr{Fun}(Y_i,\mc{C})^{\prime}$ likewise.
 Since the inclusion $X_i\rightarrow Y_i$ is an inner anodyne, the map
 $\mr{Fun}(Y_i,\mc{C})^{\prime}\rightarrow\mr{Fun}(X_i,\mc{C})^{\prime}$
 is a trivial fibration. The map
 $\mr{Fun}(X_{i+1},\mc{C})^{\prime}\rightarrow\mr{Fun}(Y_i,\mc{C})^{\prime}$
 is a trivial fibration by \cite[5.2.8.17]{HTT}. Thus, $\beta$ is a
 trivial fibration as required.
\end{proof}

\begin{proof}[Proof of Proposition \ref{propinfcatext}]
 We have already constructed the functor $\mr{Int}$.
 We will show that the functor satisfies the required properties.
 Let $\mc{C}_{n,m}:=\mr{Fun}(\Delta^n\times\Delta^m,\mc{C})^{\simeq}$.
 Then we have the functor
 $\mc{C}_{\bullet,\bullet}\colon\mbf{\Delta}^{\mr{op}}\times
 \mbf{\Delta}^{\mr{op}}\rightarrow\Spc$.
 Let $G\colon\mr{Fun}(\mbf{\Delta}^{\mr{op}},\Cat_\infty)\rightarrow
 \mr{Fun}(\mbf{\Delta}^{\mr{op}}\times\mbf{\Delta}^{\mr{op}},\Spc)$
 be the composition of functors in (\ref{defintfunc}).
 By construction, $G(M_\bullet\mc{C})\simeq\mc{C}_{\bullet,\bullet}$.
 In the situation of \ref{propinfcatext}.\ref{propinfcatext-2},
 we denote by $\widetilde{\mc{D}}_{n,m}$ be the full subcategory (thus
 space) of $\mc{D}_{n,m}$ spanned by of maps
 $\Delta^n\times\Delta^m\rightarrow\mc{D}$ such that for each vertex
 $i\in\Delta^n$, the edges $\{i\}\times\Delta^{\{j,j+1\}}$ are
 $p$-equivalent and for each $j\in\Delta^m$ the edges
 $\Delta^{\{i,i+1\}}\times\{j\}$ are $p$-Cartesian.

 We have the map $\Delta^n\rightarrow\mr{Ar}(\Delta^n)$ sending $i$ to
 $(i;i)$. This induces the functor of cosimplicial objects
 $\iota\colon\Delta^{\bullet}\rightarrow\mr{Ar}(\Delta^{\bullet})$.
 For an $\infty$-category $\mc{C}$, let $\mc{A}(\mc{C})$ be the
 simplicial spaces
 $\mr{Fun}(\mr{Ar}(\Delta^{\bullet}),\mc{C})^{\simeq}$,
 namely $\mr{Fun}(\mr{Ar}(\Delta^n),\mc{C})^{\simeq}$ is assigned to
 $[n]$. The functor $\iota$ induces the functor
 $\mc{A}(\mc{C})\rightarrow\mr{Seq}_{\bullet}(\mc{C})$.
 Moreover, we have
 \begin{align*}
  \mc{A}(\mc{C})
  &:=
  \mr{Fun}(\mr{Ar}(\Delta^{\bullet}),\mc{C})^{\simeq}
  \simeq
  \mr{Fun}(L^{\mr{op}}(\Delta^{\bullet,\bullet}),\mc{C})^{\simeq}
  \simeq
  \bigl(L(\mr{Fun}(\Delta^{\bullet,\bullet},\mc{C}))\bigr)^{\simeq}\\
  &\simeq
  L(\mr{Fun}(\Delta^{\bullet,\bullet},\mc{C})^{\simeq})
  =:L(\mc{C}_{\bullet,\bullet}).
 \end{align*}
 where the 1st equivalence follows by Corollary \ref{compofcolcat}, and
 the 3rd equivalence by Lemma \ref{propspccatinf}.
 For the 2nd equivalence, the functor can be constructed using the
 adjointness of $(p_{\mc{T},!},p_{\mc{T}}^*)$ and
 $(p_{\mc{T}}^*,p_{\mc{T},*})$. Then the equivalence is reduced to the
 equivalence for each term. In this situation, $p_{\mc{T},!}$ and
 $p_{\mc{T},*}$ can be computed by colimits and limits.
 Thus, we have the functor
 \begin{equation*}
  \mr{Int}(M_\bullet\mc{C}):=\mr{JT}(LG(M_\bullet\mc{C}))
   \simeq
   \mr{JT}(\mc{A}(\mc{C}))
   \rightarrow
   \mr{JT}(\mr{Seq}_\bullet\mc{C})\xrightarrow{\sim}\mc{C}.
 \end{equation*}
 This is the required functor of
 \ref{propinfcatext}.\ref{propinfcatext-1}.
 
 Now, let $\widetilde{\mc{A}}(p)$ be the simplicial subspaces of
 $\mc{A}(\mc{D})$ such that for each $[n]$, we consider the
 subspace spanned by the functors
 $\phi\colon\mr{Ar}(\Delta^n)\rightarrow\mc{D}$ such that
 ``vertical edges'' $\phi((i;j)\rightarrow(i;j+1))$ are $p$-equivalent and
 ``horizontal edges'' $\phi((i;j)\rightarrow(i+1;j))$ are
 $p$-Cartesian.
 Let $\mr{Fun}(\Delta^1,\mc{D})^{\mr{cart}}$ be the full
 subcategory of $\mr{Fun}(\Delta^1,\mc{D})$ spanned by $p$-Cartesian
 edges. Then the inclusion
 $\mr{Fun}(\Delta^1,\mc{D})^{\mr{cart}}\rightarrow
 \mr{Fun}(\Delta^1,\mc{D})$ is a categorical fibration by
 \cite[2.4.6.5]{HTT}. Similarly, we define
 $\mr{Fun}(\Delta^1,\mc{D})^{\mr{cons}}$ to be the full subcategory
 spanned by $p$-equivalent edges. Then we have
 \begin{equation*}
  \widetilde{\mc{A}}(p)\simeq
   \mc{A}(\mc{D})
   \times^{\mr{cat}}_{\prod_{E_\mr{h}}\mr{Fun}(\Delta^1,\mc{D})}
   \prod_{E_\mr{h}}\mr{Fun}(\Delta^1,\mc{D})^{\mr{cart}}
   \times^{\mr{cat}}_{\prod_{E_\mr{v}}\mr{Fun}(\Delta^1,\mc{D})}
   \prod_{E_\mr{v}}\mr{Fun}(\Delta^1,\mc{D})^{\mr{cons}},
 \end{equation*}
 where $E_{\mr{h}}$ (resp.\ $E_{\mr{v}}$) is the set of horizontal
 (res.\ vertical) edges in $\mr{Ar}(\Delta^n)$.
 Thus, may write $\widetilde{\mc{A}}(p)$ using limits.
 We also have similar presentation for
 $\widetilde{\mc{D}}_{\bullet,\bullet}$ using limits.
 Thus, we have
 $L(\widetilde{\mc{D}}_{\bullet,\bullet})\simeq\widetilde{\mc{A}}(p)$. It
 remains to check that the composition
 \begin{equation*}
  \widetilde{\mc{A}}(p)([n])
   \rightarrow
  \mr{Fun}(\mr{Ar}(\Delta^n),\mc{D})^{\simeq}
   \rightarrow
   \mr{Fun}(\Delta^n,\mc{D})^{\simeq}
 \end{equation*}
 is a homotopy equivalence of spaces. This follows from Corollary
 \ref{corfactsystf}.

 For $\mc{C}_\bullet$ in $\mr{Fun}(\mbf{\Delta}^{\bullet},\Cat_\infty)$,
 unwinding the definition, the space $LG(\mc{C}_\bullet)([0])$ is
 equivalent to $\mc{C}_0^{\simeq}$ and
 $LG(\mc{C}_\bullet)([1])$ is equivalent to
 $\mr{Fun}(\Delta^1,\mc{C}_0)^{\simeq}\times^{\mr{cat}}
 _{\{1\},\mc{C}_0^{\simeq},s}\mc{C}_1^{\simeq}$.
 Now, the adjunction map $LG(\mc{C}_\bullet)\rightarrow
 \mr{Seq}_\bullet\mr{JT}(LG(\mc{C}_\bullet))=:
 \mr{Seq}_\bullet\mr{Int}(\mc{C}_\bullet)$
 induces the desired maps of
 \ref{propinfcatext}.\ref{propinfcatext-3}.
\end{proof}

\begin{rem*}
 Let $\mc{C}$ be an $\infty$-category. The proof shows, in fact,
 $\mr{Fun}(\Delta^n,\mr{Int}(M_\bullet\mc{C}))^{\simeq}\simeq
 \mr{Fun}(\mr{Ar}(\Delta^n),\mc{C})^{\simeq}$.
\end{rem*}

\subsection{}
For the future use, let us make yet another analogous construction.
Under the setting of \ref{functorseup}, we define the category $\mr{Ar}^{\mr{prop}}_{\mr{sep}}(\mr{Sch})$
similarly to $\widetilde{\mr{Ar}}^{\mr{prop}}_{\mr{sep}}(\mr{Sch})$ as follows.
The objects are the same as $\widetilde{\mr{Ar}}^{\mr{prop}}_{\mr{sep}}(\mr{Sch})$.
A morphism $(X_1\rightarrow Y_1)\rightarrow (X_0\rightarrow Y_1)$ is a diagram
\begin{equation}
 \label{diagmortilar2}
 \xymatrix{
  X_1\ar[d]_{f_1}\ar[r]^-{\alpha}&
  W_{10}\ar[d]_{\widetilde{f}_0}\ar[r]^-{\widetilde{g}}\ar@{}[rd]|\square&
  X_0\ar[d]^{f_0}\\
 Y_1&Y_1\ar[r]^-{g}\ar@{=}[l]&
  Y_0,}
\end{equation}
where $\alpha\in\mr{prop}$, and we identify diagrams by the obvious equivalence relation similar to that of
$\widetilde{\mr{Ar}}^{\mr{prop}}_{\mr{sep}}(\mr{Sch})$.
In other words, this is the Grothendieck fibration over $\mr{Sch}$ associated to the pseudo-functor $\mr{Sch}^{\mr{op}}\rightarrow\Cat_1$
sending $Y$ to $\mr{Sch}_{/Y}^{\mr{prop}}$.
On the other hand, let $\mr{Un}(\mc{D})\rightarrow\mr{Sch}^{\mr{op}}$ be the coCartesian fibration classified by the functor
$\mr{Sch}^{\mr{op}}\xrightarrow{\mc{D}}\LinCat_{\mc{A}}\rightarrow\widehat{\Cat}_\infty$,
where the second functor is the one taking the underlying $\infty$-category.

\begin{thm*}
 Under the setting of {\normalfont\ref{functorseup}}, there exists a functor
 $\H_{\mr{c}}^*\colon\mr{Ar}^{\mr{prop}}_{\mr{sep}}(\mr{Sch})^{\mr{op}}\rightarrow\mr{Un}(\mc{D})$ of coCartesian fibrations over
 $\mr{Sch}^{\mr{op}}$ such that for $f\colon X\rightarrow Y\in\mr{sep}$,
 the object $\H_{\mr{c}}^*(f)\in\mr{Un}(\mc{D})$ is equivalent to $f_!(I_X)$.
 Assume we are given a morphism $m\colon f_1\rightarrow f_0$ in $\mr{Ar}^{\mr{prop}}_{\mr{sep}}(\mr{Sch})$
 given by the diagram {\normalfont(\ref{diagmortilar2})}.
 Then $\H_{\mr{c}}^*(m)$ is equivalent to the composition of the following morphisms
 \begin{equation*}
  g^*f_{0!}(I_{X_0})
   \simeq
   \widetilde{f}_{0!}\widetilde{g}^*(I_{X_0})
   \rightarrow
   \widetilde{f}_{0!}\alpha_!\alpha^*\widetilde{g}^*(I_{X_0})
   \simeq
   f_{1!}(I_{X_1}).
 \end{equation*}
 The functor $\H_{\mr{c}}^*$ preserves coCartesian edges.
\end{thm*}
\begin{proof}
 In this proof, $\widehat{\Cat}_\infty$ is abbreviated as $\Cat_\infty$, and similarly for $\coCart_\infty$.
 Consider the functor
 \begin{equation*}
  \mr{Fun}(\mc{E},\coCart_{\infty})
   \rightarrow
   \mr{Fun}(\mc{E},\mr{Fun}(\Delta^1,\Cat_\infty))
   \cong
   \mr{Fun}(\Delta^1,\mr{Fun}(\mc{E},\Cat_\infty))
   \xrightarrow[\sim]{\mr{Un}}
   \mr{Fun}(\Delta^1,\Cart^{\mr{str}}(\mc{E}^{\mr{op}})).
 \end{equation*}
 Assume we are given a functor $F\colon\mc{E}\rightarrow\coCart_{\infty}$.
 Via this functor, this functor corresponds to a diagram of Cartesian fibrations over $\mc{E}^{\mr{op}}$
 \begin{equation*}
  \xymatrix@R=10pt{
   \mc{C}\ar[rd]\ar[rr]^-{G}&&\mc{D}\ar[ld]\\
  &\mc{E}^{\mr{op}}.&
   }
 \end{equation*}
 We may take $G$ to be a categorical fibration.
 For a coCartesian fibration $\mc{A}\rightarrow\mbf{\Delta}^{\mr{op}}$,
 we put $\mc{A}^{\sharp}:=(\bp\mc{A})^{\mr{op}}\rightarrow\mbf{\Delta}^{\mr{op}}$,
 the opposite of the dual Cartesian fibration, which is a coCartesian fibration.
 By definition, the $(\infty,2)$-functor $\mbf{Corr}\rightarrow\mbf{Cat}_\infty^{2-\mr{op}}$ corresponds to a functor between
 Cartesian fibrations $\bp\mr{Corr}^{\circledast}\rightarrow\coCart_\infty\times_{\Cat_\infty}\mbf{\Delta}$
 over $\mbf{\Delta}$, where $\mbf{\Delta}\rightarrow\Cat_\infty$ is the functor sending $[n]$ to $\Delta^n$. 
 This functor preserves Cartesian edges.
 Composing with the projection, we have the induced functor $\bp\mr{Corr}^{\circledast}\rightarrow\coCart_\infty$.
 By the observation above, we have the following diagram of {\em Cartesian} fibrations over $(\bp\mr{Corr}^{\circledast})^{\mr{op}}$
 \begin{equation*}
  \xymatrix@R=10pt{
   \mc{C}\ar[rd]\ar[rr]^-{G}&&\mr{Corr}^{\circledast,\sharp}\times_{\mbf{\Delta}^{\mr{op}}}\Gamma\ar[ld]\\
  &\mr{Corr}^{\circledast,\sharp},&
   }
 \end{equation*}
 where $\Gamma$ is the simplicial set defined in \ref{gammaanddual}.
 Note that $G$ may not be a coCartesian fibration, but for each object $X$ of $\mr{Corr}^{\circledast,\sharp}$,
 the fiber $G_X$ is a coCartesian fibration by construction.

 We denote $\mr{Ar}^{\mr{prop}}_{\mr{sep}}(\mr{Sch})$ by $\mr{Ar}$ in the following.
 For a {\em coCartesian} fibration $q\colon\mc{D}\rightarrow\mc{C}$,
 we denote $(\widetilde{M}_{\bullet}p^{\mr{op}})^{\mr{op}}$, using the construction in \ref{integlconst},
 by $\widetilde{M}q$ or $\widetilde{M}\mc{D}$ by abusing notations for simplicity.
 The coCartesian fibration $\mr{Ar}^{\mr{op}}\rightarrow\mr{Sch}^{\mr{op}}$, yields $\widetilde{M}_{\bullet}\mr{Ar}^{\mr{op}}$.
 The coCartesian fibration over $\mbf{\Delta}^{\mr{op}}$ obtained by
 unstraightening this is denoted by $\widetilde{M}\mr{Ar}^{\mr{op},\circledast}$.
 One checks easily that $\widetilde{M}\mr{Ar}^{\mr{op},\circledast}$ is equivalent to
 $\bigl(\Phi^{\mr{co}}(\Gamma\times\Delta^1,\mr{Sch}^{\mr{op}})^{\mr{prop}}\bigr)^{\sharp}$ of \ref{dfnofphocoprop}.
 Since the functors in the commutative diagram (\ref{dfnofphocoprop-diag2}) all preserves coCartesian edges over $\mbf{\Delta}^{\mr{op}}$,
 this yields the following commutative diagram:
 \begin{equation*}
  \xymatrix{
   \widetilde{M}\mr{Ar}^{\mr{op},\circledast}\ar[r]^-{\alpha}\ar[d]&
   \overleftarrow{s}^*\mr{Corr}^{\circledast,\sharp}
   \ar[d]^{q}\\
  M\mr{Sch}^{\mr{op},\circledast}\ar[r]^-{\mc{D}^{\circledast}}&
   \mr{Corr}^{\circledast,\sharp}.
   }
 \end{equation*}

 Now, we have the functor $q'\colon\Gamma\rightarrow\overleftarrow{s}^*\Gamma$ whose fiber over $[n]$
 is the inclusion $\Delta^n\rightarrow(\Delta^n)^{\triangleleft}$ avoiding the cone point.
 We have the coCartesian fibration $\mr{Un}(\mc{D})\rightarrow\mr{Sch}^{\mr{op}}$,
 and let $\widetilde{M}\mc{D}^{\circledast}$ be the coCartesian fibration over $\mbf{\Delta}^{\mr{op}}$
 corresponding to the functor $\widetilde{M}_{\bullet}\mr{Un}(\mc{D})$.
 By construction, we have the following homotopy Cartesian diagram of $\infty$-categories
 \begin{equation}
  \label{cartDcomp}\tag{$\star$}
   \xymatrix@C=50pt{
   \mr{Un}(\mc{D})\ar[d]\ar@{}[rd]|\square&
   \mc{X}\ar[r]\ar[d]\ar[l]\ar@{}[rd]|\square&
   \overleftarrow{s}^*\mc{C}\ar[d]^{\overleftarrow{s}^*G}\\
  \mr{Sch}^{\mr{op}}&
   \widetilde{M}\mr{Ar}^{\mr{op},\circledast}\times_{\mbf{\Delta}^{\mr{op}}}\Gamma
   \ar[r]^-{\alpha\times q'}\ar[l]&
   \overleftarrow{s}^*\bigl(\mr{Corr}^{\circledast,\sharp}\times_{\mbf{\Delta}^{\mr{op}}}\Gamma\bigr).
   }
 \end{equation}
 On the other hand, we have the functor $\overleftarrow{s}^*\mr{Corr}^{\circledast,\sharp}\rightarrow\mr{Sch}^{\mr{op}}$ sending to the value at the cone point.
 The composition $\widetilde{M}\mr{Ar}^{\mr{op},\circledast}\rightarrow\mr{Sch}^{\mr{op}}$
 with $\alpha$ is the constant functor at the final object $*$ by construction.
 Thus, the homotopy Cartesian diagram above induces the following homotopy Carteisan diagram
 \begin{equation*}
  \xymatrix@C=30pt{
   \widetilde{M}\mr{Ar}^{\mr{op},\circledast}\times\mc{D}(*)\ar[rr]\ar[d]_{\mr{pr}_1}\ar@{}[rrd]|\square&
   &\overleftarrow{s}^*\mc{C}\ar[d]^{\overleftarrow{s}^*G}\\
  \widetilde{M}\mr{Ar}^{\mr{op},\circledast}\ar[r]^-{\alpha}&
   \overleftarrow{s}^*\mr{Corr}^{\circledast,\sharp}\times_{\mbf{\Delta}^{\mr{op}}}\mbf{\Delta}^{\mr{op}}
   \ar[r]^-{\mr{id}\times i}&
   \overleftarrow{s}^*\bigl(\mr{Corr}^{\circledast,\sharp}\times_{\mbf{\Delta}^{\mr{op}}}\Gamma\bigr),
   }
 \end{equation*}
 where the right lower horizontal map $i\colon\mbf{\Delta}\rightarrow\overleftarrow{s}^*\Gamma$
 is the inclusion sending $[n]\in\mbf{\Delta}$ to $([n]^{\triangleleft},-\infty)\in\overleftarrow{s}^*\Gamma$.
 Using the fixed object $I\in\mc{D}(*)$, we have the map
 \begin{equation*}
  \mr{cone}_I\colon
   \widetilde{M}\mr{Ar}^{\mr{op},\circledast}
   \xrightarrow{\mr{id}\times\{I\}}
   \widetilde{M}\mr{Ar}^{\mr{op},\circledast}\times\mc{D}(*)
   \rightarrow
   \overleftarrow{s}^*\mc{C}.
 \end{equation*}
 We have the following commutative diagram:
 \begin{equation*}
  \xymatrix@C=50pt{
   \widetilde{M}\mr{Ar}^{\mr{op},\circledast}\times_{\mbf{\Delta}^{\mr{op}}}\mbf{\Delta}^{\mr{op}}
   \ar[r]^-{\mr{cone}_I}\ar@{^{(}->}[d]_{\mr{id}\times i}&
   \overleftarrow{s}^*\mc{C}\ar[d]^{\overleftarrow{s}^*G}\\
  \widetilde{M}\mr{Ar}^{\mr{op},\circledast}\times_{\mbf{\Delta}^{\mr{op}}}\overleftarrow{s}^*\Gamma
   \ar[r]^-{\alpha\times\mr{id}}\ar@{-->}[ru]&
   \overleftarrow{s}^*\bigl(\mr{Corr}^{\circledast,\sharp}\times_{\mbf{\Delta}^{\mr{op}}}\Gamma\bigr).
   }
 \end{equation*}
 We take a $\overleftarrow{s}^*G$-left Kan extension as above.
 The existence is ensured by the fact that the fiber $G_X$ for each $X\in\mr{Corr}^{\circledast,\sharp}$
 is a coCartesian fibration combining with Lemma \ref{cartedgedetelem}.
 By composing this extension with $\mr{id}\times q'$, the Cartesian diagram (\ref{cartDcomp}) induces a functor
 $\widetilde{M}\mr{Ar}^{\mr{op},\circledast}\times_{\mbf{\Delta}^{\mr{op}}}\Gamma\rightarrow\mc{D}$.
 By taking the adjoint, we have the functor
 $\widetilde{M}\mr{Ar}^{\mr{op},\circledast}\rightarrow M\mc{D}^{\circledast}$.
 Unwinding the definition, this functor through $\widetilde{M}\mc{D}^{\circledast}$,
 and get a functor $\widetilde{M}\mr{Ar}^{\mr{op},\circledast}\rightarrow\widetilde{M}\mc{D}^{\circledast}$.
 Finally, applying $\mr{Int}$ of (the coCartesian version of) Proposition \ref{propinfcatext},
 we have a functor $\H_{\mr{c}}^*\colon\mr{Ar}^{\mr{op}}\rightarrow\mr{Un}(\mc{D})$
 of coCartesian fibrations over $\mr{Sch}^{\mr{op}}$ which preserves coCartesian edges.
\end{proof}

\section{Examples}
\label{secExam}
In this section, we exhibit some concrete examples of
$(\infty,2)$-functor
$\mbf{Corr}(\mr{Sch})\rightarrow\twoLinCat_{R}$ to apply the results of
previous sections.

\subsection{}
We fix a noetherian scheme $S$, and let $\mr{Sch}(S)$ be a full
subcategory of the category of noetherian $S$-schemes which is stable
under pullbacks.
We denote by $\Tri$ be the $(2,1)$-category of triangulated categories,
triangulated functors, and invertible triangulated natural
transforms. Similarly, we denote by $\Tri^{\otimes}$ the
$(2,1)$-category of triangulated symmetric monoidal categories,
triangulated symmetric monoidal functors, and invertible triangulated
symmetric monoidal natural transformations.

\begin{dfn*}
 A {\em category of coefficients} is a functor
 $D\colon\mr{Sch}(S)^{\mr{op}}\rightarrow\Tri^{\otimes}$.
 For a morphism $f\colon X\rightarrow Y$ in $\mr{Sch}(S)$, we denote
 $D(f)\colon D(Y)\rightarrow D(X)$ by $f^*$.
 The category of coefficients is said to be {\em premotivic} (cf.\
 \cite[1.4.2]{CD}) if the following conditions are satisfied:
 \begin{itemize}
  \item For any {\em smooth separated morphism of finite type} $f$ in
	$\mr{Sch}(S)$, the $1$-morphism $f^*$, considered as a morphism
	in $\Tri$, admits a left adjoint $f_{\sharp}$, and $f^*$ and
	$f_{\sharp}$ satisfies some base change property (cf.\
	\cite[1.1.10]{CD});
	
  \item For any morphism $f$, $f^*$, considered as a morphism in $\Tri$,
	admits a right adjoint, denoted by $f_*$ (cf.\
	\cite[1.1.12]{CD});

  \item For any smooth separated morphism of finite type
	$f\colon X\rightarrow Y$, the canonical
	morphism $f_{\sharp}((-)\otimes f^*(-))\rightarrow
	f_{\sharp}(-)\otimes(-)$ of functors
	$D(X)\times D(Y)\rightarrow D(Y)$ is an equivalence (cf.\
	\cite[1.1.27]{CD});
	
  \item For any $X\in\mr{Sch}(S)$, the category $D(X)$ is closed
	({\em i.e.}\ admits an internal hom).
 \end{itemize}
 The category $\mr{Sch}(S)$ is assumed to be {\em adequate} in the
 sense of \cite[2.0.1]{CD}:
 \begin{itemize}
  \item It is closed under finite sums and pullbacks along morphisms of
	finite type;
	
  \item Any quasi-projective $S$-scheme belongs to $\mr{Sch}(S)$;
	
  \item Any separated morphism of finite type\footnote{
	In \cite{CD}, they do not impose the morphism to be of finite
	type. We think this is a typo, otherwise,
	all the separated morphisms in $\mr{Sch}(S)$ need to be of
	finite type.}
	in $\mr{Sch}(S)$ admits a
	compactification (cf.\ \cite[2.0.1 (c)]{CD} for more precise
	statement);
	
  \item Chow's lemma holds (cf.\ \cite[2.0.1 (d)]{CD} for more precise
	statement).
 \end{itemize}
 
 The category of coefficients is said to be {\em motivic} if it is
 premotivic, and moreover, satisfies the
 following conditions (see \cite[2.4.45]{CD} for more details):
 \begin{itemize}
  \item For $p\colon\mb{A}^1_T\rightarrow T$, the counit map
	$p_{\sharp}p^*\rightarrow\mr{id}$ is an equivalence
	(homotopy property);
	
  \item For a smooth separated morphism of finite type $f\colon
	X\rightarrow T$ and a section $s\colon T\rightarrow X$, the
	functor $f_{\sharp}s_*$ induces a categorical equivalence
	(stability property);
	
  \item We have $D(\emptyset)=0$, $i^*i_*\rightarrow\mr{id}$ is an
	equivalence for any closed immersion $Z\hookrightarrow T$, and
	$(j^*,i^*)$ is conservative where $j$ is the open immersion
	$T\setminus Z\hookrightarrow T$ (localization property);
	
  \item For any proper morphism $f$, $f_*$ admits a right adjoint
	(adjoint property).
 \end{itemize}
\end{dfn*}

One of the main results of the theory is the following theorem, which
roughly says that the proper base change theorem holds for motivic
category of coefficients.

\begin{thm}[Voevodsky, Ayoub, Cisinski-D\'{e}glise {\cite[2.4.26,
 2.4.28]{CD}}]
 \label{mainthmmotcoef}
 Let $D\colon\mr{Sch}(S)^{\mr{op}}\rightarrow\Tri^{\otimes}$ be a
 motivic category of coefficients. Then the support property and the
 proper base change property holds: Consider a Cartesian diagram in
 $\mr{Sch}(S)$
 \begin{equation*}
  \xymatrix{
   X'\ar[r]^-{g'}\ar[d]_{f'}\ar@{}[rd]|\square&X\ar[d]^{f}\\
  Y'\ar[r]^-{g}&Y.
   }
 \end{equation*}
 \begin{enumerate}
  \item If $f$ is proper and $g$ is an open immersion, the canonical map
	$g_!f'_*\rightarrow f_*g'_!$ of functors $D(X')\rightarrow D(Y)$,
	constructed using the equivalence
	$g'_!f'^*\xrightarrow{\sim}f^*g_!$, is an equivalence
	{\normalfont(}namely, the support property holds{\normalfont)}.
	
  \item If $f$ is proper, then the adjunction map $g^*f_*\rightarrow f'_*g'^*$ of functors $D(X)\rightarrow D(Y')$ is an equivalence
	{\normalfont(}namely, the proper base change property holds{\normalfont)}.
 \end{enumerate}
\end{thm}

\subsection{}
\label{GRmainres}
The following theorem is essentially a consequence of Gaitsgory and
Rozenblyum's extension theorem as well as the theorem above.
The theorem roughly says as follows: Assume we are given a motivic
category of coefficients $D$, and assume we wish to upgrade this to an
$(\infty,2)$-functor from the category of correspondences to
$\twoLinCat$. Then all we need to construct is only an
$\infty$-enhancement of $D$, which is often easy to carry out.
The author learned the technique from the thesis of A. Khan \cite{K}.

\begin{thm*}
 Let $R$ be an $\mb{E}_\infty$-ring, and let
 $\mc{D}^*\colon\mr{Sch}(S)^{\mr{op}}\rightarrow\mr{CAlg}(\LinCat_{R})$
 be a functor. Assume that the composition
 \begin{equation*}
  \mr{Sch}(S)^{\mr{op}}\xrightarrow{\mc{D}^*}\mr{CAlg}(\LinCat_{R})
   \rightarrow
   \Tri^{\otimes}
 \end{equation*}
 is a motivic category of coefficients.
 Here, the second functor is defined by {\normalfont\cite[4.8.2.18]{HA}} and {\normalfont\cite[1.1.2.14]{HA}}.
 Then we have the following commutative diagram
 \begin{equation*}
  \xymatrix@C=50pt{
   \mr{Sch}(S)^{\mr{op}}
   \ar[r]^-{\mc{D}^*}\ar[d]&
   \LinCat_R\ar[d]\\
  \mbf{Corr}(S)^{\mr{prop}}_{\mr{sep};\mr{all}}
   \ar[r]^-{\mbf{D}^*_!}&
   \twoLinCat_R^{2\mbox{-}\mr{op}}.
   }
 \end{equation*}
 Here, $\mr{prop}$, $\mr{sep}$, $\mr{all}$ denote the classes of proper
 morphisms, separated morphisms, and all morphisms.
\end{thm*}

\begin{rem*}
 In fact, less amount of data is required to construct $\mbf{D}^*_!$ if we do not need $\infty$-enhancement of $\otimes$ and $\shom$.
 More precisely, in order to get $\mbf{D}^*_!$, it suffices to assume that we are given a functor $\mr{Sch}(S)^{\mr{op}}\rightarrow\LinCat_R$
 such that the induced functor $\mr{Sch}(S)^{\mr{op}}\rightarrow\Tri$ can be promoted to a motivic category of coefficients
 $\mr{Sch}(S)^{\mr{op}}\rightarrow\Tri^{\otimes}$.
\end{rem*}

\begin{proof}
 Consider the functor
 $\mr{Sch}(S)^{\mr{op}}\xrightarrow{\mc{D}^*}\LinCat_R\rightarrow
 \twoLinCat_R$. Let us show that this functor satisfies the right
 Beck-Chevalley condition (\cite[Ch.7, 3.1.5]{GR})
 with respect to open immersions,
 namely satisfies the following two conditions:
 \begin{itemize}
  \item For any open immersion $j\colon U\rightarrow X$ in
	$\mr{Sch}(S)$, the $1$-morphism $j^*\colon D(X)\rightarrow D(U)$
	in $\twoLinCat_R$ admits a left adjoint, denoted by $j_!$;
	
  \item For a Cartesian diagram in $\mr{Sch}(S)$
	\begin{equation}
	 \label{cartdiagschref}
	  \xymatrix{
	  X'\ar[d]_{f'}\ar[r]^-{g'}\ar@{}[rd]|\square&
	  X\ar[d]^{f}\\
	 Y'\ar[r]^-{g}&Y}
	\end{equation}
	such that $g$ is an open immersion, the canonical $2$-morphism
	of functors $g'_!\circ f'^*\rightarrow f^*\circ g_!$ is an
	equivalence.
 \end{itemize}
 For the existence of left adjoint, it suffices to check that the
 $1$-morphism $j^*$ considered as a functor between underlying
 $\infty$-category admits a left adjoint by Lemma \ref{adjcritlincat}.
 An exact functor $F$ between stable $\infty$-categories admits left or
 right adjoint if and only if so does the functor between its homotopy
 categories $\mr{h}F$ by \cite[3.3.1]{NRS}.
 Since an open immersion is separated smooth of finite type,
 this follows from the fact that the induced category of coefficients is
 premotivic. In order to show that the adjunction map is an equivalence,
 it suffices to show this for the associated homotopy category as well.
 Thus, the equivalence follows by the base change property of
 $f_{\sharp}$ in the axiom of premotivic category. Invoking
 \cite[Ch.7, 3.2.2 (b)]{GR}, we get a functor
 $\mbf{D}_1\colon\mbf{Corr}(S)^{\mr{open}}_{\mr{open};\mr{all}}
 \rightarrow\twoLinCat_R$, where $\mr{open}$ denotes the class of open
 immersions. Restricting $\mbf{D}_1$ to
 $\mbf{Corr}(S)^{\mr{iso}}_{\mr{open};\mr{all}}$, where $\mr{iso}$ is
 the class of isomorphisms, and take $(-)^{1\&2\mbox{-}\mr{op}}$ to both
 sides, we get the $2$-functor
 $(\mbf{Corr}(S)^{\mr{iso}}_{\mr{all};\mr{open}})^{2\mbox{-}\mr{op}}
 \rightarrow\twoLinCat_R^{1\&2\mbox{-}\mr{op}}$. Since the $2$-morphisms
 are equivalences in the category
 $\mbf{Corr}(S)^{\mr{iso}}_{\mr{all};\mr{open}}$, in other words it is
 an $(\infty,1)$-category, we have an equivalence
 $\mbf{Corr}(S)^{\mr{iso}}_{\mr{all};\mr{open}}\simeq
 (\mbf{Corr}(S)^{\mr{iso}}_{\mr{all};\mr{open}})^{2\mbox{-}\mr{op}}$ by
 inverting the $2$-morphisms. Thus, we obtain
 $\mbf{D}_2\colon\mbf{Corr}(S)^{\mr{iso}}_{\mr{all};\mr{open}}
 \rightarrow\twoLinCat_R^{1\&2\mbox{-}\mr{op}}$.

 Now, we wish to show that the composition
 $\mr{Sch}(S)\rightarrow\mbf{Corr}^{\mr{iso}}_{\mr{all};\mr{open}}
 \xrightarrow{\mbf{D}_2}\twoLinCat_R^{1\&2\mbox{-}\mr{op}}$ satisfies
 the left Beck-Chevalley condition (\cite[Ch.7, 3.1.2]{GR})
 with respect to ``$\mr{prop}$''.
 An adjoint pair of $1$-morphisms $(f,g)$ in a $2$-category
 $\mbf{C}$ is equivalent to giving an adjoint pair of $1$-maps $(f',g')$
 in the $2$-category $\mbf{C}^{1\&2\mbox{-}\mr{op}}$, where $f'$, $g'$
 are corresponding $1$-maps in $\mbf{C}^{1\&2\mbox{-}\mr{op}}$
 to $f$, $g$.
 Thus verifying the left Beck-Chevalley condition
 amounts to checking the following two conditions:
 \begin{itemize}
  \item For any proper morphism $f\colon X\rightarrow Y$, the
	$1$-morphism $f^*\colon D(Y)\rightarrow D(X)$ in $\twoLinCat_R$
	admits a right adjoint, denoted by $f_*$;

  \item For a Cartesian diagram (\ref{cartdiagschref}) in $\mr{Sch}(S)$
	such that $f$ is proper and $g$ is any morphism, the adjunction
	map $g^*f_*\rightarrow f'_*g'^*$ is an equivalence;
 \end{itemize}
 In order to check that the $1$-morphism $f^*$ admits a right adjoint in
 $\twoLinCat_R$, we need to show that the underlying functor, denoted by
 $(f^*)^{\circ}$, between $\infty$-category (without linear
 $\infty$-category structure)
 admits a right adjoint which commutes with small colimits by Lemma
 \ref{adjcritlincat}. Since $(f^*)^{\circ}$ is a morphism in $\PrL$ by
 definition of $\LinCat$, $(f^*)^{\circ}$ admits a right adjoint
 $f^{\circ}_*$.
 We need to check that this functor commutes with small
 colimits. Since $f_*^{\circ}$ is an exact functor by
 \cite[1.1.4.1]{HA}, it suffices to check that it commutes with small
 coproducts by \cite[4.4.2.7]{HTT}.
 This commutation is equivalent to the commutation of small
 coproducts of the functor $\mr{h}(f_*^{\circ})$ between associated homotopy categories.
 By \cite[5.2.2.9]{HTT},
 $\mr{h}(f_*^{\circ})$ is right adjoint to $\mr{h}((f^*)^{\circ})$,
 and $\mr{h}(f_*^{\circ})$ admits a right adjoint ${}^{\mr{h}}f^!$,
 because $f$ is proper, by the adjointness axiom of motivic category of
 coefficients.
 Thus the claim follows. The second condition can be checked in the
 homotopy category, which is nothing but Theorem \ref{mainthmmotcoef}.
 In addition to the left Beck-Chevalley condition, the condition
 \cite[Ch.7, 5.2.2]{GR} holds since the support property holds by
 Theorem \ref{mainthmmotcoef}.
 This enables us to invoke \cite[Ch.7, 5.2.4]{GR} for
 $\mathit{adm}=\mr{prop}$, $\mathit{co}\mbox{-}\mathit{adm}=\mr{open}$,
 and get a functor
 $\mbf{D}_3\colon\mbf{Corr}(S)^{\mr{prop}}_{\mr{all};\mr{sep}}\rightarrow\twoLinCat_R^{1\&2\mbox{-}\mr{op}}$.
 Finally, we take $(-)^{1\mbox{-}\mr{op}}$ to get $\mbf{D}^*_!$.
\end{proof}

\subsection*{Motivic theory of modules}
\subsection{}
\label{moduleconsmoti}
Assume we are in the situation of Theorem \ref{GRmainres}.
Let $R'$ be an $\mb{E}_\infty$-algebra over $R$.
Then we have the scalar extension functor
$\LinCat_{R}\rightarrow\LinCat_{R'}$ (cf.\ \cite[D.2.4]{HA}).
Thus, we have
$\mc{D}_{R'}\colon\mr{Sch}(S)\xrightarrow{\mc{D}}
\LinCat_{R}\rightarrow\LinCat_{R'}$.
Now, recall the notations of \ref{prmodfunct}.
Consider the following diagram:
\begin{equation*}
 \xymatrix{
  &\Prcat^{\mr{CAlg}}_{\mc{L}}\ar[d]^{\phi^{\mr{CAlg}}}
  \ar[r]^-{\mr{pr}_1\circ\Xi}&
  \mr{CAlg}(\LinCat_{R'}),\\
 \mr{Sch}(S)^{\mr{op}}
  \ar[r]^-{\mc{D}_{R'}}\ar@/^15pt/@{.>}[ur]^-{A}&
  \mr{CAlg}(\LinCat_{R'})
  }
\end{equation*}
Assume we are given a dotted arrow in the diagram so that the diagram
commutes. Then by composing with $\mr{pr}_1\circ\Xi$, we get a new
functor $\mr{Mod}_A(\mc{D}_{R'})\colon
\mr{Sch}(S)^{\mr{op}}\rightarrow\LinCat_{R'}$.
Assume that for any morphism $f$ in $\mr{Sch}(S)^{\mr{op}}$, the edge
$A(f)$ is $\phi^{\mr{CAlg}}$-coCartesian. In this case, by
\cite[7.2.13, 7.2.18]{CD}, the underlying theory of coefficients of
$\mr{Mod}_A(\mc{D}_{R'})$ is in fact a motivic theory of coefficients.
Indeed, the underlying category is compatible with \cite{CD} by
\cite[4.3.3.17]{HA}. Thus, we can apply Theorem \ref{GRmainres}.

Finally, the construction of $A$ is essentially the same as choosing a
commutative algebra object $A(S)$ of $\mc{D}_{R'}(S)$.
Because we assume that $A(f)$ is a coCartesian edge for any morphism $f$
in $\mr{Sch}(S)^{\mr{op}}$, the following is a $\phi^{\mr{CAlg}}$-left
Kan extension diagram:
\begin{equation*}
   \xymatrix{
  \{S\}\ar[d]\ar[r]^-{A(S)}&
  \Prcat^{\mr{CAlg}}_{\mc{L}}\ar[d]^{\phi^{\mr{CAlg}}}\\
 \mr{Sch}(S)^{\mr{op}}\ar[r]_-{\mc{D}_{R'}}\ar[ur]^-{A}&
  \mr{CAlg}(\LinCat_{R'}).
  }
\end{equation*}
Thus by \cite[4.3.2.15, 4.3.2.16]{HTT}, we have the claim.
Summing up, if we fix $A_S\in\mr{CAlg}(\mc{D}_{R'}(S))$, there exists an
$\infty$-enhancement of the motivic theory associating $f\colon
X\rightarrow S$ to $\mr{Mod}_{f^*A_S}(\mc{D}_{R'}(X))$.

\subsection*{\'{E}tale cohomology theory}

\subsection{}
\label{monoidalgrothcons}
Let $\mc{O}^{\otimes}$ be a symmetric $\infty$-operad.
Let $\mr{Mon}^{\mr{pres}}_{\mc{O}}(\Cat_\infty)$ be the subcategory of
$\mr{Mon}_{\mc{O}}(\Cat_\infty)$ (cf.\ \cite[2.4.2.1]{HA})
spanned by $\mc{O}$-monoidal $\infty$-categories which are compatible
with small colimits and each fiber over $X\in\mc{O}$ is presentable, and
those $\mc{O}$-monoidal functors which preserve small colimits.

\begin{lem*}
 Let $\mc{D}$ be a coCartesian symmetric monoidal $\infty$-category.
 \begin{equation*}
  \mr{Mon}^{\mr{pres}}_{\mc{D}}(\Cat_{\infty})
   \simeq
   \mr{Fun}(\mc{D},\mr{CAlg}(\PrL)).
 \end{equation*}
\end{lem*}
\begin{proof}
 The proof is similar to \cite[3.3.4.11]{GL}.
 Let $\mc{K}$ be the set of small simplicial sets. Then the inclusion
 $\PrL\rightarrow\Cat_\infty(\mc{K})$ is fully faithful.
 By \cite[4.8.1.9]{HA} and \cite[5.5.3.5]{HTT}, we have
 $\mr{Mon}^{\mr{pres}}_{\mc{D}}(\Cat_{\infty})
 \simeq\mr{Alg}_{\mc{D}}(\PrL)$.
 We invoke \cite[2.4.3.18]{HA} to conclude.
\end{proof}

\subsection{}
Let $S$ be a noetherian scheme, and $\mr{Sch}(S)$ be the category of noetherian $S$-schemes.
Let $\mr{\acute{E}t}$ be the full subcategory of $\mr{Fun}(\Delta^1,\mr{Sch}(S))$ spanned by \'{e}tale morphisms $X\rightarrow Y$ over $S$.
We have the functor $\mr{\acute{E}t}\rightarrow\mr{Sch}(S)$ sending $X\rightarrow Y$ to $Y$,
which is a Cartesian fibration. By straightening, we have the functor $\mr{\acute{E}t}\colon\mr{Sch}(S)^{\mr{op}}\rightarrow\Cat_\infty$
sending $T\in\mr{Sch}(S)$ to $\mr{\acute{E}t}(T)$, the category of \'{e}tale schemes over $T$.
We fix a ring $\Lambda$, and consider $\mr{Mod}^{\otimes}_{\Lambda}\in\mr{CAlg}(\PrL_{\mr{St}})$.
The construction \ref{algmonoidfuncto} induces a functor
\begin{equation*}
 \PSh_{\Lambda,*}\colon
  \mr{Sch}(S)\xrightarrow{\mr{\acute{E}t}}
  \Cat_\infty^{\mr{op}}
  \xrightarrow{\mr{op}}
  \Cat_\infty^{\mr{op}}
  \xrightarrow{\mr{Fun}(-,\mr{Mod}^{\otimes}_\Lambda)}
  \mr{CAlg}(\PrL_{\mr{st}})
\end{equation*}
sending $T\in\mr{Sch}(S)$ to $\PSh_\Lambda(T):=\mr{Fun}(\mr{\acute{E}t}(T)^{\mr{op}},\mr{Mod}^{\otimes}_{\Lambda})$,
the $\infty$-category of $\mr{Mod}_{\Lambda}$-valued presheaves with pointwise symmetric monoidal structure.
If we are given a morphism $f\colon X\rightarrow Y$ of $S$-schemes, the associated functor $\PSh_{\Lambda,*}(f)$ is given by the pushforward $f_*$.
By the (Cartesian version of) ``symmetric monoidal Grothendieck construction'' \cite[3.3.4.11]{GL}, the functor yields a symmetric monoidal {\em Cartesian} fibration
$\mc{P}_*\rightarrow\mr{Sch}(S)^{\mr{op}}$ in the sense of \cite[3.3.4.6]{GL}.
This is in fact a symmetric monoidal coCartesian fibration as well.
Indeed, since $f_*$ admits a left adjoint, it is a coCartesian fibration of $\infty$-categories.
Since the tensor product in $\mr{Mod}_{\Lambda}$ commutes with small colimits in each variable, $f^*$ is compatible with the tensor product as required.
By the (coCartesian version of) symmetric monoidal Grothendieck construction,
we get a functor $\PSh_{\Lambda}^*\colon\mr{Sch}(S)^{\mr{op}}\rightarrow\mr{CAlg}(\PrL_{\mr{st}})$.

Now, let $\mc{F}$ be a presheaf in $\PSh_\Lambda(T)$.
We say that $\mc{F}$ is a {\em sheaf} if for any \'{e}tale hypercovering $U_\bullet\rightarrow V$ where $V\in\mr{\acute{E}t}(T)$,
the induced map
\begin{equation*}
 \mc{F}(V)\rightarrow\invlim\mc{F}(U_\bullet)
\end{equation*}
is an equivalence. We define $\Shv_\Lambda(T)$ to be the full
subcategory of $\PSh_\Lambda(T)$ spanned by sheaves.
By \cite[1.3.4.3]{SAG}, the fully faithful inclusion
$\Shv_\Lambda(T)\hookrightarrow\PSh_\Lambda(T)$ admits a left
adjoint.
By \cite[2.1.2.2]{SAG}, we have an equivalence
$\mr{h}\Shv_\Lambda(T)\simeq D(T_{\mr{\acute{e}t}},\Lambda)$, where
the last category is the (ordinary) derived category.
By Lemma \ref{monoidalgrothcons}, the functor $\PSh^*_\Lambda$ gives rise to a coCartesian fibration of symmetric $\infty$-operads
$\PSh_\Lambda^{\otimes}\rightarrow\mr{Sch}(S)^{\mr{op},\times}$ with compatibility conditions.
In view of \cite[1.3.4.4]{SAG} (or \cite[3.2.2.6]{GL}),
we may invoke \cite[2.2.1.9]{HA} to get a coCartesian fibration
$\Shv_\Lambda^{\otimes}\rightarrow
\mr{Sch}(S)^{\mr{op},\times}$ which is the fiberwise localization.
Since presentability is preserved by localizations,
$\Shv_\Lambda^{\otimes}$ yields the functor
\begin{equation*}
 \Shv_\Lambda\colon\mr{Sch}(S)^{\mr{op}}\rightarrow\mr{CAlg}(\PrL_{\mr{st}}).
\end{equation*}
By construction, this is an $\infty$-enhancement of the functor
$\mr{Sch}(S)^{\mr{op}}\rightarrow\Tri^{\otimes}$ sending $T$ to
$D(T_{\mr{\acute{e}t}},\Lambda)$ with pullback functors. When $\Lambda$
is torsion and there exists an integer $n$ invertible in $S$ such that
$n\Lambda=0$, this functor forms a motivic category of
coefficients.
This is a consequence of marvelous works in SGA, but we need slightly to be careful since we are dealing with {\em unbounded} derived categories.
To check that it is premotivic, non-trivial points are to check the existence of $f_{\sharp}$ and the projection formula.
The existence of $f_{\sharp}$ follows from [SGA 4, Exp.\ XVIII, Thm 3.2.5].
For the projection formula, since all the functors involved commute with colimits,
we are reduced to checking the formula for compact objects, in which case it is well-known.
To show that it is motivic, the stability property follows from \cite[2.4.19]{CD}.
The adjoint property has been checked in \cite[1.2.3]{CDet}.
The other properties are standard.

\begin{rem*}
 We have treated the torsion cohomology theory, but we can further upgrade this to $\ell$-adic cohomology theory.
 To do this, there are at least 2 methods.
 One is to use the adic formalism as in the classical theory.
 The formalism is much complicated than the torsion theory, but crucial ideas can be found in \cite{GL}.
 The other method, suggested by the referee, is to use \cite[7.2.21]{CDet}.
 In view of this result, it suffices to construct an $\infty$-enhancement of 6-functor formalism of
 $\mr{DM}_{\mr{h}}$ using the notation of \cite{CDet}.
 This can be done in the same way as the $\infty$-enhancement of the 6-functor formalism of
 the motivic $\mb{A}^1$-homotopy theory that we will treat in the next paragraph.
 The detail is left to the reader.
\end{rem*}

\subsection*{Stable motivic $\mb{A}^1$-homotopy theory}

\subsection{}
Let $S$ be a noetherian scheme of finite Krull dimension. We put
$\mr{Sch}(S)$ to be the category of noetherian $S$-schemes of finite
Krull dimension. In Robalo's thesis \cite[9.3.1]{R}, the functor
\begin{equation*}
 \SH^{\otimes}\colon
  \mr{Sch}(S)^{\mr{op}}\rightarrow
  \mr{CAlg}(\PrL_{\mr{St}})\simeq
  \mr{CAlg}(\LinCat_{\mb{S}})
\end{equation*}
is constructed, where $\mb{S}$ denotes the sphere spectrum and the last
equivalence is from \cite[4.8.2.18]{HA}. Let us recall his construction
for the sake of completeness.
The first half of the construction is parallel to that of \'{e}tale
theory except that we use Nisnevich topology rather than \'{e}tale
topology and take $\Lambda:=\mb{S}$. Then we get a sheaf
\begin{equation*}
 \Shv^{\mr{Nis}}_{\mb{S}}\colon
  \mr{Sch}(S)^{\mr{op}}\rightarrow
  \mr{CAlg}(\PrL_{\mr{st}}).
\end{equation*}
We need two more operations to acquire $\SH$: localize $\mb{A}^1$
and invert $\mb{P}^1$. For each $T\in\mr{Sch}(S)$, let $S_T$ be the
collection of morphisms $\mbf{1}_T\rightarrow p_*p^*\mbf{1}_T$ of
$\Shv_{\mb{S}}(T)$ where $\mbf{1}_T$ is a unit object, and
$p\colon\mb{A}^1_T\rightarrow T$. We localize $\Shv_{\mb{S}}(T)$ by
$S_T$ (cf.\ \cite[5.5.4.15]{HTT}). Invoking Lemma
\ref{monoidalgrothcons} similarly to the construction of $\Shv$ out of
$\PSh$, we obtain a functor
$\Shv^{\mr{Nis},\mb{A}^1}_{\mb{S}}\colon
\mr{Sch}(S)^{\mr{op}}\rightarrow\mr{CAlg}(\PrL_{\mr{st}})$.
Finally, we need to invert $\mb{P}^1$.
In fact, this is the crucial part of Robalo's article \cite{R2}.
He constructed a map (cf.\ \cite[2.6]{R2})
\begin{equation*}
 \mr{Loc}\colon
 \mc{P}(\mathit{free}^{\otimes}(\Delta^0))^{\otimes}
  \rightarrow
  \mc{P}(\mc{L}^{\otimes}_{(\mathit{free}^{\otimes}(\Delta^0),*)}
  (\mathit{free}^{\otimes}(\Delta^0)))^{\otimes}
\end{equation*}
in $\mr{CAlg}(\PrL)$. Giving an object of $\mr{CAlg}(\PrL)
_{\mc{P}(\mathit{free}^{\otimes}(\Delta^0))^{\otimes}/}$ is equivalent
to giving a presentable symmetric monoidal category $\mc{C}^{\otimes}$
and an object $X\in\mc{C}$. Assume given
$X\in\Shv^{\mr{Nis},\mb{A}^1}_{\mb{S}}(S)$. The corresponding object of
$\mr{CAlg}(\PrL)_{\mc{P}(\mathit{free}^{\otimes}(\Delta^0))^{\otimes}/}$
is denoted by $X'$. Consider the following diagram
\begin{equation*}
 \xymatrix@C=50pt{
  \{S\}\ar[r]^-{X'}\ar[d]&\mr{CAlg}(\PrL)
  _{\mc{P}(\mathit{free}^{\otimes}(\Delta^0))^{\otimes}/}
  \ar[d]^{p}\\
 \mr{Sch}(S)^{\mr{op}}
  \ar[r]_-{\Shv^{\mr{Nis},\mb{A}^1}_{\mb{S}}}
  \ar@{-->}[ur]^{X''}&
  \mr{CAlg}(\PrL).
  }
\end{equation*}
By \cite[2.1.2.2]{HTT}, $p$ is a left fibration, and since $S$ is a
final object of $\mr{Sch}(S)$, we may take the $p$-left Kan extension.
For $\mc{C}^{\otimes}\in\mr{CAlg}(\PrL)$, let
$\mr{Cons}(\mc{C}^{\otimes})$ be the constant functor
$\mr{Sch}(S)^{\mr{op}}\rightarrow\mr{CAlg}(\PrL)$ at
$\{\mc{C}^{\otimes}\}$. The Kan extension determines a diagram in
$\mr{Fun}(\mr{Sch}(S)^{\mr{op}},\mr{CAlg}(\PrL))$
\begin{equation*}
 \Shv^{\mr{Nis},\mb{A}^1}_{\mb{S}}
  \xleftarrow{X''}
 \mr{Cons}(\mc{P}(\mathit{free}^{\otimes}(\Delta^0))^{\otimes})
  \xrightarrow{\mr{Cons}(\mr{Loc})}
 \mr{Cons}(\mc{P}(\mc{L}^{\otimes}
  _{(\mathit{free}^{\otimes}(\Delta^0),*)}
  (\mathit{free}^{\otimes}(\Delta^0)))^{\otimes}).
\end{equation*}
The pushout of this diagram is denoted by
$\Shv^{\mr{Nis},\mb{A}^1}_{\mb{S}}[X^{-1}]$.
Finally, let $p\colon\mb{P}^1_S\rightarrow S$. Let
$p^*:=\Shv^{\mr{Nis},\mb{A}^1}_{\mb{S}}(p)$, and let $p_*$ be a right
adjoint. Then we define
\begin{equation*}
 \SH^{\otimes}:=
  \Shv^{\mr{Nis},\mb{A}^1}_{\mb{S}}[(p_*p^*\mbf{1}_S)^{-1}].
\end{equation*}
By \cite[2.23]{R2}, this functor, in fact, lands in
$\mr{CAlg}(\PrL_{\mr{st}})$. By \cite[5.1.2.3]{HTT}, the pushout can be
computed object-wise. Namely, we have an equivalence
\begin{equation*}
 \SH^{\otimes}(X)\simeq
  \Shv^{\mr{Nis},\mb{A}^1}_{\mb{S}}(X)
  \coprod_{\mc{P}(\mathit{free}^{\otimes}(\Delta^0))^{\otimes}}
  \mc{P}(\mc{L}^{\otimes}
  _{(\mathit{free}^{\otimes}(\Delta^0),*)}
  (\mathit{free}^{\otimes}(\Delta^0)))^{\otimes}.
\end{equation*}
Thus by \cite[2.4.4, 2.37]{R2}, this coincides with the classical stable
$\mb{A}^1$-homotopy category of \cite[1.4.3]{CD}.
The underlying triangulated category forms a motivic
category of coefficients by \cite[2.4.48]{CD}, and we
may apply the previous theorem to get a 6-functor formalism.

\begin{ex}
 For the record, we summarize what we have constructed.
 Let $k$ be a perfect field, and take $S=\mr{Spec}(k)$.
 Let $R$ be a (discrete) ring.
 Recall that Voevodsky introduced the Eilenberg-MacLane spectrum
 $\mathit{HR}_{S}$ in $\mr{CAlg}(\SH(S))$, which is a
 $\mb{Z}$-module (cf.\ \cite[2.12]{CD2} for more detail).
 A principal application of the construction in \ref{moduleconsmoti} is
 when we take $A(S)$ to be $\mathit{HR}_{S}\otimes_{\mb{Z}}R$.
 This spectrum yields a motivic theory of coefficients as in
 \cite[4.3]{CD2}. Our construction above gives an $\infty$-enhancement
 of this theory.
 In particular, by Theorem \ref{mainthmcons} for
 $I=\mbf{1}_{\mr{Spec}(k)}$, $J=\mbf{1}_{\mr{Spec}(k)}(d)$, where $d$ is
 an integer and $(d)$ denotes the Tate twist, we have a functor
 \begin{equation*}
  \mr{H}(d)\colon
   \widetilde{\mr{Ar}}^{\mr{prop}}_{\mr{sep}}(\mr{Sch}(k))^{\mr{op}}
   \rightarrow
   \mr{Mod}_{R}\simeq\mc{D}(R),
 \end{equation*}
 such that $p\colon X\rightarrow Y$ over $k$ is sent to
 $\mr{Mor}_{\mr{Mod}_{\mathit{HR}_Y}}(p_!p^*\mbf{1}_Y,\mbf{1}_Y(d))$ in
 $\mc{D}(R)$. For example $\mr{H}(\mr{id}\colon X\rightarrow X)$
 coincides with the motivic cohomology $\mr{H}_{\mc{M}}^*(X,R(d))$, at
 least when $X$ is smooth (cf.\ \cite[11.2.3, 11.2.c]{CD}), and
 $\mr{H}(X\rightarrow\mr{Spec}(k))$ is nothing but the motivic
 Borel-Moore theory. The functor $\mr{H}$ unifies these two theories,
 and gives an $\infty$-enhancement.
\end{ex}

\subsection*{Arithmetic $\mc{D}$-module theory}
\subsection{}
Let $k$ be a perfect field of characteristic $p>0$,
and let $K$ be a complete discrete field of mixed characteristic with residue field an algebraic extension of $k$.
In \cite{A}, there exists a 6-functor formalism for schemes separated of finite type over $k$,
with $K$-linear coefficient category.
It is natural to expect that this formalism can be enhanced to an $(\infty,2)$-functor as in Theorem \ref{GRmainres}.
We might need more work than \'{e}tale or motivic case above because the construction of the pullback functor for
arithmetic $\mc{D}$-modules already requires various choices (parameterized by contractible spaces),
and it is not as straightforward as in those cases.

Tomoyuki Abe:\\
Kavli Institute for the Physics and Mathematics of the Universe
(WPI), University of Tokyo\\
5-1-5 Kashiwanoha, Kashiwa, Chiba, 277-8583, Japan\\
{\tt tomoyuki.abe@ipmu.jp}

\end{document}